\documentclass[11pt,reqno]{amsart}

\usepackage{amsmath, amsthm, amssymb}
\usepackage{amsfonts}

\usepackage{hyperref}

\usepackage[ansinew]{inputenc}
\usepackage[dvips]{epsfig}
\usepackage{graphicx}
\usepackage[english]{babel}

\usepackage{thmtools}
\theoremstyle{plain}
\declaretheorem[title=Theorem, parent=section]{theorem}
\declaretheorem[title=Lemma,sibling=theorem]{lemma}
\declaretheorem[title=Proposition,sibling=theorem]{proposition}
\declaretheorem[title=Corollary,sibling=theorem]{corollary}

\theoremstyle{definition}
\declaretheorem[title=Definition,sibling=theorem]{definition}
\declaretheorem[title=Remark,sibling=theorem]{remark}
\declaretheorem[title=Remark, numbered=no]{remark*}
\declaretheorem[title=Example, sibling=theorem]{example}
\declaretheorem[title=Assumption, numbered=no]{assumption*}

\numberwithin{equation}{section}

\usepackage[backgroundcolor=white, bordercolor=blue,
linecolor=blue]{todonotes}

\usepackage{dsfont}
\usepackage{bbm}

\newcommand{\N}{\mathbb{N}}
\newcommand{\R}{\mathbb{R}}
\newcommand{\C}{\mathbb{C}}

\newcommand{\cP}{\mathcal{P}}

\newcommand{\cN}{\mathcal{N}}

\newcommand{\cE}{\mathcal{E}}

\newcommand{\cM}{\mathcal{M}}

\newcommand{\cF}{\mathcal{F}}
\newcommand{\cX}{\mathcal{X}}

\newcommand{\eps}{\varepsilon}

\newcommand{\loc}{\mathrm{loc}}
\newcommand{\1}{\mathbbm{1}}

\DeclareMathOperator{\dist}{dist}

\DeclareMathOperator{\supp}{supp}
\DeclareMathOperator{\re}{Re}	
\DeclareMathOperator{\im}{Im}

\DeclareMathOperator{\tail}{Tail}

\renewcommand{\d}{\textnormal{\,d}}

\newcommand{\average}{{\mathchoice {\kern1ex\vcenter{\hrule height.4pt
width 6pt depth0pt} \kern-9.7pt} {\kern1ex\vcenter{\hrule
height.4pt width 4.3pt depth0pt} \kern-7pt} {} {} }}
\newcommand{\dashint}{\average\int}

\begin{document}

\allowdisplaybreaks

\title[The Neumann problem for the fractional Laplacian]{The Neumann problem for the fractional Laplacian: \\ optimal regularity via the Mellin transform}

\author{Serena Dipierro}
\author{Xavier Ros-Oton}
\author{Enrico Valdinoci}
\author{Marvin Weidner}

\address{Department of Mathematics and Statistics,
University of Western Australia, 35 Stirling Highway, WA6009, Crawley, Australia}
\email{serena.dipierro@uwa.edu.au}

\address{ICREA, Pg. Llu\'is Companys 23, 08010 Barcelona, Spain \& Universitat de Barcelona, Departament de Matem\`atiques i Inform\`atica, Gran Via de les Corts Catalanes 585, 08007 Barcelona, Spain \& Centre de Recerca Matem\`atica, Barcelona, Spain}
\email{xros@icrea.cat}
\urladdr{www.ub.edu/pde/xros}

\address{Department of Mathematics and Statistics,
University of Western Australia, 35 Stirling Highway, WA6009, Crawley, Australia}
\email{enrico.valdinoci@uwa.edu.au}

\address{Institute for Applied Mathematics, University of Bonn, Endenicher Allee 60, 53115, Bonn, Germany}
\email{mweidner@uni-bonn.de}
\urladdr{https://sites.google.com/view/marvinweidner/}

\keywords{fractional Laplacian; Neumann problem; regularity; Mellin transform}

\subjclass[2020]{47G20, 35B65}

\allowdisplaybreaks

\begin{abstract}
We establish the optimal regularity of solutions to the Neumann problem for the fractional Laplacian, $(-\Delta)^s u=h$ in $\Omega$, with the external condition $\mathcal N^s u=0$ in $\Omega^c$. 
For this, a key point is to establish a 1D Liouville theorem for functions with growth, which we prove by using complex analysis and the Mellin transform. 

More precisely, we prove a ``meta-theorem'' relating the classification of 1D solutions to general linear homogeneous equations of the type $Lu=0$ in $(0,\infty)$ to the (complex) roots of an explicit meromorphic function $f(z)$ that depends on $L$.
In case of the fractional Laplacian with Neumann conditions, we show that all solutions are $C^{2s+\alpha}$ when $s\leq 1/2$, and $C^{s+\frac12+\alpha}$ when $s\geq1/2$. 

Moreover, quite surprisingly, we prove that even in 1D there exist highly oscillating solutions of the type $u(x)=x^{a} \cos(b \log x)$ for $x>0$, with $a>0$ and $b>0$ that depend on~$s$, and $a<2s$ for $s\sim1$.
\end{abstract}

\allowdisplaybreaks

\maketitle

\section{Introduction}

Nonlocal equations on domains have attracted great interest in the PDE community in the last two decades, especially since the first works of Caffarelli and Silvestre on this topic; see 
\cite{CS0,CS1,CS2,CS3,BKK08,CKS,BCI2,DGLZ,RS-Cs,Grubb2,Grubb,BKK15,RS-Duke,AG,KD,AbRo20,DRSV,DRV,KiWe24,Gru22,FeRo24} and the references therein.

Most of these works concern nonlocal Dirichlet problems, in which we have an equation like $(-\Delta)^s u=f$ in $\Omega\subset \R^n$ together with an exterior condition $u=0$ in $\Omega^c$.
On the one hand, thanks to a number of important developments in the last years, the regularity for this kind of problems is very well understood; see in particular \cite{RS-Cs,Grubb2,Grubb,RS-Duke,AbRo20,DRSV,RoWe24}.

On the other hand, much less is known about nonlocal Neumann problems of the type
\begin{equation}  \label{NP}
\left\{
\begin{array}{rcll}
(-\Delta)^s u &=& h & \quad \textrm{in}\quad \Omega, \\
\mathcal N^s u &=& 0 & \quad \textrm{in}\quad \Omega^c,
\end{array}
\right.
\end{equation}
where
\[\mathcal N^s u(x) = c_{n,s}\int_\Omega \frac{u(x)-u(y)}{|x-y|^{n+2s}}\,dy,\qquad x\in \Omega^c.\]
This problem arises when minimizing the energy functional
\[\mathcal E(u) = \frac{c_{n,s}}{4} \int \int_{(\Omega^c\times\Omega^c)^c} \frac{\big|u(x)-u(y)\big|^2}{|x-y|^{n+2s}}\,dx\,dy - \int_\Omega h u, \]
for some $h \in L^2(\Omega)$ over all functions with finite energy and it also has a very natural probabilistic interpretation, as well as a corresponding integration by parts formula; see \cite{DRV,DGLZ} for more details.

The behavior and smoothness of solutions to \eqref{NP} near the boundary has remained open.
The only known regularity result for this problem was obtained in \cite{AFR23}, and states that solutions are H\"older continuous (namely, $C^{\max\{0,2s-1\}+\alpha}$ for some $\alpha>0$) up to the boundary.

Unfortunately, this is still very far from the optimal regularity of solutions, and does not answer the following key questions:

\begin{itemize}
\item[(i)] \em Are solutions classical, i.e. $C^{2s+\alpha}$, up to the boundary?  \vspace{2mm}
\item[(ii)] Are solutions $C^1$ up to the boundary and satisfy \,$\partial_\nu u=0$ on $\partial\Omega$?   \vspace{2mm}
\item[(iii)] What is the optimal regularity of solutions?
\end{itemize}
The goal of this paper is to answer all of these questions.
As we will see, a key difficulty is to understand solutions of \eqref{NP} in dimension $n=1$. As explained in detail below, we classify all 1D solutions by using complex analysis and the Mellin transform.

\subsection{Main result}

Our first main result is the following, which provides a higher order regularity estimate for solutions to \eqref{NP}.

\begin{theorem}\label{thm0}
Let $\Omega\subset \R^n$ be an open set with $C^{1,1}$ boundary and $u$ be a solution of \eqref{NP}.

Then, there exists $\alpha \in (0,\frac{1}{2})$, depending only on $s$, such that, if  $h \in C^\alpha(\overline\Omega)$
\[\begin{array}{llll}
\displaystyle
u \in C^{2s+\alpha}(\overline\Omega) \quad & \textrm{with}\quad &
\|u\|_{C^{2s+\alpha}(\overline\Omega)} \leq C \big( \Vert u \Vert_{L^{\infty}(\R^n)} + \|h\|_{C^{\alpha}(\overline\Omega)} \big) \quad  &  \textrm{if} \quad s\leq {\textstyle\frac12}, \vspace{2mm}\\
\displaystyle
u \in C^{s+\frac12+\alpha}(\overline\Omega) \quad & \textrm{with}\quad  &
\|u\|_{C^{s+\frac12+\alpha}(\overline\Omega)} \leq C \big( \Vert u \Vert_{L^{\infty}(\R^n)}  + \|h\|_{C^{\min\{\alpha+\frac12-s,0\}}(\overline\Omega)} \big) \quad &   \textrm{if} \quad s\geq {\textstyle\frac12}.
\end{array}\]
Moreover, when $2s+\alpha>1$ (in particular when $s \ge \frac{1}{2}$), we have $\partial_{\nu} u=0$ on $\partial\Omega$.
The constant $C$ depends only on $n$, $s$, and $\Omega$.
\end{theorem}

Note that we do not assume $\Omega$ to be connected or to be bounded.
The existence of weak solutions to \eqref{NP} was established in \cite{DRV} under the compatibility condition $\int_{\Omega} h = 0$. Moreover, if $\Omega$ is bounded, then solutions are unique up to constants. We refer \autoref{def:weak-sol-Neumann} for the weak solution concept.

As we will see, this regularity result is optimal in the sense that solutions are in general \emph{not} $C^{2s+\frac12}$ for $s \le \frac{1}{2}$, and they are \emph{not} $C^{s+1}$ for $s \ge \frac{1}{2}$. Moreover, they are not $C^{s+\frac12+\alpha}$ for $s$ close to $0$, and \emph{not} $C^{2s+\alpha}$ for $s$ close to $1$ (for any $\alpha\geq0$). We refer to \autoref{thm:main-intro} for a more detailed version of this result, where the optimal regularity exponent is identified precisely.

Notice that our result entails, in particular, that solutions of \eqref{NP} are always \emph{more} regular than those of the Dirichlet problem for the fractional Laplacian (see~\cite{RS-Cs,RS-Duke}), since the latter ones are in general not better than $C^s$.

In order to understand the optimal regularity of solutions, a useful rule of thumb is that one expects solutions to be smooth in tangential directions, while the optimal regularity is determined by the corresponding 1D problem, say for $\Omega=(0,\infty)$.
Moreover, after a contradiction and compactness argument, one expects the behavior of solutions near $x=0$ to be dictated by the nontrivial solutions to the 1D problem with $h\equiv0$. This leads to the following question:
\begin{itemize}
\item \em Can one characterize all solutions to  
\begin{equation}\label{poiu}
\begin{array}{rcll}
(-\Delta)^s u &=& 0 & \quad \textrm{in}\quad (0,\infty), \\
\mathcal N^s u &=& 0 & \quad \textrm{in}\quad (-\infty,0),
\end{array}
\end{equation}
satisfying the growth condition $|u(x)|\leq C(1+|x|^\gamma)$\,?
\end{itemize}

This turns out to be a delicate issue, and resolving this difficulty will be, in our approach, one of the key steps in this context.

\begin{remark}
Notice that this type of questions arises as well in the study of nonlocal Dirichlet problems, i.e. with $u=0$ in $(-\infty,0)$.
However, in that case one can use the extension property of $(-\Delta)^s$ and transform the question into a \emph{local} PDE problem with a local boundary condition; see \cite[Theorem~1.10.16]{FeRo24}.
This seems impossible here, since the boundary condition is purely nonlocal. Without such an extension we cannot use local PDE techniques as in \cite{RS-Duke,FeRo24}.
\end{remark}

\subsection{1D Liouville theorems and the Mellin transform}

As said above, none of the known methods to classify global solutions to elliptic equations works for \eqref{poiu}.

Since the condition $\mathcal N^s u=0$ determines the values of $u$ for $x<0$ in terms of the values of $u$ for $x>0$, equation~\eqref{poiu} can be reformulated in terms of a single nonlocal operator $L$ in the half-line $(0,\infty)$, i.e.
\[Lu(x) := \text{p.v.} \int_{0}^{\infty} (u(x) - u(y)) K_{\R_+}(x,y) \d y =0\quad \textrm{in}\quad (0,\infty),\] where
\begin{align*}
K_{\R_+}(x,y) = \frac{c_{1,s}}{|x-y|^{1+2s}} + 2s c_{1,s} \int_{0}^{\infty} \frac{|z|^{2s}}{|x+z|^{1+2s} |y+z|^{1+2s}} \d z
\end{align*}
Moreover, since both $(-\Delta)^s$ and $\mathcal N^s$ are homogeneous operators of order $2s$, then
so is $L$, in the sense 
\begin{align*}
[L u(\lambda \cdot)](x) = \lambda^{2s}  [Lu] (\lambda x) ~~ \forall x >0.
\end{align*}
In particular, for any $\beta<2s$ we have 
\begin{equation}\label{vhfu}
L(x^\beta)= f(\beta)x^{\beta-2s},
\end{equation}
where $f(\beta)=L(x^\beta)|_{x=1}\in \R$.

We can actually compute explicitly
\begin{equation}\label{f}
f(\beta) = \frac{\Gamma(\beta + 1) \sin(\pi s)}{\Gamma(\beta - 2s + 1)\sin(\pi(\beta -2s))}  \left(  \frac{\sin(\pi(\beta-s))}{\sin(\pi s)} + \frac{\Gamma(2s-\beta) \Gamma(\beta+1)}{\Gamma(2s)} \right).
\end{equation}

Rephrasing the problem stated in~\eqref{poiu},
the natural question that arises here is then:
\begin{itemize}
\item \em If $L$ is a linear and homogeneous operator in $(0,\infty)$, it is true that all solutions to $Lu=0$ in $(0,\infty)$ (with appropriate growth at infinity) must necessarily be homogeneous?
\end{itemize}

If we could answer this general question positively, the classification of solutions to \eqref{poiu} would reduce to finding the zeros of $f(\beta)$.

Our first classification result, which addresses this question, can be stated as follows:

\begin{theorem} \label{thm-intro5}
Let $s\in(0,1)$, $\eps > 0$, and $L$ be the operator associated to \eqref{poiu}, given by \eqref{hghg}-\eqref{hghg2}.

Then, any distributional solution of $Lu=0$ in $(0,\infty)$ satisfying $|u(x)|\leq C(1+|x|^{2s-\varepsilon})$ must be a linear combination of $x^{\beta_k}$, where $\beta_k\in \mathbb C$ are the \textbf{complex} zeros of $f(\beta)$ in the strip $\{0<\re(\beta)<2s\}$.
\end{theorem}

The definition of distributional solutions is somewhat delicate and depends on the test function space; see \autoref{sec2} for more details.

Note that operator $L$ has no special properties that we can use, beyond linearity and homogeneity. In fact, the question whether all solutions to $Lu = 0$ in $(0,\infty)$ must be homogeneous is a very natural and fundamental one, arising in connection with many different classes of differential operators $L$. To emphasize this point, we actually prove a general version of \autoref{thm-intro5} that applies to \emph{any} linear and homogeneous operator $L$ (see \autoref{thm:Liouville} and the discussion in \autoref{sec2}) and essentially reduces the classification of solutions to finding the \emph{complex} zeros of the corresponding function $f(\beta)$.

\vspace{2mm}

Our proof of \autoref{thm-intro5} heavily relies on the Mellin transform, which is a multiplicative version of the Fourier transform\footnote{Actually, we have $\cM [w(x)](z)= \mathcal F [w(e^{-x})](-iz)$ for all $z\in \C$.} and is often used in number theory. It is given as
\begin{align*}
\cM w(z):=\int_0^\infty x^{z-1} w(x) \d x.
\end{align*}

Notice that $\cM$ is not well defined for functions with polynomial growth at infinity, in the same way as the Fourier transform is not defined for functions with exponential growth.
Still if $u$ is a solution to $Lu = 0$, then by the self-adjointness of $L$ and by \eqref{eq:f}, at least formally, we should have
\begin{align}
\label{eq:Mellin-proof-key-intro}
\begin{split}
0 &= \cM[L u](z) = \int_0^{\infty} x^{z-1} Lu(x) \d x \\
&= \int_0^{\infty} L(x^{z-1}) u(x) \d x = f(z-1) \int_0^{\infty} x^{z-1-2s} u(x) \d x = f(z-1) \cM[u](z-2s). 
\end{split}
\end{align}
This implies that $\cM[u](z-2s)$ must be a sum of delta functions supported at the zeros of $f(z-1)$.

Since, at least formally, $\cM[w](z) = \delta(z - \alpha)$ if and only if $w(x)=x^{-\alpha}$, and since\footnote{This follows since by the same argument as in \eqref{eq:Mellin-proof-key-intro}, it holds $\cM[L x^\beta](z) = f(z-1)\cM[x^\beta](z)$, and thus, by \eqref{vhfu}, we have $f(\beta)\delta(z+\beta-2s) = f(\beta) \cM[x^{\beta-2s}](z) = f(z-1)\cM[x^\beta](z) = f(z-1) \delta(z-2s+\beta)$.} $f(\beta) = f(2s - 1 - \beta)$, this yields the desired result. 

Of course this argument is far from being correct, and one of the contributions of the present work is to develop the right setting in order to define the Mellin transform in a distributional way, for functions that grow polynomially at infinity.
This requires careful choices of function spaces and the use of some fine results from complex analysis such as the Paley-Wiener theorem.

\vspace{2mm}

\autoref{thm-intro5} is stated for functions with growth controlled by $|x|^{2s-\varepsilon}$, so that equation~\eqref{poiu} makes sense.
Still, in our proof of \autoref{thm0} we actually need a similar result for functions with growth of the type $|x|^{2s+\alpha}$, and this will require much more work.
Indeed, we first need to introduce a generalized notion of solutionA notable fact about this result is that it allows for \emph{oscillatory} solutions of the type 
\[u(x)=x^{a+ib} = x^a \cos(b \log x).\]
As we will see, the function $f$  in \eqref{f} does have complex zeros, and so equation~\eqref{poiu} possesses oscillatory solutions of the above type. In particular, {\em not all solutions are homogeneous}
(hence, the answer to the question above is in the negative).

This is quite remarkable, since the equation $(-\Delta)^s u=0$ satisfies a maximum principle and a Harnack inequality. Moreover, this property stands in stark contrast to the nonlocal Dirichlet problem. to  \eqref{poiu}, and then establish an alternative version of the classification result \autoref{thm-intro5} for the combined action of two homogeneous operators. We refer to \autoref{thm:Neumann-Liouville-growth} for the corresponding classification result of solutions to \eqref{poiu} and to \autoref{thm:Liouville-2D} for a general classification result.

\subsection{The zeros of $f$}

The use of the Mellin transform and \autoref{thm-intro5} above allow us to reduce the problem of classifying 1D solutions to the problem of understanding the complex zeros of the function $f$ in \eqref{f}.
Unfortunately, this turns out to be very delicate too.

In case $s=\frac12$ the equation $f(\beta)=0$ is equivalent to 
\[\sin(2\pi\beta)= 2\pi\beta.\]
This is a transcendental equation, with infinitely many zeros with $\re(\beta)>0$, the smallest one being 
\[\beta= 1.193292...\pm i\,0.4406488...\]

Using fine properties of Gamma functions, we are able to prove the following:

\begin{proposition}\label{prop:zeros-final}
Let $f$ be given by \eqref{f}.
Then, 
\begin{itemize}
\item[(i)] If $s\leq \frac12$, then~$f$ has no zeros in the strip $\{0<\re(\beta)<2s+\alpha\}$, for some small $\alpha>0$.

Moreover, it has at least one zero in the strip $\{2s<\re(\beta)<2s+\frac12\}$.

\vspace{2mm}

\item[(ii)] If $s\geq \frac12$, then~$f$ has no zeros in the strip $\{2s-1<\re(\beta)<s+\frac12+\alpha\}$, for some small $\alpha>0$.

Moreover, it has at least one zero in the strip $\{s+\frac12<\re(\beta)<s+1\}$.
\end{itemize}
\end{proposition}

\begin{remark}\label{rk:zeros-final}
In other words, for any $s\in(0,1)$, we may denote by $B_0$ the real part of the zero $\beta$ of $f$ with the smallest real part satisfying $\re(\beta)>\max\{0,2s-1\}$.
Then, \autoref{prop:zeros-final} yields that~$B_0>\min\{2s,s+\frac12\}$. 

This turns out to be optimal, as we prove that $B_0<s+\frac12$ for $s\sim 0$ and $B_0<2s$ for $s\sim 1$ (see \autoref{rem-s-asymp}).

By the classification result in \autoref{thm-intro5}, the value of $B_0$ determines the optimal regularity of solutions to \eqref{NP}. In fact, solutions are in general not better than $C^{B_0}$.
\end{remark}

\begin{remark}
Note that $f$ has a (trivial) zero at $\beta = 2s-1$, whenever $s \ge \frac{1}{2}$. This reflects the fact that $(-\Delta)^s |x|^{2s-1} = 0$ in $\R \setminus \{ 0 \}$ is the fundamental solution. Since $|x|^{2s-1}$ is only a distributional solution, but not a weak solution, it does not determine the boundary regularity of solutions to \eqref{NP}.
\end{remark}

\subsection{Higher regularity}

All the results described so far deal with the classification of 1D solutions for \eqref{poiu}. As was mentioned before, by passing from 1D to general open set $\Omega \subset \R^n$, we can establish higher regularity estimates for solutions to \eqref{NP} to obtain \autoref{thm0}. Actually, in this paper we also show the \emph{optimal regularity} of solutions in general solution domains up to the optimal exponent $B_0$.

\begin{theorem}
\label{thm:main-intro}
Let $\Omega\subset \R^n$ be an open set with $C^{1,1}$ boundary and $u$ be a solution of \eqref{NP}.

Then, for any $\eps \in (0,B_0-s)$ such that $B_0 - \eps \not \in \{1,2s\}$, and any $h \in \cX$, it holds
\begin{align*}
u \in C^{B_0 - \eps}(\overline\Omega) \qquad \textrm{with}\qquad \Vert u \Vert_{C^{B_0 - \eps}(\overline{\Omega})} \le C \big( \Vert u \Vert_{L^{\infty}(\R^n)} + \Vert h \Vert_{\cX} \big),
\end{align*}
where the constant $C$ depends only on $n,s$, and $\Omega$, the value $B_0$ is as in \autoref{rk:zeros-final}, and
\begin{align*}
\cX = \begin{cases}
C^{B_0 -2s -\eps}(\overline{\Omega}) & ~~ \text{ if } B_0 - \eps > 2s,\\
d^{2s-B_0 + \eps}_{\Omega} L^{\infty}(\Omega) & ~~ \text{ if } B_0 - \eps < 2s.
\end{cases}
\end{align*}

Moreover, if $B_0 - \eps > 1$, then $\partial_{\nu} u = 0$ on $\partial \Omega$.

Besides, if $\Omega$ is bounded, then we have the estimate
\begin{align*}
\left\Vert u - \dashint_\Omega u \right\Vert_{C^{B_0 - \eps}(\overline{\Omega})} \le C  \Vert h \Vert_{\cX}.
\end{align*}
\end{theorem}

In this paper we also prove a localized version of \autoref{thm:main-intro} where we merely require the equation to be satisfied in $\Omega \cap B_1$ and the Neumann condition to hold in $B_1 \setminus \Omega$. This result is contained in \autoref{thm:main-Neumann} and \autoref{thm:main-Neumann-higher}.

This result yields the optimal boundary regularity of solutions to the nonlocal Neumann problem~\eqref{NP} for all $s \in (0,1)$ and provides a Schauder-type theory for \eqref{NP}. Together with \autoref{prop:zeros-final} it immediately implies \autoref{thm0} and altogether these results yield a complete understanding of the key questions (i)-(ii)-(iii).

In fact, the answers to these key questions are:

\begin{itemize}
\item[(i)] \em On the one hand, if~$s\le\frac12$, all solutions are classical, i.e. $C^{2s+\alpha}$, up to the boundary.
On the other hand, if~$s>\frac12$, all solutions are~$C^{s+\frac12+\alpha}$ up to the boundary.
  \vspace{2mm}
\item[(ii)] If~$s\ge\frac12$,
all solutions are~$C^1$ up to the boundary and satisfy \,$\partial_\nu u=0$ on $\partial\Omega$.   \vspace{2mm}
\item[(iii)] The optimal regularity exponent is given by $B_0$ from \autoref{rk:zeros-final}. Moreover, the above answers are optimal, in the sense that solutions are not necessarily~$C^{s+\frac12+\alpha}$ when~$s$ is close to~$0$, and
not necessarily~$C^{2s+\alpha}$ when~$s$ is close to~$1$.
\end{itemize}

As was mentioned before, the key ingredient to prove \autoref{thm:main-intro} is the classification of solutions in 1D. However, we still need to pass from 1D to general open sets $\Omega\subset\R^n$. The general strategy is to prove a pointwise boundary estimate by contradiction and compactness, relying ultimately on the 1D Liouville theorem (see \autoref{thm:bdry-expansion} and \autoref{thm:bdry-expansion-higher}). Variants of this approach have previously been used for nonlocal equations in several works (see \cite{AbRo20,AFR23,DRSV,RS-Duke,RoWe24}). In our setting, however, its implementation is quite delicate. In particular, the following two important difficulties arise. 
\begin{itemize}
\item On the one hand, when the order of the estimate is greater than one, we need to subtract a linear term from the equation. In case $\Omega = \{ x_n > 0\}$ is a half-space, this term is needed to account for solutions of the form $b \cdot x$ with $b_n = 0$ that only depend on the tangential directions $x'$. In general solution domains $\Omega \subset \R^n$, we need to understand very precisely what is exactly the right function to subtract, and prove that its fractional Laplacian $(-\Delta)^s$ is well behaved near the boundary (see \autoref{subsec:correction} for more details.).

\item On the other hand, when the order of the estimate is greater than $2s$ we need to develop a theory of generalized solutions to \eqref{NP} in unbounded open sets and with polynomial growth at infinity. Thereby, we adapt some of the arguments from \cite{AbRo20,DSV2} to our setting. However, it turns out that the contradiction and compactness argument is much more delicate in this case, since it seems no longer possible to reformulate \eqref{NP} in terms of a single operator $L$ in $(0,\infty)$. Still, one can develop such a theory, separately, for the two operators $(-\Delta)^s$ and $\cN^s$, and deduce regularity from a 1D classification result for the combined action of these two operators on functions that grow like $|x|^{2s+\alpha}$ (see \autoref{thm:Neumann-Liouville-growth}).
\end{itemize}

\subsection{Outline}

This article is structured as follows. In \autoref{sec2}, we present and discuss our Liouville theorems that hold for general homogeneous operators in $\R$ (see \autoref{thm:Liouville} and \autoref{thm:Liouville-2D}). \autoref{sec3} provides several preliminary facts about the distributional Mellin transform, while \autoref{sec4} is dedicated to the proofs of these results.  In \autoref{sec5}, we apply our new 1D Liouville theorems to the nonlocal Neumann problem and establish \autoref{thm-intro5}, as well as \autoref{prop:zeros-final}. Finally, in \autoref{sec6}, we prove the higher regularity for the nonlocal Neumann problem and show our main results \autoref{thm0} and \autoref{thm:main-intro}.

\subsection{Acknowledgments}
XR and MW were supported by the European Research Council under the Grant Agreement No. 101123223 (SSNSD), and by AEI project PID2024-156429NB-I00 (Spain).
XR was also supported by the AEI--DFG project PCI2024-155066-2 (Spain--Germany), the AEI grant RED2024-153842-T (Spain), and by the Spanish State Research Agency through the Mar\'ia de Maeztu Program for Centers and Units of Excellence in R{\&}D (CEX2020-001084-M). MW was also supported by the Deutsche Forschungsgemeinschaft (DFG, German Research Foundation) under Germany's Excellence Strategy - EXC-2047/1 - 390685813 and through the CRC 1720 ``Analysis of criticality: from complex phenomena
to models and estimates'', 53930965.
SD was supported
the Australian Future Fellowship FT230100333.
EV was supported by the Australian Laureate
Fellowship FL190100081.

\section{Liouville theorems for homogeneous operators in 1D via the Mellin transform}  
\label{sec2}

The goal of this section is to present two classification results of solutions to  homogeneous linear equations in 1D that hold true in a very general setting, namely \autoref{thm:Liouville} and \autoref{thm:Liouville-2D}. These ''meta-theorems`` relate the classification of 1D solutions to the complex roots of an explicit meromorphic function $f$. 

As explained before, these results are crucial for our proof of the optimal regularity of the nonlocal Neumann problem~\eqref{NP}, since they yield \autoref{thm-intro5} (and also \autoref{thm:Neumann-Liouville} and \autoref{thm:Neumann-Liouville-growth}). Due to their wide applicability, we believe such classification results to be of general interest, and introduce them here in a rather general framework.

\subsection{A Liouville theorem for operators acting on $(0,\infty)$}

Let $L$ be a linear operator acting on functions $u : (0,\infty) \to \C$ 
that satisfies, for some $s > 0$ and a function $f : \C \to \C$,
\begin{align}
\label{eq:f}
L (x^{\beta})(x) = f(\beta) x^{\beta - 2s} ~~ \forall \beta \in \C, ~~ \forall x \in (0,\infty).
\end{align}

Given $d \in \N \cup \{0\}$ and $a,b \in \R$ we introduce the space
\begin{align*}
C^d_{a,b}(0,\infty):= \Big\{ \varphi \in C^d_{loc}(0,\infty) : \Vert \varphi \Vert_{C^d_{a,b}(0,\infty)} < \infty \Big\},
\end{align*}
where we denote
\begin{align*}
\Vert \varphi \Vert_{C^d_{a,b}(0,\infty)} = \inf \Big\{ C > 0 : ~\forall k \in \{0,\dots,d\} :~& |\varphi^{(k)}(x)| \le C x^{a}, ~\forall x \in (0,1], \\
& |\varphi^{(k)}(x)| \le C x^{-b}, ~\forall x \in (1,\infty]  \Big\}.
\end{align*}
We observe that, when~$d_1 \le d_2$, $a_1 \le a_2$, and $b_1 \le b_2$,
we have that $C^{d_2}_{a_2,b_2}(0,\infty) \subset C^{d_1}_{a_1,b_1}(0,\infty)$.

We assume that there exist $d,a,b$ such that, for any $\varphi \in C^d_{a,b}(0,\infty)$, the map $x \mapsto L \varphi(x)$ is continuous and, for some $\alpha_1 \in [0,2s)$ and $\alpha_2 \in [0,2s-\alpha_1)$,
\begin{align}
\label{eq:L-decay}
|L \varphi (x)| \le C \begin{cases}
x^{-\alpha_1}, ~~ \text{ for } x \in (0,1),\\
x^{-2s+\alpha_2}, ~~ \text{ for } x \in [1,\infty),\\
\end{cases}, 
~~ \int_{r}^{\infty} |L \varphi(x)| \d x \le C r^{-2s+\eps} ~~ \forall \eps \in (0,1), ~~ \forall r > 0,
\end{align}

Moreover, we assume that $L$ is self-adjoint in the sense that
\begin{align}
\label{eq:self-adjoint}
\int_0^{\infty} x^{\beta-1} L\varphi(x) \d x = \int_0^{\infty} L x^{\beta-1} \varphi(x) \d x,
\end{align}
for any $\varphi \in C^d_{a,b}(0,\infty)$, where $\re(\beta) \in (\alpha_1,2s-\alpha_2)$.

\begin{remark}
The first assumption in \eqref{eq:L-decay} might look a bit complicated, but it is very natural. In fact, for $L$ being the regional fractional Laplacian, i.e.
\begin{align*}
Lu(x) = c_{s} \int_0^{\infty} (u(x) - u(y)) |x-y|^{-1-2s} \d y,
\end{align*}
it holds true with $\alpha_1 = \alpha_2 = 0$, when $\varphi \in C^2_{0,1+2s}(0,\infty)$. However, when $L$ is the regional operator associated to the nonlocal Neumann problem~\eqref{hghg}-\eqref{hghg2} we need to take $\alpha_1 \in (\max \{ 0 , 2s-1 \} , 2s)$ and $\alpha_2 \in [0,\max\{2s,1\})$ (see \autoref{lemma:Neumann-L}).\\
The self-adjointness in \eqref{eq:self-adjoint} is only satisfied
under additional assumptions on the growth of $\varphi$ at zero in order to rule out the appearance of boundary terms. For both the regional fractional Laplacian and the Neumann operator, it is enough to take $\varphi \in C^2_{2,1+2s}(0,\infty)$.
\end{remark}

\begin{definition}
\label{def:distr-sol}
We say that $u : (0,\infty) \to \C$ with
\begin{align*}
|u(x)| \le C (1+ x)^{2s-\eps} ~~ \forall x \in (0,\infty)
\end{align*}
for some $\eps > 0$ is a (distributional) solution to
\begin{align*}
Lu = 0 ~~ \text{ in } (0,\infty)
\end{align*}
if for every $\varphi \in C^d_{a,b}(0,\infty)$ it holds
\begin{align}
\label{eq:distributional-sol}
\int_0^{\infty} u(x) L \varphi(x) \d x = 0.
\end{align}
\end{definition}

\begin{remark}
\label{remark:distr-well-def}
The integral in \eqref{eq:distributional-sol} converges absolutely by the same arguments as in the proof of \cite[Lemma 2.2.11]{FeRo24}. Indeed, the convergence of the integral at zero is immediate from the growth of $u$ and boundedness of $|L\varphi(x)|$. For the
convergence at infinity we use the second property in \eqref{eq:L-decay} with some $\eps' < \eps$:
\begin{align*}
\int_1^{\infty} |u(x)| |L\varphi(x)| \d x \le C \sum_{k = 0}^{\infty}  2^{k(2s-\eps)}  \int_{2^k}^{2^{k+1}} |L\varphi(x)| \d x  \le C \sum_{k = 0}^{\infty}  2^{k(\eps'-\eps)} \le C. 
\end{align*}
By an analogous reasoning, we see that both integrals in \eqref{eq:self-adjoint} converge absolutely.
\end{remark}

\begin{definition}
Let~$M > 1$.
We say that a meromorphic function $f : \C \to \C$ is $M$-admissible if it satisfies the following properties:
\begin{itemize}
\item There are only finitely many zeros of $f$ in the strip $S_M := \{ z \in \C : |\re(z)| < M \}$.
\item For any sequence $(z_n) \subset S_{M}$ with $|\im (z_n)| \to \infty$, it holds that
$$\liminf_{n \to \infty} |f(z_n)| > 0.$$
\end{itemize}
\end{definition}

Our goal is to prove the following Liouville-type theorem

\begin{theorem}
\label{thm:Liouville}
Let $L$ be as before, satisfying \eqref{eq:f}, \eqref{eq:L-decay}, and \eqref{eq:self-adjoint} for some $d,a,b$, and assume that $f$ is $M$-admissible for some $M > \max\{ 3 + 2s, 2 + 4s\}$. Let $u$ be a distributional solution to
\begin{align*}
L u = 0 ~~ \text{ in } (0,\infty)
\end{align*}
with test-function space $C^d_{a,b}(0,\infty)$ and such that for some $\eps \in (0,2s)$ and $C > 0$,
\begin{align}
\label{eq:u-growth-ass}
|u(x)| \le C (1 + x)^{2s-\eps} ~~ \forall x \in (0,\infty).
\end{align}
Then, there are coefficients $a_{\beta,l} \in \C$ such that
\begin{align}
\label{eq:Liouville-zero}
u(x) = \sum_{\beta \in \{ f = 0 \} \cap S_{2s-\eps} } \sum_{l = 0}^{k(\beta) - 1} a_{\beta,l} x^{\beta} (\log x)^l,
\end{align}
where $k(\beta)$ denotes the multiplicity of $\beta$.
\end{theorem}

\begin{remark}
\label{remark:hom}
In applications, we will consider operators $L$ acting on functions $u : (0,\infty) \to \R$. We will assume that $L$ is homogeneous of order $2s$ for some $s \in (0,1)$ in the following sense:\\
For any function $u$ and $\lambda > 0$, it holds
\begin{align}
\label{eq:hom}
[L u(\lambda \cdot)](x) = \lambda^{2s}  [Lu] (\lambda x) ~~ \forall x >0.
\end{align}
Then, we can extend $L$ to operate on functions $u : (0,\infty) \to \C$ by writing $u(x) = u_1(x) + i u_2(x)$ and setting $Lu(x) = L u_1(x) + i L u_2(x)$.
In this setting, \eqref{eq:hom} remains valid for $u : (0,\infty) \to \C$, and we have, for any $\beta \in \C$,
\begin{align*}
L(x^{\beta})(x) = x^{\beta-2s} L(x^{\beta})(1),
\end{align*}
i.e. \eqref{eq:f} holds true with $f(\beta) = L(x^{\beta})(1)$.
\end{remark}

\begin{remark}
Note that we did not assume $s \in (0,1)$. In applications to operators of fractional Laplacian type with $s \in (0,1)$, one can treat functions with growth of order $2s+k-\eps$ by differentiating the equation $k$ times (which leads to an operator of order $2s' = 2s + k$). An alternative approach is discussed in \autoref{subsec:2D-Liouville-intro}
\end{remark}

\begin{example} [Dirichlet problem for the fractional Laplacian]
Our \autoref{thm:Liouville} entails, as a particular case, the well-known Liouville theorem for
\begin{align*}
\begin{cases}
(-\Delta)^s u = 0 &~~ \text{ in } (0,\infty),\\
u = 0 &~~ \text{ in } (-\infty,0).
\end{cases}
\end{align*}
Note that the proof is completely different from previous ones since it does not rely on the Caffarelli-Silvestre extension. To prove it, we observe that we can equivalently rewrite, for $x \in (0,\infty)$,
\begin{align*}
(-\Delta)^s u (x) &= c_{s} \int_0^{\infty} (u(x) - u(y)) |x-y|^{-1-2s} \d y + u(x) c_{s} \int_{-\infty}^0 |x-y|^{-1-2s} \d y \\
&= (-\Delta)^s_{\R_+} u (x) + u(x) x^{-2s} \frac{c_{s}}{2s} =: Lu(x),
\end{align*}
where $(-\Delta)^s_{\R_+}$ denotes the regional fractional Laplacian in $(0,\infty)$ and $c_s = 4^s s \frac{\Gamma(s + \frac{1}{2})}{\Gamma(1 -s)} \pi^{-\frac{1}{2}}$. Note that
\begin{align*}
f(\beta)& = L(x^{\beta})(1) = (-\Delta)^s_{\R_+} [x^{\beta}](1) + \frac{c_{n,s}}{2s} \\
&= (-\Delta)^s (x_+)^{\beta}(1) - c_{s} \int_{-\infty}^0 |1-y|^{-1-2s} \d y + \frac{c_{s}}{2s} \\
&= (-\Delta)^s (x_+)^{\beta}(1)\\
&= \frac{\Gamma(\beta + 1)}{\Gamma(\beta - 2s + 1)} \frac{\sin(\pi(\beta-s))}{\sin(\pi(\beta -2s))},
\end{align*}
and therefore
\begin{align*}
f(\beta) = 0, ~~ \beta \in (-2s,2s) ~~ \Leftrightarrow ~~ \beta \in \{s-1,s \}.
\end{align*}
Since $u$ is bounded at zero, it follows from \autoref{thm:Liouville} that
\begin{align*}
u(x) = a (x_+)^s.
\end{align*}
\end{example}

In a similar way, it is possible to reprove the Liouville theorem for the regional fractional Laplacian from \cite{FaRo22} by using \autoref{thm:Liouville}.

\subsection{A Liouville theorem for two operators acting on $\R$}
\label{subsec:2D-Liouville-intro}

In comparison to the previous subsection, where we introduced an abstract Liouville theorem for a single homogeneous operator $L$ acting on functions $u : (0,\infty) \to \C$, the goal of the current subsection is to present an extension of \autoref{thm:Liouville} to pairs of operators acting on functions $u : \R \to \C$. This theorem will be useful in our application to the nonlocal Neumann problem, whenever we want to classify solutions that grow like $x^{2s+\eps}$ as $x \to +\infty$, since in this case it is no longer possible to find a suitable regional operator $L$ that is well-defined for such functions.

Let $L,N$ be linear operators acting on functions $u : \R \to \C$ that satisfy for some $s > 0$ and functions $f_{L,+}, f_{L,-}, f_{N,+}, f_{N,-}$ for any $\beta \in \C$:
\begin{align}
\label{eq:f-2D}
\begin{split}
L(x_+^{\beta})(x) &= f_{L,+}(\beta) x_+^{\beta-2s}, \qquad L(x_-^{\beta})(x) = f_{L,-}(\beta) x_+^{\beta-2s} \qquad \forall x \in (0,\infty),\\
N(x_+^{\beta})(x) &= f_{N,+}(\beta) x_-^{\beta-2s}, \qquad N(x_-^{\beta})(x) = f_{N,-}(\beta) x_-^{\beta-2s} \qquad \forall x \in (-\infty,0).
\end{split}
\end{align}
Moreover, given $d \in \N \cup \{ 0 \}$ and $a,b \in \R$, we define the space
\begin{align*}
C_{a,b}^d(\R) := \left\{ \varphi \in C^d_{\loc}(\R) : \varphi |_{(0,\infty)} , \varphi(-\cdot) |_{(0,\infty)} \in C_{a,b}^d(0,\infty) \right\}, \end{align*}
endowed with the norm
\begin{align*}
 \Vert \varphi \Vert_{C_{a,b}^d(\R)} := \Vert \varphi \Vert_{C_{a,b}^d(0,\infty)} + \Vert \varphi(-\cdot) \Vert_{C_{a,b}^d(0,\infty)}.
\end{align*}

Additionally, given $k \in \N$, we define
\begin{align*}
C_{a,b}^{d,k}(\R) := \left\{ \varphi \in C^d_{a,b}(\R) : \int_0^{\infty} x^j \varphi(x) \d x = \int_{0}^{\infty} x^j \varphi(-x) \d x = 0 ~~ \forall j \in \{0,\dots,k-1\} \right\},
\end{align*}
and we also set $C_{a,b}^{d,0}(\R): = C^d_{a,b}(\R)$.

We assume that there exist $d,a,b$, and $k \in \N \cup \{ 0 \}$ such that, for any $\varphi \in C^{d,k}_{a,b}(\R)$, the maps $x \mapsto L \varphi(x)$ and $x \mapsto N \varphi(x)$ are continuous and,
for some $\alpha_1 \in [0,2s)$ and $\alpha_2 \in [0,2s-\alpha_1)$,
\begin{align}
\label{eq:L-decay-2D}
\begin{split}
|N \varphi (-x)| + |L \varphi (x)| &\le C \begin{cases}
|x|^{-\alpha_1}, ~~~\quad\quad \text{ for } x \in (0,1),\\
|x|^{-2s-k+\alpha_2}, ~~ \text{ for } x \in [1,\infty),\\
\end{cases}\\
\int_{-\infty}^{-r} |N \varphi(x)| \d x + \int_{r}^{\infty} |L \varphi(x)| \d x &\le C r^{-2s-k+\eps} ~~ \forall \eps \in (0,1), ~~ \forall r > 0.
\end{split}
\end{align}

Moreover, we assume that the pair $(L,N)$ is self-adjoint in the sense that
\begin{align}
\label{eq:self-adjoint-2D}
\begin{split}
\int_0^{\infty} g(x) L\varphi(x) \d x = \int_0^{\infty} L (g\1_{\R_+})(x) \varphi(x) \d x + \int_{-\infty}^0 N (g \1_{\R_+})(x) \varphi(x) \d x,\\
\int_{-\infty}^0 g(x) N\varphi(x) \d x  = \int_0^{\infty} L (g\1_{\R_-})(x) \varphi(x) \d x + \int_{-\infty}^0 N (g \1_{\R_-})(x) \varphi(x) \d x
\end{split}
\end{align}
for any $\varphi \in C^{d,k}_{a,b}(\R)$ and $g(x) = |x|^{\beta-1}$, where $\re(\beta) \in (\alpha_1,2s + k -\alpha_2)$.

\begin{definition}
\label{def:distr-sol-2D}
We say that $u : \R \to \C$ with
\begin{align*}
|u(x)| \le C (1+ |x|)^{2s+k-\eps} ~~ \forall x \in \R
\end{align*}
for some $\eps > 0$ and $k \in \N \cup \{ 0 \}$ is a (distributional) solution to
\begin{align*}
\begin{cases}
Lu &\overset{k}{=} 0 ~~ \text{ in } (0,\infty),\\
Nu &\overset{k}{=} 0 ~~ \text{ in } (-\infty,0),
\end{cases}
\end{align*}
if for every $\varphi \in C^{d,k}_{a,b}(\R)$ it holds
\begin{align}
\label{eq:distributional-sol-2D}
\int_{-\infty}^0 u(x) N \varphi(x) \d x + \int_0^{\infty} u(x) L \varphi(x) \d x = 0.
\end{align}
\end{definition}

We can now state our Liouville theorem for the combined action of two operators:

\begin{theorem}
\label{thm:Liouville-2D}
Let $L,N$ be as before, satisfying \eqref{eq:f-2D}, \eqref{eq:L-decay-2D}, and \eqref{eq:self-adjoint-2D} for some $d,a,b,k$, and assume that $f_1 := \frac{f_{L,+} f_{N,-} - f_{L,-} f_{N,+}}{f_{N,-}}$ and $f_2 := \frac{f_{L,+} f_{N,-} - f_{L,-} f_{N,+}}{f_{L,-}}$ are $M$-admissible for some $M > \max\{ 3 + 2s, 2 + 4s, 2k + \frac{3}{2}, k + 2s + 2\}$. Let $u$ be a distributional solution to
\begin{align*}
\begin{cases}
Lu &\overset{k}{=} 0 ~~ \text{ in } (0,\infty),\\
Nu &\overset{k}{=} 0 ~~ \text{ in } (-\infty,0),
\end{cases}
\end{align*}
with test-function space $C^{d,k}_{a,b}(0,\infty)$ and such that, for some $\eps \in (0,2s)$ and $C > 0$,
\begin{align}
\label{eq:u-growth-ass-2D}
|u(x)| \le C (1 + |x|)^{2s+k-\eps}  ~~ \forall x \in  \R.
\end{align}

Then, there exist coefficients $a^{(1)}_{\beta,l},a^{(1)}_{\beta,l}, A^{(1)}_j, A^{(2)}_j \in \C$ such that
\begin{align}
\label{eq:Liouville-zero-2D}
\begin{split}
u(x) &= \sum_{\beta \in \{ f_1 = 0 \} \cap S_{2s+k-\eps} } \sum_{l = 0}^{k_1(\beta) - 1} a^{(1)}_{\beta,l} x^{\beta} (\log x)^l + \sum_{\beta \in \{ f_2 = 0 \} \cap S_{2s+k-\eps} } \sum_{l = 0}^{k_2(\beta) - 1} a^{(2)}_{\beta,l} x^{\beta} (\log x)^l\\
&\quad  + \sum_{j = 0}^{k-1} \1_{\{|f_1(2s+j)| < \infty \}} A^{(1)}_{j} x^{2s+j} + \sum_{j = 0}^{k-1} \1_{\{|f_2(2s+j)| < \infty \}} A^{(2)}_{j} x^{2s+j},
\end{split}
\end{align}
where $k_1(\beta), k_2(\beta)$ denote the multiplicities of $\beta$ with respect to $f_1, f_2$.
\end{theorem}

\begin{remark}
A useful way to think about Theorem 2.10 is as a vector valued analog of Theorem 2.5. The functions $f_{L,\pm}$ and $f_{N,\pm}$ completely describe the action of $L$ and $N$ on all homogeneous functions $u(x) = a x_+^{\beta} + b x_-^{\beta}$. For such a function, the conditions $Lu = 0$ and $Nu = 0$ reduce to the following system:
       \begin{align*}
       A_f \cdot \begin{pmatrix} a \\ b \end{pmatrix} := \begin{pmatrix} f_{L,+} & f_{L,-} \\ f_{N,+} & f_{N,-} \end{pmatrix}
       \cdot \begin{pmatrix} a \\ b \end{pmatrix} = 0
	\end{align*}      
The matrix $A_f$ can be interpreted as a matrix valued Mellin symbol. The term $D := f_{L,+} f_{N,-} - f_{L,-} f_{N,+} = 0$ which is present in both, $f_1$ and $f_2$, is the determinant of the matrix $A_f$. 
\end{remark}

\section{The distributional Mellin transform}

\label{sec3}

In this section we collect several elementary results about the Mellin transform. In particular, we make sense of the Mellin transform applied to functions with polynomial growth by introducing a distributional variant of the Mellin transform. 

Given a complex valued function $u : (0,\infty) \to \C$, we define the Mellin transform of $u$ at $z \in \C$ to be
\begin{equation}\label{ORIGIMELL}
\cM[u](z) = \int_0^{\infty} x^{z-1} u(x) \d x,
\end{equation}
whenever the integral is absolutely convergent. Moreover, given a complex function $\phi : \C \to \C$, we define the $c$-inverse Mellin transform of $\phi$ at $x \in (0,\infty)$ to be
\begin{align*}
\cM^{-1}_c[\phi](x) = \frac{1}{2\pi} \int_{-\infty}^{\infty} x^{-c - i z} \phi \left(c + i z \right) \d z, ~~ c \in \R,
\end{align*}
whenever the (line) integral is well-defined. We will also denote $\cM^{-1} := \cM^{-1}_{1/2}$, i.e. the case $c = \frac{1}{2}$ is the default definition of the inverse Mellin transform, since it gives rise to a natural Plancherel identity (see \eqref{eq:Plancherel}). Note that the freedom to choose $c \in \R$ is important for instance in the Mellin inversion theorem, which we state below. 

\begin{lemma}[Mellin inversion theorem]
\label{lemma:MIT}
Let $u$ be continuous and assume that $\cM[u](z)$ exists whenever $\re(z) \in (a,b)$ for some $a,b \in \R$. 

Then, for every $c \in (a,b)$,
\begin{align}
\label{eq:Mellin-inversion}
\cM^{-1}_c[\cM[u]](x) := \frac{1}{2\pi} \int_{-\infty}^{\infty} x^{-c - i z} \cM[u](c + iz) \d z = u(x) ~~ \forall x \in (0,\infty).
\end{align}

Moreover, if $u$ is holomorphic in $\{ z \in \C : a < \re(z) < b \}$ for some $a,b \in \R$, and $|u(z)| \le C (1 + |z|)^{-2}$ in $\{ z \in \C : a < \re(z) < b \}$, then for any $c \in (a,b)$ the function
\begin{align}
\label{eq:Mellin-inversion-reverse}
\cM^{-1}_c[u](x) = \frac{1}{2\pi} \int_{-\infty}^{\infty} x^{-c - i z} u(c+iz) \d z
\end{align}
satisfies $\cM[\cM^{-1}_c[u]] = u$ in $(0,\infty)$.
\end{lemma}

For references of these two results, we refer to \cite[p.341-343]{McL63}, \cite[p.80]{PaKa01}, and \cite[Theorem 11.1.1]{MiLa86}. The proof makes use of well-known inversion formulas for the Fourier / Laplace transform.

\subsection{Properties of the inverse Mellin transform}

Given $M,m > 0$, let us denote by $S_M = \{ z \in \C : |\re(z)| <  M \}$ the vertical strip of radius $M$ and introduce the following function space
\begin{align*}
\cX_{M,m}: = \left\{ \phi : \C \to \C ~\mid~ \phi \text{ holomorphic in } S_M ,~~ \exists C > 0 :  |\phi(z)| \le C (1 + |z|)^{-m} ~~ \forall z \in S_M \right\}.
\end{align*}

We will prove that the inverse Mellin transform $\cM^{-1}[\phi]$ is well-defined, whenever $\phi \in \cX_{M,m}$ for $M > \frac{1}{2}$ and $m \ge 2$. 
Obviously, for $M' \le M$ and $m' \le m$ we have that
\begin{align*}
\cX_{M,m} \subset \cX_{M',m'}.
\end{align*}
Moreover, by possibly choosing $M$ and $m$ larger, respectively, we gain decay properties of $\cM^{-1}[\phi]$ at zero, at infinity, and regularity of $\cM^{-1}[\phi]$, respectively (see \autoref{lemma:inverse-Mellin}). 

Before we prove this result, we need to recall the Paley-Wiener theorem (see \cite[Chapter 4, Theorem 2.1, Theorem 3.1]{StSh03}).

Here and below we use the following definition of the Fourier transform:
\begin{align*}
\cF[f](\xi): = \int_{-\infty}^{\infty} e^{-ix \xi} f(x) \d x, \qquad \cF^{-1}[f](x) := \frac{1}{2\pi} \int_{-\infty}^{\infty} e^{ix \xi} f(\xi) \d \xi.
\end{align*}

\begin{lemma}[Paley-Wiener theorem]
\label{lemma:Paley-Wiener}
Let $M > 0$.
\begin{itemize}
\item[(i)] Let $f$ be holomorphic in $\{ z \in \C : |\im (z)| < M \}$ and assume that $|f(z)| \le c_1 (1 + |z|)^{-2}$ for every $z \in \C$ such that $|\im(z)| <  M$. Then, for all~$ M' \in (0, M)$ and~$ \xi \in \R$,
\begin{align*}
|\cF[f](\xi)| \le c_2 e^{- M' |\xi|}. 
\end{align*}
\item[(ii)] Let $f$ and $\cF[f]$ be such that,  for all~$ x$, $\xi \in \R$,
\begin{align*}
|f(x)| \le c_3 (1 + |x|)^{-2} \qquad{\mbox{and}}\qquad |\cF[f](\xi)| \le c_3 e^{- M |\xi|}.
\end{align*}
Then $f$ is the restriction to $\R$ of a function $f$ that is holomorphic in $\{ z \in \C : |\im (z)| < M \}$.
\end{itemize}
\end{lemma}

The following result is a direct application of the Paley-Wiener theorem.

\begin{lemma}
\label{lemma:inverse-Mellin}
Let $c \in \R$, $\phi \in \cX_{M,m}$ for some $M > |c|$ and $m \ge 2$. Then,
\begin{align*}
\cM^{-1}_c[\phi](x) = \frac{1}{2\pi} x^{-c} \cF \left[\phi \left(c + i \cdot \right)\right](\log x). 
\end{align*}
Moreover, for every $M' \in (0,M)$ there exists $C > 0$ such that
\begin{align*}
|\cM^{-1}_c[\phi](x)| \le C 
\begin{cases}
x^{M' - 2|c|}, &~~ \forall x \in (0,1],\\
x^{-M' + 2|c|}, &~~ \forall x \in [1,\infty).
\end{cases}
\end{align*}
In addition, for any $k \in \N$ such that $k+2 \le m$, there exists $C > 0$ such that
\begin{align*}
|(\cM^{-1}_c[\phi])^{(k)}(x)| \le C
\begin{cases}
x^{M' - k - 2|c|}, &~~ \forall x \in (0,1],\\
x^{-M' - k +2|c|}, &~~ \forall x \in [1,\infty).
\end{cases}
\end{align*}
\end{lemma}

\begin{proof}
By the definitions of $\cM^{-1}_c$ and $\cF$ we immediately verify the first claim:
\begin{align*}
\cM^{-1}_c[\phi](x) = \frac{1}{2\pi} x^{-c}  \int_{-\infty}^{\infty} e^{-i z \log(x)} \phi \left(c + i z \right) \d z = \frac{1}{2\pi} x^{-c} \cF \left[\phi \left(c + i \cdot \right)\right](\log x).
\end{align*}
Since $\phi$ is holomorphic in $S_M$ and satisfies $|\phi(z)| \le C(1 + |z|)^{-2}$ for any $z \in S_M$, we deduce that $\tilde{\phi}(z) = \phi \left(c + i z \right)$ is holomorphic and satisfies $|\tilde{\phi}(z)| \le C (1 + |z|)^{-2}$ in
\begin{align*}
\left\{ z \in \C : \left|\re \left(c + i z \right)\right| < M \right\} = \left\{ z \in \C : \im(z) \in \left(-M + c , M + c \right) \right\}.
\end{align*}
Since $(-(M - |c|) , M - |c|) \subset \left(-M + c , M + c \right)$, we can apply \autoref{lemma:Paley-Wiener}(i) with $M' := M - |c|$ and deduce that 
\begin{align*}
\cF[\tilde{\phi}](x) \le C e^{- M' |x|} ~~ \forall M' \in \left( 0 , M - |c| \right), ~~ \forall x \in \R.
\end{align*}
Thus, 
\begin{align*}
|\cM^{-1}_c[\phi](x)| \le C x^{-c} | \cF [\tilde{\phi}](\log x) | \le C x^{-c} e^{- M' |\log x|} \le C \begin{cases}
x^{-c} x^{M'} &~~ \forall x \in (0,1],\\
x^{-c} x^{-M'} &~~ \forall x \in [1,\infty),
\end{cases}
\end{align*}
which yields the second claim. 

To prove the third claim, we let $k \in \N$, assume $k+2 \le m$ and observe
\begin{align*}
(\cM^{-1}_c[\phi])^{(k)}(x) &= \frac{1}{2\pi} \int_{-\infty}^{\infty} \left[x^{-c - iz}\right]^{(k)}(x) \phi(z) \d z \\
&= \frac{1}{2\pi} \int_{-\infty}^{\infty}  \prod_{j = 0}^{k-1} \left(-c - j - iz \right) x^{-c - iz - k}  \phi(z) \d z \\
&= \frac{1}{2\pi} x^{-c - k}\int_{-\infty}^{\infty} e^{-iz \log(x)}  \prod_{j = 0}^{k-1}\left(-c - j  - iz \right)  \phi \left(c + i z \right) \d z \\
&=\frac{1}{2\pi} x^{-c - k} \cF \left[ \prod_{j = 0}^{k-1}\left(-c - j - i\cdot \right)  \phi \left(c + i \cdot \right)\right](\log x).
\end{align*}
We define $$\bar{\phi}(z) :=  \prod_{j = 0}^{k-1}\left(-c - j - iz \right)  \phi \left(c + i z \right)$$ and observe that~$\bar{\phi}$ is holomorphic. 

In addition, using that $m \ge k + 2$ we find that,
for all~$z \in \{ z \in \C : |\im(z)| < M - |c| \}$,
\begin{align*}
|\bar{\phi}(z)| \le C \left|\prod_{j = 0}^{k-1}\left(-c - j - iz \right) \right| (1 + |z|)^{-m} \le C (1 + |z|)^{k-m} \le C (1 + |z|)^{-2}.
\end{align*}
Hence, by \autoref{lemma:Paley-Wiener}(i),
\begin{align*}
\cF[\bar{\phi}](x) \le C e^{- M' |x|} ~~ \forall M' \in \left( 0 , M - |c| \right), ~~ \forall x \in \R.
\end{align*}
Thus, 
\begin{align*}
|(\cM^{-1}_c[\phi])^{(k)}(x)| \le C x^{-c - k} | \cF [\tilde{\phi}](\log x) | \le C x^{-c - k} e^{- M' |\log x|} \le C \begin{cases}
x^{-c - k} x^{M'} &~~ \forall x \in (0,1],\\
x^{-c - k} x^{-M'} &~~ \forall x \in [1,\infty),
\end{cases}
\end{align*}
as desired.
\end{proof}

The previous lemma implies that $\cM^{-1}[\phi]$ (as well as its derivatives) satisfies polynomial decay at zero and at infinity whenever $\phi \in \cX_{M,m}$. Moreover, by taking $M > 0$ (and $m > 0$) large, the order of the polynomial decay of $\cM^{-1}[\phi]$ (and its derivatives) can be made arbitrarily large.

We also point out that the choice of the parameter $c \in \R$ in the definition of the inverse Mellin transform is not important, whenever $\phi \in \cX_{M,m}$. More precisely:

\begin{corollary}
\label{cor:c-change}
Let $c \in \R$, $\phi \in \cX_{M,m}$ for some $M > \max \{|c|,2|c| - \frac{1}{2}\}$ and $m \ge 2$. Then, 
\begin{align*}
\cM^{-1}_c[\phi] = \cM^{-1}[\phi].
\end{align*}
\end{corollary}

\begin{proof}
The proof relies on the application of the Mellin inversion theorem \eqref{eq:Mellin-inversion} and \eqref{eq:Mellin-inversion-reverse}:
\begin{align*}
\cM^{-1}_c[\phi] = \cM^{-1}[\cM[\cM_c^{-1}[\phi]]] = \cM^{-1}[\phi].
\end{align*}
In order to justify application of \eqref{eq:Mellin-inversion} to $\cM^{-1}_c[\phi]$ in the first step, we use that this function is continuous and that $\cM[\cM^{-1}_c[\phi]](z)$ exists for $\re(z) = \frac{1}{2}$ by \autoref{lemma:inverse-Mellin}. The application of \eqref{eq:Mellin-inversion-reverse} to $\phi$ is justified by the definition of $\cX_{M,m}$ and since $M > c$.
\end{proof}

\subsection{Definition of a distributional Mellin transform}

The main purpose of \autoref{lemma:inverse-Mellin} is that it motivates the following definition of the Mellin transform for functions $u$ which might not satisfy $\int_0^{\infty} |x^{z-1} u(x)| \d x < \infty$ for some (or any) $z \in \C$.

\begin{definition}
\label{def:Mellin}
Given $M > \frac{1}{2}$ and $m \ge 2$, we define
\begin{align*}
\cM^{-1}[\cX_{M,m}]' = \left\{ u \in L^1_{loc}(0,\infty) ~\Big|~ \int_0^{\infty} u(x) \cM^{-1}[\phi](x) \d x < \infty ~~ \forall \phi \in \cX_{M,m} \right\}.
\end{align*}
Moreover, for any $u \in \cM^{-1}[\cX_{M,m}]'$ we define $\widetilde{\cM}[u] : \cX_{M,m} \to \R$ as follows
\begin{align*}
\langle \widetilde{\cM}[u] , \phi \rangle := \int_0^{\infty} u(x) \cM^{-1}[\phi](x) \d x.
\end{align*}
\end{definition}


\begin{remark}
As a a consequence of the Plancherel formula,
we point out that, whenever $u$ is such that $\int_0^{\infty}|x^{z-1}u(x)| \d x < \infty$ for any $z \in \C$, then,
for any $\phi \in \cX_{M,m}$ with $M > \frac{1}{2}$ and $m \ge 2$,
\begin{align}
\label{eq:Plancherel}
\langle \widetilde{\cM}[u] , \phi \rangle = \int_{0}^{\infty} u(x) \cM^{-1}[\phi](x) \d x = \int_{-\infty}^{\infty} \cM[u]\left(\frac{1}{2} + iz \right)  \overline{ \phi\left(\frac{1}{2} + iz \right) } \d z.
\end{align} 
Hence, $\widetilde{\cM}$ coincides with the original definition of the Mellin transform up to a shift, whenever $u$ is nice enough. 
\end{remark}

\begin{remark}
As a direct consequence of \autoref{lemma:inverse-Mellin} we deduce that, for any $a,b > 0$,
\begin{align}
\label{eq:dual-inclusion}
\left\{ u \in L^1_{loc}(0,\infty) ~\big|~ |u(x)| \le C(|x|^{-a} + |x|^b) ~~ \forall x \in (0,\infty)  \right\} \subset \cM^{-1}[\cX_{M,m}]',
\end{align}
whenever $M > \max\{a,b\}$ and $m \ge 2$.
\end{remark}

With the help of \autoref{def:Mellin} we can give a rigorous meaning to $\widetilde{\cM}[x^{\alpha}]$ for $\alpha \in \C$ and prove:

\begin{lemma}
\label{lemma:Mellin-Dirac}
Let $\alpha \in \C$. Then,
\begin{align*}
\widetilde{\cM}[x^{\alpha}](z) = \delta(z - (\alpha+1))
\end{align*}
in the following sense: For any $\phi \in \cX_{M,m}$ with $M > \max\{ 2|\re(\alpha)| + \frac{3}{2} , |\re(\alpha)| + 1\}$ and $m \ge 2$,
\begin{align*}
\langle \widetilde{\cM}[x^{\alpha}] , \phi \rangle = \phi(\alpha + 1).
\end{align*}

Moreover, for any $l \in \N$,
\begin{align*}
\langle \widetilde{\cM}[x^{\alpha} (\log(x))^l] , \phi \rangle = (-1)^l \phi^{(l)}(\alpha + 1).
\end{align*}
\end{lemma}

\begin{proof}
Let $\phi \in \cX_{M,m}$. Recall that by \autoref{cor:c-change} and the assumption on $M$, we have that, for any $c \in \R$ with $\max\{|c| , 2|c| - \frac{1}{2}\} < M$,
\begin{align*}
\cM^{-1}[\phi](x) = \cM^{-1}_c[\phi](x) ~~ \forall x \in (0,\infty).
\end{align*}
First, we take $c_1 < \re(\alpha) + 1$ with $\max\{|c_1| , 2|c_1| - \frac{1}{2}\} < M$ and deduce by Fubini's theorem that
\begin{align*}
\int_0^{1} x^{\alpha} \cM^{-1}_{c_1}[\phi](x) \d x &=  \frac{1}{2\pi}\int_{-\infty}^{\infty} \left( \int_0^1 x^{\alpha - c_1 - iz}  \d x \right) \phi(c_1 + iz) \d z \\
& =  - \frac{1}{2\pi}\int_{-\infty}^{\infty} \frac{\phi(c_1 + iz)}{(c_1 + iz) - (\alpha + 1 )} \d z \\
&= - \frac{1}{2\pi i}\int_{c_1 -i \infty}^{c_1 + i\infty} \frac{\phi(z)}{z - (\alpha + 1 )}   \d z.
\end{align*}
Analogously, taking $c_2 > \re(\alpha) + 1$ with $\max\{|c_2| , 2|c_2| - \frac{1}{2}\} < M$, we have that
\begin{align*}
\int_1^{\infty} x^{\alpha} \cM^{-1}_{c_2}[\phi](x) \d x &= \frac{1}{2\pi} \int_{-\infty}^{\infty} \left( \int_1^{\infty} x^{\alpha - c_2 - iz}  \d x \right) \phi(c_2 + iz) \d z \\
&= \frac{1}{2\pi}\int_{-\infty}^{\infty} \frac{\phi(c_2 + iz)}{(c_2 + iz) - (\alpha + 1)} \d z \\
&= \frac{1}{2\pi i}\int_{c_2 -i\infty}^{c_2 + i\infty} \frac{\phi(z)}{z - (\alpha + 1)} \d z .
\end{align*}

Let now $R > 0$.
By the Cauchy integral formula applied to the contour $$\gamma_R := [c_2 - iR , c_2 + iR] \cup [c_2 + iR , c_1 + iR] \cup [c_1 + iR , c_1 - iR] \cup [c_1 - iR , c_2 - iR],$$we obtain that
\begin{align*}
\phi(\alpha + 1) = \frac{1}{2\pi i}\int_{\gamma_R} \frac{\phi(z)}{z - (\alpha + 1)} \d z.
\end{align*} 
Hence, taking the limit $R \to \infty$, and observing that the contribution from the horizontal parts of $\gamma_R$ goes to zero by the decay of $\phi \in \cX_{M,m}$, we deduce
\begin{align*}
\int_0^{\infty} x^{\alpha} \cM^{-1}[\phi](x) \d x &= \int_0^{1} x^{\alpha} \cM^{-1}_{c_1}[\phi](x) \d x + \int_1^{\infty} x^{\alpha} \cM^{-1}_{c_2}[\phi](x) \d x\\
&= \frac{1}{2\pi i}\int_{c_2 -i \infty}^{c_2 + i\infty} \frac{\phi(z)}{z - (\alpha + 1 )} \d z - \frac{1}{2\pi i}\int_{c_1 -i\infty}^{c_1 + i\infty} \frac{\phi(z)}{z - (\alpha + 1)} \d z =  \phi(\alpha + 1).
\end{align*}

To prove the second claim, note that for any $l \in \N$
\begin{align*}
\int_0^1 x^{a} (\log(x))^l \d x =  \frac{(-1)^l l!}{(1+a)^{l+1}} ~~ \text { if } \re(a) > -1, \qquad \int_1^{\infty} x^{a} (\log(x))^l \d x =  \frac{(-1)^{l+1} l!}{(1+a)^{l+1}} ~~ \text { if } \re(a) < -1.
\end{align*}
Thus, by similar arguments as before, we find that, when $c_1 < \re(\alpha) + 1$ and $c_2 > \re(\alpha) + 1$,
\begin{align*}
\int_0^{1} x^{\alpha} (\log(x))^l \cM^{-1}_{c_1}[\phi](x) \d x &=  -\frac{(-1)^l l!}{2\pi i}\int_{c_1 -i \infty}^{c_1 + i\infty} \frac{\phi(z)}{[z - ( \alpha + 1 )]^{l+1}} \d z, \\
\int_1^{\infty} x^{\alpha} (\log(x))^l \cM^{-1}_{c_2}[\phi](x) \d x &= \frac{(-1)^l l!}{2\pi i}\int_{c_2 -i\infty}^{c_2 + i\infty} \frac{\phi(z)}{[z - (\alpha + 1)]^{l+1}} \d z.
\end{align*}
Cauchy's integral formula (or rather Cauchy's differentiation formula) applied to the contour $\gamma_R$ yields
\begin{align*}
\phi^{(l)}(\alpha + 1) = \frac{l!}{2\pi i} \int_{\gamma_R} \frac{\phi(z)}{[z - (\alpha + 1)]^{l+1}} \d z,
\end{align*}
and therefore we deduce that
\begin{equation*}\begin{split}
\int_0^{\infty} & x^{\alpha} (\log(x))^l  \cM^{-1}[\phi](x) \d x \\
 &= \frac{(-1)^l l!}{2\pi i}\int_{c_2 -i \infty}^{c_2 + i\infty} \frac{\phi(z)}{[ z - (\alpha + 1 )]^{l+1}} \d z - \frac{(-1)^{l} l!}{2\pi i}\int_{c_1 -i\infty}^{c_1 + i\infty} \frac{\phi(z)}{[z - (\alpha + 1)]^{l+1}} \d z \\
&= (-1)^l \phi^{(l)}(\alpha + 1).\qedhere\end{split}
\end{equation*}
\end{proof}

\subsection{A fundamental lemma of the calculus of variations}

The following result can be regarded as a version of the fundamental lemma of the calculus of variations in the spaces $\cM^{-1}[\cX_{M,m}]'$ and $\cM^{-1}[\cX_{M,m}]$. It allows us to deduce that a function $u \in \cM^{-1}[\cX_{M,m}]'$ is zero by analyzing its distributional Mellin transform.

\begin{lemma}
\label{lemma:u-zero}
Let $M > \frac{1}{2}$, $m \ge 2$, and $u \in \cM^{-1}[\cX_{M,m}]'$. If for any $\phi \in \cX_{M,m}$ it holds
\begin{align*}
\langle \widetilde{\cM}[u] , \phi \rangle = 0,
\end{align*}
then $u = 0$ a.e. in $(0,\infty)$.
\end{lemma}

\begin{proof}
Let $\eta \in C_c^{\infty}(0,\infty)$. First, we claim that $\cM[\eta] \in \cX_{M,m}$. To prove it, we observe that $\cM[\eta]$ is well-defined in the classical sense and
\begin{align*}
\cM[\eta](z) = \int_0^{\infty} \eta(x) x^{z-1} \d x = \int_{-\infty}^{\infty} \eta(e^x) e^{xz} \d x = \int_{-\infty}^{\infty} \eta(e^{x}) e^{xz} \d x = 2\pi \cF^{-1}[\eta(e^{\cdot})](-iz).
\end{align*}

Since $\eta \in C_c^{\infty}(0,\infty)$, we have that~$x \mapsto \eta(e^x) \in C_c^{\infty}(\R)$. Hence, a variant of the Paley-Wiener-Schwartz theorem (see e.g.~\cite[Chapter IX.3, p.16]{ReSi75}) implies that for any  $N \in \N$ there exists $C > 0$ such that,
for all~$ z \in \C$,
\begin{align}
\label{eq:Fourier-decay}
2\pi |\cF^{-1}[\eta(e^{\cdot})](-z)| = |\cF[\eta(e^{\cdot})](z)| \le C (1 + |z|)^{-N} e^{C |\im(z)|}.
\end{align}
In particular, for all~$ z \in S_M$,
\begin{align*}
|\cM[\eta](z)| = 2\pi |\cF^{-1}[\eta(e^{\cdot})](-iz)| \le C e^{CM} (1 + |z|)^{-N} .
\end{align*}
Moreover, since $\eta(e^\cdot) \in C_c^{\infty}(\R)$, we infer from \eqref{eq:Fourier-decay} that, for all $z \in \R$ and~$a > 0$,
\begin{align*}
|\cF^{-1}[\eta(e^{\cdot})](z)| \le C(1 + |z|)^{-2}  \qquad{\mbox{and}}\qquad |\eta(e^{z})| \le C_a e^{-a|z|} .
\end{align*}
Accordingly, we can apply \autoref{lemma:Paley-Wiener}(ii) with $f := \cF^{-1}[\eta(e^{\cdot})]$ and deduce that $\cF^{-1}[\eta(e^{\cdot})]$ is holomorphic in $\{ z \in \C : |\im(z)| < a\}$ for any $a > 0$.

In particular, $\cM[\eta]$ is holomorphic in $S_M$. Hence, we have shown that $\cM[\eta] \in \cX_{M,m}$, as claimed. 

This allows us to apply the assumption of the lemma with $\phi = \cM[\eta]$, and conclude that
\begin{align*}
0 = \langle \widetilde{\cM}[u] , \cM[\eta] \rangle = \int_0^{\infty} u(x) \cM^{-1}[\cM[\eta]](x) \d x = \int_0^{\infty} u(x) \eta(x) \d x.
\end{align*}
In the last inequality we used the Mellin inversion theorem (see \eqref{eq:Mellin-inversion}), which is applicable since $\eta \in C_c^{\infty}(0,\infty)$ and $\cM[\eta](z)$ exists in the classical sense for any $z \in \C$. Since $\eta \in C_c^{\infty}(0,\infty)$ was arbitrary, we deduce that $u = 0$ almost everywhere, and the proof is complete.
\end{proof}

\subsection{The Mellin transform on $\R$}
\label{subsec:2D-Mellin}

In this subsection we introduce the Mellin transform for functions $u : \R \to \C$ (generalizing the setting in~\eqref{ORIGIMELL},
which was for functions~$u : (0,\infty)\to \C$)
and establish several basic properties which will be useful in the proof of \autoref{thm:Liouville-2D}.

Given a complex valued function $u  : \R \to \C$ we define the Mellin transform of $u$ at $z \in \C$ to be
\begin{align*}
\cM[u](z) &:= (\cM[u]^+(z) , \cM[u]^-(z)) \\
&:= (\cM[u |_{\R_+}](z) , \cM[u(-\cdot) |_{\R_-}](z)) = \left(\int_0^{\infty} x^{z-1} u(x) \d x , \int_{-\infty}^0 x_-^{z-1} u(x) \d x \right) \in \C^2,
\end{align*}
whenever the integral is absolutely convergent. Moreover, given two complex functions $(\phi_1,\phi_2) : \C \to \C^2$, we define the $c$-inverse Mellin transform of $(\phi_1,\phi_2)$ at $x \in \R$ to be
\begin{align*}
\cM_c^{-1}[\phi_1,\phi_2](x) := \cM_c^{-1}[\phi_1](x)\1_{\R_+}(x) + \cM_c^{-1}[\phi_2](-x)\1_{\R_-}(x).
\end{align*}

It is immediate from the definition, that the Mellin inversion theorem \autoref{lemma:MIT} remains valid for this extended definition of the Mellin transform and its inverse. Therefore also straightforward variants of \autoref{lemma:inverse-Mellin} and \autoref{cor:c-change} are available, simply by applying the previous theorems to both components separately.

\section{A new Liouville theorem via the Mellin transform}
\label{sec4}

The goal of this section is to present a rigorous proof of \autoref{thm:Liouville} and \autoref{thm:Liouville-2D} with the help of the distributional Mellin transform.

\subsection{Proof of \autoref{thm:Liouville}}

We deduce from \autoref{lemma:inverse-Mellin} that, for every $M > \max\{ 1 + b , 1 + a  + d  \}$ and $m \ge d+2$,
\begin{align}
\label{eq:X-C-inclusion}
\cM^{-1}[\cX_{M,m}] \subset C^d_{a,b}(0,\infty).
\end{align}

The following result is a consequence of \eqref{eq:f} and the self-adjointness of $L$.

\begin{lemma}
\label{lemma:Mellin-magic}
Let $L$ be as in \autoref{thm:Liouville}. 
Let also~$M > \max\{ 1 + b , 1 + a  + d  \}$, $m \ge d+2$, and $\phi \in \cX_{M,m}$.

Then, for every $z \in \C$ with $\re(z) \in (\alpha_1,2s-\alpha_2)$ we have that
\begin{align}
\label{eq:Mellin-well-def}
\cM[L(\cM^{-1}[\phi])](z) = f(z-1) \phi(z-2s).
\end{align}
\end{lemma}

\begin{proof}
By \eqref{eq:X-C-inclusion} we have $\cM^{-1}[\phi] \in C_{a,b}^d(0,\infty)$ for every $\phi \in \cX_{M,m}$. Hence, \eqref{eq:L-decay} implies that
\begin{align*}
\cM[L(\cM^{-1}[\phi])](z) = \int_0^{\infty} x^{z-1} L(\cM^{-1}[\phi])(x) \d x
\end{align*}
is well-defined for every $z \in \C$ with $\re(z) \in (\alpha_1,2s - \alpha_2)$. For any such $z$, we compute, using that $L$ is self-adjoint (recall~\eqref{eq:self-adjoint}) and satisfies \eqref{eq:f},
\begin{align}
\label{eq:Mellin-magic}
\begin{split}
\cM[L(\cM^{-1}[\phi])](z) &= \int_0^{\infty} x^{z-1} L(\cM^{-1}[\phi])(x) \d x \\
&= \int_0^{\infty} L(x^{z-1})(x) \cM^{-1}[\phi](x) \d x \\
&= f(z-1) \int_0^{\infty} x^{z-2s-1} \cM^{-1}[\phi](x) \d x \\
&= f(z-1) \cM[\cM^{-1}[\phi]](z-2s) \\
&= f(z-1) \phi(z-2s).
\end{split}
\end{align}

In the last step, we used the Mellin inversion theorem (see \eqref{eq:Mellin-inversion-reverse}), which is applicable to $\phi \in \cX_{M,m}$ by definition of the space $\cX_{M,m}$. 
\end{proof}

\begin{lemma}
\label{lemma:PDE-Mellin-1}
Let $L$ be as in \autoref{thm:Liouville}. Assume that $|u(x)| \le C (1+ x)^{2s-\eps}$ for some $\eps \in (0,1)$ and that, in the distributional sense,
\begin{align*}
Lu = 0 ~~ \text{ in } (0,\infty).
\end{align*}
Then, for any $M > \max\{ 1 + b , 1 + a  + d  \}$, $m \ge d+2$, $c \in (\alpha_1,2s-\alpha_2)$, and $\phi \in \cX_{M,m}$,
\begin{align*}
0 = \int_0^{\infty} u(x) \cM^{-1}_c[\phi(\cdot -2s) f(\cdot - 1)](x) \d x.
\end{align*}
\end{lemma}

\begin{proof}
By \eqref{eq:X-C-inclusion} we have that $\cM^{-1}[\phi] \in C^d_{a,b}(0,\infty)$ for any $\phi \in \cX_{M,m}$. Hence, recalling~\eqref{eq:L-decay}, we know that $\cM^{-1}[\phi]$ is a valid test function for the equation $Lu = 0$. As a result,
\begin{align}
\label{eq:PDE-Mellin-1-help}
0 = \int_{0}^{\infty} u(x) L (\cM^{-1}[\phi])(x) \d x,
\end{align}
and the integral on the right-hand side converges absolutely by \autoref{remark:distr-well-def}. 
Moreover, we recall that
\begin{align*}
\cM[L(\cM^{-1}[\phi])](z) = \int_0^{\infty} x^{z-1} L(\cM^{-1}[\phi])(x) \d x
\end{align*}
is well-defined for every $z \in \C$ with $\re(z) \in (\alpha_1,2s-\alpha_2)$ by \eqref{eq:L-decay}. Hence, we can apply the Mellin inversion theorem (see \eqref{eq:Mellin-inversion}) and deduce
\begin{align*}
L (\cM^{-1}[\phi]) = \cM^{-1}_c[\cM[ L (\cM^{-1}[\phi])]] ~~ \forall c \in (\alpha_1,2s-\alpha_2).
\end{align*}
Moreover, since by the definition of $\cM^{-1}_c$ only the values of $\cM[ L (\cM^{-1}[\phi])](z)$ for $z \in \C$ with $\re(z) \in (\alpha_1,2s-\alpha_2)$ contribute to the inverse Mellin transform, we can apply  \autoref{lemma:Mellin-magic} to deduce
\begin{align}
\label{eq:appl-Mellin-inversion}
L (\cM^{-1}[\phi]) = \cM^{-1}_c[ f(\cdot -1) \phi(\cdot -2s) ] ~~ \forall c \in (\alpha_1,2s - \alpha_2).
\end{align}

Finally, by combining \eqref{eq:PDE-Mellin-1-help} and \eqref{eq:appl-Mellin-inversion},
\begin{align*}
0 = \int_{0}^{\infty} u(x) L (\cM^{-1}[\phi])(x) \d x =  \int_{0}^{\infty} u(x) \cM^{-1}_c[ f(\cdot -1) \phi(\cdot -2s) ](x) \d x.
\end{align*}
This implies the desired result.
\end{proof}

The following result is a reformulation of \autoref{lemma:PDE-Mellin-1}. We will apply \autoref{lemma:PDE-Mellin-1} only in this reformulated version.

\begin{lemma}
\label{lemma:PDE-Mellin-2}
Let $L,u$ be as in \autoref{thm:Liouville}. Then, for any $M > \max\{ 1 + b , 1 + a  + d , 2 +2 s \}$, $m \ge d+2$, and $\psi \in \cX_{M,m}$ such that $z \mapsto \psi(z+2s)/f(z+2s-1) \in \cX_{M,m}$,
\begin{align*}
0 = \int_{0}^{\infty} u(x) \cM^{-1}[\psi](x) \d x.
\end{align*}  
\end{lemma}

\begin{proof}
We apply \autoref{lemma:PDE-Mellin-1} with $\phi(z) = \psi(z+2s)/f(z+2s-1)$ and obtain
\begin{align*}
0 = \int_0^{\infty} u(x) \cM^{-1}_c[f(\cdot - 1) \phi(\cdot - 2s)] \d x = \int_0^{\infty} u(x) \cM^{-1}_c[\psi] \d x
\end{align*}
for any $c \in (\alpha_1,2s-\alpha_2)$. Since also $\psi \in \cX_{M,m}$ by assumption, we are allowed to take $c = \frac{1}{2}$ due to \autoref{cor:c-change}, even if $\frac{1}{2} \not\in (\alpha_1,2s-\alpha_2)$. Note that the resulting integral converges absolutely due to the growth assumption on $u$ and \autoref{lemma:inverse-Mellin}. This concludes the proof.
\end{proof}

Before we prove \autoref{thm:Liouville} we need the following result, which yields the existence of functions in $\cX_{M,m}$ with prescribed values (and prescribed derivatives up to a finite order). We believe that the result is standard, but we did not find it in the literature, so provide a proof in order for our argument to be exhaustive and
self-contained. The proof is inspired by \cite[Theorem 15.13]{Rud87}.

\begin{lemma}
\label{lemma:prescribed-values}
Let $M, m > 0$ and let $F_0 \subset S_M$ be a finite set. For every $\beta \in F_0$ let $k(\beta) \in \N \cup \{ 0 \}$ and $a_{\beta,l} \in \C$ for any $l \in \{ 0 , \dots , k(\beta) \}$ be given. 

Then, there exists $\psi \in \cX_{M,m}$ such that
\begin{align*}
\psi^{(l)}(\beta) = A_{\beta,l} ~~ \forall \beta \in F_0, ~~ l \in \{0,\dots, k(\beta) \}.
\end{align*}
\end{lemma}

\begin{proof}
We define 
\begin{align*}
\eta(z) := e^{-(z/i)^2} \prod_{\beta \in F_0} (z - \beta)^{k(\beta) + 1}.
\end{align*}
We observe that $\eta$ has a zero of order $k(\beta) + 1$ at each $\beta \in F_0$ and vanishes faster than any polynomial in $S_M$. We claim that for each $\beta \in F_0$ we can find a function $P_{\beta}$ of the form
\begin{align}
\label{eq:P-beta-ex}
P_{\beta}(z) = \sum_{l = 1}^{k(\beta) + 1} c_{\beta,l} (z - \beta)^{-l}
\end{align}
such that $\eta P_{\beta}$ has the following power series expansion in some disc centered at $\beta$:
\begin{align}
\label{eq:P-beta-prop}
\eta(z) P_{\beta}(z) = A_{\beta,0} + A_{\beta,1}(z-\beta) + \dots + A_{\beta,k(\beta)} \frac{(z - \beta)^{k(\beta)}}{k(\beta)!} + \dots .
\end{align}
To prove the existence of $P_{\beta}$, let $\beta \in F_0$. For simplicity, let us assume that $\beta = 0$ and write $k(\beta) = k$ and $A_{\beta,l} = A_l$. Then, for some coefficients $\eta_j \in \C$ with $\eta_1 \not= 0$, and $c_j \in \C$, we find that
\begin{align*}
\eta(z) = \eta_1 z^{k + 1} + \eta_2 z^{k + 2} + \dots ~, \qquad P(z) = c_1 z^{-1} + c_2 z^{-2} + \dots + c_k z^{-k} + c_{k+1} z^{-k-1}.\end{align*}
In this setting, \eqref{eq:P-beta-prop} becomes equivalent to finding $c_{l}$ such that
\begin{align*}
\eta(z) P(z) &= (c_{1} z^{k} + c_{2} z^{k - 1} + \dots + c_{k} z + c_{k + 1} ) (\eta_1 + \eta_2 z + \eta_3 z^2 + \dots) \\
& \overset{!}{=}  A_{0} + A_{1} z + \dots + A_{k} \frac{z^{k}}{k!} + \dots.
\end{align*}
However, by comparing the coefficients of the two expressions, we can solve the resulting equations successively for $c_{k+1}, c_k, \dots, c_2 , c_1$, which yields the existence of $P$ as in \eqref{eq:P-beta-ex}-\eqref{eq:P-beta-prop}. 

Let us observe that, by construction, for any $\beta \in F_0$ and $\alpha \in F_0 \setminus \{ \beta \}$,
\begin{align*}
(\eta P_{\alpha})^{(l)}(\beta) = 0 ~~ \forall l \in \{ 0 , \dots , k(\beta) \}.
\end{align*}
Indeed, this follows immediately from the fact that $\eta$ has a zero of order $k(\beta) + 1$ at $\beta$, while $P_{\alpha}$ is holomorphic in a neighborhood of $\beta$. In particular, by the construction of $P_{\alpha}$, also the function $\eta P_{\alpha}$ is holomorphic in a neighborhood of $\beta$.

Now, we define
\begin{align}
\label{eq:P-beta-prop-alpha}
\psi(z) = \eta(z) \sum_{\beta \in F_0} P_{\beta}(z).
\end{align} 
We deduce from \eqref{eq:P-beta-prop} and \eqref{eq:P-beta-prop-alpha} that $\psi^{(l)}(\beta) = A_{\beta,l}$ for any $\beta \in F_0$ and $l \in \{ 0 , \dots , k(\beta)\}$. Moreover, by construction, the function $\psi$ is holomorphic in $S_M$, and decays exponentially since the $P_{\beta}$ are bounded away from $\beta$. Altogether, we deduce that $\psi \in \cX_{M,m}$, as desired.
\end{proof}

We are now in a position to prove \autoref{thm:Liouville}.

\begin{proof}[Proof of \autoref{thm:Liouville}]
Let~$M > \max\{ 1 + b , 1 + a  + d  \} + 2s$ and~$m \ge d+2$. We denote 
\begin{align*}
F_0: = \{ \beta \in \C : f(\beta) = 0 , ~|\re(\beta)| < M \},
\end{align*}
and recall that we assumed that $F_0$ only contains finitely many elements. 

Given~$\beta \in F_0$, we also denote  by $k(\beta) \in \N$ the order of the zero of $f$ at $\beta$. We define
\begin{align*}
v(x) := u(x) - \sum_{\beta \in F_0}  \sum_{l = 0}^{k(\beta) - 1}  c_{\beta,l} x^{\beta} (\log(x))^{l}  ,
\end{align*}
where the constants $c_{\beta,l} \in \R$ will be chosen later. 

Let also~$\overline{M} > 2(M + 1) + \frac{3}{2}$. Given $g \in \cX_{\overline{M} + 1,m}$, we define
\begin{align*}
\psi(z) := g(z) - \Psi_0(z) \sum_{\beta \in F_0} \left( \Upsilon_{\beta}(z)  \sum_{l = 0}^{k(\beta) - 1} \frac{(z - 1 - \beta)^l g^{(l)}(\beta +1)}{l!}  \right),
\end{align*}
where $\Psi_0 \in \cX_{\overline{M}+1,m}$ is such that, for all~$ \beta \in F_0$,
\begin{align*}
\Psi_0(\beta+1) = 1 
\end{align*}
and, for all~$\beta \in F_0$ and~$ l \in \{ 1, \dots, k(\beta) - 1 \}$,
$$ \Psi_0^{(l)}(\beta+1) = 0 .$$
Moreover, given $\beta \in F_0$, the function $\Upsilon_{\beta} \in \cX_{\overline{M}+1,m}$ is such that,
for all~$ \alpha \in F_0$,
\begin{align*}
\Upsilon_{\beta}(\alpha + 1) = \delta_{\alpha,\beta}  ~~ \forall \alpha \in F_0, \qquad \Upsilon_{\beta}^{(l)}(\alpha + 1) = 0 ~~ \forall l \in \{ 0 , \dots , k(\beta) - 1 \} .
\end{align*}
Such functions $\Psi_0$ and $\Upsilon_{\beta}$ do exist by \autoref{lemma:prescribed-values}.

Note that $\psi$ is constructed in such a way that
\begin{align*}
\psi^{(l)}(\beta + 1) = 0 ~~ \forall \beta \in F_0, ~~ l \in \{ 0 , \dots , k(\beta) - 1\},
\end{align*}
i.e. $\psi$ has a zero of order $k(\beta)$ at all points $\beta+1$ with $\beta \in F_0$.

Moreover, since we assumed that there are only finitely many zeros in $S_{M}$, and also that~$g(\cdot +  1)$ is analytic around any $\beta \in F_0$, the function $\psi(\cdot+2s) /f(\cdot+2s-1)$ must be holomorphic in $S_{M - 2s}$. Moreover it has sufficient decay in $S_{M-2s}$ in order to satisfy $\psi(\cdot+2s) /f(\cdot+2s-1) \in \cX_{M',m}$ for some $M' > \max\{ 1 + b , 1 + a  + d  \} $, since we assumed that $f$ does not vanish as $|\im(z)| \to \infty$ within $S_M$. Also, by construction we have $\psi \in \cX_{M,m}$, so we are able to apply \autoref{lemma:PDE-Mellin-2}  with $\psi$.
 
Moreover, note that by application of \autoref{lemma:Mellin-Dirac},
\begin{align*}
\langle \widetilde{\cM} [x^{\beta} (\log(x))^{l} ] , g \rangle = (-1)^l g^{(l)}(\beta + 1) ~~ \forall l \in \{ 0 , \dots , k(\beta) - 1 \},
\end{align*}
where we used that $\overline{M} > \max\{2|\re(\beta)| + \frac{3}{2} , |\re(\beta)| + 1 \}$ for all $\beta \in F_0$. 

Then, by \autoref{lemma:PDE-Mellin-2} applied with $\psi$ we obtain
\begin{align*}
\int_0^{\infty} & v(x) \cM^{-1}[g](x) \d x = \int_0^{\infty} u(x) \cM^{-1}[g](x) \d x - \sum_{\beta \in F_0} \sum_{l = 0}^{k(\beta) - 1}  c_{\beta,l} \langle \widetilde{\cM} [x^{\beta} (\log(x))^{l}] , g \rangle \\
&= \int_0^{\infty} u(x) \cM^{-1}[g - \psi](x) \d x  -  \sum_{\beta \in F_0} \sum_{l = 0}^{k(\beta) - 1} c_{\beta,l} (-1)^l g^{(l)}(\beta +1) \\
&= \sum_{\beta \in F_0} \sum_{l = 0}^{k(\beta) - 1} \frac{g^{(l)}(\beta +1)}{l!} \int_0^{\infty} u(x) \cM^{-1}\left[ \Psi_0(\cdot) \Upsilon_{\beta}(\cdot) (\cdot - 1- \beta)^l \right](x) \d x \\
&\quad - \sum_{\beta \in F_0}  \sum_{l = 0}^{k(\beta) - 1} c_{\beta,l} (-1)^l g^{(l)}(\beta +1) \\
&= \sum_{\beta \in F_0}  \sum_{l = 0}^{k(\beta) - 1} \frac{g^{(l)}(\beta +1)}{l!} \left\{  \int_0^{\infty} u(x) \cM^{-1}\left[ \Psi_0(\cdot) \Upsilon_{\beta}(\cdot) (\cdot - 1- \beta)^l  \right](x) \d x - c_{\beta,l} (-1)^l l! \right\}.
\end{align*}

Hence, if we choose
\begin{align*}
c_{\beta,l} := \frac{\int_0^{\infty} u(x) \cM^{-1}\left[ \Psi_0(\cdot) \Upsilon_{\beta}(\cdot) (\cdot - 1- \beta)^l  \right](x) \d x }{(-1)^l l! },
\end{align*}we deduce that,
for every $\beta \in F_0$ and $l \in \{0,\dots,k(\beta) - 1\}$, 
\begin{align*}
\int_0^{\infty} v(x) \cM^{-1}[g](x) \d x = 0 ~~ \forall g \in \cX_{\overline{M}+1,m}.
\end{align*}
Since by \eqref{eq:u-growth-ass} and \eqref{eq:dual-inclusion} it holds $u \in \cM^{-1}[\cX_{\overline{M}+1,m}]'$, \autoref{lemma:u-zero} implies $v = 0$, and thus we conclude
\begin{align*}
u(x) = \sum_{\beta \in F_0}  \sum_{l = 0}^{k(\beta) - 1}  c_{\beta,l} x^{\beta} (\log(x))^{l}.
\end{align*}
By \eqref{eq:u-growth-ass}, we can discard all summands $\beta \in F_0$ with $\re(\beta) \not\in [0,2s-\eps]$ and conclude the proof.
\end{proof}

\subsection{Proof of \autoref{thm:Liouville-2D}}

The goal of this subsection is to prove the Liouville theorem for operators acting on functions $u : \R \to \C$. The proof is largely analogous to that  of \autoref{thm:Liouville}.

Given $k \in \N$, we define the spaces
\begin{align*}
\cX_{M,m}^k := \big\{ \phi \in \cX_{M,m} : \phi(j+1) = 0 ~~ \forall j \in \{ 0 , \dots, k-1\} \big\}, \qquad \cX_{M,m}^0 := \cX_{M,m}.
\end{align*}

In analogy to \eqref{eq:X-C-inclusion} we deduce from \autoref{lemma:inverse-Mellin} applied to our definition of the Mellin transform from \autoref{subsec:2D-Mellin} that, whenever $k \in \N \cup \{ 0 \}$, $M > \max\{ 1 + b , 1 + a  + d  , 2k + \frac{3}{2} \}$, and $m \ge d+2$,
\begin{align}
\label{eq:X-C-inclusion-2D}
\cM^{-1}[\cX_{M,m}^k \times \cX_{M,m}^k] \subset C^{d,k}_{a,b}(\R).
\end{align}
This property follows immediately from \eqref{eq:X-C-inclusion}, using also the definition of $\widetilde{\cM}$ and \autoref{lemma:Mellin-Dirac}, which yield that for any $j \in \{ 0 , \dots, k-1 \}$,
\begin{align*}
\int_0^{\infty} x^j \cM^{-1}[\phi](x) \d x = \langle \widetilde{\cM} [x^j] , \phi \rangle = \phi(j+1).
\end{align*}

First, we establish the following counterpart of \autoref{lemma:Mellin-magic}.

\begin{lemma}
\label{lemma:Mellin-magic-2D}
Let $L$ be as in \autoref{thm:Liouville-2D} and $k \in \N \cup \{ 0 \}$. Let also $M > \max\{ 1 + b , 1 + a  + d ,2k+\frac{3}{2} \}$,
$m \ge d+2$, and $(\phi_1,\phi_2) \in \cX_{M,m}^k \times \cX_{M,m}^k$.

Then, for every $z \in \C$ with $\re(z) \in (\alpha_1,2s+k -\alpha_2)$,
\begin{align}
\label{eq:Mellin-well-def-2D}
\begin{split}
\cM[L(\cM^{-1}[\phi_1,\phi_2])]^+(z) &=  f_{L,+}(z-1) \phi_1(z-2s) + f_{N,+}(z-1) \phi_2(z-2s), \\
\cM[N(\cM^{-1}[\phi_1,\phi_2]) ]^-(z) &= f_{N,-}(z-1) \phi_2(z-2s) + f_{L,-}(z-1) \phi_1(z-2s).
\end{split}
\end{align}
\end{lemma}

\begin{proof}
By \eqref{eq:X-C-inclusion-2D} and \eqref{eq:L-decay-2D} we have that
\begin{align*}
\cM[L(\cM^{-1}[\phi_1,\phi_2])]^+(z) &= \int_0^{\infty} x_+^{z-1} L(\cM^{-1}[\phi_1]\1_{\R_+} + \cM^{-1}[\phi_2(-\cdot)]\1_{\R_-})(x) \d x \\ 
\cM[N(\cM^{-1}[\phi_1,\phi_2]) ]^-(z) &= \int_{-\infty}^{0} x_-^{z-1} N(\cM^{-1}[\phi_1]\1_{\R_+} + \cM^{-1}[\phi_2(-\cdot)]\1_{\R_-})(x) \d x
\end{align*}
are well-defined for every $z \in \C$ with $\re(z) \in (\alpha_1 , 2s+k-\alpha_2)$. 

Then, using \eqref{eq:self-adjoint-2D} and \eqref{eq:f-2D},
\begin{align*}
\cM & [L(\cM^{-1}[\phi_1,\phi_2])]^+(z) = \int_0^{\infty} x_+^{z-1} L(\cM^{-1}[\phi_1]\1_{\R_+} + \cM^{-1}[\phi_2(-\cdot)]\1_{\R_-})(x) \d x \\
&= \int_0^{\infty} L(x_+^{z-1})(x) \cM^{-1}[\phi_1] (x) \d x  + \int_{-\infty}^0 N(x_+^{z-1})(x) \cM^{-1}[\phi_2(-\cdot)](x) \d x \\
&= f_{L,+}(z-1) \int_0^{\infty}x_+^{z-2s-1} \cM^{-1}[\phi_1](x) \d x + f_{N,+}(z-1) \int_{-\infty}^0 x^{z-2s-1} \cM^{-1}[\phi_2(-\cdot)](x) \d x \\
&= f_{L,+}(z-1) \phi_1(z-2s) + f_{N,+}(z-1) \phi_2(z-2s),
\end{align*}
where we used the Mellin inversion theorem in the last step. An analogous proof yields the result for $\cM[N(\cM^{-1}[\phi_1,\phi_2]) ]^-$.
\end{proof}

The following two lemmas are analogs of \autoref{lemma:PDE-Mellin-1} and \autoref{lemma:PDE-Mellin-2}. Again, the proofs goes in a similar way as before, using \autoref{lemma:Mellin-magic-2D}.

\begin{lemma}
\label{lemma:PDE-Mellin-1-2D}
Let $L$ be as in \autoref{thm:Liouville-2D} and $k \in \N \cup \{ 0 \}$. Assume that $|u(x)| \le C (1+ x)^{2s+k-\eps}$ for some $\eps \in (0,1)$ and that in the distributional sense
\begin{align*}
\begin{cases}
Lu &\overset{k}{=} 0 ~~ \text{ in } (0,\infty),\\
Nu &\overset{k}{=} 0 ~~ \text{ in } (-\infty,0).
\end{cases}
\end{align*}
Then, for any $M > \max\{ 1 + b , 1 + a  + d , 2k + \frac{3}{2} \}$, $m \ge d+2$, $c \in (\alpha_1,2s+k-\alpha_2)$, and $(\phi_1,\phi_2) \in \cX_{M,m}^k \times \cX_{M,m}^k$, it holds
\begin{align*}
0 &= \int_0^{\infty} u(x) \big( \cM^{-1}_c[\phi_1(\cdot -2s) f_{L,+}(\cdot - 1) + \phi_2(\cdot -2s) f_{N,+}(\cdot - 1)](x)\big) \d x \\
&\quad + \int_{-\infty}^0 u(x) \big( \cM^{-1}_c[\phi_1(\cdot -2s) f_{L,-}(\cdot - 1) + \phi_2(\cdot -2s) f_{N,-}(\cdot - 1)](x)\big) \d x.
\end{align*}
\end{lemma}

\begin{proof}
Due to \eqref{eq:X-C-inclusion-2D}, we have that $\cM^{-1}[\phi_1,\phi_2]$ is a valid test function for the equation of $u$ whenever $\phi_1,\phi_2 \in \cX_{M,m}^{k}$. Therefore, we see that
\begin{align*}
0 = \int_0^{\infty} u(x) L(\cM^{-1}[\phi_1,\phi_2])(x) \d x + \int_{-\infty}^{0} u(x) N(\cM^{-1}[\phi_1,\phi_2])(x) \d x,
\end{align*}
and the integrals on the right-hand side both converge absolutely by \eqref{eq:L-decay-2D} (see also \autoref{remark:distr-well-def}). Then, by \autoref{lemma:Mellin-magic-2D} and the Mellin inversion theorem, arguing exactly as in the proof of \autoref{lemma:PDE-Mellin-1}, we deduce that, for any $c \in (\alpha_1 , 2s+k - \alpha_2)$,
\begin{align*}
L(\cM^{-1}[\phi_1,\phi_2]) &= \cM_c^{-1}[\cM[L(\cM^{-1}[\phi_1,\phi_2])]^+] = \cM_c^{-1}[ f_{L,+}(\cdot-1) \phi_1(\cdot-2s) + f_{N,+}(\cdot-1) \phi_2(\cdot-2s)],\\
N(\cM^{-1}[\phi_1,\phi_2]) &= \cM_c^{-1}[\cM[N(\cM^{-1}[\phi_1,\phi_2])]^-] =\cM_c^{-1}[f_{N,-}(\cdot-1) \phi_2(\cdot-2s) + f_{L,-}(\cdot-1) \phi_1(\cdot-2s)].
\end{align*}
A combination of the two previous equations on display concludes the proof.
\end{proof}

\begin{lemma}
\label{lemma:PDE-Mellin-2-2D}
Let $L,u,k,f$ be as in \autoref{thm:Liouville-2D}. Then, for any $M > \max\{ 1 + b , 1 + a  + d  , 2k + \frac{3}{2} , k+2s+2 \}$, $m \ge d+2$, and $\psi \in \cX_{M,m}$ such that $z \mapsto \frac{\psi(z+2s)}{f_1(z+2s-1)} \in \cX_{M,m}^k$, $z \mapsto \frac{\psi(z+2s)}{f_2(z+2s-1)} \in \cX_{M,m}^k$,
\begin{align*}
0 = \int_{0}^{\infty} u(x) \cM^{-1}[\psi](x) \d x.
\end{align*}  
\end{lemma}

\begin{proof}
Let us choose $\phi_2(z) = \frac{f_{L,-}(z+2s-1)}{f_{N,-}(z+2s-1)}\phi_1(z)$, which is a valid test function by assumption. We observe that in that case, by \autoref{lemma:PDE-Mellin-1-2D}, for any $\phi_1 \in \cX_{M,m}^k$ we have that
\begin{align*}
0 = \int_0^{\infty} u(x) \cM_c^{-1} [f_1(\cdot - 1) \phi_1(\cdot -2s)](x) \d x,
\end{align*}
where $f_1 = \frac{f_{L,+} f_{N,-} - f_{L,-} f_{N,+}}{f_{N,-}}$. 

Then, the desired result follows by choosing $\phi_1(z) := \psi(z+2s)/f(z+2s-1)$ in the same way as in \autoref{lemma:PDE-Mellin-2}, using also that, by assumption,
\begin{equation*}
\phi_2(z) = \frac{f_{L,-}(z+2s-1)}{f_{N,-}(z+2s-1)}\phi_1(z) = \frac{f_{L,-}(z+2s-1)}{f_{N,-}(z+2s-1)}\frac{\psi(z+2s)}{f_1(z+2s-1)} = \frac{\psi(z+2s)}{f_2(z+2s-1)} \in \cX_{M,m}^k.\qedhere
\end{equation*}
\end{proof}

Now, we are in a position to give a proof of \autoref{thm:Liouville-2D}.

\begin{proof}[Proof of \autoref{thm:Liouville-2D}]
As in the proof of \autoref{thm:Liouville}, we take $M > \max\{ 1 + b , 1 + a + d , 2k + \frac{3}{2} , k + 2s + 2 \} + 2s$ and $m \ge d + 2$. Then, letting $\overline{M} > 2 (M+1) + \frac{3}{2}$, we pick $g \in \cX_{\overline{M}+1,m}$ and our goal is to define $\psi \in \cX_{M,m}$ such that 
\begin{align*}
P_1(z) := \frac{\psi(z + 2s)}{f_1(z + 2s -1)}, \qquad P_2(z) := \frac{\psi(z+2s)}{f_2(z+2s-1)}
\end{align*}
are holomorphic in $S_{M-2s}$ and satisfy $P_1(j+1) = P_2(j+1) = 0$ whenever $j \in \{ 0, \dots, k - 1 \}$. 

In analogy to the proof of \autoref{thm:Liouville}, we define 
\begin{align*}
F_0 &:= \{ \beta \in \C : f_1(\beta) = 0 ~~ \text{ or } ~~ f_2(\beta) = 0,  ~~ |\re(\beta)| < M \}, \\
F_{\infty} &:= \{ \beta = j + 2s : |f_1(\beta)| + |f_2(\beta)| < \infty, ~~ j \in \{ 1 , \dots, k\} \} \setminus F_0,
\end{align*}
 and set
\begin{align*}
\psi(z) := g(z) & - \Psi_0(z)  \sum_{\beta \in F_0 \cup F_{\infty}} \left( \Upsilon_{\beta}(z)  \sum_{l = 0}^{k(\beta) - 1} \frac{(z - 1 - \beta)^l g^{(l)}(\beta +1)}{l!}  \right),
\end{align*}
where we denote by $k(\beta)$ the maximum of the multiplicities of $\beta$ with respect to $f_1$ and $f_2$, whenever $\beta \in F_0$, and we set $k(\beta) = 1$ whenever $\beta \in F_{\infty}$. 

Then, we let $\Psi_0 \in \cX_{\overline{M}+1,m}$ be such that 
\begin{align*}
\Psi_0(\beta+1) = 1 ~~ \forall \beta \in F_0 \cup F_{\infty}, \qquad \Psi_0^{(l)}(\beta+1) = 0 ~~ \forall l \in \{ 1, \dots, k(\beta) - 1 \} ~~ \forall \beta \in F_0,
\end{align*}
and for any $\beta \in F_0$ the function $\Upsilon_{\beta} \in \cX_{\overline{M}+1,m}$ is such that 
\begin{align*}
\Upsilon_{\beta}(\alpha + 1) = \delta_{\alpha,\beta}  ~~ \forall \alpha \in F_0 \cup F_{\infty}, \qquad \Upsilon_{\beta}^{(l)}(\alpha + 1) = 0 ~~ \forall l \in \{ 0 , \dots , k(\beta) - 1 \} ~~ \forall \alpha \in F_0.
\end{align*}
Such functions $\Psi_0$ and $\Upsilon_{\beta}$ do exist by \autoref{lemma:prescribed-values}. Note that $\psi$ is constructed in such a way that
\begin{align*}
\psi^{(l)}(\beta + 1) = 0 ~~ \forall \beta \in F_0 \cup F_{\infty}, ~~ l \in \{ 0 , \dots , k(\beta) - 1\},
\end{align*}
and, in particular, $P_1$ and $P_2$ satisfy all the aforementioned properties so that we can apply \autoref{lemma:PDE-Mellin-2-2D} with $\psi$. From here, the proof goes in the exact same way as the proof of \autoref{thm:Liouville}, choosing
\begin{align*}
v(x) := u(x) - \sum_{\beta \in F_0 \cup F_{\infty}}  \sum_{l = 0}^{k(\beta) - 1}  c_{\beta,l} x^{\beta} (\log(x))^{l}  ,
\end{align*}
where the constants $c_{\beta,l} \in \R$ will be chosen as in the proof of \autoref{thm:Liouville} so that we can deduce
\begin{align*}
\int_0^{\infty} v(x) \cM^{-1}[g](x) \d x = 0.
\end{align*} 
Since $g \in \cX_{\overline{M}+1,m}$ was arbitrary, we deduce $v = 0$ by \autoref{lemma:u-zero} and conclude the proof.
\end{proof}

\section{1D Liouville theorems for the nonlocal Neumann problem}
\label{sec5}

The goal of this section is to prove a 1D Liouville theorem for the nonlocal Neumann problem
\begin{align}
\label{eq:Neumann-1D-prep}
\begin{cases}
(-\Delta)^s_{\R} u &= 0 ~~ \text{ in } (0,\infty),\\
\mathcal{N}^s_{(0,\infty)} u &= 0 ~~ \text{ in } (-\infty,0).
\end{cases}
\end{align}
In \cite{AFR23}, it was already proved that solutions which do not grow faster than $t^{2s-1+\eps}$ at infinity are constant, in case $s \in (\frac{1}{2},1)$. With the help of \autoref{thm:Liouville}, we can finally give a {\em
complete classification of all solutions to the nonlocal Neumann problem in 1D}.

Since we are ultimately interested in the optimal regularity of solutions to the nonlocal Neumann problem, it will be sufficient to classify solutions that do not grow faster than $t^{2s+1-\eps}$ at infinity. 
Note that our proof distinguishes between solutions that do not grow faster than $t^{2s-\eps}$ at infinity and those that do. 

In case $u(t) \lesssim t^{2s-\eps}$ at infinity we work with weak solutions to \eqref{eq:Neumann-1D-prep} and we rewrite \eqref{eq:Neumann-1D-prep} in terms of a regional problem on $(0,\infty)$. Indeed, we recall from \cite[Proposition A.3]{AFR23} that $u$ is also a weak solution to 
\begin{align*}
L u = 0 ~~ \text{ in } (0,\infty),
\end{align*}
where
\begin{equation}\label{hghg}
Lu(x) = \text{p.v.} \int_{0}^{\infty} (u(x) - u(y)) K_{\R_+}(x,y) \d y,
\end{equation}
and
\begin{equation}\label{hghg2} \begin{split}
K_{\R_+}(x,y) &= \frac{c_s}{|x-y|^{1+2s}} + k_{\R_+}(x,y), \qquad c_s = 4^s s \frac{\Gamma(s + \frac{1}{2})}{\Gamma(1 -s)} \pi^{-\frac{1}{2}}, \\
k_{\R_+}(x,y) &= 2s c_s \int_{0}^{\infty} \frac{|z|^{2s}}{|x+z|^{1+2s} |y+z|^{1+2s}} \d z.
\end{split}\end{equation}

By the homogeneity of the kernel, i.e. $K_{\R_+}(\lambda x,\lambda y) = \lambda^{-1-2s} K_{\R_+}(x,y)$, it follows immediately that $L$ is homogeneous of order $2s$ in the sense of \eqref{eq:hom}, and we will use \autoref{thm:Liouville} to classify all such solutions.

Here, we define weak solutions of $Lu = 0$ as follows:
\begin{definition}
\label{def:weak-sol-1D}
Given an open, bounded interval $B \subset \R$, we say that $u \in H_{K,\loc}$ is a weak solution to $L_{\Omega} u = 0$ in $B \cap (0,\infty)$ with Neumann conditions on $\{0\} \cap B$, if, for any $\eta \in H_K(B)$ with $\eta \equiv 0$ in $(0,\infty) \setminus B$, 
\begin{align*}
B_{\R_+}(u,\eta):= \int_{\R_+}\int_{\R_+} (u(x) - u(y)) (\eta(x) - \eta(y)) K_{\R_+}(x,y)\d y \d x = 0,
\end{align*}
and denote by $H_K(\Omega)$ the space of functions 
\begin{align*}
H_K(\Omega) := \left\{ w \in H^s(\Omega) ~~\big|~~ \Vert w \Vert_{H_K(\Omega)} := \Vert w \Vert_{L^2(\Omega)} + B_{\R_+}(u,u) < \infty \right \}.
\end{align*}
Moreover, $H_{k,\loc}$ is the space of functions $u \in L^2_{\loc}(0,\infty)$ for which $B_{\R \cap B}(u,u) < \infty$ for any $B \subset \R^n$.\\
We say that $u$ is a weak solution to $Lu = 0$ in $(0,\infty)$, if $u$ is a weak solution to $L_{\Omega} u = 0$ in $B \cap (0,\infty)$ with Neumann conditions on $\{0\} \cap B$ for any $B \subset \R$.
\end{definition}

Note that this definition is essentially equivalent to \autoref{def:weak-sol} in case $n=1$ and $h = 0$.

Weak solutions to \eqref{eq:Neumann-1D-prep} are defined as follows:

\begin{definition}
\label{def:weak-sol-Neumann-1D}
Given two open, bounded intervals $\Omega, B \subset \R$, we say that $u \in V^s(\Omega | \R)$ is a weak solution to 
\begin{align}
\label{eq:Neumann-1D-prep-2}
\begin{cases}
(-\Delta)^s u &= 0 ~~ \text{ in } \Omega \cap B,\\
\cN^s_{\Omega} u &= 0 ~~ \text{ in } B \setminus \Omega,
\end{cases}
\end{align}
if for all $\eta \in H^s(\R)$ with $\eta \equiv 0$ in $B^c$ it holds
\begin{align*}
\cE_{\Omega}(u,\eta) := \frac{c_{1,s}}{2} \iint_{(\Omega^c \times \Omega^c)^c} (u(x) - u(y)) (\eta(x) - \eta(y)) |x-y|^{-1-2s} \d y \d x = 0,
\end{align*}
where $V^s(\Omega | \R)$ is the space of all functions $u$ such that $\cE_{\Omega}(u,u) < \infty$. Moreover, we say that $u$ is a weak solution to \eqref{eq:Neumann-1D-prep}, if it is a weak solution to \eqref{eq:Neumann-1D-prep-2} with $\Omega = (0,\infty)$ for any ball $B \subset \R$.
\end{definition}

Note that this definition is essentially equivalent to \autoref{def:weak-sol-Neumann} in case $n = 1$ and $h = 0$.

In order to apply \autoref{thm:Liouville}, we have to determine the corresponding function $f$ (see \eqref{eq:f}). This is achieved in the following result, which we will prove in the next subsection.

\begin{lemma}
\label{lemma:f-def}
Let $L$ be as above. It holds for $\beta \in \C$:
\begin{align*}
f(\beta) &= \frac{\Gamma(\beta + 1)}{\Gamma(\beta - 2s + 1)} \frac{\sin(\pi(\beta-s))}{\sin(\pi(\beta -2s))} - 2s^2 4^s \frac{\Gamma(s + \frac{1}{2})}{\Gamma(1 -s)} \pi^{-\frac{1}{2}} \left( \frac{\Gamma(2s-\beta) \Gamma(\beta+1)}{\Gamma(1+2s)} \right)^2\\
&= \frac{\Gamma(\beta + 1) \sin(\pi s)}{\Gamma(\beta - 2s + 1)\sin(\pi(\beta -2s))}  \left(  \frac{\sin(\pi(\beta-s))}{\sin(\pi s)} + \frac{\Gamma(2s-\beta) \Gamma(\beta+1)}{\Gamma(2s)} \right).
\end{align*}
\end{lemma}

The main work is to locate the zeros of the function $f$. The corresponding result is given in \autoref{prop:zeros-final} and its proof will also be postponed to \autoref{subsec:zeros}.


We denote by 
\begin{align}
\label{eq:B0}
B_0: = \inf\{ \re(\beta) : f(\beta) = 0, ~~ \re(\beta) > 0, ~~ \beta \not= 2s-1 \}
\end{align}
the real part of the first (non-trivial) zero of $f$.
In particular, by \autoref{prop:zeros-final}, it holds $B_0 \in (2s, 2s+\frac{1}{2})$ if $s \le \frac{1}{2}$ and $B_0 \in (s+\frac{1}{2},s+1)$ if $s \ge \frac{1}{2}$.

It is of some interest to locate the zeros of $f$ at a higher level of precision. Although this is in general very complicated, the following two remarks provide exact information about the zeros of $f$ for $s = \frac{1}{2}$, as well as $s \searrow 0$ and $s \nearrow 1$.

\begin{remark}
\label{rem:s-half}
In case $s = \frac{1}{2}$, the identity $f(\beta) = 0$ is equivalent to
\begin{align*}
\sin(2\pi \beta) = 2 \pi \beta,
\end{align*}
which is known to have the solution $\beta = 0$ and no further solution in $\{ \beta \in \C : \re(\beta) \in [0,1] \}$. However, there are further solutions outside this strip (in fact, there are infinitely many of them). The one with smallest real part is given by $\beta = 1.193292\ldots + i 0.4406488\ldots$,
see \cite{Hil43}.
\end{remark}

\begin{remark}
\label{rem-s-asymp}
We observe that (see also \eqref{eq:zero-equivalence}) finding the nontrivial zeros of $f$ is equivalent to finding the nontrivial zeros of
\begin{align*}
F(s,\beta):=
(2s-\beta)\sin(\pi(s-\beta))-\frac{\sin(\pi s)\Gamma(2s+1-\beta)\,\Gamma(\beta+1)}{\Gamma(2s)}.
\end{align*}
Then, there exist $\eta_1, \eta_2 >0$ sufficiently small such that, for all $s\in[1-\eta_1,1]$ and $s \in [0,\eta_2]$, there exist~$\beta_1(s) , \beta_2(s) \in\C$ satisfying
$$ F(s,\beta_1(s))=0 = F(s,\beta_2(s)).$$
Moreover, it holds
\begin{align*}
\beta_1(s)=2\pm i\sqrt{2(1-s)}-3(1-s)+o(1-s) ~~ \text{ as } s \to 1, \qquad \beta_2(s)=3 s+o(s) ~~ \text{ as } s \to 0.
\end{align*}
We provide the proof of this remark at the end of \autoref{subsec:zeros}.
\end{remark}

We can show the following Liouville theorem in case $u$ does not grow faster than $t^{2s-\eps}$ at infinity.

\begin{theorem}
\label{thm:Neumann-Liouville}
Let $u$ be a weak solution to
\begin{align}
\label{eq:Neumann-1D}
\begin{cases}
(-\Delta)^s_{\R} u &= 0 ~~ \text{ in } (0,\infty),\\
\mathcal{N}^s_{(0,\infty)} u &= 0 ~~ \text{ in } (-\infty,0)
\end{cases}
\end{align}
in the sense of \autoref{def:weak-sol-Neumann-1D} with $u(0) = 0$ and
\begin{align*}
|u(x)| \le C (1 + x)^{\min\{B_0,2s\}-\eps}
\end{align*}
for some $C > 0$ and $\eps > 0$. Then,
\begin{align*}
u = 0 ~~ \text{ in } [0,\infty ).
\end{align*}
\end{theorem}

In addition, in case $u$ grows faster than $t^{2s}$ at infinity, we first need to give a proper meaning to \eqref{eq:Neumann-1D-prep} since weak solutions to \eqref{eq:Neumann-1D-prep} are not defined anymore.

\begin{definition}[Distributional solutions with faster growth]
\label{def:faster-growth}
We say that $u$ satisfying
\begin{align*}
|u(x)| \le C (1 + |x|)^{2s+1 -\eps} ~~ \forall x \in \R
\end{align*}
for some $\eps > 0$ is a distributional solution with faster growth to
\begin{align}
\label{eq:faster-growth}
\begin{cases}
(-\Delta)_{\R}^s u &\overset{1}{=} 0 ~~ \text{ in } (0,\infty),\\
\mathcal{N}^s_{(0,\infty)} u &\overset{1}{=} 0 ~~ \text{ in } (-\infty,0),
\end{cases}
\end{align}
if, for any $\eta \in C^{\infty}_c(\R)$ with $\int_0^{\infty} \eta(x) \d x = \int_{-\infty}^0 \eta(x) \d x = 0$, it holds that
\begin{align*}
\int_{-\infty}^0 u(x) \mathcal{N}^s_{(0,\infty)} \eta(x) \d x + \int_0^{\infty} u(x) (-\Delta)_{\R}^s \eta(x) \d x = 0.
\end{align*}
\end{definition}

Clearly, both operators $(-\Delta)^s_{\R}$ and $\mathcal{N}_{(0,\infty)}^s$ are homogeneous of order $2s$. In case $u$ solves \eqref{eq:Neumann-1D-prep} in the sense of \autoref{def:faster-growth} we will use \autoref{thm:Liouville-2D} to classify all such solutions.

It turns out that the characterizing functions $f_1$ and~$f_2$ are the same (or closely related) to the function $f$ from \autoref{lemma:f-def}, as pointed out in the following result.

\begin{lemma}
\label{lemma:f-def-2D}
Let $L = (-\Delta)_{\R}^s$ and $N = \mathcal{N}^s_{(0,\infty)}$. Then, it holds for $\beta \in \C$:
\begin{align*}
f_1(\beta) = f(\beta),\qquad f_2(\beta) = - \frac{\Gamma(1+2s)}{2s \Gamma(2s-\beta) \Gamma(\beta+1)} f(\beta).
\end{align*}
\end{lemma}

We will provide the proof in the next subsection. As a direct consequence of \autoref{prop:zeros-final} we get the following result.

\begin{corollary}
\label{cor:zeros-f2}
Let $f_1,f_2$ be as above. Then, the following hold true:
\begin{itemize}
\item[(i)] $\{ f_1 = 0 \} = \{ f = 0 \} $,
\item[(ii)] $\{ f_2 = 0 \} = \{ f = 0\}$,
\item[(iii)] $|f_1(2s)| = |f_2(2s)| = \infty$.
\end{itemize}
\end{corollary}

\begin{proof}
The first claim is trivial since $f_1 = f$. To prove the second and third claims it suffices to observe that $f(\beta)$ has a pole of order two at $\beta = 2s+j$ for any $j \in \N \cup\{ 0 \}$ and therefore in particular $|f_2(2s)| = |f_1(2s)| = \infty$.
\end{proof}

Therefore, we are able to show the following {\em Liouville theorem for distributional solutions of \eqref{eq:Neumann-1D-prep} with faster growth}.

\begin{theorem}
\label{thm:Neumann-Liouville-growth}
Let $u \in C^{\max\{2s-1,0\}+\eps}_{loc}(\R)$ be a distributional solution with faster growth to
\begin{align*}
\begin{cases}
(-\Delta)^s u &\overset{1}{=} 0 ~~ \text{ in } (0,\infty),\\
\mathcal{N}^s_{\R_+} u &\overset{1}{=} 0 ~~ \text{ in } (-\infty,0),
\end{cases}
\end{align*}
with $u(0) = 0$ and
\begin{align*}
|u(x)| \le C (1 + x)^{B_0-\eps}
\end{align*}
for some $C > 0$ and $\eps > 0$. Then, 
\begin{align*}
u = 0 ~~ \text{ in } \R.
\end{align*}
\end{theorem}

Here, we can rule out $u = |x|^{2s-1}$ since we assume $u \in C^{2s-1+\eps}_{loc}(\R)$, which we can also do in the application.

\subsection{Determining the Mellin symbols}

The goal of this subsection is to prove \autoref{lemma:f-def} and \autoref{lemma:f-def-2D}, and also
to establish some basic properties of the function $f$ (and therefore also of $f_1$ and~$f_2$).

\begin{proof}[Proof of \autoref{lemma:f-def}]
Since $L$ is homogeneous of order $2s$, by \autoref{remark:hom} we have  $f(\beta) = L(x^{\beta})(1)$. Moreover, 
\begin{align}
\label{eq:k_R-int}
\begin{split}
\int_0^{\infty} k_{\R_+}(x,y) \d y &= c_s \int_0^{\infty} \left(\int_{-\infty}^0 \frac{|x-z|^{-1-2s} |y-z|^{-1-2s}}{\int_0^{\infty} |z-\bar{z}|^{-1-2s} \d \bar{z}} \d z \right) \d y \\
&= c_s\int_{-\infty}^0 \frac{\int_0^{\infty} |z-y|^{-1-2s} \d y}{|x-z|^{1+2s} \int_0^{\infty} |z-\bar{z}|^{-1-2s} \d \bar{z}} \d z = c_s\int_{-\infty}^0 |x-z|^{-1-2s} \d z = c_s \frac{x^{-2s}}{2s}.
\end{split}
\end{align}

We denote, for all $\beta \in \C$,
\begin{align}
\label{eq:C-beta}
C_{\beta} := \int_0^{\infty} t^{\beta} (1 + t)^{-1-2s} \d t = \frac{\Gamma(2s-\beta) \Gamma(\beta+1)}{\Gamma(1+2s)},
\end{align}
and compute:

\begin{align*}
L(x^{\beta}) &= (-\Delta)^s_{\R_+} (x^{\beta}) + \int_{0}^{\infty} (x^{\beta} - y^{\beta}) k_{\R_+}(x,y) \d y \\
&= (-\Delta)^s_{\R_+} (x^{\beta}) + c_s \frac{x^{\beta-2s}}{2s} -  2s c_s \int_0^{\infty}z^{2s} \left( \int_0^{\infty} y^{\beta} (z + y)^{-1-2s} \d y \right) (x+z)^{-1-2s} \d z.
\end{align*}
We recall from \cite[equation~(2.8)]{FaRo22} that
\begin{align*}
(-\Delta)^s_{\R_+} (x^{\beta})(1) = (-\Delta)^s(x_+)^{\beta}(1) - c_s \int_{-\infty}^0 |1-y|^{-1-2s} \d y = \frac{\Gamma(\beta + 1)}{\Gamma(\beta - 2s + 1)} \frac{\sin(\pi(\beta-s))}{\sin(\pi(\beta -2s))} - \frac{c_s}{2s} .
\end{align*}We point out that
\begin{align*}
2s c_s & \int_0^{\infty} z^{2s} \left( \int_0^{\infty} y^{\beta} (z + y)^{-1-2s} \d y \right) (1+z)^{-1-2s} \d z \\
&= 2s c_s \int_0^{\infty} z^{\beta-1} \left( \int_0^{\infty} \left(\frac{y}{z} \right)^{\beta} \left(1 + \frac{y}{z} \right)^{-1-2s} \d y \right) (1+z)^{-1-2s} \d z \\
&= 2s c_s \int_0^{\infty} z^{\beta} \left( \int_0^{\infty} t^{\beta} \left(1 + t \right)^{-1-2s} \d y \right) (1+z)^{-1-2s} \d z \\
&= 2s c_s C_{\beta}^2.
\end{align*}
Altogether, we have proved
\begin{align*}
f(\beta) = L(x^{\beta})(1) &= \frac{\Gamma(\beta + 1)}{\Gamma(\beta - 2s + 1)} \frac{\sin(\pi(\beta-s))}{\sin(\pi(\beta -2s))} - 2s c_s C_{\beta}^2 \\
&= \frac{\Gamma(\beta + 1)}{\Gamma(\beta - 2s + 1)} \frac{\sin(\pi(\beta-s))}{\sin(\pi(\beta -2s))} - 2s^2 4^s \frac{\Gamma(s + \frac{1}{2})}{\Gamma(1 -s)} \pi^{-\frac{1}{2}} \left( \frac{\Gamma(2s-\beta) \Gamma(\beta+1)}{\Gamma(1+2s)} \right)^2,
\end{align*}
as desired. 

Let us now derive the second formula for $f$. First, observe that we can write
\begin{align*}
f(\beta) &= \frac{\Gamma(\beta + 1) \sin(\pi s)}{\Gamma(\beta - 2s + 1)\sin(\pi(\beta -2s))} \Bigg( \frac{\sin(\pi(\beta-s))}{\sin(\pi s)} \\
&\qquad\qquad\qquad - 2s^2 4^s \frac{\Gamma(s + \frac{1}{2})}{\Gamma(1 -s)} \pi^{-\frac{1}{2}} \frac{\Gamma(2s-\beta)^2 \Gamma(\beta+1) \sin(\pi(\beta-2s)) \Gamma(\beta - 2s + 1)}{\Gamma(1+2s)^2 \sin(\pi s)} \Bigg).
\end{align*}

By elementary transformations of the Gamma-function, namely using in the first step that $\Gamma(1+2s) = 2s \Gamma(2s)$ and $\Gamma(s + \frac{1}{2}) = \frac{2}{4^s} \pi^{1/2} \frac{\Gamma(2s)}{\Gamma(s)}$, and in the second step that $\Gamma(1-s) = \frac{\pi}{\Gamma(s) \sin(\pi s)}$ and $\Gamma(2s-\beta) = -\frac{\pi}{\Gamma(\beta-2s+1)\sin(\pi(\beta - 2s))}$, we obtain
\begin{align*}
- 2s^2 4^s & \frac{\Gamma(s + \frac{1}{2})}{\Gamma(1 -s)} \pi^{-\frac{1}{2}} \frac{\Gamma(2s-\beta)^2 \Gamma(\beta+1) \sin(\pi(\beta-2s)) \Gamma(\beta - 2s + 1)}{\Gamma(1+2s)^2 \sin(\pi s)} \\
&= - \frac{1}{\Gamma(s)\Gamma(1 -s)} \frac{\Gamma(2s-\beta)^2 \Gamma(\beta+1) \sin(\pi(\beta-2s)) \Gamma(\beta - 2s + 1)}{\Gamma(2s) \sin(\pi s)} = \frac{\Gamma(2s-\beta) \Gamma(\beta+1)}{\Gamma(2s)},
\end{align*}
whence we deduce the desired formula for $f$.
\end{proof}

\begin{proof}[Proof of \autoref{lemma:f-def-2D}]
We recall from \cite[equation~(2.8)]{FaRo22} that
\begin{align*}
f_{L,+}(\beta) = (-\Delta)^s (x_+^{\beta})(1) = \frac{\Gamma(\beta + 1)}{\Gamma(\beta - 2s + 1)} \frac{\sin(\pi(\beta-s))}{\sin(\pi(\beta -2s))}.
\end{align*}
Thus we compute
\begin{align*}
f_{L,-}(\beta) = (-\Delta)^s (x_-^{\beta})(1) = -c_s \int_{-\infty}^0 y_-^{\beta} |1 - y|^{-1-2s} \d y = - c_s C_{\beta} ,\\
f_{N,+}(\beta) = \mathcal{N}^s_{\R_+} (x_+^{\beta})(-1) = - c_s \int_0^{\infty} y^{\beta} |1 + y|^{-1-2s} \d y = - c_s C_{\beta},
\end{align*}
where $C_{\beta}$ was defined in \eqref{eq:C-beta}. 

Moreover, we note that
\begin{align*}
f_{N,-}(\beta) =  \mathcal{N}^s_{\R_+} (x_-^{\beta})(-1) = c_s \int_0^{\infty} |1 + y|^{-1-2s} \d y = \frac{c_s}{2s}.
\end{align*}
Therefore, we get as desired,
\begin{equation*}\begin{split}
f_1(\beta) &= \frac{f_{L,+}(\beta)f_{N,-}(\beta) - f_{L,-}(\beta) f_{N,+}(\beta)}{f_{N,-}(\beta)} = \frac{\Gamma(\beta + 1)}{\Gamma(\beta - 2s + 1)} \frac{\sin(\pi(\beta-s))}{\sin(\pi(\beta -2s))} - 2s c_s C_{\beta}^2 = f(\beta),\\
f_2(\beta) &= \frac{f_{N,-}(\beta)}{f_{L,-}(\beta)} f_1(\beta) = - (2sC_{\beta})^{-1} f(\beta) = - \frac{\Gamma(1+2s)}{2s \Gamma(2s-\beta) \Gamma(\beta+1)} f(\beta).\qedhere
\end{split}\end{equation*}
\end{proof}

In order to apply \autoref{thm:Liouville}, we also need the following result:
\begin{lemma}
\label{lemma:f-prop}
$f$ is meromorphic. Moreover, for any $M > 0$ there are only finitely many zeros of $f$ in the strip $S_M = \{ z \in \C : |\re(z)| < M \}$, and for any sequence $(z_n) \subset S_M$ with $|\im(z_n)| \to \infty$ it holds that $$\liminf_{n \to \infty} |f(z_n)| > 0.$$ The same properties hold for $f_1$ and~$f_2$.
\end{lemma}

\begin{proof}
It is easy to see from the definition that $f$ is meromorphic. To prove the other results, we heavily rely on an asymptotic analysis. 
First, since the only poles of the Gamma-function occur at the non-positive integers, and since the Gamma-function has no zeros, for the prefactor of $f$ we see that, for all~$\beta \in \C$,
\begin{align*}
\frac{\Gamma(\beta + 1) \sin(\pi s)}{\Gamma(\beta - 2s + 1)\sin(\pi(\beta -2s))} = 0 \Leftrightarrow \beta \in 2s -\N.
\end{align*} 
Hence, we see that, for all~$\beta \in \C$ with $\beta \not\in 2s - \N$ (and, in fact,
$\beta =2s-1$ is still a solution of the equivalences below):
\begin{align}
\label{eq:zero-equivalence}
f(\beta) = 0 &\Leftrightarrow  \frac{\sin(\pi(s - \beta))}{\sin(\pi s)} = \frac{\Gamma(2s-\beta) \Gamma(\beta+1)}{\Gamma(2s)} \\
\label{eq:zero-equivalence-2}
&\Leftrightarrow \frac{\sin(\pi(\beta-s))\sin(\pi(\beta-2s)) \Gamma(2s)}{\sin(\pi s) \pi} = \frac{\Gamma(\beta+1)}{\Gamma(\beta-2s+1)}.
\end{align}
Indeed, it is easy to see that the property \eqref{eq:zero-equivalence} holds true for $\beta = 0$. Let us now pcik~$M > 0$ and investigate solutions to \eqref{eq:zero-equivalence-2} inside the strip $S_M = \{ z \in \C : |\re(z)| < M \}$. By Stirling's formula (see \cite[equation~(8.328)]{GrRy07}) it holds 
\begin{align}
\label{eq:Gamma-asymp}
\lim_{|y| \to \infty} |\Gamma(x+iy)| \exp(\pi |y|/2) |y|^{\frac{1}{2} - x} = \sqrt{2\pi}.
\end{align}
As a result, writing $\beta = a + ib$ with~$a,b\in\R$,
\begin{align}
\label{eq:Gamma-quotient-asymp}
\left| \frac{\Gamma(\beta+1)}{\Gamma(\beta+1-2s)} \right| \asymp |b|^{2s} ~~ \text{ as } |b| \to \infty.
\end{align}
Moreover, by the definition of the complex sine function, one can immediately read off that
\begin{align}
\label{eq:sin-asymp}
C^{-1} \exp(\pi |y|) \le |\sin(\pi(x + iy))| \le C \exp(\pi |y|) ~~ \text{ for } |y| \ge 1,
\end{align}
and therefore
\begin{align}
\label{eq:sin-product-asymp}
C(s)^{-1} \exp(2\pi|b|) \le \left|\frac{\sin(\pi(\beta-s))\sin(\pi(\beta-2s)) \Gamma(2s)}{\sin(\pi s) \pi}\right| \le C(s) \exp(2\pi|b|) ~~ \text{ for } |b| \ge 1.
\end{align}
Hence, the asymptotic behavior of the two functions in \eqref{eq:Gamma-quotient-asymp} and \eqref{eq:sin-product-asymp} does not match for large imaginary part $|b|$. Thus, since meromorphic functions can only have finitely many zeros in any bounded domain in $\C$, \eqref{eq:zero-equivalence-2} can only have finitely many solutions inside $S_M$. This means that there are only finitely many zeros of $f$ inside $S_M$.

Using the identity $\Gamma(2s-\beta) = \pi [ \Gamma(\beta - 2s + 1) \sin(\pi(2s-\beta)) ]^{-1}$, we can rewrite
\begin{align*}
f(\beta) = -\sin(\pi s) \pi \left(\frac{\Gamma(\beta + 1)}{\Gamma(\beta - 2s + 1)\sin(\pi(\beta -2s))} \right)^2  + \frac{\Gamma(\beta + 1) \sin(\pi(\beta - s))}{\Gamma(\beta - 2s + 1)\sin(\pi(\beta -2s))}.
\end{align*}
By \eqref{eq:Gamma-quotient-asymp} and \eqref{eq:sin-asymp} the first summand vanishes as $|b| \to \infty$, while the second summand behaves like $|b|^{2s}$ as $|b| \to \infty$. In particular, for any sequence $(z_n) \subset S_M$ with $|\im(z_n)| \to \infty$, it holds that~$\liminf_{n \to \infty}|f(z_n)| > 0$, as desired.
\end{proof}

\subsection{Finding the zeros of $f$}
\label{subsec:zeros}

The goal of this subsection is to prove \autoref{prop:zeros-final}. Moreover, at the very end of this subsection we will give a proof of \autoref{rem-s-asymp}.

 Let us recall from \eqref{eq:zero-equivalence} that finding the nontrivial zeros $\beta \not\in \{0 , 2s-1\}$ is equivalent to finding complex non-trivial zeros $\beta \in \{ \re(\beta) > 0\}$ of 
\begin{align}
\label{LAFUNZ}
g(z):=\frac{\Gamma(2s-z)\,\Gamma(z+1)}{\Gamma(2s)}-\frac{\sin(\pi(s-z))}{\sin(\pi s)}.
\end{align}

\subsubsection{Some classical results from complex analysis}

In the forthcoming \autoref{lemma:COMPLEM2} and \autoref{lemma:COMPLEM1} we present
some versions of classical results in complex analysis which will be used in our framework. We begin with a restatement of the folk wisdom that
if the image of the boundary of a region under a holomorphic function does not wind around the origin, then the function has no zeros inside the region. The precise statement that we need is as follows:

\begin{lemma}\label{lemma:COMPLEM2}
Let~$\Omega\subset\C$
be a bounded, open, simply connected set, with~$\partial\Omega$ of Lipschitz class.

Suppose that~$g$ is a holomorphic function in a neighborhood of~$\overline\Omega$.

If~$g$ has no zeros on~$\partial\Omega$ and the curve~$g(\partial\Omega)$ can be decomposed into a finite number of closed curves~$C_1,\dots,C_N$, such that~$C_i\cap C_j$ is a finite number of points for all~$i\ne j$, and each~$C_i$ is homotopic to a constant curve in the punctured complex plane~$\C\setminus\{0\}$, then~$g$ has no zeros in~$\Omega$.
\end{lemma}

\begin{proof} On the one hand, the winding number of~$g(\partial\Omega)$ coincides with the sum of the winding numbers
of the components~$C_i$. On the other hand,
the homotopy invariance of the winding number yields that the winding number of each component~$C_i$ is null.

All in all, the winding number of~$g(\partial\Omega)$ is null and the desired result thus follows from
the argument principle.
\end{proof}

\begin{lemma}\label{lemma:COMPLEM1}
Let~$\Omega\subset\C$
be a bounded, open, simply connected set, with~$\partial\Omega$ of Lipschitz class.

Suppose that~$g$ is a holomorphic function in a neighborhood of~$\overline\Omega$.
Assume also that
\begin{equation}\label{0PRELIH}
{\mbox{$\im( g(z))\ge\min\big\{0,c_1-c_2|\re( g(z))|\big\}$
for every~$z\in\partial\Omega$,}}
\end{equation}for some~$c_1$, $c_2>0$.

Then, $g(z)\ne0$ for every~$z\in\Omega$, unless~$g$ vanishes identically.
\end{lemma}

\begin{proof} Let us first suppose that, for every~$z\in\partial\Omega$,
\begin{equation}\label{PRELIH}
{\mbox{$\im( g(z))\ge\epsilon+\min\big\{0,c_1-c_2|\re( g(z))|\big\}$
for every~$z\in\partial\Omega$,}}
\end{equation}for some~$\epsilon>0$.

We consider the closed curve~$\partial\Omega$ and we
claim that\begin{equation}\label{ZOSQWDF}{\mbox{$g$ has no zeros on~$\partial\Omega$.}}\end{equation}
Indeed, by contradiction, if~$g(z_0)=0$ for some~$z_0\in\partial\Omega$,
we use~\eqref{PRELIH}
to see that
$$0=\im( g(z_0))\ge\epsilon+\min\big\{0,c_1-c_2|\re( g(z_0))|\big\}=
\epsilon+\min\big\{0,c_1\big\}=\epsilon,$$
which is impossible and proves~\eqref{ZOSQWDF}.

Then, \eqref{ZOSQWDF} allows us to employ \autoref{lemma:COMPLEM2}: indeed,
by~\eqref{PRELIH}, the curve~$g(\partial\Omega)$ does not
encircle the origin and therefore
we obtain that~$g$ has no zeros in~$\Omega$,
proving the desired claim in this setting.

Let us now suppose that~\eqref{0PRELIH} holds true. We pick~$\epsilon>0$ and let~$g_\epsilon:=g+i\epsilon$. Now, condition~\eqref{PRELIH} is fulfilled by~$g_\epsilon$ and we thereby know that~$g_\epsilon$ has no zeros in~$\Omega$.
We need to show that~$g$ has no zeros in~$\Omega$ too and, for this purpose, we argue by contradiction, supposing that~$g(z_0)=0$ for some~$z_0\in\Omega$, without~$g$ vanishing identically.

By the analyticity of~$g$, we have that~$z_0$ is isolated and therefore there exists~$\rho>0$ small enough such that~$g$ has no zeros on the circle~$C_\rho$ centered at~$z_0$ and of radius~$\rho$. As a result, by
the argument principle and the dominated convergence theorem,
$$ 1\le \oint_{C_\rho}\frac{g'(z)}{g(z)}\,dz=\lim_{\epsilon\to0}
\oint_{C_\rho}\frac{g'(z)}{g_\epsilon(z)}\,dz=\lim_{\epsilon\to0}
\oint_{C_\rho}\frac{g'_\epsilon(z)}{g_\epsilon(z)}\,dz=0,$$
which is a contradiction.
\end{proof}

\subsubsection{Four preparatory lemmas}

Following are the applications of \autoref{lemma:COMPLEM2} and \autoref{lemma:COMPLEM1} (combined
with suitable contour integration arguments)
to some cases of interest.
We first state all these consequences as separate results and then we present
their proofs. At the end of this subsection, we explain how these results imply \autoref{prop:zeros-final}.

\begin{lemma}\label{LEM:CONSEG:1}
Let $g$ be as in \eqref{LAFUNZ}. Assume that~$s\in\left(\frac12,1\right)$. Then,
there exists~$M>0$ sufficiently large (possibly in dependence of~$s$) such that~$g$ has no zeros in the complex region
\begin{equation}\label{LAREG1} \left\{ {\mbox{$z\in\C$ s.t. $\re z\in\left(2s-1,s+\displaystyle\frac12
\right)$ and~$\im z\in[0,M]$}}\right\}.\end{equation}
\end{lemma}

\begin{lemma}\label{LEM:CONSEG:2}
Let~$g$ be as in~\eqref{LAFUNZ}. Assume that~$s\in\left(0,\frac12\right)$. Then,
there exists~$M>0$ sufficiently large (possibly in dependence of~$s$) such that~$g$ has no zeros in the complex region
\begin{equation}\label{LAREG2} \left\{ {\mbox{$z\in\C$ s.t. $\re z\in(0,2s)$ and~$\im z\in[0,M]$}}\right\}.\end{equation}
\end{lemma}

\begin{lemma}\label{LEM:CONSEG:3}
Let~$g$ be as in~\eqref{LAFUNZ}. Assume that~$s\in\left(\frac12,1\right)$. Then,
there exists~$M>0$ sufficiently large (possibly in dependence of~$s$) such that~$g$ has at least one zero in the complex region
\begin{equation}\label{LAREG3} \left\{ {\mbox{$z\in\C$ s.t. $\re z\in\left(\displaystyle s+\frac12,s+1\right)$ and~$\im z\in[0,M]$}}\right\}.\end{equation}
\end{lemma}

\begin{lemma}\label{LEM:CONSEG:4}
Let~$g$ be as in~\eqref{LAFUNZ}. Assume that~$s\in\left(0,\frac12\right)$. Then,
there exists~$M>0$ sufficiently large (possibly in dependence of~$s$) such that~$g$ has at least one zero in the complex region
\begin{equation*} \left\{ {\mbox{$z\in\C$ s.t. $\re z\in\left(\displaystyle 2s, 2s+\frac12\right)$ and~$\im z\in[0,M]$}}\right\}.\end{equation*}
\end{lemma}

We now proceed with the proofs of the above results.

\begin{proof}[Proof of \autoref{LEM:CONSEG:1}] First of all, we rule out the possibility of real zeros.
For this, we use the digamma function
\begin{equation}\label{10p2rfjgrqadscv} \psi (z):={\frac {\Gamma '(z)}{\Gamma (z)}}\end{equation}
and note that, when~$x\in\left(2s-1,s+\frac12
\right)$, hence~$\pi(s-x)\in\left(-\frac\pi2,\frac\pi2\right)$,
\begin{eqnarray*}
g'(x)&=&
-\frac{\Gamma'(2s-x)\,\Gamma(x+1)}{\Gamma(2s)}
+\frac{\Gamma(2s-x)\,\Gamma'(x+1)}{\Gamma(2s)}+\frac{\pi\cos(\pi(s-x))}{\sin(\pi s)}\\&=&
\frac{\Gamma(2s-x)\,\Gamma(x+1)}{\Gamma(2s)}
\big( \psi(x+1)-\psi(2s-x)\big)+\frac{\pi\cos(\pi(s-x))}{\sin(\pi s)}\\&\ge&
\frac{\Gamma(2s-x)\,\Gamma(x+1)}{\Gamma(2s)}
\big( \psi(x+1)-\psi(2s-x)\big).
\end{eqnarray*}

Also, the digamma function is monotonically increasing for positive real numbers and therefore, since~$x+1>2s-x>0$, we have that~$\psi(x+1)>\psi(2s-x)$. We have thereby found that~$g'(x)>0$ for all~$x\in\left(2s-1,s+\frac12
\right)$, leading to
$$ g(x)>g(2s-1)=0,$$
as desired.

Let now~$\Omega$ be the region in~\eqref{LAREG1} and note that~$g$ is holomorphic in
a neighborhood of~$\overline\Omega$. We claim
\begin{equation}\label{q0dp:3}
{\mbox{$\im( g(z))\ge\min\big\{0,1-|\re( g(z))|\big\}$
for every~$z\in\partial\Omega$.}}
\end{equation}
For this, we
decompose~$\partial\Omega$ into~$L_1\cup L_2\cup L_3\cup L_4$, with
\begin{eqnarray*}
&&L_1:=\left\{ {\mbox{$z\in\C$ s.t. $\re z\in\left(2s-1,s+\displaystyle\frac12\right)$ and~$\im z=\{0\}$}}\right\},\\
&&L_2:=\left\{ {\mbox{$z\in\C$ s.t. $\re z=s+\displaystyle\frac12$ and~$\im z\in[0,M]$}}\right\},\\
&&L_3:=\left\{ {\mbox{$z\in\C$ s.t. $\re z\in\left(2s-1,s+\displaystyle\frac12\right)$ and~$\im z=M$}}\right\},\\
{\mbox{and }}\quad
&&L_4:=\left\{ {\mbox{$z\in\C$ s.t. $\re z=2s-1$ and~$\im z\in[0,M]$}}\right\}.
\end{eqnarray*}
The proof of~\eqref{q0dp:3} will be performed by showing that~$\im( g(z))\ge\min\big\{0,1-|\re( g(z))|\big\}$
for every~$z\in L_j$, for all~$j\in\{1,\dots,4\}$. To this end, consider first~$z\in L_1\subseteq\R$.
In this case, $g$ is real-valued and therefore the claim is obvious.

Let now~$z\in L_2$. Then, $z=s+\frac12+iy$, with~$y\in[0,M]$. We observe that, for all~$w\in\C$,
\begin{equation}\label{GAMMOVE}
\Gamma(\overline w)=\overline{\Gamma(w)}\end{equation} and, as a consequence,
\begin{eqnarray*}
&&\Gamma\left(s-\frac12-iy\right)\,\Gamma\left(s+\frac12+iy+1\right)\\&&\qquad=
\left(s+\frac12+iy\right)\left(s-\frac12+iy\right)
\Gamma\left(s-\frac12-iy\right)\,\Gamma\left(s-\frac12+iy\right)\\&&\qquad=
\left( s^2 - y^2 - \frac14+ 2 i s y\right)\,\Gamma\left(\overline{s-\frac12+iy}\right)\,\Gamma\left(s-\frac12+iy\right)\\&&\qquad=
\left( s^2 - y^2 - \frac14+ 2 i s y\right)\,\left|\Gamma\left({s-\frac12+iy}\right)\right|^2.
\end{eqnarray*}

Accordingly,
\begin{equation}\label{0qFFewdiofj30mrtu9v45} \im\left(\frac{\Gamma\left(s-\frac12-iy\right)\,\Gamma\left(s+\frac12+iy+1\right)}{\Gamma(2s)}\right)=
\frac{2 s y\,\left|\Gamma\left({s-\frac12+iy}\right)\right|^2}{\Gamma(2s)}\ge0.\end{equation}

For this reason,
\begin{equation}\label{FTR293de3136425t}
\begin{split}\im\left(
g\left(s+\frac12+iy\right)\right)&=\im\left(\frac{\Gamma\left(s-\frac12-iy\right)\,\Gamma\left(s+\frac12+iy+1\right)}{\Gamma(2s)}\right)+\im\left(\frac{\sin\left(\pi\left(\frac12+iy\right)\right)}{\sin(\pi s)}\right)
\\&\ge0+\im\left(\frac{
\cosh(\pi y) }{\sin(\pi s)}\right)=0,
\end{split}\end{equation}
which implies~\eqref{q0dp:3},
as desired.

We now focus on the case~$z\in L_3$, that is~$z=x+iM$ with~$x\in \left(2s-1,s+\frac12\right)$.
We use an integration by parts to point out that, if~$\tau>0$,
\begin{eqnarray*}
\Gamma (\tau\pm iM)&=&\int _{0}^{+\infty }t^{\tau-1\pm i M}e^{-t}\,dt\\
&=&\frac1{\pm iM}\lim_{j\to+\infty}\int _{1/j}^{j}t^{\tau}\frac{d}{dt}(t^{\pm iM})e^{-t}\,dt
\\&=&\frac1{\pm iM}\lim_{j\to+\infty}\Bigg[
j^{\tau\pm iM}e^{-j}-\left(\frac1j\right)^{\tau\pm iM}e^{-\frac1j}\\&&\qquad\qquad-
\tau\int _{1/j}^{j}t^{\tau-1\pm iM} e^{-t}\,dt
+\int _{1/j}^{j}t^{\tau\pm iM} e^{-t}\,dt
\Bigg]\\&=&
-\frac1{\pm iM}\Bigg[
\tau\int _{0}^{+\infty}t^{\tau-1\pm iM} e^{-t}\,dt
-\int _{0}^{+\infty}t^{\tau\pm iM} e^{-t}\,dt
\Bigg].
\end{eqnarray*}
On this account,
for all~$b>a>0$,
\begin{equation}\label{GASU}\begin{split}&
\sup_{\tau\in[a,b]}|\Gamma (\tau\pm iM)|\\&\quad\le\frac1M
\Bigg[
b\int _{0}^{1}t^{a-1}\,dt+
b\int _{1}^{+\infty}t^{b-1} e^{-t}\,dt
+\int _{0}^{1}t^{a} e^{-t}\,dt+\int _{1}^{+\infty}t^{b} e^{-t}\,dt
\Bigg]\le \frac{C_{a,b}}M,
\end{split}
\end{equation}for some~$C_{a,b}>0$.

Thus, we deduce from~\eqref{GASU} that
\begin{equation*}
\lim_{M\to+\infty}\sup_{x\in \left(2s-1,s+\frac12\right)}
\left|\frac{\Gamma(2s-x-iM)\,\Gamma(x+iM+1)}{\Gamma(2s)}\right|=0.
\end{equation*}
As a result,
\begin{equation}\label{SDdvf}\begin{split} g(x+iM)&=-\frac{\sin(\pi(s-x-iM))}{\sin(\pi s)}+\mu_M(x)\\&=
\frac{ -\sin(\pi (s - x)) \cosh(\pi M) + i \cos(\pi (s - x)) \sinh(\pi M)}{\sin(\pi s)}+\mu_M(x),\end{split}\end{equation}
with~$\sup_{x\in \left(2s-1,s+\frac12\right)}|\mu_M(x)|$ as small as we wish, provided that~$M$ is large enough.

Now, we distinguish two cases. First, if~$x\in \left[s+\frac14,s+\frac12\right)$, we deduce from~\eqref{SDdvf} that
\begin{eqnarray*}&&
|\re( g(x+iM))|+\im( g(x+iM))\\&&\qquad\ge
\frac{ \sin(\pi (x-s)) \cosh(\pi M) }{\sin(\pi s)}
+
\frac{  \cos(\pi (x-s)) \sinh(\pi M)}{\sin(\pi s)}
-2|\mu_M(x)|\\&&\qquad\ge
\frac{ \sqrt2 \cosh(\pi M) }{2\sin(\pi s)}
+0
-2|\mu_M(x)|\\&&\qquad\ge1,
\end{eqnarray*}as long as~$M$ is sufficiently large,
which implies~\eqref{q0dp:3} in this case.

If instead~$x\in \left(2s-1,s+\frac14\right)$, we observe that for some~$c_s>0$
\begin{eqnarray*}
\inf_{x\in \left(2s-1,s+\frac14\right)}\frac{ \cos(\pi (s - x)) \sinh(\pi M)}{\sin(\pi s)}&=&
\inf_{\theta\in \left(-\frac\pi4,(1-s)\pi\right)}
\frac{ \cos\theta \sinh(\pi M)}{\sin(\pi s)} \ge c_s\sinh(\pi M).
\end{eqnarray*}

Hence, one infers from~\eqref{SDdvf} that
$$ \im( g(x+iM))\ge c_s\sinh(\pi M)-|\mu_M(x)|\ge0,$$ for~$M$ large enough,
which gives~\eqref{q0dp:3} in this case too, as desired.

We now take care of~$L_4$. To this end, 
it is useful to observe that, for all~$x_0$, $x$, $y>0$,
\begin{equation}\label{LAGAMO.0}
{\mbox{the phase of }}\frac{\Gamma(x+iy)}{\Gamma(x_0+iy)}{\mbox{ belongs
to }}\left[0, \frac{\pi(x-x_0)}2\right]
\end{equation}
and thus
\begin{equation}\label{LAGAMO}
\im\left(\frac{\Gamma(x+iy)}{\Gamma(x_0+iy)}\right)\ge0,
\end{equation}as long as~$x\in[x_0,x_0+2]$.

To check this, we recall the digamma function's
series representation 
$$ \psi (z)=-\gamma +\sum _{k=0}^{+\infty }{\frac {z-1}{(k+1)(k+z)}},$$
where~$\gamma$ is the Euler-Mascheroni constant.

In particular, when~$x>0$ and $y>0$,
\begin{equation}\label{20erfjer}\im(\psi(x+iy))=
\sum _{k=0}^{+\infty }\im\left({\frac {x-1+iy}{(k+1)(k+x+iy)}}\right)
=\sum _{k=0}^{+\infty }\frac{ y}{(k+x)^2+y^2}\end{equation}
and therefore
\begin{equation}\label{20erfjer.2}\im(\psi(x+iy))>0.\end{equation}
Moreover (see e.g.~\cite[Example~7, page~136]{WhWa62}), for all~$t>0$,
$$ \frac{\pi \coth(\pi t) }{2 }-\frac{ 1}{2 t}
=\sum_{k=1}^{+\infty} \frac{t}{k^2 + t^2}$$
Also, since~$\sinh(\pi t)\ge\pi t$,
$$\frac{d}{dt}\left(\frac{\pi \coth(\pi t) }{2 }-\frac{ 1}{2 t}\right)=
\frac1{2t^2}-\frac{\pi^2}{2\sinh^2(\pi t)}\ge0$$
and therefore
$$\sup_{t>0}\left(\frac{\pi \coth(\pi t) }{2 }-\frac{ 1}{2 t}\right)=\lim_{t\to+\infty}\left(
\frac{\pi \coth(\pi t) }{2 }-\frac{ 1}{2 t}\right)=\frac\pi2.$$
{F}rom these observations and~\eqref{20erfjer} we arrive at
\begin{equation}\label{20erfjer.3} \im(\psi(x+iy))
\le\sum _{k=0}^{+\infty }\frac{ y}{k^2+y^2}
= \frac{\pi \coth(\pi y) }{2 }-\frac{ 1}{2 y}\le\frac\pi2.\end{equation}

Now, given~$y>0$,
for all~$x\in[x_0,x_0+2]$, we consider a continuous phase~$\vartheta(x)\in\R$ such that~$\Gamma(x+iy)=
\varrho(x)\,e^{i\vartheta(x)}$, where~$\varrho(x):=|\Gamma(x+iy)|$. Taking the derivative in~$x$, it follows that
$$\frac{\Gamma'(x+iy)}{\Gamma(x+iy)}=\frac{\varrho'(x)}{\varrho(x)}+i\vartheta'(x)$$ and then, considering the imaginary part of this identity,
$$\vartheta'(x)=\im\left( \frac{\Gamma'(x+iy)}{\Gamma(x+iy)}\right)=\im(\psi(x+iy))\in
\left[0,\frac\pi2\right],$$
thanks to~\eqref{10p2rfjgrqadscv}, \eqref{20erfjer.2}, and~\eqref{20erfjer.3}.

Therefore,
$$ \vartheta(x)-\vartheta(x_0)\in\left[0, \frac{\pi(x-x_0)}2\right],$$
showing~\eqref{LAGAMO.0} and, as a byproduct, \eqref{LAGAMO}.

Now, to deal with~$L_4$, we take~$z=2s-1+iy$ with~$y\in[0,M]$ and we see that, using~\eqref{GAMMOVE},
\begin{eqnarray*}
g(z)&=&\frac{\Gamma(1-iy)\,\Gamma(2s+iy)}{\Gamma(2s)}-\frac{\sin(\pi(1-s-iy))}{\sin(\pi s)}
\\&=&
\frac{|\Gamma(1+iy)|^2\,\Gamma(2s+iy)}{\Gamma(1+iy)\Gamma(2s)}
-\frac{ \sin(\pi s)\cosh(\pi y) + i\cos(\pi s) \sinh( \pi y)}{\sin(\pi s)}.
\end{eqnarray*}
Therefore,
\begin{equation*}
\im(g(z))\ge\im\left(\frac{|\Gamma(1+iy)|^2\,\Gamma(2s+iy)}{\Gamma(1+iy)\Gamma(2s)}\right)=\frac{|\Gamma(1+iy)|^2}{\Gamma(2s)}
\im\left(\frac{\Gamma(2s+iy)}{\Gamma(1+iy)}\right)
.\end{equation*}
This and~\eqref{LAGAMO} yield that~$\im(g(z))\ge0$ and
this completes the proof of~\eqref{q0dp:3}.

The desired result now follows from~\eqref{q0dp:3} and \autoref{lemma:COMPLEM1}.
\end{proof}

\begin{proof}[Proof of \autoref{LEM:CONSEG:2}] First of all, we rule out the possibility of real zeros.
We observe that, when~$x\in(0,2s)$, and thus~$\pi(s-x)\in(-\pi s,\pi s)$,
\begin{eqnarray*}
g'(x)&=&
\frac{\Gamma(2s-x)\,\Gamma(x+1)}{\Gamma(2s)}
\big( \psi(x+1)-\psi(2s-x)\big)+\frac{\pi\cos(\pi(s-x))}{\sin(\pi s)}\\&\ge&\frac{\Gamma(2s-x)\,\Gamma(x+1)}{\Gamma(2s)}
\big( \psi(x+1)-\psi(2s-x)\big)+\frac{\pi\cos(\pi(s-x))}{\sin(\pi s)}.
\end{eqnarray*}
Since~$x+1>2s-x$, the monotonicity of the digamma function yields that~$g'(x)>0$ for all~$x\in(0,2s)$,
therefore
$$g(x)>g(0)=0,$$as desired.

Let now~$\Omega$ be as in~\eqref{LAREG2}.
Our goal is to show that
\begin{equation}\label{1w2erfq0dp:3}
{\mbox{$\im( g(z))\ge\min\big\{0,1-|\re( g(z))|\big\}$
for every~$z\in\partial\Omega$.}}
\end{equation}
Indeed, this, in combination with \autoref{lemma:COMPLEM1}, gives the desired result in \autoref{LEM:CONSEG:2}.

To prove~\eqref{1w2erfq0dp:3}, we
decompose~$\partial\Omega$ into~$L_1\cup L_2\cup L_3\cup L_4$, with
\begin{eqnarray*}
&&L_1:=\left\{ {\mbox{$z\in\C$ s.t. $\re z\in(0,2s)$ and~$\im z=\{0\}$}}\right\},\\
&&L_2:=\left\{ {\mbox{$z\in\C$ s.t. $\re z=2s$ and~$\im z\in[0,M]$}}\right\},\\
&&L_3:=\left\{ {\mbox{$z\in\C$ s.t. $\re z\in(0,2s)$ and~$\im z=M$}}\right\},\\
{\mbox{and }}\quad
&&L_4:=\left\{ {\mbox{$z\in\C$ s.t. $\re z=0$ and~$\im z\in[0,M]$}}\right\}.
\end{eqnarray*}
and we prove the statement in~\eqref{1w2erfq0dp:3} in each segment~$L_j$ for~$j\in\{1,\dots,4\}$.

If~$z\in L_1$, then~$g(z)$ is real and therefore the inequality in~\eqref{1w2erfq0dp:3} is obvious.

If~$z\in L_2$, we write~$z=2s+iy$, with~$[0,M]$ and consequently, recalling~\eqref{GAMMOVE},
\begin{equation}\label{1-0peirf3jr12efv}
\begin{split}
g(2s+iy)&=\frac{\Gamma(-iy)\,\Gamma(2s+1+iy)}{\Gamma(2s)}+\frac{\sin(\pi(s+iy))}{\sin(\pi s)}\\
&=\frac{i\,|\Gamma(1+iy)|^2\,\Gamma(2s+1+iy)}{y\Gamma(2s)\,\Gamma(1+iy)}+
\cosh(\pi y)+i
\frac{\cos(\pi s)\,\sinh(\pi y)}{\sin(\pi s)}.
\end{split}\end{equation}
By~\eqref{LAGAMO.0},
the phase of~$\frac{\Gamma(2s+1+iy)}{\Gamma(1+iy)}$ belongs to~$[0,\pi s]$.

After a rotation, we find that the phase of~$\frac{i\Gamma(2s+1+iy)}{\Gamma(1+iy)}$,
which coincides with the phase of~$\frac{i\,|\Gamma(1+iy)|^2\,\Gamma(2s+1+iy)}{y\Gamma(2s)\,\Gamma(1+iy)}$, belongs to~$\left[\frac\pi2, \frac\pi2+\pi s\right]\subseteq[0,\pi]$.

For this reason,
$$ \im\left(\frac{i\,|\Gamma(1+iy)|^2\,\Gamma(2s+1+iy)}{y\Gamma(2s)\,\Gamma(1+iy)}\right)\ge0$$
and therefore\begin{equation}\label{0euorjh234foi}\im\big(
g(2s+iy)\big)\ge0+
\frac{\cos(\pi s)\,\sinh(\pi y)}{\sin(\pi s)}\ge0,
\end{equation}which yields~\eqref{1w2erfq0dp:3} in this case.

We now deal with the case in which~$z\in L_3$, namely~$z=x+iM$ with~$x\in(0,2s)$.
In this framework, \begin{eqnarray*}
g(z)&=&\frac{\Gamma(2s-x-iM)\,\Gamma(x+1+iM)}{\Gamma(2s)}-\frac{\sin(\pi(s-x-iM))}{\sin(\pi s)}
\\&=&
\frac{ -\sin(\pi (s - x)) \cosh(\pi M) + i \cos(\pi (s - x)) \sinh(\pi M)}{\sin(\pi s)}
+\mu_M(x),
\end{eqnarray*}
where, by virtue of~\eqref{GASU}, $\sup_{x\in(0,2s)}|\mu_M(x)|$ can be made arbitrarily small as long as we choose~$M$ sufficiently large.

For this reason,
\begin{equation*}\begin{split}&
\im\big(g(z)\big)\ge \frac{\cos(\pi (s - x)) \sinh(\pi M)}{\sin(\pi s)}
-|\mu_M(x)| \ge\frac{\cos(\pi s) \sinh(\pi M)}{\sin(\pi s)}-|\mu_M(x)|>0,
\end{split}
\end{equation*}as long as~$M$ is large enough,
and~\eqref{1w2erfq0dp:3} follows.

We now suppose that~$z\in L_4$, that is~$z=iy$ with~$y\in[0,M]$.
In this setting,
\begin{eqnarray*}
g(z)&=&\frac{\Gamma(2s-iy)\,\Gamma(1+iy)}{\Gamma(2s)}-\frac{\sin(\pi(s-iy))}{\sin(\pi s)}
\\&=&\frac{|\Gamma(2s+iy)|^2\,\Gamma(1+iy)}{\Gamma(2s)\,\Gamma(2s+iy)}-\cosh(\pi y)
+\frac{i\cos(\pi s)\sinh(\pi y)}{\sin(\pi s)}.
\end{eqnarray*}

Hence, since~$\im\left(\frac{\Gamma(1+iy)}{\Gamma(2s+iy)}\right)\ge0$ thanks to~\eqref{LAGAMO}, we conclude that\begin{eqnarray*}\im(
g(z))\ge\frac{\cos(\pi s)\sinh(\pi y)}{\sin(\pi s)}\ge0.
\end{eqnarray*}
The proof of~\eqref{1w2erfq0dp:3} is thereby complete.
\end{proof}

\begin{proof}[Proof of \autoref{LEM:CONSEG:3}]
The boundary of the region in~\eqref{LAREG3} 
can be decomposed into the four segments
\begin{eqnarray*}
&&L_1:=\left\{ {\mbox{$z\in\C$ s.t. $\re z\in\left(\displaystyle s+\frac12,s+1\right)$ and~$\im z=\{0\}$}}\right\},\\
&&L_2:=\left\{ {\mbox{$z\in\C$ s.t. $\re z=s+1$ and~$\im z\in[0,M]$}}\right\},\\
&&L_3:=\left\{ {\mbox{$z\in\C$ s.t. $\re z\in\left(\displaystyle s+\frac12,s+1\right)$ and~$\im z=M$}}\right\},\\
{\mbox{and }}\quad
&&L_4:=\left\{ {\mbox{$z\in\C$ s.t. $\re z=s+\displaystyle\frac12$ and~$\im z\in[0,M]$}}\right\}.
\end{eqnarray*}
Since~$g$ has a pole in~$2s$ it is convenient to perform a contour integration for the argument principle
in which~$L_1$ is replaced by~$L_1'$ which is obtained by picking a conveniently small~$\epsilon>0$
and by replacing the small interval~$[2s-\epsilon,2s+\epsilon]$ with an upper semicircle of radius~$\epsilon$
and centered at~$2s$. The gist will be that \begin{equation}\label{mmowdf3}
{\mbox{the image under the map~$g$ of $L_1'\cup L_2\cup L_3\cup L_4$
encircles the origin,}}\end{equation} i.e., it is homotopic in~$\C\setminus\{0\}$ to a curve with nonzero winding number;
from this, the desired result will follow from the argument principle.

The details are as follows. We first consider~$z\in L_1'$. If~$x\in\left(s+\frac12,2s-\epsilon\right)\cup
\left(2s+\epsilon,s+1\right)$, we have that
$$ g(x)=\frac{\Gamma(2s-x)\,\Gamma(x+1)}{\Gamma(2s)}-\frac{\sin(\pi(s-x))}{\sin(\pi s)},$$
which is real-valued and such that
\begin{equation*}
g\left(s+\frac12\right)=\frac{\Gamma\left(s-\frac12\right) \,\Gamma\left(s + \frac32\right)}{\Gamma(2 s)} + \frac1{\sin(\pi s)}>0.
\end{equation*}
Moreover,
\begin{equation*}
g(s+1)=-\frac{\Gamma(s)\,\Gamma(s+2)}{(1-s)\Gamma(2 s)} <0.
\end{equation*}
We also see that~$g$ approaches~$+\infty$ when~$x$ tends to~$2s$ from the left,
and it approaches~$-\infty$ when~$x$ tends to~$2s$ from the right.

Furthermore, without loss of generality, we can suppose that~$g(x)\ne0$ for all~$x\in\left(s+\frac12,2s-\epsilon\right)\cup
\left(2s+\epsilon,s+1\right)$, otherwise the claim in \autoref{LEM:CONSEG:3} would be true
(actually, with a real zero).

These observations entail that
\begin{equation}\label{asc32}
\begin{split}&{\mbox{$g\left(s+\frac12,2s-\epsilon\right)$ is a segment in the strictly positive reals}}\\&{\mbox{with an endpoints on the extreme right}}\\&{\mbox{and~$g\left(2s+\epsilon,s+1\right)$ is 
a segment in the strictly negative reals}} \\&{\mbox{with an endpoints on the extreme left.}}\end{split}\end{equation}

Let us now analyze the behavior of~$g(z)$ when~$z$ travels on the small upper semicircle
of radius~$\epsilon$ centered at~$2s$.
In this case, since~$|g(2s)|=+\infty$, we know that the norm of~$g(z)$ is very large
for small~$\epsilon$. Also, for~$w$ close to~$0$ we know that~$\Gamma(w)=\frac1w+O(1)$ and therefore,
if~$\theta\in[C\epsilon,\pi-C\epsilon]$ and~$C>0$ is large enough,
\begin{eqnarray*}&&
\im (g(2s+\epsilon e^{i\theta}))=\im\left(
\frac{\Gamma(-\epsilon e^{i\theta})\,\Gamma(2s+1+\epsilon e^{i\theta})}{\Gamma(2s)}+\frac{\sin(\pi(s+\epsilon e^{i\theta}))}{\sin(\pi s)}\right)\\&&\qquad=\im\left(-\frac{e^{-i\theta}\,\Gamma(2s+1)}{\epsilon \Gamma(2s)}+O(1)\right)>0,
\end{eqnarray*}
while, when~$\theta\in[0,C\epsilon)\cup(\pi-C\epsilon,\pi]$
\begin{eqnarray*}&&
\left|\re (g(2s+\epsilon e^{i\theta}))\right|=\left|\re\left(-\frac{e^{-i\theta}\,\Gamma(2s+1)}{\epsilon \Gamma(2s)}+O(1)\right)\right|\ge\frac{c}\epsilon,
\end{eqnarray*}for some~$c>0$.

On this account,\begin{equation}\label{asc32.1}
\begin{split}&{\mbox{the image under~$g$ of  the small upper semicircle
of radius~$\epsilon$ centered at~$2s$}}\\&{\mbox{connects the two remote endpoints of the disjoint, real intervals
in~\eqref{asc32}}}\\&{\mbox{by a large curve which, away from a
neighborhood of these endpoints}}\\&{\mbox{which is separated from the origin,
lies in the upper halfplane.}}\end{split}\end{equation}

We now analyze the image of~$L_2$ under the map~$g$. For this, we point out that, when~$y\in[0,M]$,
\begin{eqnarray*}&&
g(s+1+iy)\\&&\quad=\frac{\Gamma(s-1-iy)\,\Gamma(s+2+iy)}{\Gamma(2s)}+\frac{\sin(\pi(1+iy))}{\sin(\pi s)}\\&&\quad=
\frac{(s+1+iy)(s+iy)(s-1+iy)
\,|\Gamma(s-1+iy)|^2}{\Gamma(2s)}-i\frac{\sinh(\pi y)}{\sin(\pi s)}\\&&\quad=
\frac{[s(s^2-3y^2-1)-iy(1+y^2-3s^2)]
\,|\Gamma(s-1+iy)|^2}{\Gamma(2s)}-i\frac{\sinh(\pi y)}{\sin(\pi s)}.
\end{eqnarray*}
As a result, since the Gamma-function has no zeros,
\begin{equation}\label{12fgr}\re(
g(s+1+iy))=
\frac{s(s^2-3y^2-1)\,|\Gamma(s-1+iy)|^2}{\Gamma(2s)}<0.
\end{equation}
Moreover,
\begin{eqnarray*}&&\im(
g(s+1+iy))=
-\frac{y(1+y^2-3s^2)
\,|\Gamma(s-1+iy)|^2}{\Gamma(2s)}-\frac{\sinh(\pi y)}{\sin(\pi s)}
\end{eqnarray*}
and in particular~$\im(g(s+1+iM))<0$ for large~$M$.
This observation and~\eqref{12fgr} give that
\begin{equation}\label{xcs0df}\begin{split}&
{\mbox{$g(L_2)$ is confined in the left halfplane}}\\&{\mbox{and has an endpoint in the third quadrant.}}\end{split}
\end{equation}

In addition, when~$M$ is large, by means of~\eqref{SDdvf} we see that
\begin{equation}\label{xcs0df2}\begin{split}&
{\mbox{$g(L_3)$ is a small perturbation of a circle of large radius}}\\&{\mbox{traveling through the fourth quadrant.}}\end{split}
\end{equation}

Moreover, in light of~\eqref{FTR293de3136425t}, we have that~$g(L_4)$ lies in the upper halfplane.
This observation,
\eqref{asc32}, \eqref{asc32.1}, \eqref{xcs0df}, and~\eqref{xcs0df2}
give the desired claim in~\eqref{mmowdf3}.
\end{proof}

\begin{proof}[Proof of \autoref{LEM:CONSEG:4}]
We pick~$\epsilon>0$ conveniently small and we will prove that~$g$ has at least one zero in
\begin{equation}\label{LAREG44} \left\{ {\mbox{$z\in\C$ s.t. $|z-2s|>\epsilon$,
$\re z\in\left(\displaystyle 2s, 2s+\frac12\right)$ and~$\im z\in[0,M]$}}\right\}.\end{equation}

We write the boundary of the region in~\eqref{LAREG44} 
as the union of the four segments and a circle, namely
\begin{eqnarray*}
&&L_1:=\left\{ {\mbox{$z\in\C$ s.t. $\re z\in\left(\displaystyle 2s+\epsilon,2s+\frac12\right)$ and~$\im z=\{0\}$}}\right\},\\
&&L_2:=\left\{ {\mbox{$z\in\C$ s.t. $\re z=2s+\displaystyle\frac12$ and~$\im z\in[0,M]$}}\right\},\\
&&L_3:=\left\{ {\mbox{$z\in\C$ s.t. $\re z\in\left(\displaystyle 2s,2s+\frac12\right)$ and~$\im z=M$}}\right\},\\
&&L_4:=\left\{ {\mbox{$z\in\C$ s.t. $\re z=2s$ and~$\im z\in[\epsilon,M]$}}\right\}\\
{\mbox{and }}\quad&& L_5:=\left\{ {\mbox{$z\in\C$ s.t. $z=2s+\epsilon e^{i\theta}$ and~$\theta\in\left[0,\displaystyle\frac\pi2\right]$}}\right\},
\end{eqnarray*}
where the direction of travel is here disregarded for the sake of simplicity, since
we focus on the topological curve.

Indeed, the result will follow from the argument principle once we show that
\begin{equation}\label{mmowdf3Dsqdx35t}
{\mbox{the image under the map~$g$ of $L_1\cup L_2\cup L_3\cup L_4\cup L_5$
encircles the origin.}}\end{equation}

With this in mind, we observe that~$g(2s+\epsilon)$ diverges to~$+\infty$ as~$\epsilon$ approaches~$0$ from the right, hence we can suppose that~$\epsilon$ is small enough such that~$g(2s+\epsilon)>0$.

Consequently, we may assume that~$g(x)>0$ for all~$x\in\left[ 2s+\epsilon,2s+\frac12\right]$
(otherwise~$g$ would have a real zero along this interval and
the claim of \autoref{LEM:CONSEG:4} would have been proven), and therefore
\begin{equation}\label{2-wfp4356yuj}
{\mbox{$g(L_1)$ is a segment contained in the strictly positive reals.}}
\end{equation}

Now, to study the image of~$L_2$, we take~$y\in[0,M]$ and notice that
\begin{eqnarray*}&&
g\left(2s+\frac12+iy\right)=\frac{\Gamma\left(-\frac12-iy\right)\,\Gamma\left(2s+\frac32+iy\right)}{\Gamma(2s)}+\frac{\sin\left(\pi\left(s+\frac12+iy\right)\right)}{\sin(\pi s)}
\\&&\quad=-\frac{4\left|\Gamma\left(\frac32+iy\right)\right|^2\,\Gamma\left(2s+\frac32+iy\right)}{(4y^2+1)\,
\Gamma(2s)\,\Gamma\left(\frac32+iy\right)}
+\frac{\cos(\pi s) \cosh(\pi y) }{\sin(\pi s)}
- i \sinh(\pi y).
\end{eqnarray*}

Now, in view of~\eqref{LAGAMO.0}, we know that
the phase of~$\frac{\Gamma\left(2s+\frac32+iy\right)}{\Gamma\left(\frac32+iy\right)}$ belongs
to~$[0,s\pi]$. 
It follows that
$$\im\left(\frac{4\left|\Gamma\left(\frac32+iy\right)\right|^2\,\Gamma\left(2s+\frac32+iy\right)}{(4y^2+1)\,
\Gamma(2s)\,\Gamma\left(\frac32+iy\right)}\right)\ge0$$ and thus
\begin{equation*}\im\left(
g\left(2s+\frac12+iy\right)\right)\le 0
- \sinh(\pi y)\le0.\end{equation*}
{F}rom this, we conclude that
\begin{equation}\label{12-i03fprjbg}
{\mbox{$g(L_2)$ lies in the lower halfplane.}}
\end{equation}

Regarding the analysis along~$L_3$, we point out that, when~$x\in\left( 2s,2s+\frac12\right)$, we have that
$$\sin(\pi (x - s))\ge \min\{ \sin(\pi s),\cos(\pi s)\}=:c_s>0$$ and therefore, by~\eqref{SDdvf},
if~$M$ is large enough,
$$\re( g(x+iM))\ge \frac{ c_s\,\cosh(\pi M) }{2\sin(\pi s)}$$
and, as a result,
\begin{equation}\label{12-i03fprjbg.2}
{\mbox{$g(L_3)$ is confined in the right halfplane,}}
\end{equation}
and, in fact, towards the extreme right for large~$M$.

We now note that,
if~$y\in[\epsilon,M]$,
\begin{equation}\label{axcwE2Fl21}\begin{split}
g(2s+iy)&=
\frac{\Gamma(-iy)\,\Gamma(2s+1+iy)}{\Gamma(2s)}+\frac{\sin(\pi(s+iy))}{\sin(\pi s)}
\\&=-\frac{\Gamma(1-iy)\,\Gamma(2s+1+iy)}{iy\,\Gamma(2s)} +\frac{
\cosh(\pi y) \sin(\pi s ) + i \sinh(\pi y) \cos(\pi s )}{\sin(\pi s)}\\&=\frac{i|\Gamma(1+iy)|^2\,\Gamma(2s+1+iy)}{y\,\Gamma(1+iy)\,\Gamma(2s)}+
\cosh(\pi y) +i\frac{ \sinh(\pi y) \cos(\pi s )}{\sin(\pi s)}.
\end{split}\end{equation}

We now recall~\eqref{LAGAMO.0} and we deduce that
the phase of~$\frac{\Gamma(2s+1+iy)}{\Gamma(1+iy)}$
belongs to~$[0,\pi s]$.

As a result, the phase of~$\frac{i|\Gamma(1+iy)|^2\,\Gamma(2s+1+iy)}{y\,\Gamma(1+iy)\,\Gamma(2s)}$ belongs
to~$\left[\frac\pi2,\pi s+\frac\pi2\right]$, entailing that
$$\im\left(\frac{i|\Gamma(1+iy)|^2\,\Gamma(2s+iy)}{y\,\Gamma(1+iy)\,\Gamma(2s)}\right)>0.$$
This and~\eqref{axcwE2Fl21} yield that~$\im\big(g(2s+iy)\big)>0$ and therefore
\begin{equation}\label{alertenvqwd0jfocv}
{\mbox{$g(L_4)$ lies in the upper halfplane.}}
\end{equation}

We now remark that, when~$\theta\in\left[0,\frac\pi2\right]$, expanding the Gamma-function near
the origin we see that, for small~$\epsilon$,
\begin{eqnarray*}
g(2s+\epsilon e^{i\theta})&=&
\frac{\Gamma(-\epsilon e^{i\theta})\,\Gamma(2s+1+\epsilon e^{i\theta})}{\Gamma(2s)}+\frac{\sin(\pi(s+\epsilon e^{i\theta}))}{\sin(\pi s)}\\
&=&
-\left(\frac1{\epsilon e^{i\theta}}+\gamma\right)
\frac{\Gamma(2s+1)+\epsilon\Gamma'(2s+1) e^{i\theta}}{\Gamma(2s)}+1+O(\epsilon)\\&=&
-\frac{\Gamma(2s+1)}{\epsilon\Gamma(2s)\,e^{i\theta}}+
\frac{\gamma \Gamma(2s+1)}{\Gamma(2s)}
-\frac{\Gamma'(2s+1)}{\Gamma(2s)}+1+O(\epsilon).
\end{eqnarray*}
As~$\theta$ varies in~$\left[0,\frac\pi2\right]$,
this describes a quarter of circle of large radius~$\frac{\Gamma(2s+1)}{\epsilon\Gamma(2s)}$ in the upper halfplane, up to a constant horizontal translation that is independent of~$\epsilon$
and an additional small perturbation of order~$\epsilon$.

The proof of~\eqref{mmowdf3Dsqdx35t} is thus completed, thanks
to the latter observation,
\eqref{2-wfp4356yuj}, \eqref{12-i03fprjbg}, \eqref{12-i03fprjbg.2}, and~\eqref{alertenvqwd0jfocv}.\end{proof}

\subsubsection{Proof of \autoref{prop:zeros-final}}

First, we point out a strengthening of \autoref{LEM:CONSEG:1} and \autoref{LEM:CONSEG:2}:

\begin{corollary}\label{USDFVFGu1}
Let~$g$ be as in~\eqref{LAFUNZ}. 
Assume that~$s\in\left(\frac12,1\right)$. Then,
there exist~$M>0$ sufficiently large and~$\epsilon>0$ sufficiently small
(possibly in dependence of~$s$) such that~$g$ has no zeros in the complex region
\begin{equation}\label{LAREG1bb} \left\{ {\mbox{$z\in\C$ s.t. $\re z\in\left(2s-1,s+\displaystyle\frac12+\epsilon
\right]$ and~$\im z\in[0,M]$}}\right\}.\end{equation}
\end{corollary}

\begin{corollary}\label{USDFVFGu2}
Let~$g$ be as in~\eqref{LAFUNZ}. Assume that~$s\in\left(0,\frac12\right)$. Then,
there exist~$M>0$ sufficiently large and~$\epsilon>0$ sufficiently small
(possibly in dependence of~$s$) such that~$g$ has no zeros in the complex region
\begin{equation}\label{LAREG2bb} \left\{ {\mbox{$z\in\C$ s.t. $\re z\in(0,2s+\epsilon]$ and~$\im z\in[0,M]$}}\right\}.\end{equation}
\end{corollary}

\begin{proof}[Proof of \autoref{USDFVFGu1}] We let~$M$ be as in \autoref{LEM:CONSEG:1}.
We argue for a contradiction and we suppose that for all~$\epsilon\in(0,1)$ there exists~$z_\epsilon$
in the region described by~\eqref{LAREG1bb} with~$g(z_\epsilon)=0$.

By \autoref{LEM:CONSEG:1} we know that~$\re z_\epsilon\in\left[s+\frac12,s+\frac12+\epsilon\right]$
and therefore, up to a subsequence, $z_\epsilon\to s+\frac12+iy_\star$ as~$\epsilon\to0$, for some~$y_\star\in[0,M]$, and~$g\left(s+\frac12+i y_\star\right)=0$.

This, together with~\eqref{0qFFewdiofj30mrtu9v45} and~\eqref{FTR293de3136425t}, yields that~$y_\star=0$.
But this is impossible, since
\begin{equation*}
g\left(s+\frac12\right)=
\frac{\Gamma\left(s-\frac12\right)\,\Gamma\left(s+\frac32\right)}{\Gamma(2s)}+\frac{1}{\sin(\pi s)}>0.
\qedhere\end{equation*}\end{proof}

\begin{proof}[Proof of \autoref{USDFVFGu2}] This is similar to the proof
of \autoref{USDFVFGu1}. In this case, the limit procedure gives the existence
of some~$y_\star\in[0,M]$ such that~$g(2s+i y_\star)=0$.

In view of~\eqref{0euorjh234foi}, we see that~$y_\star=0$. Then,
recalling~\eqref{1-0peirf3jr12efv},
$0=|g(2s)|
=+\infty$,
which is absurd.
\end{proof}

We need one more result in order to conclude the proof of \autoref{prop:zeros-final}.

\begin{lemma}
\label{lemma:M-large}
Let $s \in (0,1)$ and $g$ be as in~\eqref{LAFUNZ}. Then, there exists $M > 0$ such that $g(\beta) \not= 0$ for all $\beta \in (0,2s+\frac{1}{2}) \times (-M,M)^c$.
\end{lemma}

\begin{proof}
Note that, for all~$z \in \C$ with $z=x+iy$, $x > 0$ and~$y\in\R$,
\begin{align*}
|\Gamma(x+iy)| \le \Gamma(x), \qquad |\sin(x+iy)| \ge \frac{1}{2}(e^{|y|} - 1).
\end{align*}
The first property is an obvious consequence from the definition of the Gamma-function:
\begin{align*}
|\Gamma(x+iy)| \le \int_0^{\infty} |t^{x+iy - 1}| e^{-t} \d t = \int_0^{\infty} t^{x-1} e^{-t} \d t = \Gamma(x),
\end{align*}
while the second property can be proved as follows:
\begin{align*}
2|\sin(x+iy)| &= |e^{ix-y} - e^{-ix+y}| \ge ||e^{ix-y}| - |e^{-ix+y}||\\
&= e^{|y|}|1 - e^{-2|y|}| = e^{|y|} |1 - e^{-|y|}||1 + e^{-|y|}| \ge e^{|y|} -1.
\end{align*}
Writing $\beta = a +ib$ we thus have the lower bound
\begin{align*}
\frac{|\sin(\pi(\beta-s))|}{|\sin(\pi s)|} \ge \frac{\exp(\pi |b|) - 1}{2 \sin(\pi s)}.
\end{align*}
Moreover, using that $\Gamma(2s-\beta) = \frac{\Gamma(2s+1-\beta)}{2s-\beta}$, we have for $a \in (0,2s+\frac{1}{2})$,
\begin{align*}
|\Gamma(2s-\beta) \Gamma(\beta+1)| = \frac{|\Gamma(2s+1-\beta) \Gamma(\beta+1)|}{|2s-\beta|} \le \max_{c \in [0,2s+\frac{1}{2}]}|\Gamma(1+c)| |\Gamma(2s+1-a)| |b|^{-1}.
\end{align*}
Since $A_1 := \max_{c \in [0,2s+\frac{1}{2}]}|\Gamma(1+c)| < \infty$, and $A_2 := \max_{a \in [0,2s+\frac{1}{2}]}|\Gamma(2s+1-a)| < \infty$ we have
\begin{align*}
|\Gamma(2s-\beta) \Gamma(\beta+1)| < \frac{|\sin(\pi(\beta-s))|}{|\sin(\pi s)|},
\end{align*}
as long as $|b|$ is large enough, uniformly in $a$. This proves the desired result.
\end{proof}

Now, we are in a position to prove \autoref{prop:zeros-final}.

\begin{proof}[Proof of \autoref{prop:zeros-final}]
In case $s = \frac{1}{2}$, the claim follows immediately from \autoref{rem:s-half}. 

In case $s \in (0,\frac{1}{2})$, the first claim follows from \autoref{USDFVFGu2} and \autoref{lemma:M-large} by choosing $M$ large enough. The second claim follows from \autoref{LEM:CONSEG:4}. In case $s \in (\frac{1}{2},1)$, the first claim follows from \autoref{USDFVFGu1} and \autoref{lemma:M-large} by choosing $M$ large enough. The second claim follows from \autoref{LEM:CONSEG:3}.
\end{proof}

\subsubsection{Proof of \autoref{rem-s-asymp}}

First, we give a proof of the asymptotic expansion as $s \to 1$ in \autoref{rem-s-asymp}.

\begin{proof}[Proof of the expansion as $s \to 1$]
 Given~$\lambda\in\C$, we write~$\sigma:=1-s$
and
\begin{equation}\label{0odjfreg}
\beta_{\sigma,\lambda}:=2\pm i\sqrt{2\sigma}+\lambda\sqrt{\sigma}\end{equation} and,
for small~$\sigma>0$ (and~$\lambda$ in a bounded region),
\begin{equation}\label{s1:-1}\begin{split}
G(\sigma,\lambda)&:=\frac{F(1-\sigma,\beta_{\sigma,\lambda})}{\sin(\pi (1-\sigma))}\\&=
\frac{(2-2\sigma-\beta_{\sigma,\lambda})\sin(\pi(1-\sigma-\beta_{\sigma,\lambda}))}{\sin(\pi \sigma)}-\frac{\Gamma(3-2\sigma-\beta_{\sigma,\lambda})\,\Gamma(\beta_{\sigma,\lambda}+1)}{\Gamma(2-2\sigma)}\\&=
\frac{(\pm i\sqrt{2\sigma}+\lambda\sqrt{\sigma}+2\sigma)\sin(\pi(1
\pm i\sqrt{2\sigma}+\lambda\sqrt{\sigma}+\sigma))}{\sin(\pi \sigma)}
\\&\qquad-\frac{\Gamma(1
\mp i\sqrt{2\sigma}-\lambda\sqrt{\sigma}-2\sigma)\,\Gamma(3\pm i\sqrt{2\sigma}+\lambda\sqrt{\sigma})}{\Gamma(2-2\sigma)}.
\end{split}\end{equation}

We observe that
\begin{eqnarray*}&& \sin(\pi(1
\pm i\sqrt{2\sigma}+\lambda\sqrt{\sigma}+\sigma))=
-\sin(\pi(\pm i\sqrt{2\sigma}+\lambda\sqrt{\sigma}+\sigma))\\&&\qquad=-\pi(\pm i\sqrt{2\sigma}+\lambda\sqrt{\sigma}+\sigma)+O(\sigma^{3/2})
\end{eqnarray*}
and accordingly
\begin{equation}\label{s1:-2}\begin{split}& \frac{(\pm i\sqrt{2\sigma}+\lambda\sqrt{\sigma}+2\sigma)\sin(\pi(1
\pm i\sqrt{2\sigma}+\lambda\sqrt{\sigma}+\sigma))}{\sin(\pi \sigma)}\\&\qquad=\frac{-\pi(\pm i\sqrt{2\sigma}+\lambda\sqrt{\sigma}+2\sigma)(\pm i\sqrt{2\sigma}+\lambda\sqrt{\sigma}+\sigma+O(\sigma^{3/2}))}{\pi \sigma+O(\sigma^3)}\\&\qquad=\frac{-\pi(\pm i\sqrt{2}+\lambda+2\sqrt\sigma)(\pm i\sqrt{2}+\lambda+\sqrt\sigma+O(\sigma))}{\pi +O(\sigma^2)}
\\&\qquad=-\lambda^2 \mp 2i\sqrt{2}\lambda + 2 
-3\sqrt{\sigma}\Big(\lambda \pm i\sqrt{2}\Big) + O(\sigma).
\end{split}\end{equation}

Furthermore,
\begin{eqnarray*}
\Gamma(1
\mp i\sqrt{2\sigma}-\lambda\sqrt{\sigma}-2\sigma)&=&
\Gamma(1)+\Gamma'(1)
(\mp i\sqrt{2\sigma}-\lambda\sqrt{\sigma}-2\sigma)
+O(\sigma)\\&=&1+\gamma
(\pm i\sqrt{2\sigma}+\lambda\sqrt{\sigma})
+O(\sigma),
\end{eqnarray*}where~$\gamma$ is the Euler-Mascheroni constant,
and
\begin{eqnarray*}
\Gamma(3\pm i\sqrt{2\sigma}+\lambda\sqrt{\sigma})&=&
\Gamma(3)+\Gamma'(3)(\pm i\sqrt{2\sigma}+\lambda\sqrt{\sigma})
+O(\sigma)\\&=&
2+(3-2\gamma)(\pm i\sqrt{2\sigma}+\lambda\sqrt{\sigma})
+O(\sigma).
\end{eqnarray*}
This gives that
\begin{eqnarray*}
&&
\Gamma(1
\mp i\sqrt{2\sigma}-\lambda\sqrt{\sigma}-2\sigma)\,\Gamma(3\pm i\sqrt{2\sigma}+\lambda\sqrt{\sigma})\\&&\qquad=
\Big(1+\gamma
(\pm i\sqrt{2\sigma}+\lambda\sqrt{\sigma})
+O(\sigma)\Big)\Big(
2+(3-2\gamma)(\pm i\sqrt{2\sigma}+\lambda\sqrt{\sigma})
+O(\sigma)\Big)\\&&\qquad=2 \pm 3i\sqrt{2\sigma} + 3\lambda\sqrt{\sigma} + O(\sigma)
\end{eqnarray*}
and therefore
\begin{align*}
\frac{\Gamma(1\mp i\sqrt{2\sigma}-\lambda\sqrt{\sigma}-2\sigma)\,\Gamma(3\pm i\sqrt{2\sigma}+\lambda\sqrt{\sigma})}{\Gamma(2-2\sigma)} &=\frac{2 \pm 3i\sqrt{2\sigma} + 3\lambda\sqrt{\sigma} + O(\sigma)}{1+O(\sigma)}\\
&=2 \pm 3i\sqrt{2\sigma} + 3\lambda\sqrt{\sigma} + O(\sigma)
\end{align*}

Hence, recalling~\eqref{s1:-1} and~\eqref{s1:-2}, we conclude that
\begin{equation*}\begin{split}
G(\sigma,\lambda)&=-\lambda^2 \mp 2i\sqrt{2}\lambda + 2 
-3\sqrt{\sigma}\Big(\lambda \pm i\sqrt{2}\Big) -\Big(2 \pm 3i\sqrt{2\sigma} + 3\lambda\sqrt{\sigma} \Big)+ O(\sigma)\\&=\mp 6i\sqrt{2\sigma} \mp 2i\sqrt{2}\lambda-\lambda^2 - 6\lambda\sqrt{\sigma} .
\end{split}\end{equation*}

Consequently, if
$$ H(\tau,\lambda):=G(\tau^2,\lambda)=\mp 6i\sqrt{2}\tau \mp 2i\sqrt{2}\lambda-\lambda^2 - 6\tau\lambda ,$$
we have that~$H(0,0)=0$ and~$\partial_\lambda H(0,0)=\mp 2i\sqrt{2}\ne0$.

This gives that, for small~$\tau$, we can find a unique~$\lambda(\tau)$ with~$\lambda(0)=0$ satisfying
$$H(\tau,\lambda(\tau))=0.$$
By implicit derivation, we also see that
$$ 0=\partial_\tau H(0,0)+\partial_\lambda H(0,0)\lambda'(0)=\mp 6i\sqrt{2}\mp 2i\sqrt{2}\lambda'(0),
$$
yielding that~$\lambda'(0)=-3$ and therefore
$$\lambda(\tau)=-3\tau+O(\tau^2).$$

As a result, for small~$\tau>0$,
\begin{eqnarray*}
0=H(\tau,\lambda(\tau))=G(\tau^2,\lambda(\tau))=
\frac{F(1-\tau^2,\beta_{\tau^2,\lambda(\tau)})}{\sin(\pi (1-\tau^2))}
\end{eqnarray*}and so, writing~$s=1-\tau^2$,
\begin{eqnarray*}
0=F(1-\tau^2,\beta_{\tau^2,\lambda(\tau)})=F(s,\beta_{1-s,\lambda(\sqrt{1-s})}).
\end{eqnarray*}
This and~\eqref{0odjfreg} imply that
\begin{eqnarray*}&&
0=F(1-\tau^2,\beta_{\tau^2,\lambda(\tau)})=F\Big(s,\,
2\pm i\sqrt{2(1-s)}+\lambda(\sqrt{1-s})\sqrt{1-s}\Big)\\&&\qquad\qquad
= F\Big(s,\,
2\pm i\sqrt{2(1-s)}-3(1-s)+O((1-s)^{3/2})\Big),
\end{eqnarray*}
as desired.
\end{proof}

Finally, we present the proof of the asymptotic expansion as $s \to 0$ in \autoref{rem-s-asymp}.

\begin{proof}[Proof of the expansion as $s \to 0$]
Let~$\lambda\in\C$ belong to a bounded region and
\begin{eqnarray*}
G(s,\lambda) := \frac{F(s,\lambda s)}{2s^2} = \frac{(2-\lambda) \sin(\pi(1-\lambda )s)}{2s}-\frac{\sin(\pi s)\Gamma(1+(2-\lambda )s)\,\Gamma(1+\lambda s)}{s\Gamma(1+2s)}
.\end{eqnarray*}
Note that, for small~$s>0$,
\begin{equation}\label{pre.ex}\begin{split}
G(s,\lambda)&=
\frac{(2-\lambda) \big(\pi(1-\lambda )+O(s^2)\big)}{2}-\frac{\big(\pi +O(s^2)\big)\Gamma(1+(2-\lambda )s)\,\Gamma(
1+\lambda s)}{\Gamma(1+2s)}\\&=
\frac{\pi(2-\lambda)(1-\lambda )}{2}-\pi+O(s)
.\end{split}\end{equation}
As a result, we have that~$G(0,3)=0$ and~$\partial_\lambda G(0,3)=\frac{3\pi}2\ne0$. Therefore, for small~$s>0$, we can find~$\lambda(s)$ such that~$\lambda(0)=3$ and~$G(s,\lambda(s))=0$.

Consequently, for small~$s$,
\begin{eqnarray*}
&&0=G(s,\lambda(s))=G(s,3+o(1))=\frac{F(s,\,3s+o(s))}{2s^2}
\end{eqnarray*}
and thus~$F(s,\,3s+o(s))=0$.
\end{proof}

\subsection{A 1D Liouville theorem}

The goal of this subsection is to prove \autoref{thm:Neumann-Liouville}, i.e. we restrict ourselves to solutions that grow slower than $t^{2s-\eps}$ at infinity for some $\eps > 0$.

It remains to verify several important properties of the operator $L$. We recall that, by \cite[Proposition 2.1]{AFR23},
\begin{align}
\label{eq:k-upper}
k_{\R_+}(x,y) \le C \begin{cases}
\left( 1+ \left| \log \left( \frac{\min\{|x|,|y|\}}{|x-y|} \right) \right| \right) |x-y|^{-1-2s}, &~~ \text{ if } \min\{|x|,|y|\} \le |x-y|,\\
\min\{|x|,|y|\}^{-1-2s}, &~~ \text{ if } \min\{|x|,|y|\} \ge |x-y|.
\end{cases}
\end{align}

\begin{lemma}
\label{lemma:Neumann-L}
Let $\varphi \in C^2_{0,1+2s}(0,\infty)$. Then $x \mapsto L \varphi(x)$ is continuous and for some $C > 0$:
\begin{align*}
|L \varphi (x)| \le C (1 + |\log x|) \begin{cases}1 + x^{1-2s}, ~~ &\text{ for } x \in (0,2],\\
 (1 + x)^{-1-2s} &\text{ for } x \in (2,\infty). 
\end{cases}
\end{align*} 
In particular, $L$ satisfies \eqref{eq:L-decay} for any $\alpha_1 \in (\max\{ 0 ,2s-1\} , 2s)$ and $\alpha_2 = 0$.
\end{lemma}

\begin{proof}
We skip the proof of the continuity and refer instead to \cite[Lemma 2.2.6]{FeRo24}. The proof goes in the same way in our setting.
We split the proof of the second claim into several parts. First, we consider $x \in (0,2]$ and observe that, by the symmetry of the first summand in $K_{\R_+}(x,y)$,
\begin{align*}
L\varphi(x) &= \frac{c_s}{2} \int_{B_{x/2}} (2\varphi(x) - \varphi(x+h) - \varphi(x-h)) |h|^{-1-2s} \d h \\
&\quad + c_s \int_{B_{x/2}} ( \varphi(x) - \varphi(x+h)) k_{\R_+}(x,x+h) \d h \\
&\quad + c_s \int_{(0,\infty) \setminus B_{x/2}(x)} (\varphi(x) - \varphi(y)) K_{\R_+}(x,y) \d y \\
&= I_1 + I_2 + I_3.
\end{align*}
For $I_1$, we compute
\begin{align*}
I_1 \le C \Vert D^2 \varphi \Vert_{L^{\infty}(0,3)} \int_{B_{x/2}} |h|^{-1-2s+2} \d h \le C x^{2-2s}.
\end{align*}
For $I_2$, we use that $|k_{\R_+}(x,x+h)| \le C x^{-1-2s}$ in the domain of integration and therefore
\begin{align*}
I_2 \le C \Vert \nabla \varphi \Vert_{L^{\infty}(0,3)} \int_{B_{x/2}}  x^{-1-2s}|h|  \d h \le C x^{1 - 2s}.
\end{align*}
For $I_3$ we estimate
\begin{align*}
I_3 &\le C \Vert \nabla \phi \Vert_{L^{\infty}(0,1)} \int_{0}^{x/2} |x-y| K_{\R_+}(x,y) \d y \\
&\quad + C \Vert \nabla \phi \Vert_{L^{\infty}(0,2)} \int_{3x/2}^2  |x-y| K_{\R_+}(x,y) \d y + 2 \Vert \phi \Vert_{L^{\infty}(0,\infty)} \int_2^{\infty}  K_{\R_+}(x,y) \d y\\
&\le C \int_{0}^{x/2} |x-y|^{-2s} \left( 1+ \left| \log \left( \frac{|y|}{|x-y|} \right) \right| \right) \d y \\
&\quad + C \int_{3x/2}^{2} |x-y|^{-2s} \left( 1+ \left| \log \left( \frac{|x|}{|x-y|} \right) \right| \right) \d y + C \int_{2}^{\infty}  |x-y|^{-1-2s} \left( 1+ \left| \log \left( \frac{|x|}{|x-y|} \right) \right| \right)  \d y  \\
&\le C (1 + x^{1-2s}) (1 + |\log x|).
\end{align*}
Altogether, we have shown
\begin{align*}
|L \varphi(x)| \le C (1 + |\log x|)(1 + x^{1-2s}) ~~ \forall x \in (0,1).
\end{align*}
Next, we consider $x > 1$. In that case, we split
\begin{align*}
L\varphi(x) &= \frac{c_s}{2} \int_{B_{1}} (2\varphi(x) - \varphi(x+h) - \varphi(x-h)) |h|^{-1-2s} \d h \\
&\quad + c_s \int_{B_{1}} (\varphi(x) - \varphi(x+h)) k_{\R_+}(x,x+h) \d h \\
&\quad + c_s \int_{(0,\infty) \setminus B_{1}(x)} \varphi(x) K_{\R_+}(x,y) \d y - c_s \int_{(0,\infty) \setminus B_{1}(x)} \varphi(y) K_{\R_+}(x,y) \d y \\
&= I_1 + I_2 + I_3 + I_4.
\end{align*}
For $I_1$ and $I_2$, by the same arguments as before,
\begin{align}
\label{eq:I1-help}
\begin{split}
I_1+ I_2 &\le C \Vert D^2 \varphi \Vert_{L^{\infty}(x-1,x+1)} \int_{B_{1}} |h|^{-1-2s+2} \d h + C \Vert \nabla \varphi \Vert_{L^{\infty}(x-1,x+1)} \int_{B_{1}} x^{-1-2s}|h| \d h \\
& \le C x^{-1-2s} \left( 1 + x^{-1 - 2s} \right).
\end{split}
\end{align}
For $I_3$, using the upper bound for $k_{\R_+}$ and splitting the integral, we see that
\begin{align*}
I_3 &\le C (1 + x)^{-1-2s} \int_{(0,\infty) \setminus B_1(x)} K_{\R_+}(x,y) \d y \\
& \le C (1 + x)^{-1-2s} \int_0^{x/2} |x-y|^{-1-2s} \left( 1+ \left| \log \left( \frac{|y|}{|x-y|} \right) \right| \right) \d y \\
 &\quad +  C (1 + x)^{-1-2s} \int_{x/2}^{x-1} |x-y|^{-1-2s} + |y|^{-1-2s} \d y \\
 &\quad +   C (1 + x)^{-1-2s}  \int_{x+1}^{2s}  |x-y|^{-1-2s} + |y|^{-1-2s} \d y \\ 
 &\quad + (1 + x)^{-1-2s} \int_{2x}^{\infty}  |x-y|^{-1-2s} \left( 1+ \left| \log \left( \frac{|x|}{|x-y|} \right) \right| \right) \d y  \\
& \le C (1 + x)^{-1-2s} (1 + |\log x|).
\end{align*}
For $I_4$, we have by analogous arguments
\begin{align*}
I_4 \le C \int_{(0,\infty) \setminus B_1(x)} (1 + y)^{-1-2s} K_{\R_+}(x,y) \d y \le  C (1 + x)^{-1-2s} (1 + |\log x|).
\end{align*}

Altogether, we have shown
\begin{align*}
|L \varphi(x)| \le C (1 + |\log x|)(1 + x)^{-1-2s} ~~ \forall x \in (1,\infty).
\end{align*}
and the proof of the second claim is complete. The proof of the third claim follows immediately by integration of the first claim.
\end{proof}

Finally, we prove that $L$ is self-adjoint in the sense of \eqref{eq:self-adjoint}.

\begin{lemma}
\label{lemma:L-selfadjoint}
It holds
\begin{align*}
\int_0^{\infty} g(x) L \varphi(x) \d x = \int_0^{\infty} L g(x) \varphi(x) \d x
\end{align*}
for any $g(x) = x^{\beta - 1}$ with $b := \re(\beta) \in (0,2s)$ and $\varphi \in C^2_{loc}(0,\infty)$ such that
\begin{align*}
|\varphi^{(i)}(x)| \le C x^{2} ~~ \forall x \in (0,1), ~i \in \{ 0 , 1 , 2\}, \qquad |\varphi(x)| \le C x^{-1-2s} ~~ \forall x \in [1,\infty).
\end{align*}
In particular, $L$ satisfies \eqref{eq:self-adjoint} for any $\varphi \in C^2_{2,1+2s}(0,\infty)$.
\end{lemma}

\begin{proof}
Let $\eps \in (0, \frac{1}{2})$. Then,
\begin{align*}
\int_{\eps}^{\infty} g(x) L \varphi(x) \d x & = \iint_{((0,\eps) \times (0,\eps))^c} (g(x) - g(y))(\varphi(x) - \varphi(y)) K_{\R_+}(x,y) \d y \d x \\
&\quad - \int_{0}^{\eps} g(x) \int_{\eps}^{\infty} (\varphi(x) - \varphi(y)) K_{\R_+}(x,y) \d y \d x \\
&= \int_{\eps}^{\infty} L g(x)  \varphi(x) \d x + \int_{0}^{\eps} \varphi(x) \int_{\eps}^{\infty} (g(x) - g(y)) K_{\R_+}(x,y) \d y \d x \\
&\quad - \int_{0}^{\eps} g(x) \int_{\eps}^{\infty} (\varphi(x) - \varphi(y)) K_{\R_+}(x,y) \d y \d x.
\end{align*}
By the dominated convergence theorem (note that the integrals converge absolutely by \autoref{lemma:Neumann-L}, respectively \eqref{eq:L-decay} and \autoref{remark:distr-well-def}),
\begin{align*}
\int_{\eps}^{\infty} g(x) L \varphi(x) \d x \to \int_{0}^{\infty} g(x) L \varphi(x) \d x, \qquad \int_{\eps}^{\infty} L g(x)  \varphi(x) \d x \to \int_{0}^{\infty} L g(x)  \varphi(x) \d x.
\end{align*}
We claim that, under our assumptions on $g$ and~$\varphi$,
\begin{align*}
\int_{0}^{\eps} g(x) \int_{\eps}^{\infty} (\varphi(x) - \varphi(y)) K_{\R_+}(x,y) \d y \d x \to 0, \qquad \int_{0}^{\eps} \varphi(x) \int_{\eps}^{\infty} (g(x) - g(y)) K_{\R_+}(x,y) \d y \d x \to 0.
\end{align*}
To prove it, let us observe that, by \eqref{eq:k-upper}, for any $x \in (0,\eps)$ and $a < 2s$,
\begin{align*}
\int_{\eps}^{1} |x-y|^i K_{\R_+}(x,y) \d y \le C (\eps - x)^{i-2s} + C|\log x|(1 + \eps^{i-2s}), \qquad \int_1^{\infty} y^{a} K_{\R_+}(x,y) \d y \le C |\log x|. 
\end{align*}
Hence, estimating
\begin{align}
\label{eq:varphi-estimate}
|\varphi(x) - \varphi(y)| \le 
\begin{cases}
|\varphi'(x)||x-y| + C \sup_{z \in [x,1]}|\varphi''(z)||x-y|^2,& ~~ y \in (\eps,1],\\
|\varphi(x)| + |\varphi(y)|,& ~~ y \in [1,\infty),
\end{cases}
\end{align}
we obtain
\begin{align*}
\Big| \int_{0}^{\eps} & g(x) \int_{\eps}^{\infty} (\varphi(x) - \varphi(y)) K_{\R_+}(x,y) \d y \d x \Big| \\
&\le C \int_0^{\eps} x^{b}  \int_{\eps}^1 |x-y| K_{\R_+}(x,y) \d y \d x + C \int_0^{\eps} x^{b-1}  \int_{\eps}^1 |x-y|^2 K_{\R_+}(x,y) \d y \d x\\
& \quad + C \int_0^{\eps} x^{b+1} \int_1^{\infty} K_{\R_+}(x,y) \d y \d x + C \int_0^{\eps} x^{b-1} \int_1^{\infty} y^{-1 - 2s} K_{\R_+}(x,y) \d y \d x\\
&\le C \int_0^{\eps}  x^{b}  [(\eps - x)^{1-2s} + |\log x|(1 + \eps^{1-2s})] \d x \\
&\quad + C\int_0^{\eps} x^{b-1} |\log x| \d x + C \int_0^{\eps} x^{b-1}(x^{2} + 1)|\log x| \d x \\
&\le C \eps^{b + 2 - 2s}(1 + |\log \eps|) + C\eps^{b} |\log \eps| + C \eps^{b}(\eps^{2} + 1) |\log \eps| \to 0,
\end{align*}
where we used that $-b_1 < 2s$ in order for the last integral in the second estimate to converge, and $b > 0$ in the third estimate. By interchanging the roles of $g,\varphi$, we obtain
\begin{equation*}\begin{split}
\Big| \int_{0}^{\eps} & \varphi(x) \int_{\eps}^{\infty} (g(x) - g(y)) K_{\R_+}(x,y) \d y \d x \Big| \\
&\le C \int_0^{\eps} x^{2} x^{b-2}  \int_{\eps}^1 |x-y| K_{\R_+}(x,y) \d y \d x + C \int_0^{\eps} x^{2} (1+x^{b - 3})  \int_{\eps}^1 |x-y|^2 K_{\R_+}(x,y) \d y \d x\\
& \quad + C \int_0^{\eps} x^{2} x^{b-1} \int_1^{\infty} K_{\R_+}(x,y) \d y + C \int_0^{\eps} x^{2} \int_1^{\infty} y^{b-1} K_{\R_+}(x,y) \d y \\
&\le C \eps^{b + 2 - 2s}(1 + |\log \eps|) + C\eps^{3} (1 + \eps^{b-3}) |\log \eps| + C \eps^{3}(\eps^{b-1} + 1) |\log \eps| \to 0.\qedhere\end{split}
\end{equation*}
\end{proof}

Our next goal is to prove that $u$ is also a distributional solution to this equation in the sense of \autoref{def:distr-sol}. Our proof follows \cite[p.46]{AFR23}. The only difference comes from the fact that test functions in \autoref{def:distr-sol} do not need to have compact support.

\begin{lemma}
\label{lemma:weak-distr}
Let $L$ be as above and $u$ be a weak solution to
\begin{align*}
Lu = 0 ~~ \text{ in } (0,\infty)
\end{align*}
with Neumann condition on $\{ 0 \}$ in the sense of \autoref{def:weak-sol-1D} such that
\begin{align*}
|u(x)| \le C (1 + x)^{2s-\eps}
\end{align*}
for some $C > 0$ and $\eps \in (0,2s)$. Then, $u$ is also a distributional solution in the sense of \autoref{def:distr-sol} with test-function space $C^2_{2,1+2s}(0,\infty)$.
\end{lemma}

\begin{proof}
Let $\varphi \in C^{2}_{c}([0,\infty))$ and observe that, by the symmetry of $K_{\R_+}$, for any $\eps \in (0,1)$ we have that
\begin{align*}
\int_{0}^{\infty} u(x) L_{\eps} \varphi(x) \d x &= \iint_{D_{\eps}} u(x) (\varphi(x) - \varphi(y))K_{\R_+}(x,y) \d x \d y \\
&= \frac{1}{2} \iint_{D_{\eps}} (u(x) - u(y))(\varphi(x) - \varphi(y)) K_{\R_+}(x,y) \d x \d y,
\end{align*}
where
\begin{align*}
L_{\eps} \varphi(x) &= \int_{(0,\infty) \setminus B_{\eps}(x)} (\varphi(x) - \varphi(y)) K_{\R_+}(x,y) \d y,\\
D_{\eps} &= \{ (x,y) \in (0,\infty) \times (0,\infty) : |x-y| > \eps \}.
\end{align*}
Since $D_{\eps} \to (0,\infty) \times (0,\infty)$, and since $\varphi \in C_{c}^{2}([0,\infty))$ is a valid test function for the weak formulation of the equation for $u$, we get
\begin{align}
\label{eq:conv-weak-sol}
\int_{0}^{\infty} u(x) L_{\eps} \varphi(x) \d x \to 0 ~~ \text{ as } \eps \to 0.
\end{align}
In addition, by repeating the proof of \autoref{lemma:Neumann-L} with $L$ replaced by $L_{\eps}$, we obtain
\begin{align*}
|L_{\eps} \varphi (x)| \le C (1 + |\log x|) \begin{cases}1 + x^{1-2s}, ~~ &\text{ for } x \in (0,2],\\
 (1 + x)^{-1-2s} &\text{ for } x \in (2,\infty),  
\end{cases}
\end{align*} 
for some constant $C > 0$ independent of $\eps$. 

Hence,
\begin{align*}
\int_{0}^{\infty} |u(x)| |L_{\eps} \varphi(x)| \d x &\le C \int_0^{2} (1 + |\log x|)(1 + x^{1-2s}) \d x \\
&\quad + C \int_2^{\infty} x^{2s-\eps} (1 + |\log x|) (1 + x)^{-1-2s} \d x \le C < \infty ,
\end{align*}
and therefore by dominated convergence, 
\begin{align*}
\int_{0}^{\infty} u(x) L_{\eps} \varphi(x) \d x \to \int_{0}^{\infty} u(x) L \varphi(x) \d x.
\end{align*}
Then, by combining this observation with \eqref{eq:conv-weak-sol}, we deduce that for any $\varphi \in C^{2}_c([0,\infty))$ it holds
\begin{align}
\label{eq:distr-help-1}
\int_{0}^{\infty} u(x) L \varphi(x) \d x = 0.
\end{align}
In order to deduce that $u$ is a distributional solution, it remains to prove that \eqref{eq:distr-help-1} holds true for any $\varphi \in C_{2,1+2s}^2(0,\infty)$. This can be achieved by an approximation argument. We will show a slightly stronger result, namely that we can even take $C_{0,1+2s}^2(0,\infty)$ as our test-function space. Indeed, let $\varphi \in C_{0,1+2s}^2(0,\infty)$ and $\varphi_k \in C_c^{2}([0,\infty))$ be such that $\psi_k := \varphi_k - \varphi \to 0$ locally uniformly such that
\begin{align*}
\supp(\varphi_k) \subset [0,k+1], \qquad \varphi_k = \varphi ~~ \text{ in } [0,k] \qquad \Vert \varphi_k - \varphi \Vert_{C^2_{0,1+2s}(0,\infty)} \le C.
\end{align*}
Set $\psi_k = \varphi_k - \varphi$. Then, for any $R \ge 2$, $x \in (0,R)$, and $k \in \N$ large enough, we find that
\begin{align*}
|L \psi_k(x)| \le 2\Vert\psi_k\Vert_{L^{\infty}(0,\infty)} \int_k^{\infty}  K_{\R_+}(x,y) \d y \le C \int_k^{\infty}  |x-y|^{-1-2s} (1 + |\log x|)\d y \le C (1 + |\log x|) k^{-2s}.
\end{align*}
Moreover, by \autoref{lemma:Neumann-L}, for all~$x \in (R,\infty)$ we have that
\begin{align*}
|L \varphi(x)| \le C(1 + |\log x|) (1 + x)^{-1-2s}.
\end{align*}
Therefore, using also the growth of $u$:
\begin{align*}
\int_{0}^{\infty} u(x) L \psi_k(x) \d x &\le C k^{-2s} R^{2s-\eps}\int_0^R  (1 + |\log x|) \d x + C \int_R^{\infty} (1 + |\log x|) (1 + x)^{-1-\eps} \d x \\
&\le C^{-2s} R^{2s-\frac{\eps}{2}} + C R^{-\frac{\eps}{2}}.
\end{align*}
Since $R \ge 2$ was arbitrary, this estimate immediately implies that
\begin{align*}
\int_{0}^{\infty} u(x) L \varphi_k(x) \d x \to \int_{0}^{\infty} u(x) L \varphi(x) \d x,
\end{align*}
which yields \eqref{eq:distr-help-1} also for $\varphi \in C_{0,1+2s}^2(0,\infty)$. The proof is complete.
\end{proof}

We are now in a position to apply \autoref{thm:Liouville} and to deduce \autoref{thm:Neumann-Liouville}, as well as \autoref{thm-intro5}.

\begin{proof}[Proof of \autoref{thm:Neumann-Liouville}]
By combining the results of this section (namely \autoref{lemma:Neumann-L},  \autoref{lemma:L-selfadjoint}, \autoref{lemma:weak-distr}, \autoref{lemma:f-def}, and \autoref{lemma:f-prop}) we can apply the Liouville theorem in \autoref{thm:Liouville} to $u,L$, choosing as a test-function space $\varphi \in C^2_{2,1+2s}(0,\infty)$, i.e. $d = 2$, $a = 2$, and $b = 1 +2s$.
Since by \autoref{prop:zeros-final} the only zeros of $f$ in $[0,B_0) \times i \R$ are given by $\beta \in \{ 0, 2s-1\}$ (there are no zeros $\beta$ with $\re(\beta) \in (0,2s-1)$ due to \cite[Theorem 5.1]{AFR23}), it follows that, for some $A,B \in \R$,
\begin{align*}
u(x) = A + B x^{2s -1}.
\end{align*}

We know by \cite[Lemma 5.9]{AFR23} that $u$ is continuous at $0$, and therefore it must be $A = 0$. Note that this result holds true for any $s \in (0,1)$. Moreover, if $s \le 1/2$ we must also have $B = 0$ by the same reason. 

It remains to prove that $B = 0$ in case $s > 1/2$.
 Let us assume that $B > 0$. Since $u$ is also a weak solution, we can proceed as in \cite[Proof of Corollary 5.8]{AFR23}, take a test function $\eta \in C_c^{\infty}(-\infty,1]$ with $\eta'\le 0$ and $\eta \not\equiv 0$, and use that $x \mapsto x^{2s-1}$ is strictly increasing in $(0,\infty)$, finding that
 \begin{align*}
 0 &= B \int_0^{\infty}\int_0^{\infty} (x^{2s-1} - y^{2s-1}) (\eta(x) - \eta(y)) K_{\R_+}(x,y) \d x \d y \\
 &= B \iint_{(0,\infty)^2 \cap \{ x < y \}} (x^{2s-1} - y^{2s-1}) (\eta(x) - \eta(y)) K_{\R_+}(x,y) \d x \d y \\
 &\quad + B \iint_{(0,\infty)^2 \cap \{ x \ge y \}} (x^{2s-1} - y^{2s-1}) (\eta(x) - \eta(y)) K_{\R_+}(x,y) \d x \d y < 0,
 \end{align*}
which is a contradiction. Hence, necessarily $B = 0$ and the proof is complete.
\end{proof}

\begin{proof}[Proof of \autoref{thm-intro5}]
As before, we combine \autoref{lemma:Neumann-L}, \autoref{lemma:L-selfadjoint}, \autoref{lemma:f-def}, and \autoref{lemma:f-prop} and apply the Liouville theorem in \autoref{thm:Liouville} to $u,L$.
\end{proof}

\subsection{A 1D Liouville theorem with faster growth at infinity} 

The goal of this subsection is to prove \autoref{thm:Neumann-Liouville-growth} with faster growth, i.e. we allow solutions to grow faster than $t^{2s}$ at infinity. 

The proof goes in the exact same way as for \autoref{thm:Neumann-Liouville} in the previous subsection, but using \autoref{thm:Liouville-2D} instead of \autoref{thm:Liouville}. Still, it suffices to establish analogs of  \autoref{lemma:Neumann-L}, \autoref{lemma:L-selfadjoint}, and  \autoref{lemma:weak-distr} for $L = (-\Delta)^s_{\R}$ and $N = \mathcal{N}_{(0,\infty)}^s$.

\begin{lemma}
\label{lemma:Neumann-L-growth}
Let $\varphi \in C^{2,1}_{0,3+2s}(\R)$. Then $x \mapsto (-\Delta)^s \varphi(x)$ and $x \mapsto \cN_{\R_+}^s \varphi(x)$ are continuous and for some $C > 0$:
\begin{align*}
|\cN_{(0,\infty)}^s \varphi (-x)| + |(-\Delta)_{\R}^s \varphi (x)| \le C 
\begin{cases}
1 + x^{1-2s}, ~~ &\text{ for } x \in (0,2],\\
 (1 + x)^{-2-2s} &\text{ for } x \in (2,\infty).
\end{cases}
\end{align*} 
In particular, the pair $(L,N)$ satisfies \eqref{eq:L-decay-2D} for any $\alpha_1 \in (\max\{ 0 ,2s-1\} , 2s)$ and $\alpha_2 = 0$.
\end{lemma}

\begin{proof}
The claim for $x \le 2$ is standard and can be established as in \autoref{lemma:Neumann-L}. Hence, we skip the proof. Let us now take $x > 2$ and compute
\begin{align*}
|(-\Delta)^s_{\R} \varphi(x)| &\le C \int_{B_1} |2 \varphi(x) - \varphi(x+h) - \varphi(x-h)| |h|^{-1-2s} \d h \\
&\quad + C |\varphi(x)| \int_{B_1(x)^c} |x-y|^{-1-2s} \d y + C \left| \int_{B_1(x)^c} \varphi(y) |x-y|^{-1-2s} \d y \right| \\
&= I_1 + I_2 + I_3.
\end{align*}
For $I_1$ we recall the estimate from \eqref{eq:I1-help} and for $I_2$ we simply recall the growth of $\varphi$ to deduce
\begin{align*}
I_1 + I_2 \le C x^{-3-2s}.
\end{align*}
To estimate $I_3$, we recall that $\int_{\R} \varphi(y) \d y = 0$ and deduce from the growth assumption on $\varphi$ that
\begin{align*}
I_3 &\le C \left| \int_{B_1(x)^c} \varphi(y) \left(|x-y|^{-1-2s} - |x|^{-1-2s} \right) \d y \right| + C x^{-1-2s} \left|\int_{B_1(x)} \varphi(y) \d y \right| \\
&\le C \int_{B_1(x)^c} \min\{ 1 , |y|^{-3-2s} \} |y| \min\{ |x-y| , |x| \}^{-2-2s} \d y +  C x^{-1-2s} x^{-3-2s}.
\end{align*}
Hence, it remains to estimate the first summand in the previous estimate. To do so, we split the integration domain into several parts. First, we have
\begin{align*}
\int_{(-\infty,x/2) \cup (2x,\infty)} & \min\{ 1 , |y|^{-3-2s} \} |y| \min\{ |x-y| , |x| \}^{-2-2s} \d y \\
&\le C x^{-2-2s} \int_{\R} \min\{ 1 , |y|^{-3-2s} \} |y|  \d y \le C x^{-2-2s}.
\end{align*}
Next, we observe that 
\begin{align*}
\int_{(x/2 , x-1) \cup (x+1,2x)} & \min\{ 1 , |y|^{-3-2s} \} |y| \min\{ |x-y| , |x| \}^{-2-2s} \d y \\
&\le C x^{-2-2s} \int_{(x/2 , x-1) \cup (x+1,2x)} \min\{ |x-y| , |x| \}^{-2-2s} \d y \le C x^{-2-2s},
\end{align*}
which yields the desired estimate for $|(-\Delta)^s_{\R} \varphi(x)|$. Note that the proof of the estimate for $|\cN_{(0,\infty)}^s \varphi(x)|$ is completely analogous to the estimates of $I_2,I_3$, with the sole two differences that we integrate only over $(0,\infty) \setminus B_{1}(x)$ and that we use $\int_0^{\infty} \varphi(y) \d y = 0$.
\end{proof}

\begin{lemma}
\label{lemma:L-growth-selfadjoint}
It holds
\begin{align*}
\int_0^{\infty} g(x) (-\Delta)_{\R}^s\varphi(x) \d x = \int_0^{\infty} (-\Delta)_{\R}^s (g\1_{\R_+})(x) \varphi(x) \d x + \int_{-\infty}^0 \cN_{(0,\infty)}^s (g \1_{\R_+})(x) \varphi(x) \d x,\\
\int_{-\infty}^0 g(x) \cN_{(0,\infty)}^s \varphi(x) \d x  = \int_0^{\infty} (-\Delta)_{\R}^s (g\1_{\R_-})(x) \varphi(x) \d x + \int_{-\infty}^0 \cN_{(0,\infty)}^s (g \1_{\R_-})(x) \varphi(x) \d x
\end{align*}
for any $g(x) = |x|^{\beta-1}$ with $b = \re(\beta) \in (0,1+2s)$ and $\varphi \in C^{2,1}_{2,3+2s}(0,\infty)$.
In particular, the pair $(L,N)$ satisfies \eqref{eq:self-adjoint-2D} for any $\varphi \in C^{2,1}_{2,3+2s}(0,\infty)$.
\end{lemma}

\begin{proof}
The proof is similar to the one of \autoref{lemma:L-selfadjoint}. We let $\eps \in (0, \frac{1}{2})$ and compute
\begin{align*}
\int_{\eps}^{\infty} g(x) (-\Delta)^s_{\R} \varphi(x) \d x &= \frac{c_s}{2} \iint_{((-\eps,\eps) \times (-\eps,\eps))^c} (g\1_{\R_+}(x) - g\1_{\R_+}(y))(\varphi(x) - \varphi(y)) |x-y|^{-1-2s} \d y \d x \\
&\quad - c_s\int_{0}^{\eps} g(x) \int_{\eps}^{\infty} (\varphi(x) - \varphi(y)) |x-y|^{-1-2s} \d y \d x \\
&= \int_{\eps}^{\infty} (-\Delta)_{\R}^s (g\1_{\R_+})(x)  \varphi(x) \d x + \int_{-\infty}^{-\eps} \cN_{(0,\infty)}^s (g\1_{\R_+})(x) \varphi(x) \d x \\
&\quad + c_s\int_{0}^{\eps} \varphi(x) \int_{\eps}^{\infty} (g\1_{\R_+}(x) - g\1_{\R_+}(y)) |x-y|^{-1-2s} \d y \d x \\
&\quad -c_s \int_{0}^{\eps} g\1_{\R_+}(x) \int_{\eps}^{\infty} (\varphi(x) - \varphi(y)) |x-y|^{-1-2s} \d y \d x \\
&\quad  + c_s\int_{-\eps}^{0} \varphi(x) \int_{\eps}^{\infty} (g\1_{\R_+}(x) - g\1_{\R_+}(y)) |x-y|^{-1-2s} \d y \d x.
\end{align*}
By the dominated convergence theorem (note that the integrals converge absolutely by \autoref{lemma:Neumann-L-growth}),
\begin{align*}
\int_{\eps}^{\infty} g(x) (-\Delta)_{\R}^s \varphi(x) \d x &\to \int_{0}^{\infty} g(x) (-\Delta)_{\R}^s \varphi(x) \d x, \\
 \int_{\eps}^{\infty} (-\Delta)_{\R}^s (g\1_{\R_+})(x)  \varphi(x) \d x &\to \int_{0}^{\infty} (-\Delta)_{\R}^s (g\1_{\R_+})(x)  \varphi(x) \d x,\\
 \int_{-\infty}^{-\eps} \cN_{(0,\infty)}^s (g\1_{\R_+})(x) \varphi(x) \d x &\to \int_{-\infty}^{0} \cN_{(0,\infty)}^s (g\1_{\R_+})(x) \varphi(x) \d x.
\end{align*}
Moreover, under our assumptions on $g$ and~$\varphi$,
\begin{align*}
\int_{0}^{\eps} g(x) \int_{\eps}^{\infty} (\varphi(x) - \varphi(y)) |x-y|^{-1-2s} \d y \d x &\to 0, \\
 \int_{0}^{\eps} \varphi(x) \int_{\eps}^{\infty} (g(x) - g(y)) |x-y|^{-1-2s} \d y \d x &\to 0,\\
\int_{-\eps}^{0} \varphi(x) \int_{\eps}^{\infty} (g\1_{\R_+}(x) - g\1_{\R_+}(y)) |x-y|^{-1-2s} \d y \d x &\to 0,
\end{align*}
which can be proved in an analogous way as \autoref{lemma:L-selfadjoint}. The second claim of the lemma follows in a similar way.
\end{proof}

\begin{lemma}
\label{lemma:weak-distr-2}
Let $u$ be a distributional solution with faster growth to \eqref{eq:faster-growth} with $u(0) = 0$ in the sense of \autoref{def:faster-growth} such that
\begin{align*}
|u(x)| \le C (1 + x)^{B_0-\eps}
\end{align*}
for some $C > 0$ and $\eps > 0$. Then, $u$ is also a distributional solution in the sense of \autoref{def:distr-sol-2D} for $(L,N) = ((-\Delta)^s_{\R} , \cN_{(0,\infty)}^s)$ with test-function space $C^{2,1}_{2,3+2s}(\R)$.
\end{lemma}

\begin{proof}
The proof goes by an approximation argument exactly as in the proof of \autoref{lemma:weak-distr} (after \eqref{eq:distr-help-1}), using also \autoref{lemma:Neumann-L-growth} and the fact that $B_0 - \eps < 1 + 2s$ due to \autoref{prop:zeros-final}.
\end{proof}

We are now in a position to prove \autoref{thm:Neumann-Liouville-growth}.

\begin{proof}[Proof of \autoref{thm:Neumann-Liouville-growth}]
By \autoref{lemma:weak-distr-2} we deduce that $u$ is a distributional solution in the sense of \autoref{def:distr-sol-2D} with $(L,N) = ((-\Delta)^s_{\R} , \cN^s_{(0,\infty)})$. Additionally, we recall \autoref{lemma:Neumann-L-growth} and \autoref{lemma:L-growth-selfadjoint}, as well as \autoref{lemma:f-def-2D}, \autoref{cor:zeros-f2}, and \autoref{lemma:f-prop}, which allows us to apply the Liouville theorem in \autoref{thm:Liouville-2D} to $u$ with $(L,N)$, choosing as a test-function space $\varphi \in C^{2,1}_{2,3+2s}(\R)$. Since, by \autoref{prop:zeros-final}, the only zeros of $f_1,f_2$ in $[0,B_0) \times i\R$ are given by $\beta \in \{ 0 , 2s-1\}$, and since $f_1,f_2$ have poles at $\beta = 2s$ by \autoref{cor:zeros-f2}, it follows that, for some $A,B \in \R$,
\begin{align*}
u(x) = A + B x^{2s-1} ~~ \forall x > 0.
\end{align*}
Since $u \in C^{\max\{2s-1,0\}+\eps}_{loc}(\R)$ by assumption, we deduce that $A=B=0$. Moreover, $u(x) = 0$ for all~$x < 0$,
since
\begin{equation*}
0 = \cN_{\R_+}^s u(x) = u(x) \left( c_s \int_0^{\infty} |x-y|^{-1-2s} \d y \right) = \frac{c_s}{2s} \frac{u(x)}{|x|^{2s}} .\qedhere
\end{equation*}
\end{proof}

\section{Boundary regularity for the nonlocal Neumann problem}
\label{sec6}

The goal of this section is to prove the optimal boundary regularity for the Neumann problem, namely \autoref{thm0} and \autoref{thm:main-intro}.  This section is structured as follows: In \autoref{subsec:correction}, we analyze the corrector terms that are needed to prove regularity results of order larger that one. In \autoref{subsec:reg-outside-sol-dom}, we prove regularity results that hold true outside the solution domain, and in \autoref{subsec:prep-expansion} we collect several lemmas that will be used in the proofs of boundary expansions. \autoref{subsec:unbounded-source-terms} and \autoref{subsec:bdry-Holder} provide auxiliary lemmas that are needed in order to treat unbounded source terms in \autoref{thm:main-intro}. Finally, \autoref{subsec:expansion-less-2s}, \autoref{subsec:reg-less-2s} and \autoref{subsec:expansion-bigger-2s}, \autoref{subsec:reg-bigger-2s}, respectively, provide expansions and regularity results of order less than $2s$ and bigger than $2s$, respectively, and allow us to prove our main results in \autoref{subsec:main-proofs}.

Given an open set $\Omega \subset \R^n$, we define the regional operator, corresponding to the nonlocal Neumann problem in $\Omega$, as 
\begin{align*}
L_{\Omega} u(x): = \text{p.v.} \int_{\Omega} (u(x) - u(y)) K_{\Omega}(x,y) \d y,
\end{align*}
where $K_{\Omega}$ is given by
\begin{align}
\label{eq:K-Omega-def}
\begin{split}
K_{\Omega}(x,y) &:= \frac{c_{n,s}}{|x-y|^{n+2s}} + k_{\Omega}(x,y), \qquad c_{n,s} := 4^s s \pi^{-\frac{n}{2}} \frac{\Gamma\left( \frac{n}{2} + s \right)}{\Gamma(1-s)},\\
k_{\Omega}(x,y) &:= c_{n,s} \int_{\Omega^c} \frac{|x-z|^{-n-2s} |y-z|^{-n-2s} \d z}{\int_{\Omega} |z-w|^{-n-2s} \d w}, ~~ x,y \in \Omega.
\end{split}
\end{align}

We would like to emphasize that no boundedness or connectedness of $\Omega$ are imposed at this stage or in the main results of this section.

Moreover, as in \cite{AFR23} we introduce the corresponding bilinear form
\begin{align*}
B_{\Omega}(u,v) := \int_{\Omega}\int_{\Omega} (u(x) - u(y)) (v(x) - v(y)) K_{\Omega}(x,y)\d y \d x,
\end{align*}
and denote by $H_K(\Omega)$ the space of functions 
\begin{align*}
H_K(\Omega) := \left\{ w \in H^s(\Omega) ~~\big|~~ \Vert w \Vert_{H_K(\Omega)} := \Vert w \Vert_{L^2(\Omega)} + B(u,u) < \infty \right \}.
\end{align*}

We let $H_{K,loc}(\Omega)$ be the space of functions $u \in L^2_{loc}(\Omega)$ for which the following localized energy is finite for any ball $B \subset \R^n$:
\begin{align*}
\int_{\Omega \cap B} \int_{\Omega \cap B} |u(x) - u(y)|^2 K_{\Omega}(x,y) \d y \d x < \infty,
\end{align*}
and  for which
\begin{align*}
B_{\Omega}(u,\eta) < \infty \quad \forall \eta \in \big\{ \eta \in H_K(B) ~:~ \eta = 0 ~\text{ in } \Omega \setminus B ~~ \text{for some ball } B \subset \R^n \big\}.
\end{align*}

Weak solutions are defined in the same way as in \cite{AFR23}:

\begin{definition}
\label{def:weak-sol}
Let $\Omega \subset \R^n$ be an open set with Lipschitz boundary and $B \subset \R^n$ be a ball.
We say that $u \in H_{K,\loc}$ is a weak supersolution to $L_{\Omega} u = h$ in $B \cap \Omega$ with Neumann conditions on $\partial \Omega \cap B$ for some $h \in (H_{K,\loc}(\Omega))'$, if, for any
nonnegative function $\eta \in H_K(B)$ with $\eta \equiv 0$ in $\Omega \setminus B$, \begin{align*}
B_{\Omega}(u,\eta) \ge \int_{B \cap \Omega} h \eta \d x.
\end{align*}
We say that $u$ is a weak subsolution in $\Omega \cap B$, if the previous estimate holds true with $\ge$ replaced by $\le$. We say that $u$ is a weak solution in $\Omega \cap B$ if it is a weak supersolution and a weak subsolution. We sat that $u$ is a weak (super/sub)-solution if the previous definitions hold true for all balls $B \subset \R^n$.
\end{definition}

Moreover, we define weak solutions to the nonlocal Neumann problem as follows.

\begin{definition}
\label{def:weak-sol-Neumann}
Let $\Omega \subset \R^n$ be an open set with Lipschitz boundary and $B \subset \R^n$ be a ball. We say that $u \in V^s_\loc(\Omega | \R^n)$ is a weak solution to 
\begin{align*}
\begin{cases}
(-\Delta)^s u &= h ~~ \text{ in } \Omega \cap B,\\
\cN^s_{\Omega} u &= 0 ~~ \text{ in } B \setminus \Omega
\end{cases}
\end{align*}
where $h \in H^s(\Omega \cap B')'$ for any ball $B' \subset \R^n$ if for all $\eta \in H^s(\R^n)$ with $\eta \equiv 0$ in $B^c$ it holds
\begin{align*}
\cE_{\Omega}(u,\eta) := \frac{c_{n,s}}{2} \iint_{(\Omega^c \times \Omega^c)^c} (u(x) - u(y)) (\eta(x) - \eta(y)) |x-y|^{-n-2s} \d y \d x = \int_{\Omega} h(x) \eta(x) \d x.
\end{align*}
\end{definition}

The space $V^s(\Omega | \R^n)$ is defined as the space of all functions $u \in L^2(\Omega)$ such that $\cE_{\Omega}(u,u) < \infty$. Moreover, the space $V^s_{\loc}(\Omega | \R^n)$  is defined as the space of all functions $u \in L^2_\loc(\R^n)$ such that for any ball $B \subset \R^n$,
\begin{align*}
\iint_{(\Omega^c \times \Omega^c)^c \cap \big( (B \times \R^n) \cup (\R^n \times B) \big)} |u(x) - u(y)|^2 |x-y|^{-n-2s} \d y \d x < \infty.
\end{align*}
Note that if $\Omega \subset \R^n$ is bounded, then $V^s_{\loc}(\Omega | \R^n) = V^s(\Omega | \R^n)$.

Finally, we give a meaning to the nonlocal Neumann problem in possibly unbounded domains in case the solution grows too fast at infinity. This definition will be used in our proof of \autoref{thm:bdry-expansion-higher}.

\begin{definition}
\label{def:up-to-poly}
Let $\Omega \subset \R^n$ with $\partial \Omega \in C^{0,1}$. Let $u \in C^{2s+\eps}_{loc}(\Omega)$ and suppose that
\begin{align*}
|u(x)| \le C (1 + |x|)^{2s + 1 -\eps} ~~ \forall x \in \R^n
\end{align*}
for some $\eps \in (0,1)$.

We say that
\begin{align*}
\begin{cases}
(-\Delta)^s u &\overset{1}{=} 0 ~~ \text{ in } \Omega,\\
\cN_{\Omega} u &\overset{1}{=} 0 ~~ \text{ in } \Omega^c,
\end{cases}
\end{align*}
if, for any $R \ge 2$,
\begin{align*}
\begin{cases}
(-\Delta)^s (u \1_{B_R})(x) &= E_R(x) + c_R ~~ \text{ in } B_{R/2} \cap \Omega,\\
\cN^s_{\Omega}(u\1_{B_R})(x) &= F_R(x) + d_R ~~ \text{ in } B_{R/2} \setminus \Omega
\end{cases}
\end{align*}
for some constants $c_R,d_R \in \R$ and functions $E_R : B_{R/2} \cap \Omega \to \R$ and $F_R : B_{R/2} \setminus \Omega \to \R$ such that $E_R \to 0$ and $F_R \to 0$ locally uniformly in $\Omega$ and in $\R^n \setminus \overline{\Omega}$, respectively.
\end{definition}

\subsection{First order correction}
\label{subsec:correction}

In order to prove that solutions possess  $C^{1,\alpha}$ regularity by a blow-up procedure, we will need to rule out linear functions of the form $b \cdot x$ with $b \in \R^n$ and $b \cdot e_n = 0$ in the Liouville theorem (see \autoref{thm:Neumann-Liouville-nd}).
This will be achieved by subtracting variants of $b \cdot x$ that are adapted to the geometry of the domain $\Omega$.

Let $s \in (0,1)$ and $\partial \Omega \in C^{1,\alpha}$ for some $\alpha \in (0 , 1]$ and with $0 \in \partial \Omega$ and $\nu(0) = e_n$. Let $\Phi : \R^n \to \R^n$ be a $C^{1,\alpha}$ diffeomorphism  with $\Phi(0) = 0$ and $D \Phi(0) = I$ and such that $\Phi(\{ x_n > 0 \} \cap B_1) = \Omega \cap B_1$. Moreover, we can choose $\Phi$ in such a way that $(x_n)_+ = d_{\Omega}(\Phi(x))$ and that it is $C^{\infty}$ in the interior of $\Omega \cap B_1$. Moreover, note that we can extend the function $\Phi$ in a bounded, smooth way in $\Omega \setminus B_1$, so that $\Phi^{-1} \in C^{1,\alpha}(\overline{\Omega}) \cap C^{\infty}_{loc}(\Omega)$.

In addition, we can extend the diffeomorphism $\Phi$ to $\{ x_n < 0 \}$ in such a way that $\Phi(\{ x_n \le 0 \} \cap B_{1/2}) = B_{1/2} \setminus \Omega$ and $d_{\Omega^c}(\Phi(y)) = (y_n)_-$ for all $y \in \{ y_n \le 0 \} \cap B_{1/2}$.

Let $b \in \R^n$ with $b_n = b \cdot e_n = 0$ and define $P_b(x) = b \cdot \Phi^{-1}(x)$ for $x \in \Omega$. Note that, by the
construction of $\Phi$, one has that $e_n = D \Phi(x)^T \nabla d_{\Omega}(\Phi(x))$ and therefore $D \Phi(x) e_n = \nabla d_{\Omega}(\Phi(x)) = \nu(\Phi(x))$. Moreover, by construction of $\Phi$, it holds $P_b \in C^{1,\alpha}(\overline{\Omega}) \cap C^{\infty}_{loc}(\Omega)$.

By taking inverses, we get that~$D \Phi^{-1}(\Phi(x)) \nu(\Phi(x)) = e_n$, and thus, since $b_n = 0$,
\begin{align}
\label{eq:Pb-normal-zero}
\partial_{\nu} P_b(x) = b \cdot (D\Phi^{-1}(x) \cdot \nu(x)) = 0 ~~ \text{ for } x \in \partial \Omega \cap B_1.
\end{align}
Moreover, by the extension of $\Phi$, we can assume that \eqref{eq:Pb-normal-zero} holds true on the whole boundary $\partial \Omega$, and we can assume that $P_b$ is bounded in $\R^n \setminus (B_1 \cup \Omega)$. Finally, note that $P_b(0) = 0$ by construction.

We have the following result:

\begin{lemma}
\label{lemma:LPb-estimate}
Let $\Omega$ be a bounded open set with $\partial \Omega \in C^{1,\alpha}$ for some $\alpha \in (0,1]$ and $P_b$ be as above for some $b \in \R^n$ with $b \cdot e_n = 0$. Then, for any $x \in \Omega \cap  B_{1/4}$,
\begin{align}
\label{eq:LPb-estimate}
|L_{\Omega} P_b(x)| \le C |b| \left( 1 + |\log(d_{\Omega}(x))| + d_{\Omega}(x)^{1+\alpha - 2s} + |x|^{\alpha} ( 1+ d_{\Omega}(x)^{1-2s}  + |\log d_{\Omega}(x)|) \right),
\end{align}
where $C > 0$ only depends on $n,s,\alpha$, and the $C^{1,\alpha}$ norm of $\Omega$.
\end{lemma}

\begin{proof}

First of all, note that it suffices to prove the result in case $|b| = 1$. The general case can then be immediately deduced by scaling. Let us define, for any~$y \in \Omega^c$,
\begin{align}
\label{eq:Pb-complement-rep}
P_b(y) := \left( \int_{\Omega} |y-z|^{-n-2s} \d z \right)^{-1} \int_{\Omega} P_b(z) |y-z|^{-n-2s} \d z.
\end{align}
We observe that $\mathcal{N}^s_{\Omega} P_b = 0$ in $\Omega^c$ and hence
\begin{align*}
L_{\Omega} P_b (x) = (-\Delta)^s P_b(x) ~~ \forall x \in \Omega,
\end{align*}
so it remains to estimate the fractional Laplacian applied to $P_b$ in order to prove the claim \eqref{eq:LPb-estimate}.

Note that $P_b \in C^{1,\alpha}(\overline{\Omega})$. Therefore, for all $x \in \Omega$,
\begin{align*}
\int_{\Omega} (P_b(x) - P_b(y)) |x-y|^{-n-2s} \d y &= \int_{\Omega}  (P_b(x) - P_b(y) + (x-y) \nabla P_b(x)) |x-y|^{-n-2s} \d y \\
&\quad - \int_{\Omega } (x-y) \nabla P_b(x) |x-y|^{-n-2s} \d y \\
&=  \int_{\Omega}  (P_b(x) - P_b(y) + (x-y) \nabla P_b(x)) |x-y|^{-n-2s} \d y \\
&\quad - (n-2s+2)^{-1} \int_{\partial \Omega} (\nabla P_b(x) - \nabla P_b(y)) \nu(y) |x-y|^{-n-2s+2} \d y,
\end{align*}
where we used the divergence theorem together with the identity
\begin{align*}
\nabla_y |x-y|^{-n-2s+2} = (n-2s+2) (x-y)|x-y|^{-n-2s},
\end{align*}
as well as $\partial_{\nu} P_b (y) = \nabla P_b(y) \cdot \nu(y) = 0$ for all $y \in \partial \Omega$ by \eqref{eq:Pb-normal-zero}. 

Since $P_b \in C^{1,\alpha}(\overline{\Omega})$, we can estimate
\begin{align*}
|P_b(x) - P_b(y) + (x-y) \nabla P_b(x)| + |\nabla P_b(x) - \nabla P_b(y)||x-y| \le C |x-y|^{1+\alpha} ~~ \forall x,y \in \overline{\Omega}.
\end{align*}
Moreover, given $y \in B_{d_{\Omega}(x)/2}(x)$, since $P_b \in C^{\infty}_{loc}(\Omega)$ we see that
\begin{align*}
|P_b(x) - P_b(y) + (x-y) \nabla P_b(x)| \le C |x-y|^2.
\end{align*}
As a consequence of these observations,
\begin{align*}
& \left| \int_{\Omega} (P_b(x) - P_b(y)) |x-y|^{-n-2s} \d y \right| \\
&\quad \le C \int_{B_{\frac{d_{\Omega}(x)}{2}}(x)} \hspace{-0.2cm}|x-y|^{-n-2s+2} \d y + C \int_{\Omega \setminus B_{\frac{d_{\Omega}(x)}{2}}(x)} \hspace{-0.2cm}|x-y|^{-n-2s+1+\alpha} \d y + C \int_{\partial \Omega} \hspace{-0.2cm}|x-y|^{-(n-1)-2s+1+\alpha} \d y \\
&\quad \le C d_{\Omega}^{-2s+1+\alpha}(x) + C.
\end{align*}

It remains to estimate the integral over $\Omega^c$. To do so, let us define for $y \in B_{1/2} \setminus \Omega$ the projection of $y$ to $\partial \Omega$ by $\bar{y} \in \partial \Omega$ (if the projection is not unique, we replace $B_{1/2}$ by a smaller ball and proceed in the same way). It thereby follows that, for all~$x \in \Omega$,
\begin{align*}
& \int_{B_{1/2} \setminus \Omega} (P_b(y) - P_b(\bar{y})) |x-y|^{-n-2s} \d y \\
&=  \int_{B_{1/2} \setminus \Omega} \left( \int_{\Omega} (P_b(z) - P_b(\bar{y}) ) |y-z|^{-n-2s} \d z \right) \left(\int_{\Omega} |y-w|^{-n-2s} \d w \right)^{-1} |x-y|^{-n-2s} \d y  \\
&= \int_{\{ y_n \le 0 \} \cap B_{1/2}} \left( \int_{\{ z_n > 0\} \cap B_1} b \cdot \left[ z - \Phi^{-1}( \overline{\Phi(y)}) \right] |\Phi(y) - \Phi(z)|^{-n-2s} |\det D \Phi(z)| \d z \right) \\
&\qquad\qquad\qquad \left(\int_{\Omega} |\Phi(y)-w|^{-n-2s} \d w \right)^{-1} |x- \Phi(y)|^{-n-2s} |\det D \Phi(y)| \d y \\
&\quad + \int_{B_{1/2} \setminus \Omega} \left( \int_{\Omega \setminus B_1} (P_b(z) - P_b(\bar{y}) ) |y-z|^{-n-2s} \d z \right) \left(\int_{\Omega} |y-w|^{-n-2s} \d w \right)^{-1} |x-y|^{-n-2s} \d y .
\end{align*}

Since $\Phi$ is a $C^{1,\alpha}$ diffeomorphism, we observe that
\begin{align*}
|\det D \Phi(z)| = 1 + O(|z|^{\alpha}), \quad |\det D \Phi(y)| = 1 + O(|y|^{\alpha}) ~~ \forall z \in \{ z_n \ge 0 \} \cap B_1, ~~ \forall y \in \{ y_n \le 0 \} \cap B_{1/2}.
\end{align*}

Furthermore, for all $y \in \{ y_n \le 0\} \cap B_{1/2}$ and $z \in \{ z_n \ge 0 \} \cap B_1$,
\begin{align}
\label{eq:Phi-Lipschitz}
 C^{-1} |y-z| \le |\Phi(y) - \Phi(z)| \le C |y-z|.
\end{align}
Also, for all $y \in \{ y_n \le 0 \} \cap B_{1/2}$,
\begin{align*}
W(y) := \left(\int_{\Omega} |\Phi(y)-w|^{-n-2s} \d w \right)^{-1} &\le C \left(\int_{\{ w_n \ge 0 \} \cap B_1} |\Phi(y)-\Phi(w)|^{-n-2s} \d w \right)^{-1} \\
&\le C \left(\int_{\{ w_n \ge 0 \} \cap B_1} |y-w|^{-n-2s} \d w \right)^{-1} \\
&\le C (y_n)_-^{2s} = C d_{\Omega}^{2s}(\Phi(y)).
\end{align*}

We now observe that
\begin{align}
\label{eq:kernel-est-trafo}
\begin{split}
\left| |\Phi(y) - \Phi(z)|^{-n-2s} - |y-z|^{-n-2s} \right| &= \frac{| |y-z|^{n+2s} - |\Phi(y) - \Phi(z)|^{n+2s}|}{|y-z|^{n+2s} |\Phi(y) - \Phi(z)|^{n+2s}} \\
&\le \frac{C}{|y-z|^{n+2s} |\Phi(y) - \Phi(z)|^{n+2s}} \left| \int_{|y-z|}^{|\Phi(y) - \Phi(z)|} \tau^{n+2s-1} \d \tau \right| \\
&\le C \frac{|y-z|^{n+2s-1} + |\Phi(y) - \Phi(z)|^{n+2s-1}}{|y-z|^{n+2s} |\Phi(y) - \Phi(z)|^{n+2s}} ||y-z| - |\Phi(y) - \Phi(z)|| \\
&\le C |y-z|^{-n-2s} \left( |y-z|^{\alpha} + |y|^{\alpha} \right),
\end{split}
\end{align}
where we used \eqref{eq:Phi-Lipschitz} and
\begin{align}
\label{eq:b-difference-est}
\begin{split}
||y-z| - |\Phi(y) - \Phi(z)|| &\le |\Phi(y) - \Phi(z) - (z-y)D \Phi(y)| + |(y-z)D \Phi(0) - (y-z) D \Phi(y)| \\
&\le C |y-z|^{1+\alpha} + C|y-z| |y|^{\alpha}.
\end{split}
\end{align}


Moreover, for any $y \in B_{1/2} \setminus \{ y_n \le 0 \}$
we know that $\Phi(y) - \overline{\Phi(y)}$ is a multiple of $\nu(\overline{\Phi(y)})$ and therefore, by \eqref{eq:Pb-normal-zero}, we have $b \cdot [D \Phi^{-1}(\overline{\Phi(y)}) \cdot (\Phi(y) - \overline{\Phi(y)})] = 0$. This gives that
\begin{align}
\label{eq:b-linear-term-est}
\begin{split}
|b \cdot [ y - \Phi^{-1}(\overline{\Phi(y)})] | &= |b \cdot \Phi^{-1}(\Phi(y)) - b \cdot \Phi^{-1}(\overline{\Phi(y)})  - b \cdot [D \Phi^{-1}(\overline{\Phi(y)}) \cdot (\Phi(y) - \overline{\Phi(y)})]| \\
&\le C |\Phi(y) - \overline{\Phi(y)}|^{1+\alpha} \le C d_{\Omega}^{1+\alpha}(\Phi(y)).
\end{split}
\end{align}

As a result,
\begin{align*}
& \left| \int_{B_{1/2} \setminus \Omega} (P_b(y) - P_b(\bar{y})) |x-y|^{-n-2s} \d y \right| \\
&\le \Bigg|  \int_{\{ y_n \le 0 \} \cap B_{1/2}} \left( \int_{\{ z_n > 0\} \cap B_1} b \cdot \left[ z - y \right] |y-z|^{-n-2s} \d z \right)  W(y) |x- \Phi(y)|^{-n-2s} |\det D \Phi(y)| \d y \Bigg| \\
&+C \int_{\{ y_n \le 0 \} \cap B_{1/2}} \left( \int_{\{ z_n > 0\} \cap B_1} \left\{ |z-y|(|z-y|^{\alpha} + |y|^{\alpha}) + d_{\Omega}^{1+\alpha}(\Phi(y))  \right\} |y - z|^{-n-2s} \d z \right)\\
&\qquad\qquad\qquad\qquad\qquad\qquad\qquad\qquad\qquad\qquad\qquad d_{\Omega}^{2s}(\Phi(y)) |x- \Phi(y)|^{-n-2s} |\det D \Phi(y)| \d y  \\
& \quad + \int_{B_{1/2} \setminus \Omega} \left| \int_{\Omega \setminus B_1} (P_b(z) - P_b(\bar{y}) ) |y-z|^{-n-2s} \d z \right| \left(\int_{\Omega} |y-w|^{-n-2s} \d w \right)^{-1} |x-y|^{-n-2s} \d y \\
&=: I_1 + I_2 + I_3.
\end{align*}

For $I_1$ let us observe that a cancellation is happening within the integral in $z$, making it bounded from above, uniformly in $y$. Hence, transforming back from the half-space to $\Omega$,
\begin{align*}
I_1 &\le C \int_{\{ y_n \le 0 \} \cap B_{1/2}} W(y) |x- \Phi(y)|^{-n-2s} |\det D \Phi(y)| \d y \\
&\le C \int_{\{ y_n \le 0 \} \cap B_{1/2}} d_{\Omega^c}^{2s}(\Phi(y)) |x - \Phi(y)|^{-n-2s} |\det D \Phi(y)| \d y  \\
&=C \int_{\Omega^c \cap B_{1/2}} d_{\Omega^c}^{2s}(y) |x-y|^{-n-2s} \d y \\
&\le C (1 + |\log d_{\Omega}(x)|).
\end{align*}

For $I_2$, we transform back from the half-space to $\Omega$ and use \eqref{eq:Phi-Lipschitz}, finding that
\begin{align*}
I_2 &\le C \int_{\{ y_n \le 0 \} \cap B_{1/2}} \left( \int_{\{ z_n > 0\} \cap B_1} |y - z|^{-n-2s + 1 + \alpha} \d z \right) d_{\Omega^c}^{2s}(\Phi(y)) |x - \Phi(y)|^{-n-2s} |\det D \Phi(y)| \d y \\
&\quad + C \int_{\{ y_n \le 0 \} \cap B_{1/2}} \left( \int_{\{ z_n > 0\} \cap B_1}  |y - z|^{-n-2s+1} \d z \right)  |y|^{\alpha} d_{\Omega^c}^{2s}(\Phi(y)) |x - \Phi(y)|^{-n-2s} |\det D \Phi(y)| \d y \\
&\quad + C \int_{\{ y_n \le 0 \} \cap B_{1/2}} \left( \int_{\{ z_n > 0\} \cap B_1}  |y - z|^{-n-2s} \d z \right) d_{\Omega^c}^{1+\alpha + 2s}(\Phi(y)) |x - \Phi(y)|^{-n-2s} |\det D \Phi(y)| \d y \\
&\le C \int_{\{ y_n \le 0 \} \cap B_{1/2}} (1 + d_{\Omega^c}^{1+\alpha - 2s}(\Phi(y))) d_{\Omega^c}(\Phi(y))^{2s} |x - \Phi(y)|^{-n-2s} |\det D \Phi(y)| \d y \\
&\quad  + C \int_{\{ y_n \le 0 \} \cap B_{1/2}} (1 + d_{\Omega^c}^{1-2s}(\Phi(y))) |\Phi(y)|^{\alpha} d_{\Omega^c}^{2s}(\Phi(y)) |x - \Phi(y)|^{-n-2s} |\det D \Phi(y)| \d y \\
&\quad  + C \int_{\{ y_n \le 0 \} \cap B_{1/2}} d_{\Omega^c}^{1+\alpha}(\Phi(y)) |x - \Phi(y)|^{-n-2s} |\det D \Phi(y)| \d y \\
&=C \int_{\Omega^c \cap B_{1/2}} (1 + d_{\Omega^c}^{1+\alpha - 2s}(y)) d_{\Omega^c}^{2s}(y) |x - y|^{-n-2s} \d y \\
&\quad  + C \int_{\Omega^c \cap B_{1/2}} (1 + d_{\Omega^c}^{1-2s}(y)) |y|^{\alpha} d_{\Omega^c}^{2s}(y) |x - y|^{-n-2s} \d y \\
&\quad  + C \int_{\Omega^c \cap B_{1/2}} d_{\Omega^c}^{1+\alpha}(y) |x - y|^{-n-2s} \d y \\
&\le C (1 + |\log d_{\Omega}(x)| + d_{\Omega}^{1+\alpha-2s}(x) + |x|^{\alpha} d_{\Omega}^{1-2s}(x) + |x|^{\alpha} (1 + |\log d_{\Omega}(x)|)).
\end{align*}

For $I_3$, we recall that $|P_b(z)| + |P_b(\overline{y})| \le C$ and we conclude that
\begin{align*}
I_3 &\le \int_{B_{1/2} \setminus \Omega} \left| \int_{\Omega \setminus B_1} |y-z|^{-n-2s} \d z \right| \left(\int_{\Omega} |y-w|^{-n-2s} \d w \right)^{-1} |x-y|^{-n-2s} \d y \\
&\le C  \int_{B_{1/2} \setminus \Omega} d_{\Omega^c}^{2s}(y) |x-y|^{-n-2s} \d y  \\
&\le C (1 + |\log d_{\Omega}(x)|).
\end{align*}

Altogether, we have shown that
\begin{align*}
& \left| \int_{B_{1/2} \setminus \Omega} (P_b(y) - P_b(\bar{y})) |x-y|^{-n-2s} \d y \right| \le C(1 + |\log d_{\Omega}(x)|) +  C|x|^{\alpha}( d_{\Omega}^{1-2s}(x) + 1 + |\log d_{\Omega}(x)|).
\end{align*}

Next, we perform the same change of variable as before
and apply \eqref{eq:kernel-est-trafo} with $y := y$ and $z := \Phi^{-1}(x)$. In this way, keeping in mind~\eqref{eq:b-difference-est} and \eqref{eq:b-linear-term-est}, we arrive at
\begin{align*}
& \left| \int_{B_{1/2} \setminus \Omega} (P_b(\bar{y}) - P_b(x)) |x-y|^{-n-2s} \d y \right| \\
&= \int_{\{ y_n \le 0 \} \cap B_{1/2}} b \cdot \left[ \Phi^{-1}(\overline{\Phi(y)}) - \Phi^{-1}(x) \right] |x  - \Phi(y)|^{-n-2s} |\det D \Phi(y)| \d y \\
&\le \left| \int_{\{ y_n \le 0 \} \cap B_{1/2}} b \cdot \left[ y - \Phi^{-1}(x) \right] | y  - \Phi^{-1}(x)|^{-n-2s}  \d y \right| \\
&\qquad + C \int_{\{ y_n \le 0 \} \cap B_{1/2}} \left\{ |y  - \Phi^{-1}(x)|(|y  - \Phi^{-1}(x)|^{\alpha}  + |y|^{\alpha}) + (y_n)_-^{1+\alpha} \right\} |y - \Phi^{-1}(x)|^{-n-2s} \d y \\
&=: J_1 + J_2.
\end{align*}

Note that $J_1$ is bounded from above, uniformly in $x \in B_{1/4}$, since a cancellation takes place. Thus, it remains to estimate $J_2$. To this end, after an elementary computation we see that
\begin{align*}
J_2 &\le C (1 + ((\Phi^{-1}(x))_n)_-^{1+\alpha - 2s} + |\Phi^{-1}(x)|^{\alpha} ((\Phi^{-1}(x))_n)_-^{1-2s} + |\log (\Phi^{-1}(x))_n)_- | ) \\
&\le C (1 + d_{\Omega}^{1+\alpha - 2s}(x) + |x|^{\alpha} (d_{\Omega}^{1-2s}(x) + 1  + |\log d_{\Omega}(x)|) ).
\end{align*}

Hence, \begin{align*}
& \left| \int_{B_{1/2} \setminus \Omega} (P_b(x) - P_b(y)) |x-y|^{-n-2s} \d y \right| \\
&\qquad \le C(1 + |\log d_{\Omega}(x)|) + Cd_{\Omega}^{1-2s+\alpha}(x) + C|x|^{\alpha}( d_{\Omega}^{1-2s}(x) + 1 + |\log d_{\Omega}(x)|).
\end{align*}


Finally, given $y \in B_{1/2}^c \setminus \Omega$, we recall \eqref{eq:Pb-complement-rep} and the fact that $|P_b(z)| \le C \min\{|z| , 1\} \le C$ for any $z \in \Omega$. As a consequence,
\begin{align*}
|P_b(y)| \le C d_{\Omega^c}^{2s}(y) \int_{\Omega} |y-z|^{-n-2s} \d z \le C.
\end{align*}
Accordingly, since $x \in \Omega \cap B_{1/4}$,  
\begin{align*}
\left| \int_{B_{1/2}^c \setminus \Omega} (P_b(x) - P_b(y)) |x-y|^{-n-2s} \d y \right| &\le C \int_{B_{1/2}^c \setminus \Omega} |x-y|^{-n-2s} \d y  \le C.
\end{align*}
The proof of \eqref{eq:LPb-estimate} is thereby complete.
\end{proof}

In order to show that solutions to the nonlocal Neumann problem possess regularity of order $2s+\eps > 1$ for some $\eps > 0$ (for certain values of $s$), we need to prove that $L_{\Omega}P_b \in C^{\eps}$. To establish this, we need to improve the $L^{\infty}$ estimate in \autoref{lemma:LPb-estimate}.

To do so, we observe that if $1+\alpha > 2s$, then we can always construct $P_b$ in such a way that
\begin{align}
\label{eq:Pb-integral-vanish}
\int_{\Omega} P_b(y) |y|^{-n-2s} \d y = 0
\end{align}
by adding a bump function $\kappa \in C^{\infty}(\R^n)$ with $\supp(\kappa) \subset \Omega \setminus B_{1/2}$ to $P_b$ that is normalized such that
\begin{align*}
\int_{\R^n} \kappa(y) |y|^{-n-2s} \d y =  -\int_{\Omega} P_b(y)|y|^{-n-2s} \d y.
\end{align*}

\begin{remark}
Note that if $1+\alpha > 2s$, then
\begin{align*}
\left| \int_{\Omega} P_b(y)|y|^{-n-2s} \d y \right| < \infty.
\end{align*}
Indeed, recalling~\eqref{eq:kernel-est-trafo} and using that~$|\det D \Phi(y)| = 1 + O(|y|^{\alpha})$, we have that 
\begin{align*}
\int_{\Omega \cap B_1} P_b(y) |y|^{-n-2s} \d y = \int_{\{ y_n \ge 0 \} \cap B_{1}} (b \cdot y) |y|^{-n-2s} \d y + \int_{\{ y_n \ge 0 \} \cap B_{1}} O(|y|^{-n-2s+1+\alpha}) \d y.
\end{align*}
Then, the first integral is finite due to a cancellation and the second integral is finite since $1+\alpha > 2s$.
\end{remark}

\begin{remark}
\label{remark:Pb-subspace}
Note that the space $\cP := \{ P_b : b \in \{ b_n = 0 \} \}$ is a subspace of $L^2(B_r)$ for any $r \in (0,\frac{1}{4}]$. To see this, if~$b = e_i$ in \eqref{eq:Pb-integral-vanish}
we consider bump functions $\kappa_{e_i}$, while if~$b = \sum_i \lambda_i e_i \in \R^n$ 
we define
the corresponding bump function by $\kappa_b := \sum_i \lambda_i e_i$. 

In this way, the definitions
\begin{align*}
P_b(x) =
\begin{cases}
 b \cdot \Phi^{-1}(x),& ~~ x \in B_{1/4} \cap \Omega, \\
 \Big(\int_{\Omega} |x-z|^{-n-2s} \d z \Big)^{-1} \int_{\Omega} \big(b \cdot \Phi_k^{-1}(z) + \kappa_b(z) \big) |x-z|^{-n-2s} \d z,& ~~ x \in B_{1/4} \setminus \Omega
 \end{cases}
\end{align*}
are compatible with the properties \eqref{eq:Pb-integral-vanish} and \eqref{eq:Pb-complement-rep}, giving that $P_{\lambda b_1 + b_2} = \lambda P_{b_1} + P_{b_2}$ in $B_{1/4}$ for any $\lambda \in \R$ and $b_1,b_2 \in \R^n$. This makes $\cP$ indeed a linear subspace of $L^2(B_r)$.
\end{remark}

The following result contains the $C^{\eps}$ estimate for $L_{\Omega} P_b$. 

\begin{lemma}
\label{lemma:LPb-estimate-2}
Let $\Omega$ be a bounded open set with $\partial \Omega \in C^{1,\alpha}$ for some $\alpha \in (0,1]$ and $P_b$ be as above for some $b \in \R^n$ with $b \cdot e_n = 0$ and satisfying \eqref{eq:Pb-integral-vanish}. 
Assume that $1+\alpha \ge 2s + \eps$ for some $\eps \in (0,1)$. 

Then, for any $x \in \Omega \cap  B_{1/4}$ and $r = d_{\Omega}(x)/2$,
\begin{align}
\label{eq:LPb-estimate-2}
[L_{\Omega} P_b]_{C^{\eps}(B_r(x))} \le C |b| \left(1 + r^{\frac{1+\alpha-2s}{1+\alpha} -\eps} + |x|^{\frac{1+\alpha-2s}{1+\alpha}} r^{-\eps} \right),
\end{align}
where $C > 0$ only depends on $n,s,\alpha,\eps$, and the $C^{1,\alpha}$ norm of $\Omega$.
\end{lemma}

\begin{proof}
As in the proof of \autoref{lemma:LPb-estimate} we restrict ourselves to the case $|b| = 1$ and define $P_b$ in $\Omega^c$ by the expression in \eqref{eq:Pb-complement-rep}, so that $L_{\Omega} P_b(x) = (-\Delta)^s P_b(x)$ for any $x \in \Omega$. Moreover, for $y \in B_{1/2} \setminus \Omega$ we denote the projection of $y$ to $\partial \Omega$ by $\bar{y} \in \partial \Omega$ (if the projection is not unique, we replace $B_{1/2}$ by a smaller ball and proceed in the same way). 

The proof is split into several steps.

\textbf{Step 1}: First, we claim that
when~$\delta = \frac{1+\alpha-2s}{1+\alpha} \in (0,1]$
we have that, for any $y \in B_{1/2} \setminus \Omega$,
\begin{align}
\label{eq:outside-projection-estimate-improved}
|P_b(y) - P_b(\bar{y})| \le C d^{2s}_{\Omega^c}(y) |y|^{\delta}.
\end{align}

To see this, note that, by the construction of $P_b$ in $\Omega^c$, 
\begin{align}
\label{eq:Pb-projection-formula-1}
P_b(y) &= \left( \int_{\Omega} |y-z|^{-n-2s} \d z \right)^{-1} \int_{\Omega} (P_b(z) - P_b(\bar{y})) |y-z|^{-n-2s} \d z  + P_b(\bar{y}).
\end{align}
Hence, it remains to estimate the first summand in the previous identity. Using \eqref{eq:Pb-integral-vanish}, we obtain for some $\rho > 0$ to be chosen later, that
\begin{align}
\label{eq:Pb-projection-formula-2}
\begin{split}
\Big| \int_{\Omega} & (P_b(z) - P_b(\bar{y})) |y-z|^{-n-2s} \d z \Big| \\
&\le \Big|  \int_{\Omega} (P_b(z) - P_b(\bar{y})) |\bar{y}-z|^{-n-2s} \d z - \int_{\Omega} P_b(z) |z|^{-n-2s} \d z \Big| \\
&\quad + \Big|  \int_{\Omega} (P_b(z) - P_b(\bar{y})) \big( |y-z|^{-n-2s} - |\bar{y}-z|^{-n-2s} \big) \d z \Big| \\
&\le \Big|  \int_{\Omega} (P_b(z) - P_b(\bar{y})) |\bar{y}-z|^{-n-2s} \d z - \int_{\Omega} P_b(z) |z|^{-n-2s} \d z \Big| \\
&\quad + \Big| \int_{\Omega \setminus B_{\rho}(\bar{y})} (P_b(z) - P_b(\bar{y})) \big( |y-z|^{-n-2s} - |\bar{y}-z|^{-n-2s} \big) \d z \Big| \\
&\quad + \Big|  \int_{\Omega \cap  B_{\rho}(\bar{y})} (P_b(z) - P_b(\bar{y})) \big( |\bar{y}-z|^{-n-2s} \big) \d z \Big| \\
&\quad + \Big|  \int_{\Omega \cap  B_{\rho}(\bar{y})} (P_b(z) - P_b(\bar{y})) \big( |y-z|^{-n-2s} \big) \d z \Big|\\
& = I_1 + I_2 + I_3 + I_4.
\end{split}
\end{align}

To estimate $I_1$, we recall from \cite[Lemma 4.1]{FaRo22} that 
\begin{align}
\label{eq:rfl-Holder}
[(-\Delta)_{\Omega}^s P_b]_{C^{1+\alpha-2s}(\overline{\Omega})} \le C \Vert P_b \Vert_{C^{1+\alpha}(\overline{\Omega})},
\end{align}
where $(-\Delta)^s_{\Omega}$ denotes the regional fractional Laplacian.

Therefore, since $P_b(0) = 0$,
\begin{align*}
I_1 \le C |(-\Delta)^s_{\Omega} P_b(\bar{y}) - (-\Delta)^s_{\Omega} P_b(0)| \le C |\bar{y}|^{1+\alpha-2s} \le C(d_{\Omega^c}^{1+\alpha-2s}(y) + |y|^{1+\alpha-2s}). 
\end{align*}
To estimate $I_2$, we observe that $|\bar{y} - z| \le d_{\Omega^c}(y) + |y-z| \le 2 |y-z|$ for any $z \in \Omega$ and that $P_b \in C^{0,1}(\overline{\Omega})$, which leads to
\begin{align*}
I_2 &\le C |y-\bar{y}| \int_{\Omega \setminus B_{\rho}(\bar{y})} (P_b(z) - P_b(\bar{y})) \min\{|y-z|,|\bar{y}-z|\}^{-n-2s-1} \d z \\
&\le C d_{\Omega^c}(y) \int_{\Omega \setminus B_{\rho}(\bar{y})} |z - \bar{y}| |\bar{y}-z|^{-n-2s-1} \d z \le C d_{\Omega^c}(y) \rho^{-2s}.
\end{align*} 

To estimate $I_3$, we first recall that $\partial_{\nu} P_b(\bar{y}) = 0$. hence, $\nabla P_b(\bar{y}) = \tau(\bar{y})$, where $\tau(\bar{y})$ is tangential to $\partial \Omega$ at $\bar{y}$. 
Hence, using that $|P_b(z) - P_b(\bar{y}) - \nabla P_b(\bar{y})\cdot (\bar{y} -z)| \le C |\bar{y} - z|^{1+\alpha}$, we deduce
\begin{align*}
I_{3} &= \Big| \int_{\Omega \cap B_{\rho}(\bar{y})} (P_b(\bar{y}) - P_b(z)) |\bar{y} - z|^{-n-2s} \d z \Big| \\
&\le \Big| \int_{\Omega \cap B_{\rho}(\bar{y})} \tau(\bar{y}) \cdot (\bar{y} - z) |\bar{y} - z|^{-n-2s} \d z  \Big| + C\int_{\Omega \cap B_{\rho}(\bar{y})} |\bar{y} - z|^{-n-2s+1+\alpha} \d z \\
&\le C \Big| \int_{\Omega \cap B_{\rho}(\bar{y}) \cap \{ (\bar{y} - z) \cdot \tau(\bar{y}) \ge 0 \}} |\bar{y} - z|^{-n-2s+1} \d z  \Big| + C \Big| \int_{B_{\rho}(\bar{y}) \setminus ( \Omega \cap \{ (\bar{y} - z) \cdot \tau(\bar{y}) \ge 0 \}) } |\bar{y} - z|^{-n-2s+1} \d z  \Big| \\ 
&\quad + C \rho^{1+\alpha -2s},
\end{align*}
where we used that 
\begin{align*}
\int_{B_{\rho}(\bar{y}) \cap \{ (\bar{y} - z) \cdot \tau(\bar{y}) \ge 0 \}} \tau(\bar{y}) \cdot (\bar{y} - z) |\bar{y} - z|^{-n-2s} \d z = 0,
\end{align*}
due to a cancellation in the integral.

Since $\partial \Omega \in C^{1,\alpha}$, for small enough $t$ one has that
\begin{align*}
|B_t(\bar{y}) \cap \{ (\bar{y} - z) \cdot \tau(\bar{y}) \ge 0 \} \cap \Omega| \le C t^{n+\alpha}.
\end{align*}
Therefore, by the coarea formula,
\begin{align*}
\Big| \int_{\Omega \cap B_{\rho}(\bar{y}) \cap \{ (\bar{y} - z) \cdot \tau(\bar{y}) \ge 0 \}} |\bar{y} - z|^{-n-2s+1} \d z  \Big| &\le C \int_0^{\rho} t^{-n-2s+1} |B_t(\bar{y}) \cap \{ (\bar{y} - z) \cdot \tau(\bar{y}) \ge 0 \} \cap \Omega| \d t \\
&\le C \int_0^{\rho} t^{-2s+\alpha} \d t \le C \rho^{1+\alpha - 2s},
\end{align*}
where we used that $1+\alpha > 2s$ for the integral to converge. 

It follows from these observations that
\begin{align*}
I_{3} \le C \rho^{1+\alpha - 2s}.
\end{align*}

It remains to consider $I_4$, which goes by analogous arguments. In this situation we see that 
\begin{align*}
I_{4} &\le \Big| \int_{\Omega \cap B_{\rho}(\bar{y})} (P_b(\bar{y}) - P_b(z)) |y - z|^{-n-2s} \d z \Big| \\
&\le \Big| \int_{\Omega \cap B_{\rho}(\bar{y})} \tau(\bar{y}) \cdot (\bar{y} - z) |y - z|^{-n-2s} \d z  \Big| + C\int_{\Omega \cap B_{\rho}(\bar{y})} |y - z|^{-n-2s} |\bar{y} - z|^{1+\alpha} \d z =: I_{4,1} + I_{4,2}.
\end{align*}
For $I_{4,1}$, since $\bar{y}$ is the projection of $y$, one finds that
\begin{align*}
\int_{B_{\rho}(\bar{y})} \tau(\bar{y})\cdot (\bar{y} - z) |y - z|^{-n-2s} \d z = 0.
\end{align*}
Therefore, using again the coarea formula and the fact that $|\bar{y} - z| \le |y-z|$,
\begin{align*}
I_{4,1} &\le \Big| \int_{\Omega \cap B_{\rho}(\bar{y}) \cap \{ (\bar{y} - z) \cdot \tau(\bar{y}) \ge 0 \}} |\bar{y} - z|^{-n-2s+1} \d z  \Big| + C \Big| \int_{B_{\rho}(\bar{y}) \setminus ( \Omega \cap \{ (\bar{y} - z) \cdot \tau(\bar{y}) \ge 0 \}) } |\bar{y} - z|^{-n-2s+1} \d z  \Big| \\
&\le C \rho^{1+\alpha-2s}.
\end{align*}

For $I_{4,2}$, since $|y-z| \ge |\bar{y}-z|$, we see that 
\begin{align*}
I_{4,2} &\le C\int_{\Omega \cap B_{\rho}(\bar{y})} |\bar{y} - z|^{-n-2s+1+\alpha} \d z \le C \rho^{1+\alpha-2s}.
\end{align*}

By combining the estimates for $I_1,I_2,I_3,I_{4,1},I_{4,2}$ with \eqref{eq:Pb-projection-formula-1} and \eqref{eq:Pb-projection-formula-2}, we conclude that
\begin{align*}
|P_b(y) - P_b(\bar{y})| &\le C d^{2s}_{\Omega^c}(y) \Big| \int_{\Omega} (P_b(z) - P_b(\bar{y})) |y-z|^{-n-2s} \d z \Big| \\
&\le C d_{\Omega^c}^{1+\alpha-2s}(y) + C |y|^{1+\alpha - 2s} + C d_{\Omega^c}(y) \rho^{-2s} + C \rho^{1+\alpha - 2s}.
\end{align*}
Let us now choose $\rho = d_{\Omega^c}^{\frac{1}{1+\alpha}}(y)$ (this choice minimizes the right-hand side in the previous display). Then, we obtain \eqref{eq:outside-projection-estimate-improved}, as desired.

\textbf{Step 2}: Let us now write $P_b = P_b^{(1)} + P_b^{(2)}$, where we let $P_b^{(1)} \in C^{1,\alpha}(B_1) \cap L^{\infty}(\R^n)$ be such that $P_b^{(1)} = P_b$ in $\Omega$, i.e. $P_b^{(1)}$ is a $C^{1,\alpha}$ extension of $P_b$ from $B_1 \cap \Omega$ to $B_1$. Then, it is well-known (see \cite[Lemma 2.2.6]{FeRo24}) that
\begin{align*}
[(-\Delta)^s P_b^{(1)}]_{C^{\eps}(B_r(x))} \le C.
\end{align*} 
It remains to prove
\begin{align}
\label{eq:Delta-s-Pb-2-est}
[(-\Delta)^s P_b^{(2)}]_{C^{\eps}(B_r(x))} \le C \left(1 + r^{\delta -\eps} + |x|^{\delta} r^{-\eps} \right).
\end{align}
To check this, we pick $x_1,x_2 \in B_r(x)$ and observe that, since $P_b^{(2)} \equiv 0$ in $\Omega$, 
\begin{align*}
|(-\Delta)^s P_b^{(2)}(x_1) - (-\Delta)^s P_b^{(2)}(x_2)| &\le \Big| \int_{\R^n} (P_b^{(2)}(x_1) - P_b^{(2)}(x_2)) |x_1-y|^{-n-2s} \d y \Big| \\
&\quad + \Big| \int_{\R^n} (P_b^{(2)}(x_2) - P_b^{(2)}(y)) \left( |x_1 - y|^{-n-2s} - |x_2 -y |^{-n-2s} \right) \d y \Big| \\
&\le C |x_1 - x_2|^{\eps} \int_{\Omega^c} |P_b^{(2)}(y)| \min\{|x_1 - y|,|x_2 - y|\}^{-n-2s-\eps} \d y.
\end{align*}
Also, since $P_b^{(2)}$ is bounded and $x_1,x_2 \in B_{1/4}$, we have that
\begin{align*}
\int_{B_{1/2}^c \setminus \Omega} |P_b^{(2)}(y)| \min\{|x_1 - y|,|x_2 - y|\}^{-n-2s-\eps} \d y \le C.
\end{align*}
Now we recall~\eqref{eq:outside-projection-estimate-improved}
and, since $P_b(\bar{y}) = P_b^{(1)}(\bar{y})$ and $\partial_{\nu} P_b^{(1)}(\bar{y}) = 0$, 
we deduce that,
for all~$y \in B_{1/2} \setminus \Omega$, 
\begin{align*}
|P_b^{(2)}(y)| \le |P_b(y) - P_b(\bar{y})| + |P_b^{(1)}(y) - P_b(\bar{y})| \le C d_{\Omega^c}^{2s}(y) |y|^{\delta} + C d_{\Omega^c}^{1+\alpha}(y).
\end{align*}
Therefore, using that $d_{\Omega}(x_1) + d_{\Omega}(x_2) \le 4r$ and $|x_1| + |x_2| \le |x| + r$,
and recalling also~\cite[Lemma B.2.4]{FeRo24},
\begin{align*}
\int_{B_{1/2} \setminus \Omega} & |P_b^{(2)}(y)| \min\{|x_1 - y|,|x_2 - y|\}^{-n-2s-\eps} \d y \\
&\le C \int_{B_{1/2} \setminus \Omega} d^{1+\alpha}_{\Omega^c}(y) \min\{|x_1 - y|,|x_2 - y|\}^{-n-2s-\eps} \d y \\
&\quad + C \int_{B_{1/2} \setminus \Omega} d^{2s}_{\Omega^c}(y) \min\{|x_1 - y|,|x_2 - y|\}^{-n-2s-\eps+\delta} \d y\\
&\quad + C (|x| + r)^{\delta} \int_{B_{1/2} \setminus \Omega} d^{2s}_{\Omega^c}(y) \min\{|x_1 - y|,|x_2 - y|\}^{-n-2s-\eps} \d y \\
&\le C + C r^{1+\alpha-2s-\eps} + C r^{\delta -\eps} + C |x|^{\delta} r^{-\eps}.
\end{align*}
All in all, we have shown that
\begin{align*}
|(-\Delta)^s P_b^{(2)}(x_1) - (-\Delta)^s P_b^{(2)}(x_2)| \le C |x_1 - x_2|^{\eps} \left(1 + r^{1+\alpha-2s-\eps} + r^{\delta -\eps} + |x|^{\delta} r^{-\eps} \right),
\end{align*}
which implies \eqref{eq:Delta-s-Pb-2-est}, yielding the desired result.
\end{proof}

\subsection{Regularity outside the solution domain}
\label{subsec:reg-outside-sol-dom}

The following result shows that the regularity of solutions is transferred to the exterior of the solution domain through the nonlocal Neumann condition.

\begin{lemma}
\label{lemma:outside-bound}
Let $s \in (0,1)$ and $\Omega$ be a bounded open set with $\partial \Omega \in C^{1,\alpha}$, for some $\alpha > 0$, with $0 \in \partial \Omega$. 
Let $x_0 \in B_{1/2} \setminus \Omega$, $\rho = d_{\Omega^c}(x_0)/4$, and $u$ be such that 
\begin{align*}
\cN_{\Omega}^s u = 0 ~~ \text{ in } B_{\rho}(x_0).
\end{align*}
Let $x_0^{\ast} \in \partial \Omega$ with $|x_0 - x_0^{\ast}| = d_{\Omega}(x_0)$,  and assume that, for some $\gamma, \gamma' \in (0,1]$ and~$C_1,C_2\in(0,+\infty)$,
\begin{align}
\label{eq:outside-u-bound-ass}
\begin{split}
|u(y) - u(x_0^{\ast})| \le C_1 |y - x_0^{\ast}|^{\gamma'} ~~ \forall y \in B_1 \cap \Omega,\\
|u(y) - u(x_0^{\ast})| \le C_2 |y - x_0^{\ast}|^{\gamma} ~~ \forall y \in B_{\rho}(x_0).
\end{split}
\end{align}
Then, it holds
\begin{align*}
[u]_{C^{\gamma}(B_{\rho}(x_0))} \le C \big((C_1 + \Vert u \Vert_{L^{\infty}(\Omega \setminus B_1)}) \rho^{\gamma' -\gamma} + C_2 \big),
\end{align*}
where $C > 0$ depends only on $n,s,\gamma,\gamma'$.
\end{lemma}

Note that variations of the previous result also hold for $\gamma > 1$ by subtracting the functions $P_b$ from \autoref{subsec:correction}. We do not investigate higher order extensions of \autoref{lemma:outside-bound}, since they are not needed for our main result.

\begin{proof}
We denote 
\begin{align}
\label{eq:outside-u-help-0}
W(x) := \int_{\Omega} |x-y|^{-n-2s} \d y \ge C d_{\Omega^c}^{-2s}(x) \ge C \rho^{-2s} , ~~ x \in B_1 \setminus \Omega.
\end{align}
Moreover, we set $v(x) = u(x) - u(x_0^{\ast})$, and observe that
\begin{align}
\label{eq:outside-u-help-1}
\begin{split}
\cN_{\Omega}^s v &= 0 ~~~\qquad\qquad\qquad\qquad\qquad\qquad \text{ in } B_{\rho}(x_0),\\
|v(y)| & \le (C_1 + C\Vert v \Vert_{L^{\infty}(\Omega \setminus B_1)}) |y - x_0^{\ast}|^{\gamma'} ~~ \forall y \in \Omega,\\
|v(y)| & \le C_2 |y - x_0^{\ast}|^{\gamma} ~~\qquad\qquad\qquad\qquad~ \forall y \in B_{\rho}(x_0).
\end{split}
\end{align}
Let us now consider~$x,\bar{x} \in B_{\rho}(x_0)$. Then, \eqref{eq:outside-u-help-1} implies
\begin{align*}
|W(x) v(x) & - W(\bar{x})v(\bar{x})| \le C \int_{\Omega} |v(y)| \big| |x - y|^{-n-2s} - |\bar{x} - y|^{-n-2s} \big| \d y \\
&\le C |x - \bar{x}| \int_{\Omega} |v(y)| \big(|x - y|^{-n-2s-1} + |\bar{x} - y|^{-n-2s-1} \big) \d y \\
&\le C(C_1 + \Vert v \Vert_{L^{\infty}(\Omega \setminus B_1)}) |x - \bar{x}| \int_{\Omega} |y- x_0^{\ast}|^{\gamma'} \big(|x - y|^{-n-2s-1} + |\bar{x} - y|^{-n-2s-1} \big) \d y \\
&=: I.
\end{align*}
Next, using that $|x - x_0^{\ast}|^{\gamma} \le |x - x_0| + 4\rho \le 5 \rho$ and $d_{\Omega^c}(x) \ge \rho$ (and similarly for $\bar{x}$), one sees that
\begin{align*}
I &\le C(C_1 + \Vert v \Vert_{L^{\infty}(\Omega \setminus B_1)}) |x - \bar{x}| \Big( |x - x_0^{\ast}|^{\gamma'} \int_{\Omega}  |x - y|^{-n-2s-1} \d y + \int_{\Omega} |x - y|^{-n-2s-1+\gamma'} \d y \\
&\qquad\qquad\qquad\qquad\qquad \quad + |\bar{x} - x_0^{\ast}|^{\gamma'} \int_{\Omega}  |\bar{x} - y|^{-n-2s-1} \d y + \int_{\Omega} |\bar{x} - y|^{-n-2s-1+\gamma'} \d y  \Big) \\
&\le C(C_1 + \Vert v \Vert_{L^{\infty}(\Omega \setminus B_1)}) |x - \bar{x}| \Big( |x- x_0^{\ast}|^{\gamma'} ( d_{\Omega^c}^{-2s-1}(x) +  d_{\Omega^c}^{-2s-1}(\bar{x})) + d_{\Omega^c}^{-2s-1+\gamma'}(x) + d_{\Omega^c}^{-2s-1+\gamma'}(\bar{x}) \Big) \\
&\le C(C_1 + \Vert v \Vert_{L^{\infty}(\Omega \setminus B_1)}) |x - \bar{x}|^{\gamma} \rho^{-2s+\gamma'-\gamma},
\end{align*}
where we also used that $\gamma' \le 1 < 1 + 2s$ by assumption and that $|x-\bar{x}| \le 2 \rho$.

Therefore, one concludes that
\begin{align}
\label{eq:outside-u-help-2}
|W(x) v(x) - W(\bar{x})v(\bar{x})| \le C |x - \bar{x}|^{\gamma} (C_1 + \Vert v \Vert_{L^{\infty}(\Omega \setminus B_1)}) \rho^{-2s+\gamma'-\gamma}.
\end{align}
Moreover, we observe that
\begin{align}
\label{eq:outside-u-help-3}
\begin{split}
|W(x) - W(\bar{x})| &\le C \int_{\Omega} \big| |x-y|^{-n-2s} - |\bar{x} - y|^{-n-2s} \big| \d y \\
&\le C |x-\bar{x}|^{\gamma} \int_{\Omega} \big( |x-y|^{-n-2s-\gamma} + |\bar{x}-y|^{-n-2s-\gamma} \big) \d y \\
&\le C |x - \bar{x}|^{\gamma} (d_{\Omega^c}^{-2s-\gamma}(x) + d_{\Omega^c}^{-2s-\gamma}(\bar{x})) \le C |x - \bar{x}|^{\gamma} \rho^{-2s-\gamma}.
\end{split}
\end{align}
Combining \eqref{eq:outside-u-help-2} and \eqref{eq:outside-u-help-3} with the second estimate in \eqref{eq:outside-u-help-1} and \eqref{eq:outside-u-help-0}, we finally obtain
\begin{align*}
|v(x) - v(\bar{x})| &\le |W(x)|^{-1} |W(x) v(x) - W(\bar{x})v(\bar{x})| + |v(\bar{x})| |W(x)|^{-1} |W(x) - W(\bar{x})| \\
&\le C |x-\bar{x}|^{\gamma} \big( \rho^{2s}  (C_1 + \Vert v \Vert_{L^{\infty}(\Omega \setminus B_1)}) \rho^{-2s+\gamma'-\gamma} + C_2 |\bar{x} - x_0^{\ast}|^{\gamma}  \rho^{2s} \rho^{-2s-\gamma}  \big) \\
&\le C |x- \bar{x}|^{\gamma} ((C_1 + \Vert u \Vert_{L^{\infty}(\Omega \setminus B_1)}) \rho^{\gamma' -\gamma} + C_2).
\end{align*}
In the last step, we also used that $\Vert v \Vert_{L^{\infty}(\Omega \setminus B_1)} \le 2 \Vert u \Vert_{L^{\infty}(\Omega \setminus B_1)}$. This implies the desired result, using that $v(x) - v(\bar{x}) = u(x) - u(\bar{x})$.
\end{proof}

\subsection{Expansions involving first order corrections}
\label{subsec:prep-expansion}

In order to prove regularity estimates by blowup, we need to prove two auxiliary results involving the first order corrections $P_b$ from \autoref{subsec:correction}. The first one is reminiscent of \cite[Lemma 4.3]{AbRo20}, and the second one is a refinement of several estimates in \cite[Proof of Proposition 4.1]{AbRo20}. See also \cite[Lemmata 2.14 and 2.15]{RoWe25}. 

Here we denote~$\Omega_{\rho} := \Omega \cap B_{\rho}$, Also, we let $[u]_{C^{\gamma}(B_{\rho})} := \Vert u \Vert_{L^{\infty}(\Omega_{\rho})}$ when $\gamma = 0$, i.e. in that case we only take into account information inside $\Omega$ in an $L^{\infty}$ sense. 

\begin{lemma}
\label{lemma:theta-increasing}
Let $\gamma \in [0,\min\{1,2s\}]$ and $\Omega$ be a bounded open set with $\partial\Omega \in C^{1,\alpha}$ for some $\alpha \in (0,1]$ with $\alpha \ge \gamma$. Let $u \in C^{\gamma}(B_{1/4})$. Let $\beta > 0$ such that $\beta + \gamma > 1$. 

Assume that for any $\rho \in [0,\frac{1}{4}]$ there exists $b_{\rho} \in \{ b \cdot e_n = 0\}$ such that, for some $c_0 > 0$,
\begin{align*}
[ u - P_{b_\rho} ]_{C^{\gamma}(B_{\rho})} \le c_0 \rho^{\beta},
\end{align*}
where $P_{b_{\rho}}$ is as in \autoref{subsec:correction} satisfying \eqref{eq:Pb-complement-rep}, and also \eqref{eq:Pb-integral-vanish} in case $\gamma > 0$. 

Then there exists $b_0 \in \{ b \cdot e_n = 0 \}$ such that, for any $\rho \in (0,\frac{1}{4}]$,
\begin{align*}
[ u - P_{b_0} ]_{C^{\gamma}(B_{\rho})} \le C c_0 \rho^{\beta},
\end{align*}
where $C > 0$ depends only on $n$, $\beta$, $\gamma$, and the $C^{1,\alpha}$ norm of $\Omega$.
\end{lemma}

\begin{proof}
We set $P_{\rho} := P_{b_{\rho}}$. We have that, for any $\rho \in (0,\frac{1}{4}]$,
\begin{align}\label{JAK9}
[ P_\rho - P_{2\rho} ]_{C^{\gamma}(B_{\rho})} \le [ P_\rho -u ]_{C^{\gamma}(B_{\rho})} + [ u - P_{2\rho} ]_{C^{\gamma}(B_{2\rho})} \le c_0(1 + 2^{\beta}) \rho^{\beta}.
\end{align}
We also recall that $0 \in \partial \Omega$, $(\Phi^{-1}(x))_n = d_{\Omega}(x)$, that $P_b(x) = b \cdot \Phi^{-1}(x)$ for $x \in \Omega_{\rho} \subset B_{\rho}$, and that $[d_{\Omega}]_{C^{\gamma}(\Omega_{\rho})} \ge C \rho^{1-\gamma}$. Thus, we deduce from \eqref{eq:Phi-Lipschitz} and~\eqref{JAK9} that
\begin{align}
\label{eq:b-difference}
\begin{split}
|b_{\rho} - b_{2\rho}|  \le C |b_{\rho} - b_{2\rho}| \rho^{-1+\gamma} [ d_{\Omega} ]_{C^{\gamma}(\Omega_{\rho})} &=  C |b_{\rho} - b_{2\rho}| \rho^{-1+\gamma} [ \Phi^{-1}(\cdot)_n ]_{C^{\gamma}(B_{\rho})} \\
& = C \rho^{-1+\gamma} [ P_\rho - P_{2\rho} ]_{C^{\gamma}(\Omega_{\rho})} \le C c_0 r^{\beta + \gamma -1}.
\end{split}
\end{align}
Moreover, a similar argument gives that
\begin{align}
\label{eq:b-estimate-2}
|b_{1/4}| \le C [ P_{1/4} ]_{C^{\gamma}(\Omega_{1/4})} \le C [ u - P_{1/4} ]_{C^{\gamma}(B_{1/4})} + C [ u ]_{C^{\gamma}(B_{1/4})} < \infty.
\end{align}
Therefore, since $\beta + \gamma - 1 > 0$, the limit $b_0 := \lim_{\rho \searrow 0} b_{\rho}$ exists. 

Accordingly, by \eqref{eq:b-difference},
\begin{align*}
|b_0 - b_{\rho}| \le \sum_{j = 0}^{\infty} |b_{2^{-j} \rho} - b_{2^{-j-1} \rho}| \le C c_0 \sum_{j = 0}^{\infty} (2^{-j} \rho)^{\beta + \gamma -1} \le C c_0 \rho^{\beta + \gamma -1}.
\end{align*}
Next, we observe that, for any $x,y \in \Omega_{\rho}$, 
\begin{equation}\label{JA8}
[\Phi^{-1}]_{C^{\gamma}(\Omega_{\rho})} \le C \rho^{1 - \gamma}\end{equation}Indeed, when $\gamma = 0 $, this fact follows immediately from \eqref{eq:Phi-Lipschitz} and when $\gamma > 0$ from the following computation, using that $\Phi \in C^{0,1}$: 
\begin{align}
\label{eq:Phi-Cgamma-est}
|\Phi^{-1}(x) - \Phi^{-1}(y)| \le C|x-y| \le C |x-y|^{\gamma} \rho^{1- \gamma}.
\end{align}

{F}rom~\eqref{JA8} we deduce that,
for $b = (b_0 - b_{\rho})/|b_0 - b_{\rho}|$, 
\begin{align}
\label{eq:Phi-Cgamma-est-claim}
\begin{split}
[ P_{\rho} - P_0 ]_{C^{\gamma}(B_{\rho})} &\le |b_0 - b_{\rho}| [P_{b}]_{C^{\gamma}(B_{\rho})} \\
&\le C |b_0 - b_{\rho}|[\Phi^{-1}]_{C^{\gamma}(\Omega_{\rho})} + |b_0 - b_{\rho}| [P_{b}]_{C^{\gamma}(B_{\rho} \setminus \Omega)} \le  C |b_0 - b_{\rho}| \rho^{1-\gamma}.
\end{split}
\end{align}

Now we prove that $[P_b]_{C^{\gamma}(B_{\rho} \setminus \Omega)} \le C \rho^{1-\gamma}$. To see this, we recall that $\gamma \le \min\{1,2s\}$. Hence, by \eqref{eq:outside-projection-estimate-improved},
we see that, for any $y \in B_{\rho} \setminus \Omega$ and $\bar{y} \in \partial \Omega$ with $|y-\bar{y}| = d_{\Omega^c}(y)$,
\begin{align*}
|P_b(y) - P_b(\bar{y})| \le C |y-\bar{y}|^{2s} |y|^{1 - \frac{2s}{1+\alpha}} \le C |y - \bar{y}|^{\gamma} \rho^{2s - \gamma +1 - \frac{2s}{1+\alpha}}.
\end{align*}
Moreover, using that $\Phi^{-1} \in C^{0,1}(\overline{\Omega})$,
we see that, when $y \in \Omega_{\rho}$,
\begin{align*}
|P_b(y) - P_b(\bar{y})| \le C |y - \bar{y}| \le C |y - \bar{y}|.
\end{align*}
Hence, \autoref{lemma:outside-bound} (applied with $\gamma'= 1$) implies that
\begin{align}
\label{eq:Phi-Cgamma-est-2}
[P_b]_{C^{\gamma}(B_{\rho} \setminus \Omega)} \le C (\rho^{1-\gamma} + \rho^{2s-\gamma + 1 - \frac{2s}{1+\alpha}} ) \le C \rho^{1-\gamma},
\end{align}
where we used that $2s - \gamma + 1 - \frac{2s}{1+\alpha} > 1 - \gamma$.

These observations give that
\begin{align*}
[ u - P_{0} ]_{C^{\gamma}(B_{\rho})} \le [ u - P_{\rho} ]_{C^{\gamma}(B_{\rho})} + [ P_{\rho} - P_{0} ]_{C^{\gamma}(B_{\rho})} \le c_0 \rho^{\beta} + C \rho^{1-\gamma} |b_0 - b_{\rho}| \le C c_0 \rho^{\beta},
\end{align*} as desired.
\end{proof}

\begin{lemma}
\label{lemma:aux-theta}
Let $\gamma \in [0,\min\{1,2s\}]$ and $\Omega$ be a bounded open set with $\partial\Omega \in C^{1,\alpha}$ for some $\alpha \in (0,1]$ with $\alpha \ge \gamma$. Let $u \in C^{\gamma}(B_{1/4})$. Let $\beta > 0$ such that $\beta + \gamma > 1$. 

Assume that for any $\rho \in (0,\frac{1}{4}]$ there exists $b_{\rho} \in \{ b \cdot e_n = 0\}$ such that, for some function $\theta : (0,\frac{1}{4}] \to (0,\infty)$ with $\theta(\rho) \nearrow \infty$, as $\rho \searrow 0$ we have that
\begin{align*}
[ u - P_{b_{\rho}} ]_{C^{\gamma}(B_{\rho})} \le \theta(\rho) \rho^{\beta},
\end{align*}
where $P_{b_{\rho}}$ is as in \autoref{subsec:correction} satisfying \eqref{eq:Pb-complement-rep}, and also \eqref{eq:Pb-integral-vanish} in case $\gamma > 0$. 

Then, for any $R \ge 1$ with $R r \le \frac{1}{4}$ and any $K \in \N$
we have that
\begin{align*}
\frac{[ u - P_{b_r} ]_{C^{\gamma}(B_{Rr})}}{r^{\beta} \theta(r)} &\le C R^{\beta}, \\
\frac{|b_r|}{\theta(r)} &\le C \theta(r)^{-1} ( \theta(8^{-1}) + [ u ]_{C^{\gamma}(B_{1/4})} + C(K))  + C 2^{-K(\beta + \gamma -1)},
\end{align*}
for $C > 0$ depending only on $n,\beta,\gamma$, and the $C^{1,\alpha}$ norm of $\Omega$, where $C(K) = \sum_{i = 1}^K \theta(2^{-i-3}) 2^{-i(\beta + \gamma - 1)}$.
\end{lemma}

Note that the second statement proves that $|b_r| \theta(r)^{-1} \to 0$, as $r \to 0$, and that the modulus of convergence does not depend on $|b_r|$.

\begin{proof}
By the same arguments as in \eqref{eq:b-difference}, and using that $\theta$ is non-increasing, we have
\begin{align*}
|b_r - b_{2r}| \le C \theta(r) r^{\beta+\gamma-1}.
\end{align*}
Iterating this inequality, we get for any $N \in \N$, as long as $2^N r \le \frac{1}{4}$:
\begin{align}
\label{eq:b-difference-2}
\begin{split}
|b_r - b_{2^Nr}| \le \sum_{i = 0}^{N-1} |b_{2^i r} - b_{2^{i+1}r}| &\le C \sum_{i = 0}^{N-1} \theta(2^i r) (2^i r)^{\beta+\gamma-1} \\
&\le C \theta(r) r^{\beta+\gamma-1} \sum_{i = 0}^{N-1} \frac{\theta(2^i r)}{\theta(r)} 2^{i(\beta+\gamma-1)} \le C \theta(r) (2^N r)^{\beta+\gamma-1}.
\end{split}
\end{align}
We also remark that $[\Phi^{-1}]_{C^{\gamma}(\Omega_{\rho})} \le C \rho^{1-\gamma}$, thanks to \eqref{eq:Phi-Cgamma-est}.
Hence,
by \eqref{eq:Phi-Lipschitz} and \eqref{eq:Phi-Cgamma-est-2}, arguing in the same way as in the proof of \eqref{eq:Phi-Cgamma-est-claim},
one deduces from~\eqref{eq:b-difference-2}
that, for all~$R \ge 1$ with $Rr \le \frac{1}{4}$,
$$[ P_{Rr} - P_{r} ]_{C^{\gamma}(B_{Rr})} \le C \theta(r) (Rr)^{\beta}.$$ Thus,
\begin{align*}
\frac{[ u - P_{r} ]_{C^{\gamma}(B_{Rr})}}{r^{\beta} \theta(r)} \le \frac{[ u - P_{Rr} ]_{C^{\gamma}(B_{Rr})}}{r^{\beta} \theta(r)}  + \frac{[ P_{Rr} - P_{r} ]_{C^{\gamma}(B_{Rr})}}{r^{\beta} \theta(r)} \le \frac{\theta(Rr)(Rr)^{\beta}}{r^{\beta} \theta(r)} + C \frac{\theta(r)(Rr)^{\beta}}{r^{\beta} \theta(r)} \le c R^{\beta}.
\end{align*}
This proves the first claim. The second claim follows by choosing for any $r > 0$ a number $N \in \N$ such that $2^N r \in [\frac{1}{8},\frac{1}{4}]$ and deducing from \eqref{eq:b-difference-2} that, for any $K \in \N$,
\begin{align*}
\frac{|b_r - b_{2^Nr}|}{\theta(r)} &\le C\sum_{i=1}^{N} \frac{\theta(2^{N-i}r)}{\theta(r)} (2^{N-i}r)^{\beta+\gamma-1} \\
&\le C\sum_{i=1}^{K} \frac{\theta(2^{-i-3})}{\theta(r)} (2^{-i})^{\beta+\gamma-1} + C \sum_{i = K+1}^{\infty} (2^{-i})^{\beta-1} =: C(K) \theta(r)^{-1} + C 2^{-K(\beta+\gamma-1)}.
\end{align*}
Moreover, by the same argument as in \eqref{eq:b-estimate-2} we have, for any $\rho \in [\frac{1}{8},\frac{1}{4}]$,
\begin{align*}
|b_{\rho}| \le C \theta(8^{-1}) + C [ u ]_{C^{\gamma}(B_{1/4})},
\end{align*}
and therefore
\begin{align*}
\frac{|b_r|}{\theta(r)} \le \frac{|b_{2^Nr}|}{\theta(r)} + \frac{|b_r - b_{2^Nr}|}{\theta(r)} \le C \big( \theta(8^{-1}) + [ u ]_{C^{\gamma}(B_{1/4})} + C(K) \big) \theta(r)^{-1} + C 2^{-K(\beta+\gamma-1)},
\end{align*}as claimed.
\end{proof}

\subsection{Treating unbounded source terms}
\label{subsec:unbounded-source-terms}

Let us recall that in \cite{AFR23} boundary regularity was only established for weak solutions to the nonlocal Neumann problem $L_{\Omega} u = h$ when $h \in L^q(\Omega)$ for some $q > \frac{n}{2s}$. However, when applied to $L_{\Omega}$, the functions $P_b$ from \autoref{subsec:correction} in general
do  not produce source terms that are in $L^q(\Omega)$, but which rather explode at the boundary like $d_{\Omega}^{1-2s}$ (see \autoref{lemma:LPb-estimate} and \autoref{lemma:LPb-estimate-2}). Therefore, we need to adapt several auxiliary results from \cite{AFR23} to such general family of source terms.

In this subsection we present an auxiliary result which is reminiscent of \cite[Lemma 3.1]{KiWe24} and prove an $H_K$-bound for weak solutions to nonlocal Neumann problems with source terms exhibiting singular boundary behavior.

Given $\Omega \subset \R^n$, $R > 0$, and $x_0 \in \Omega$, let us write $\Omega_R(x_0) := \Omega \cap B_R(x_0)$.

\begin{lemma}
\label{lemma:fv-integral-estimate}
Let $\Omega \subset \R^n$ be a bounded open set with Lipschitz boundary. Let $s \in (0,1)$, $R > 0$, $x_0 \in \Omega$, and $d_{\Omega}^{s-\eps} h \in L^{\infty}(\Omega)$ for some $\eps \in (0,s]$. 

Then, for any $v \in H^s(\Omega)$ with $v = 0$ in $\Omega \setminus \Omega_R(x_0)$,
\begin{align*}
\int_{\Omega_R(x_0)} |h v| \d x \le C R^{\frac{n}{2} + \eps} \Vert d^{s-\eps}_{\Omega} h \Vert_{L^{\infty}(\Omega_R(x_0))} [v]_{H^s(\Omega)},
\end{align*}
where $C > 0$ only depends on $n,s,\eps$ and the Lipschitz norm of $\Omega$.
\end{lemma}

\begin{proof}
First, let us assume that $d_{\Omega}(x_0) \le 2R$. Then, 
\begin{align}
\label{eq:v-Hs-estimate-help-1}
\begin{split}
\int_{\Omega_R(x_0)} |h v| \d x &\le C R^{\eps} \Vert d^{s-\eps}_{\Omega} h \Vert_{L^{\infty}(\Omega_R(x_0))}  \left\Vert \frac{v}{d_{\Omega_R(x_0)}^s} \right\Vert_{L^1(\Omega_R(x_0))} \\
& \le C R^{\frac{n}{2} + \eps} \Vert d^{s-\eps}_{\Omega} h \Vert_{L^{\infty}(\Omega_R(x_0))} \left\Vert \frac{v}{d_{\Omega_R(x_0)}^s} \right\Vert_{L^2(\Omega_R(x_0))},
\end{split}
\end{align}
where we used that $d_{\Omega_R(x_0)} \le \min\{ 2R , d_{\Omega}(x) \}$ in $\Omega_R(x_0)$ and H\"older's inequality together with the fact that $|\Omega_R(x_0)| \asymp R^n$.

Next, we observe that for any $x \in \Omega_R(x_0)$ we can find $y_x \in \Omega$ such that $B_R(y_x) \subset \Omega \setminus B_R(x_0)$ and for any $y \in B_R(y_x)$ it holds $|x-y| \ge d_{\Omega_R(x_0)}(x)$. Thus, we have
\begin{align*}
\int_{B_R(y_x)} |x-y|^{-n-2s} \d y \ge c R^n d_{\Omega_R(x_0)}^{-n-2s}(x) \ge c d_{\Omega_R(x_0)}^{-2s}(x).
\end{align*} Consequently, since $v = 0$ in $B_R(y_x)$,
\begin{align*}
\left\Vert \frac{v}{d_{\Omega_R(x_0)}^s} \right\Vert_{L^2(\Omega_R(x_0))}^2  &\le c \int_{\Omega_R(x_0)} |v(x)|^2 \left( \int_{B_R(y_x)} |x-y|^{-n-2s} \d y  \right) \d x \\
& \le c \int_{\Omega_R(x_0)} \int_{B_R(y_x)} |v(x) - v(y)|^2 |x-y|^{-n-2s} \d y \d x \le c [v]_{H^s(\Omega)}^2.
\end{align*}
A combination of the previous estimate and \eqref{eq:v-Hs-estimate-help-1} yields the claim.

In case $d_{\Omega}(x_0) > 2R$ we have that~$B_R(x_0) = \Omega_R(x_0)$. Accordingly, by testing the weak formulation for $v$ with $v$ itself, following similar arguments as in \eqref{eq:v-Hs-estimate-help-1}, and using the fractional Poincar\'e-Wirtinger inequality, we conclude that
\begin{align*}
\int_{\Omega_R(x_0)} |h v| \d x \le C R^{\frac{n}{2} + \eps - s} \Vert d_{\Omega}^{s-\eps} f \Vert_{L^{\infty}(\Omega_R(x_0))} \Vert v \Vert_{L^2(B_R(x_0))} \le C R^{\frac{n}{2}+\eps}\Vert d_{\Omega}^{s-\eps} f \Vert_{L^{\infty}(\Omega_R(x_0))}  [v]_{H^s(B_R(x_0))}.
\end{align*}
This proves the claim also in the second case, and the proof is complete.
\end{proof}

In order to get the boundary regularity of order $C^{\min\{2s,1\} - \eps}$, the only difference compared to the proof in \cite{AFR23} will be a version of \cite[Lemma 6.1]{AFR23} for functions with growth $|u(x)| \le M_0(1 + |x|^{2s-\eps})$. This follows by using a Caccioppoli inequality involving tail terms. 
Since we also want to treat source terms $h$ that explode at the boundary, we need to make some further adaptations to the proof, compared to \cite[Lemma 6.1]{AFR23}.

\begin{lemma}
\label{lemma:Hk-bound}
Let $\Omega \subset \R^n$ be a bounded open set with Lipschitz boundary and $L_{\Omega}$ be as before. Let $d_{\Omega}^{s-\eps} h \in L^{\infty}(\Omega)$ for some $\eps \in (0,s]$. Assume that $u \in H_K(\Omega)$ is a weak solution to 
\begin{align*}
L_{\Omega} u = h ~~ \text{ in } \Omega.
\end{align*}
Assume also that, for some $\eps \in (0,2s)$ and $M_0 > 0$,
\begin{align*}
|u(x)| \le M_0 (1 + |x|^{2s-\eps}) ~~ \forall x \in \R^n.
\end{align*}
Then, for any $0 < r < R$ and any $x_0 \in \Omega$,
\begin{align*}
[u]_{H_K(B_r(x_0))} \le C \left( \Vert d_{\Omega}^{s-\eps} h \Vert_{L^{\infty}(\Omega_R(x_0))} + M_0 \right),
\end{align*}
where $C> 0$ depends only on $n,s,x_0,\eps,r,R$, and the Lipschitz norm of $\Omega$.
\end{lemma}

\begin{proof}
Let $\varphi \in C_c^{\infty}(B_R(x_0))$ with $0 \le \varphi \le 1$ and $\varphi \equiv 1$ in $B_r(x_0)$. Then, by testing the weak formulation with $\eta = u \varphi^2$ and using the symmetry of $K_{\Omega}$, we obtain that
\begin{align*}
\frac{1}{2} \int_{\Omega} h u \varphi^2 \d x &= \int_{\Omega} \int_{\Omega} (u(x) - u(y)) u(x) \varphi^2(x) K_{\Omega}(x,y) \d y \d x \\
& = \frac{1}{2}\int_{\Omega_{2R}(x_0)}\int_{\Omega_{2R}(x_0)} (u(x) - u(y))^2(\varphi^2(x) + \varphi^2(y)) K_{\Omega}(x,y) \d y \d x \\
&\quad + \int_{\Omega_{2R}(x_0)}\int_{\Omega_{2R}(x_0)} (u(x) - u(y))u(x)(\varphi^2(x) - \varphi^2(y)) K_{\Omega}(x,y) \d y \d x \\
&\quad + \int_{\Omega \setminus B_{2R}(x_0)} \int_{\Omega_{R}(x_0)} (u(x) - u(y)) u(x) \varphi^2(x) K_{\Omega}(x,y) \d y \d x \\
&= I_1 + I_2 + I_3.
\end{align*}
Note that, by \cite[Lemma 3.2]{KaWe22},
\begin{align*}
I_1 &\ge \frac{1}{2} \int_{\Omega_{2R}(x_0)}\int_{\Omega_{2R}(x_0)} (\varphi u(x) - \varphi u(y))^2 K_{\Omega}(x,y) \d y \d x \\
&\quad - 2\int_{\Omega_{2R}(x_0)} \int_{\Omega_{2R}(x_0)} u^2(x) (\varphi(x) - \varphi(y))^2 K_{\Omega}(x,y) \d y \d x.
\end{align*}
Therefore, by H\"older's inequality,
\begin{align*}
I_1 + I_2 &\ge \frac{1}{2} I_1 - C\int_{\Omega_{2R}(x_0)}\int_{\Omega_{2R}(x_0)} u^2(x)(\varphi(x) - \varphi(y))^2 K_{\Omega}(x,y) \d y \d x \\
&\ge c [u \varphi]_{H_K(\Omega_{2R}(x_0))}^2  - C M_0^2 \int_{\Omega_{2R}(x_0)}\int_{\Omega_{2R}(x_0)} |x-y|^2 K_{\Omega}(x,y) \d y \\
&\ge c [u \varphi]_{H_K(\Omega_{2R}(x_0))}^2 - C M_0^2 \int_{\Omega_{2R}(x_0)} (1 + d_{\Omega}(x)^{1-2s})(1 + |\log d_{\Omega}(x)|) \d x \\
&\ge c [u \varphi]_{H_K(\Omega_{2R}(x_0))}^2  - C M_0^2 ,
\end{align*}
where $c > 0$ is independent of $r,R$, but $C > 0$ depends on $r,R$. In the second to last step, we used the same arguments as in the estimates of $|I|$ and $|J|$ from the proof of \cite[Lemma 6.1]{AFR23}.

For $I_3$, using \cite[equation~(6.5)]{AFR23} we see that
\begin{align*}
I_3 &\ge - \int_{\Omega \setminus B_{2R}(x_0)}  |u(x)| \varphi^2(x)  \left( \int_{\Omega_{R}(x_0)}  u(y)K_{\Omega}(x,y) \d y \right) \d x \\
&\ge - C M_0^2 \int_{\Omega \setminus B_{2R}(x_0)} \left( \int_{\Omega _{R}(x_0)} K_{\Omega}(x,y) \d y \right) \d x \\
&\ge - C M_0^2 \int_{\Omega \setminus B_{2R}(x_0)}  \frac{|\log d_{\Omega}(x)|}{(1 + |x|)^{n+2s}} \d x \ge - C M_0^2.
\end{align*}

Finally, by an application of \autoref{lemma:fv-integral-estimate} with $v := u \varphi$ and $\Omega:= \Omega \cap B_{2R}(x_0)$ we obtain that
\begin{align*}
\int_{\Omega} h u \varphi^2 \d x &\le \int_{\Omega} |h u \varphi| \d x \le C R^{\frac{n}{2} + \eps} \Vert d_{\Omega}^{s-\eps} h \Vert_{L^{\infty}(\Omega_{2R}(x_0))} [u \varphi]_{H^s(\Omega_{2R}(x_0))}.
\end{align*}

By combining the aforementioned estimates, we conclude the proof.
\end{proof}

\subsection{A boundary H\"older estimate for equations with unbounded source terms}
\label{subsec:bdry-Holder}

In this subsection we prove a H\"older regularity estimate that covers source terms with a singular boundary behavior.
This can be achieved by an adaptation of the arguments in \cite{AFR23}. 

\begin{lemma}
\label{lemma:Holder-exploding}
Let $\Omega \subset \R^n$ be a bounded open set with Lipschitz boundary. Let $s \in (0,1)$ and $u$ be a weak solution of
\begin{align*}
L_{\Omega} u = h ~~ \text{ in } \Omega \cap B_R(x_0) =: \Omega_R(x_0)
\end{align*}
with Neumann conditions on $\partial \Omega \cap B_R(x_0)$ for some $R > 0$ and $x_0 \in \Omega$, and $d_{\Omega}^{s-\eps} h \in L^{\infty}(\Omega)$ for some $\eps \in (0,s]$. 

Then, there exist $\alpha \in (0,1)$ and $C > 0$ such that
\begin{align*}
[u]_{C^{\alpha}(\Omega_{R/2}(x_0))} \le C R^{-\alpha} \left(\Vert u \Vert_{L^{\infty}(\Omega)} + R^{s+\eps} \Vert d^{s-\eps}_{\Omega} h \Vert_{L^{\infty}(\Omega_R(x_0))} \right),
\end{align*}
where $\alpha,C$ depend only on $n, s, q, \eps$, and the Lipschitz norm of $\Omega$ in $B_R(x_0)$.
\end{lemma}

\begin{proof}
We prove the result by mimicking the procedure in \cite{AFR23}. Since large parts of the proof remain completely the same we only explain the main differences. The proof is split into several steps.

\textbf{Step 1:} First, we prove that any weak solution $v$ to 
\begin{align*}
\begin{cases}
L_{\Omega} v &\le h v ~~ \text{ in } \Omega,\\
v &\ge 0 ~~ \text{ in } \Omega
\end{cases}
\end{align*}
with Neumann conditions on $\partial \Omega$, where $\Vert d_{\Omega}^{s-\eps} h \Vert_{L^{\infty}(\Omega)} \le 1$, satisfies the estimate
\begin{align}
\label{eq:loc-bd-1}
\Vert v \Vert_{L^{\infty}(\Omega)} \le C \Vert v \Vert_{L^2(\Omega)},
\end{align}
where $C > 0$ depends only on $n,s,\eps,\Omega$, as a counterpart of \cite[Lemma 3.1]{AFR23}. 

To prove it, let us consider $\beta \ge 2$, $k \in \N$, and $u_k = \min \{ u , k\}$, and define $v = u_k^{\frac{\beta}{2} - 1} u \in H_K(\Omega)$, as in \cite{AFR23}. Following \cite[equations~(3.2) and~(3.3)]{AFR23}, we have
\begin{align*}
B_{\Omega}(v,v) \le \beta \int_{\Omega} h(x) v^2(x) \d x.
\end{align*}
By Hardy's inequality (see \cite{ChSo03,Dyd04,GrHe24}), which is applicable to $v \in H^s(\Omega)$ (i.e. without compact support in $\Omega$) since $\frac{s-\eps}{2} < \frac{1}{2}$, we estimate
\begin{align*}
\int_{\Omega} h(x) v^2(x) \d x \le \Vert d^{s-\eps}_{\Omega} h \Vert_{L^{\infty}(\Omega)} \Vert v / d^{\frac{s-\eps}{2}} \Vert_{L^2(\Omega)}^2 \le C \Vert v \Vert_{H^{s-\eps}(\Omega)}^2.
\end{align*}
Note that, for any $\delta \in (0,1)$,
\begin{align*}
[v]_{H^{s-\eps}(\Omega)}^2 &\le C \iint_{(\Omega \times \Omega) \times \{|x-y| \le \delta \}} \frac{|v(x) - v(y)|^2}{|x-y|^{n+2s}} |x-y|^{2\eps} \d y \d x \\
&\quad + C \int_{\Omega} |v(x)|^2 \left( \int_{\Omega \setminus B_{\delta}(x)} |x-y|^{-n-2s + 2 \eps} \d y \right) \d x \\
&\le C \delta^{2\eps} [v]_{H^s(\Omega)}^2 + C \delta^{-2(s-\eps)} \Vert v \Vert_{L^2(\Omega)}^2,
\end{align*}
and therefore, combining the previous three inequalities on display, we deduce that
\begin{align*}
B_{\Omega}(v,v) \le C \beta \delta^{2\eps}[v]_{H^s(\Omega)}^2 + C \beta (1 + \delta^{-2(s-\eps)}) \Vert v \Vert_{L^2(\Omega)}^2.
\end{align*}
Since $[v]_{H^s(\Omega)}^2 \le B_{\Omega}(v,v)$, we obtain, by choosing $\delta := (2C \beta)^{-\frac{1}{2\eps}}$,
\begin{align*}
B_{\Omega}(v,v) \le C \beta^{\frac{s}{\eps}} \Vert v \Vert_{L^2(\Omega)}^2.
\end{align*}
By the Sobolev inequality, this yields exactly the same estimate as \cite[equation~(3.7)]{AFR23}, and from here the proof goes in the exact same way as in \cite{AFR23}.

\textbf{Step 2:} Second, we prove that any weak solution $v$ to
\begin{align*}
\begin{cases}
L_{\Omega} v &= |h| ~~ \text{ in } \Omega \cap B_{R}(x_0),\\
v &= 0 ~~ \text{ in } \Omega \setminus \Omega_{R}(x_0)
\end{cases}
\end{align*}
with Neumann conditions on $\partial \Omega \cap B_{R}(x_0)$ for some $x_0 \in \Omega$ and $R > 0$ satisfies
\begin{align}
\label{eq:v-Linfty-estimate}
\Vert v \Vert_{L^{\infty}(\Omega_{R}(x_0))} \le C R^{s + \eps} \Vert d_{\Omega}^{s-\eps} h \Vert_{L^{\infty}(\Omega_{R}(x_0))}, 
\end{align}
where $C > 0$ depends only on $n,s,\eps$, and the Lipschitz norm of $\partial \Omega \cap B_{2R}(x_0)$, which is the counterpart of the main estimate in \cite[Lemma 3.3]{AFR23}.

Furthermore, by the same arguments as in \cite{AFR23} (in particular, using \eqref{eq:loc-bd-1}), we obtain that
\begin{align*}
\Vert v \Vert_{L^{\infty}(\Omega_{R}(x_0))} \le C R^{-\frac{n}{2}} \Vert v \Vert_{L^2(\Omega_R(x_0))} + C R^{s+\eps} \Vert d_{\Omega}^{s-\eps} h \Vert_{L^{\infty}(\Omega_{R}(x_0))}.
\end{align*}
Hence, it remains to prove that 
\begin{align}
\label{eq:v-Hs-estimate}
\Vert v \Vert_{L^2(\Omega_R(x_0))} \le c R^{\frac{n}{2} + s + \eps} \Vert d_{\Omega}^{s-\eps} h \Vert_{L^{\infty}(\Omega_R(x_0))}.
\end{align}

Then, by testing the weak formulation for $v$ with $v$ itself, and applying \autoref{lemma:fv-integral-estimate}, we obtain
\begin{align*}
[v]_{H^s(\Omega)}^2 &\le C B(v,v) = C \int_{\Omega_R(x_0)} h v \le C R^{\frac{n}{2} + \eps} \Vert d^{s-\eps}_{\Omega} h \Vert_{L^{\infty}(\Omega_R(x_0))} [v]_{H^s(\Omega)}.
\end{align*}
From here, the estimate in \eqref{eq:v-Hs-estimate} follows after dividing by $[v]_{H^s(\Omega)}$ and applying the Poincar\'e-Wirtinger inequality to $v$.

\textbf{Step 3:} Note that in Step 2 we have already established the boundedness of weak solutions. Moreover, by following the arguments in \cite[Section 4]{AFR23} we immediately obtain the desired H\"older regularity estimate. \\
To be precise, there are only two changes need to be made:
firstly, in \cite[Proof of Theorem 4.7]{AFR23},
instead of \cite[Lemma 3.3]{AFR23} we here employ \eqref{eq:v-Linfty-estimate}, secondly, in \cite[Proof of Theorem 4.1]{AFR23}, we replace everywhere $R^{2s-\frac{n}{q}} \Vert h \Vert_{L^q(\Omega_R(x_0))}$ by $R^{s+\eps} \Vert d^{s-\eps}_{\Omega} h \Vert_{L^{\infty}(\Omega_R(x_0))}$. 
\end{proof}

\subsection{Estimates of localization errors}

In this subsection we prove two lemmas that will only be used in the proofs of the expansion estimates below, in case $\Omega \subset \R^n$ is unbounded. The idea is to first prove the expansion estimates for bounded domains. In case $\Omega$ is unbounded, we construct a bounded set $\Omega'$ that coincides with $\Omega$ in $B_{3/2}$ and tapply the expansion estimate for $\Omega'$. The two lemmas below estimate the error terms that arise from this construction.

\begin{lemma}
\label{lemma:unbd-error}
Let $\Omega \subset \R^n$ be an open set with $C^{1,\alpha}$ boundary for some $\alpha > 0$. Let $\Omega' \Subset B_2$ with 
\begin{align*}
\Omega'\subset\Omega, \qquad \Omega'\cap B_{3/2} = \Omega\cap B_{3/2}.
\end{align*}
Let $u \in L^{\infty}(\Omega) \cap H_{K,\loc}(\Omega)$. Then, there is $Q \in L^2(\Omega' \cap B_{1/2})$ such that for any $\eta \in H_K(\Omega')$ with $\supp \eta \subset \overline{\Omega' \cap B_{1/2}}$,
\begin{align*}
B_{\Omega'} (u,\eta) - B_{\Omega} (u,\eta)  = \int_{\Omega' \cap B_{1/2}} Q \eta \d x, \quad |Q(x)| \le C \Vert u \Vert_{L^\infty(\Omega)} \big( 1 + |\log d_{\Omega}(x)| \big) ~~ \forall x \in \Omega' \cap B_{1/2},
\end{align*}
where $C > 0$ depends only on $n,s,\Omega$.
\end{lemma}

\begin{proof}
Given $D \subset \R^n$ and a function $w \in L^\infty(D)$, we define
\begin{align*}
E_Dw(x) := \begin{cases}
w(x) \quad &\forall x \in D,\\
\big(\int_D |x-y|^{-n-2s}\d y \big)^{-1} \int_{D} w(y) |x-y|^{-n-2s} \d y \quad &\forall x \in D^c.
\end{cases} 
\end{align*}
Then, it is easy to see that for any $\eta \in H_K(D)$ with $\eta \equiv 0$ in $D^c$,
\begin{align*}
B_D(w,\eta) = \cE_{D}(Ew,\eta),
\end{align*}
where the second claim is also proved in \cite[Appendix A]{AFR23}. Hence, we conclude that for any 
\begin{align*}
B_{\Omega'} (u,\eta) - B_{\Omega} (u,\eta) &= \cE_{\Omega'} (E_{\Omega'}u,\eta) - \cE_{\Omega} (E_{\Omega}u,\eta) \\
&= c \int_{\Omega \cap B_{1/2}} \eta(x) \left( \int_{(\Omega \cap B_{3/2})^c} \big( E_{\Omega}(y) - E_{\Omega'}(y) \big) |x-y|^{-n-2s} \d y \right) \d x \\
&=: \int_{\Omega' \cap B_{1/2}} \eta(x) Q(x) \d x,
\end{align*}
where we used that $E_{\Omega}u(x) = E_{\Omega'}u(x) = u(x)$ for $x \in \Omega \cap B_{3/2} = \Omega' \cap B_{3/2}$ by construction. Moreover, note that for $x \in \Omega \cap B_{1/2}$ it holds
\begin{align}
\label{eq:Q-est-prelim}
|Q(x)| \le c \Vert u \Vert_{L^{\infty}(\Omega)} + \int_{B_1 \setminus \Omega} \big( E_{\Omega}u(y) - E_{\Omega'}u(y) \big) |x-y|^{-n-2s} \d y,
\end{align}
since $\Vert E_{\Omega} u \Vert_{L^{\infty}(\R^n)} + \Vert E_{\Omega'} u \Vert_{L^{\infty}(\R^n)} \le c \Vert u \Vert_{L^{\infty}(\Omega)}$.
Next, we observe that for $y \in B_1 \setminus \Omega$ it holds
\begin{align*}
\left( \int_{\Omega} |z-y|^{-n-2s} \d z \right) E_{\Omega}u(y) &= \int_{\Omega \setminus \Omega'} u(z) |z-y|^{-n-2s} \d z  + \int_{\Omega'} u(z) |z-y|^{-n-2s} \d z \\
&= \int_{\Omega \setminus \Omega'} u(z) |z-y|^{-n-2s} \d z + \left( \int_{\Omega'} |z-y|^{-n-2s} \d z \right) E_{\Omega'} u(y),
\end{align*}
and therefore,
\begin{align*}
\left( \int_{\Omega} |z-y|^{-n-2s} \d z \right) \big| E_{\Omega}u(y) - E_{\Omega'} u(y) \big| &\le \int_{\Omega \setminus \Omega'} |u(z)| |z-y|^{-n-2s} \d z \\
&\quad + \left( \int_{\Omega \setminus \Omega'} |z-y|^{-n-2s} \d z \right) |E_{\Omega'}u(y)| \\
&\le C \Vert u \Vert_{L^{\infty}(\Omega)},
\end{align*}
where we used that $\Omega \setminus \Omega' \in B_{3/2}^c$. Hence, we deduce
\begin{align*}
 \big| E_{\Omega}u(y) - E_{\Omega'} u(y) \big| \le C \Vert u \Vert_{L^{\infty}(\Omega)} d_{\Omega}^{2s}(y).
\end{align*}
Plugging this estimate into \eqref{eq:Q-est-prelim}, we obtain for any $x \in \Omega \cap B_{1/2}$,
\begin{align*}
|Q(x)| \le C \Vert u \Vert_{L^{\infty}(\Omega)} \left( 1 + \int_{B_1 \setminus \Omega} d^{2s}_{\Omega}(y) |x-y|^{-n-2s} \d y \right) \le C \Vert u \Vert_{L^{\infty}(\Omega)}  \big(1 + |\log d_\Omega(x)| \big).
\end{align*}
This concludes the proof.
\end{proof}

\begin{lemma}
\label{lemma:unbd-error-2}
Let $\Omega \subset \R^n$ be an open set with $C^{1,\alpha}$ boundary for some $\alpha > 0$. Let $\Omega' \Subset B_2$ with 
\begin{align*}
\Omega'\subset\Omega, \qquad \Omega'\cap B_{3/2} = \Omega\cap B_{3/2}.
\end{align*}
Let $u \in L^{\infty}(\Omega) \cap V^s_{\loc}(\Omega|\R^n)$. Then, there are $a,b \in L^\infty(B_{1/2} \setminus \Omega')$ such that for any $\eta \in H^s(\R^n)$ with $\eta = 0$ in $B_{1/2}^c$,
\begin{align*}
\cE_{\Omega'} (u,\eta) - \cE_{\Omega} (u,\eta)  = \int_{B_{1/2} \setminus \Omega'} \big(a(x) + u(x) b(x) \big) \eta(x) \d x, 
\end{align*}
where for any $k \in \N$,
\begin{align*}
\Vert a \Vert_{C^k(B_{1/2} \setminus \Omega')}  \le C \Vert u \Vert_{L^\infty(\Omega)}, \quad \Vert b \Vert_{C^k(B_{1/2} \setminus \Omega')} \le C
\end{align*}
for some $C > 0$ depending only on $n,s,\Omega,k$.
\end{lemma}

\begin{proof}
First, note that
\begin{align*}
(\Omega^c \times \Omega^c)^c \setminus (\Omega^c \times \Omega^c)^c = (\Omega \setminus \Omega' \times \Omega \setminus \Omega') \cup (\Omega \setminus \Omega' \times \Omega^c) \cup (\Omega^c \times \Omega \setminus \Omega').
\end{align*}
Therefore, and using that $\eta \equiv 0$ in $B_{1/2}^c \supset \Omega \setminus \Omega'$,
\begin{align*}
\cE_{\Omega'} (u,\eta) - \cE_{\Omega} (u,\eta) &= c \int_{\Omega^c} \eta(x) \left( \int_{\Omega \setminus \Omega'} \big( u(x) - u(y) \big) |x-y|^{-n-2s} \d y \right) \d x \\
&= - c \int_{B_{1/2} \setminus \Omega'} \eta(x) \left(\int_{\Omega \setminus \Omega'} u(y) |x-y|^{-n-2s} \d y \right) \d x \\
&\quad + c \int_{B_{1/2} \setminus \Omega'} \eta(x) u(x) \left( \int_{\Omega \setminus \Omega'} |x-y|^{-n-2s} \d y \right) \d x\\
&=: \int_{B_{1/2} \setminus \Omega'} \eta(x) \big(a(x) + u(x) b(x) \big) \d x.
\end{align*}
The regularity estimates for $a,b$ follow immediately by the regularity of $x \mapsto |x-y|^{-n-2s}$ away from the singularity.

\end{proof}

\subsection{An expansion at boundary points of order less than $2s$}
\label{subsec:expansion-less-2s}

The goal of this subsection is to prove an expansion of solutions to the nonlocal Neumann problem of order less than $2s$. This is the main auxiliary result in our proof of the boundary regularity estimate of the same order. By \autoref{thm:Neumann-Liouville}, such an expansion can only hold true if the order does not exceed $B_0$ (see \eqref{eq:B0} and \autoref{prop:zeros-final}). In case $B_0 > 2s$, we can also prove expansions of higher order. Since its proof requires several highly non-trivial modifications, we will treat that case separately (see \autoref{thm:bdry-expansion-higher}).

\begin{theorem}
\label{thm:bdry-expansion}
Let $s \in (0,1)$ and $\eps \in (0,\min\{2s,B_0\}-s)$ be such that $\min\{2s , B_0\} - \eps \not= 1$, and let $\alpha \in [\max\{0 , \min\{2s,B_0\} - \eps - 1 \}, 1)$. Let $\Omega \subset \R^n$ be an open set with $C^{1,\alpha}$ boundary with $0 \in \partial \Omega$ and $\nu(0) = e_n$. 

Let $u \in H_K(\Omega)$ be a weak solution in the sense of \autoref{def:weak-sol} to
\begin{align*}
L_{\Omega} u = h ~~ \text{ in } \Omega \cap B_{1/2}
\end{align*}
with Neumann conditions on $\partial \Omega \cap B_{1/2}$ and $d_{\Omega}^{2s - \min\{2s,B_0\} + \eps} h \in L^{\infty}(\Omega \cap B_{1/2})$. 

Then, there exist $C > 0$, and $b \in \R^n$ with $b \cdot e_n = 0$ such that, for any $x \in \Omega \cap B_{1/4}$, 
\begin{align*}
|u(x) - u(0) & - \1_{\{ \min\{2s,B_0\} -\eps > 1 \}} b \cdot \Phi^{-1}(x)| \\
&\le C |x|^{\min\{2s,B_0\} - \eps} \left( \Vert u \Vert_{L^{\infty}(\R^n)} + \Vert d_{\Omega}^{2s-\min\{2s,B_0\}+\eps} h \Vert_{L^{\infty}(\Omega \cap B_{1/2})} \right),
\end{align*}
where $B_0$ is as in \eqref{eq:B0}. Moreover, $C > 0$ depends only on $n,s,\eps,\alpha$, and the $C^{1,\alpha}$ norm of $\Omega$. 

The function $\Phi$ is as in \autoref{subsec:correction} if $\Omega$ is bounded. If $\Omega$ is unbounded, we choose $\Omega' \subset B_2$ with $C^{1,\alpha}$ boundary such that $\Omega \cap B_{3/2} = \Omega' \cap B_{3/2}$ and choose $\Phi$ with respect to $\Omega'$ as in \autoref{subsec:correction}.
\end{theorem}

Let us first prove a Liouville theorem of order $\min\{2s,B_0\}-\eps$ in the half-space. Its proof follows immediately from \autoref{thm:Neumann-Liouville}.

\begin{theorem}
\label{thm:Neumann-Liouville-nd}
Let $u$ be a weak solution to
\begin{align*}
\begin{cases}
(-\Delta)^s u &= 0 ~~ \text{ in } \{ x_n > 0 \},\\
\mathcal{N}^s_{\{ x_n > 0 \}} u &= 0 ~~ \text{ in } \{ x_n \le 0 \}
\end{cases}
\end{align*}
with
\begin{align*}
|u(x)| \le C (1 + |x|)^{\min\{B_0,2s\}-\eps}
\end{align*}
for some $C > 0$ and $\eps \in (0,2s)$, where $B_0$ is as in \eqref{eq:B0}. Then,
\begin{align*}
u = a + b \cdot x.
\end{align*}
for some $a \in \R$ and $b \in \R^{n}$ with $b_n = 0$.
\end{theorem}

\begin{proof}
The proof goes exactly as in \cite[Theorem 5.1]{AFR23}, only replacing the 1D Liouville theorem from \cite{AFR23} by \autoref{thm:Neumann-Liouville} of this paper. The restriction $s \in (\frac{1}{2},1)$ can be dropped without any further change in the proof, since in \cite{AFR23} it only came from the 1D Liouville theorem. In fact, by \cite[Proposition 5.10]{AFR23} we have that
\begin{align*}
u(x) = w_0(x_n) + \sum_{i = 1}^{n-1} w_i(x_n) x_i
\end{align*}
for some functions $w_0,w_1,\dots, w_{n-1}$, depending only on $x_n$. Moreover, by \cite[Lemma 5.11, Proof of Theorem 5.1]{AFR23} the function $w_i$ are weak solutions to \eqref{eq:Neumann-1D} and have the same growth as $u$. Therefore, we can apply \autoref{thm:Neumann-Liouville} to deduce the desired result.
\end{proof}

We are now in a position to give the proof of \autoref{thm:bdry-expansion}.

\begin{proof}[Proof of \autoref{thm:bdry-expansion}]

First, we prove the result in case $\Omega \subset \R^n$ is bounded.

In case $\min\{2s,B_0\} -\eps < 1$, the proof goes exactly as in \cite[Proposition 6.2]{AFR23}, but using \autoref{thm:Neumann-Liouville-nd},  \autoref{lemma:Hk-bound}, and \autoref{lemma:Holder-exploding} instead of \cite[Theorem 5.1]{AFR23}, \cite[Lemma 6.1]{AFR23}, and \cite[Theorem 4.1]{AFR23}, respectively. Let us therefore explain only the case $\min\{2s,B_0\} -\eps > 1$. In particular,  we can assume from now on that $s > \frac{1}{2}$.

\textbf{Step 1:} We now use \cite[Proposition 3.2]{AFR23}
and~\autoref{lemma:Holder-exploding}.
In this way, we see that $u$ is bounded
and
we can assume that $\Vert u \Vert_{L^{\infty}(\R^n)} + \Vert d_{\Omega}^{2s - \min\{2s,B_0\} + \eps} h \Vert_{L^{\infty}(\Omega)} \le 1$.

Accordingly, it suffices to show that
\begin{align}
\label{eq:blowup-claim}
|u(x) - u(0) - b \cdot \Phi^{-1}(x) | \le C |x|^{\min\{2s,B_0\}-\eps} ~~ \forall x \in \Omega \cap B_{1/4}.
\end{align}
Let us assume by contradiction that \eqref{eq:blowup-claim} does not hold true. In that case there are sequences $(\Omega_k)$ as in the statement of the lemma with corresponding diffeomorphisms $(\Phi_k)$, as well as $(u_k) \subset H_K(\Omega_k)$, $(h_k)$ with $d_{\Omega_k}^{2s - \min\{2s,B_0\} + \eps} h_k \in L^{\infty}(\Omega_k)$, such that
\begin{align}
\label{eq:blowup-normalization}
\Vert u_k \Vert_{L^{\infty}(\R^n)} + \Vert d_{\Omega_k}^{2s-\min\{2s,B_0\}+\eps} h_k \Vert_{L^{\infty}(\Omega_k \cap B_{1/2})}  \le 1,
\end{align}
and 
\begin{align*}
L_{\Omega_k} u_k = h_k ~~ \text{ in } \Omega_k \cap B_{1/2}
\end{align*}
in the weak sense with Neumann conditions on $\partial \Omega_k \cap B_{1/2}$. 

Moreover, there are $(x_k) \subset \Omega_k$, and $C_k \to +\infty$, such that 
\begin{align*}
\inf_{b \in \{  b_n = 0 \}} \sup_{r \in [0,\frac{1}{4}]} \frac{\Vert u_k - u_k(0) - b \cdot \Phi_k^{-1} \Vert_{L^{\infty}(\Omega_k \cap B_r)}}{|x_k|^{\min\{2s,B_0\}-\eps}} \ge \inf_{b \in \{  b_n = 0 \}} \frac{|u_k(x_k) - u_k(0) - b \cdot \Phi_k^{-1}(x_k) |}{|x_k|^{\min\{2s,B_0\}-\eps}} \ge C_k.
\end{align*}

For any $k \in \N$ and $r \in (0,\frac{1}{4}]$ we consider the $L^2(\Omega_k \cap B_{r})$ projections of $u_k(\cdot) - u_k(0)$ over the space $\cP_{k,r} := \{ b \cdot \Phi^{-1}_k(\cdot) : b \in \{ b_n = 0 \} \}$ and denote them by $P_{k,r}$. Moreover, we will extend $P_{k,r}$ in a bounded way outside $\Omega_k \cap B_1$, as in \autoref{subsec:correction}. In this way, we see that
\begin{align*}
\Vert u_k - u_k(0) - P_{k,r} \Vert_{L^2(\Omega_k \cap B_{r})} &\le \Vert  u_k - u_k(0) - P \Vert_{L^2(\Omega_k \cap B_{r})} ~~ \forall P \in \cP_{k,r},\\
\int_{\Omega_k \cap B_{r}} (u_k(x) - u_k(0) - P_{k,r}(x)) P(x) \d x &= 0 ~~ \forall P \in \cP_{k,r}.
\end{align*}

It follows from \autoref{lemma:theta-increasing} that the function 
\begin{align*}
\theta(r): = \sup_{k \in \N} \sup_{\rho \in [r , \frac{1}{4}]} \rho^{-\min\{2s,B_0\}+\eps} \Vert u_k - u_k(0) - P_{k,\rho} \Vert_{L^{\infty}(\Omega_k \cap B_{\rho})}
\end{align*} 
satisfies 
\begin{align*}
\theta(r) \nearrow +\infty ~~ \text{ as } ~~ r \searrow 0,
\end{align*}
and therefore there exist subsequences $(k_j) \subset \N$ and $r_j \searrow 0$ with $r_j \in (0,\frac{1}{8}]$ such that
\begin{align*}
\frac{\Vert u_{k_j} - u_{k_j}(0) - P_{k_j,r_j} \Vert_{L^{\infty}(\Omega_{k_j} \cap B_{r_j})}}{\theta(r_j) r_j^{\min\{2s,B_0\}-\eps}} \ge \frac{1}{2} ~~ \forall j \in \N.
\end{align*}

\textbf{Step 2:} Next, we introduce the blow-up sequence
\begin{align*}
v_j(x) = \frac{u_{k_j}(r_j x) - u_{k_j}(0) - P_{k_j,r_j}(r_j x)}{r_j^{\min\{2s,B_0\}-\eps} \theta(r_j)}, ~~ j \in \N,
\end{align*}
and observe that, for all $j \in \N$,
\begin{align}
\label{eq:vj-properties}
v_j(0) = 0, \qquad \Vert v_j \Vert_{L^{\infty}(\tilde{\Omega}_{j} \cap B_1)} \ge \frac{1}{2}, \qquad \int_{\tilde{\Omega}_j \cap B_1} v_j(x) r_j^{-1} P(r_j x) \d x = 0 ~~ \forall P \in \cP_{k_j,r_j},
\end{align}
where $\tilde{\Omega}_j = r_{_j}^{-1} \Omega_{k_j}$. 

Moreover, we claim that
\begin{align}
\label{eq:vj-growth}
\Vert v_j \Vert_{L^{\infty}(\tilde{\Omega}_j \cap B_R)} \le C R^{\min\{2s,B_0\}-\eps} ~~ \forall R \in [1 , r_j^{-1}/4 ],
\end{align}
and
\begin{align}
\label{eq:b-coeff-convergence}
\frac{|b_{k_j,r_j}|}{\theta(r_j)} \to 0 ~~ \text{ as } j \to \infty.
\end{align}

Both of these claims follow from \autoref{lemma:aux-theta}, applied with $u := u_{k_j}- u_{k_j}(0)$. Note that the assumptions of \autoref{lemma:aux-theta} are satisfied, since, by the definition of $\theta(r)$ and by \eqref{eq:blowup-normalization}, for any $r \le \frac{1}{4}$ we have that
\begin{align*}
\Vert u_{k_j} - u_{k_j}(0) - P_{k_j,r_j} \Vert_{L^{\infty}(\Omega_{k_j} \cap B_{r})} \le \theta(r) r^{\min\{2s,B_0\}-\eps}, \qquad \Vert u_{k_j} - u_{k_j}(0) \Vert_{L^{\infty}(\Omega_{k_j} \cap B_{1/4})} \le 2.
\end{align*}

Hence, \eqref{eq:b-coeff-convergence} follows immediately by the second claim in \autoref{lemma:aux-theta}, since the bound is independent of $k_j$. The growth of $v_j$ from \eqref{eq:vj-growth} follows from the first claim of \autoref{lemma:aux-theta}, which yields
\begin{align*}
\Vert v_j \Vert_{L^{\infty}(\tilde{\Omega}_j \cap B_R)} = \frac{\Vert u_{k_j} - u_{k_j}(0) - P_{k_j,r} \Vert_{L^{\infty}(\Omega_{k_j} \cap B_{r_j})}}{r_j^{\min\{2s,B_0\}-\eps} \theta(r_j)} \le C R^{\min\{2s,B_0\}-\eps}.
\end{align*}

Finally, observe that $v_j$ satisfies
\begin{align}
\label{eq:vj-PDE}
L_{\tilde{\Omega}_j} v_j = \tilde{h}_j ~~ \text{ in } \tilde{\Omega}_j \cap B_{r_j^{-1}/2}
\end{align}
with Neumann conditions on $\partial \tilde{\Omega}_j \cap B_{r_j^{-1}/2}$, where
\begin{align*}
\tilde{h}_j(x) = \frac{r_j^{2s - \min\{2s,B_0\} + \eps}}{\theta(r_j)} h_{j}(r_j x) - \frac{r_j^{2s - \min\{2s,B_0\} + \eps}}{\theta(r_j)} (L_{\Omega_j} P_{k_j,r_j})(r_j x).
\end{align*}

By applying \autoref{lemma:Holder-exploding} to $w_j := v_j \1_{B_{4R}}$, we deduce that there exists $\gamma \in (0,s)$ such that, for any $R \in [1,r_j^{-1}/8]$,
\begin{align}
\label{eq:regularity-estimate-w}
\begin{split}
&[v_j]_{C^{\gamma}(\tilde{\Omega}_j \cap B_R)} \le [w_j]_{C^{\gamma}(\tilde{\Omega}_j \cap B_R)} \\
&\qquad \le C R^{-\gamma} \Big(  \Vert w_j \Vert_{L^{\infty}(\tilde{\Omega}_j)}  + R\Vert d_{\tilde{\Omega}_j}^{2s-1} \tilde{h}_j \Vert_{L^{\infty}(\tilde{\Omega}_j \cap B_{2R})} + R \Vert d_{\tilde{\Omega}_j}^{2s-1} \tail_j(v_j ; 4R , \cdot) \Vert_{L^{\infty}(\tilde{\Omega}_j \cap B_{2R})} \Big),
\end{split}
\end{align}
where
\begin{align*}
\tail_j(v_j ; 4R , x) := \int_{\tilde{\Omega}_j \setminus B_{4R}} v_j(y) K_{\tilde{\Omega}_j}(x,y) \d y.
\end{align*}

By the same computation as in \cite[Lemma 5.9]{AFR23}, and using the growth \eqref{eq:vj-growth} (and the fact that~$\min\{2s,B_0\}-\eps < 2s$), we obtain that, for any $x \in \tilde{\Omega}_j \cap B_{2R}$,
\begin{align*}
|\tail_j(v_j ; 4R , x)| \le C(R) (1 + |\log d_{\tilde{\Omega}_j}(x)|), \qquad \Vert d_{\tilde{\Omega}_j}^{2s-1} \tail_j(v_j ; 4R , \cdot) \Vert_{L^{\infty}(\tilde{\Omega}_j \cap B_{2R})} \le C(R)
\end{align*}
for some constant $C(R) > 0$, depending on $R$.

In addition, we notice that $d_{\tilde{\Omega}_j}(x) = r_j^{-1} d_{\Omega_j}(r_j x)$.
Also, since $s > \frac{1}{2}$, there exists $\delta \in (0, \eps)$ such that $2s-1-\delta > 0$. 

Consequently, using also \autoref{lemma:LPb-estimate}, we have, when $R \le r_j^{-1}/8$,
\begin{align*}
& \frac{r_j^{2s - \min\{2s,B_0\} + \eps}}{\theta(r_j)} \Vert d_{\tilde{\Omega}_j}^{2s-1} (L_{\Omega_j} P_{k_j,r_j})(r_j \cdot) \Vert_{L^{\infty}(\tilde{\Omega}_j \cap B_{2R})} = \frac{r_j^{- \min\{2s,B_0\} + \eps + 1}}{\theta(r_j)} \Vert d_{\Omega_j}^{2s-1} (L_{\Omega_j} P_{k_j,r_j}) \Vert_{L^{\infty}(\Omega_j \cap B_{2 r_j R})} \\
&\quad\le C |b_{k_j,r_j}| \frac{r_j^{-\min\{2s,B_0\} + \eps + 1}}{\theta(r_j)} \Vert d_{\Omega_j}^{2s-1} (1 + |\log d_{\Omega_j}| + d_{\Omega_j}^{1+\alpha-2s} + |\cdot|^{\alpha} ( d_{\Omega_j}^{1-2s} + |\log d_{\Omega_j}|)) \Vert_{L^{\infty}(\Omega_j \cap B_{2 r_j R})} \\
&\quad\le  C |b_{k_j,r_j}| \frac{r_j^{-\min\{2s,B_0\} + \eps + 1}}{\theta(r_j)} \Vert(d_{\Omega_j}^{2s-1} + d_{\Omega_j}^{2s-1-\delta} + d_{\Omega_j}^{\alpha} + |\cdot|^{\alpha}) \Vert_{L^{\infty}(\Omega_j \cap B_{2 r_j R})} \\
&\quad\le C(R) |b_{k_j,r_j}| \frac{r_j^{-\min\{2s,B_0\} + \eps + 1 + \alpha} + r_j^{-\min\{2s,B_0\} + 2s + \eps - \delta}}{\theta(r_j)} \le C(R) \frac{|b_{k_j,r_j}|}{\theta(r_j)},
\end{align*}
where we used that $1+\alpha \ge \min\{2s,B_0\} - \eps$. 

Furthermore, by \eqref{eq:blowup-normalization} and since $\eps < \min\{2s,B_0\} - 1$,
\begin{align*}
\frac{r_j^{2s - \min\{2s,B_0\} + \eps}}{\theta(r_j)} & \Vert d_{\tilde{\Omega}_j}^{2s-1} h_j(r_j \cdot) \Vert_{L^{\infty}(\tilde{\Omega}_j \cap B_{2R})} = \frac{r_j^{- \min\{2s,B_0\} + \eps + 1}}{\theta(r_j)} \Vert d_{\Omega_j}^{2s-1} h_j \Vert_{L^{\infty}(\Omega_j \cap B_{2r_j R})} \\
&\le C(R) \theta(r_j)^{-1} \Vert d_{\Omega_j}^{2s - \min\{2s,B_0\} + \eps} h_j \Vert_{L^{\infty}(\Omega_j \cap B_{2r_j R})} \le C(R)\theta(r_j)^{-1}.
\end{align*}

As a result, recalling~\eqref{eq:b-coeff-convergence}, 
\begin{align}
\label{eq:f-bound}
\Vert d_{\tilde{\Omega}_j}^{2s-1} \tilde{h}_j \Vert_{L^{\infty}(\tilde{\Omega}_j \cap B_{2R})} \le C(R) \left(\frac{1 + |b_{k_j,r_j}|}{\theta(r_j)} \right) \to 0 ~~ \text{ as } j \to \infty.
\end{align}
In particular, the above norm is uniformly bounded.

We also point out that $$\Vert w_j \Vert_{L^{\infty}(\tilde{\Omega}_j)} \le \Vert v_j \Vert_{L^{\infty}(\tilde{\Omega}_j)} \le 1,$$ due to~\eqref{eq:blowup-normalization}.

Thus, by combining the previous estimates with \eqref{eq:vj-growth}
and~\eqref{eq:regularity-estimate-w}, we deduce that
\begin{align*}
\Vert v_j \Vert_{C^{\gamma}(\tilde{\Omega}_j \cap B_R)} \le C(R) \left( 1 +  \frac{1 + |b_{k_j,r_j}|}{\theta(r_j)} \right) \le C(R).
\end{align*}
Hence, by Arzel\`a-Ascoli's theorem, the sequence $(v_j)_j$ converges locally uniformly in $\{ x_n \ge 0\}$, up to a subsequence, to some $v \in C^{\gamma'}_{loc}(\{ x_n \ge 0\})$, for any $\gamma' \in (0,\gamma)$. Moreover, by \eqref{eq:vj-growth} and \eqref{eq:f-bound}, we can apply \autoref{lemma:Hk-bound} to deduce that for any $R \le r_j^{-1}/8$ it holds
\begin{align*}
[v_j]_{H^s(\tilde{\Omega}_j \cap B_R)} \le C(R) + C(R) \Vert d_{\tilde{\Omega}_j}^{2s-1} \tilde{h}_j \Vert_{L^{\infty}(\tilde{\Omega}_j \cap B_{2R})} \le C(R).
\end{align*}
Hence, the sequence $(v_j)$ is also uniformly bounded in $H_{K,loc}(\overline{\tilde{\Omega}_j})$, and therefore it also holds $v \in H_{K,loc}(\{ x_n \ge 0 \})$. 

Moreover, by the uniform convergence and \eqref{eq:vj-properties}, we also have
\begin{align}
\label{eq:v-properties}
v(0) = 0, \qquad \Vert v \Vert_{L^{\infty}(\{ x_n \ge 0 \} \cap B_1)} \ge \frac{1}{2}, \qquad \int_{\{ x_n > 0 \} \cap B_1} v(x) ( b \cdot x ) \d x = 0 ~~ \forall b \in \{ b_n = 0 \}.
\end{align}
The latter property holds true since for any $P_j = b \cdot \Phi_j^{-1} \in \cP_j$ with $b_n = 0$, it holds $r_j^{-1} P_j(r_j x) \to b \cdot x$ locally uniformly since $P_j(0) = 0$, $\nabla P_j(0) = b$, and $P_j$ is uniformly bounded in $C^{1,\alpha}$.

\textbf{Step 3:} 
Let us now pass to the limit in the equation for $v_j$.
Since $v_j$ satisfies \eqref{eq:vj-PDE} in the weak sense with Neumann conditions on $\partial \tilde{\Omega}_j \cap B_{r_j^{-1}/2}$, it also satisfies the equation in the distributional sense, namely 
\begin{align}
\label{eq:vj-distr}
\int_{\tilde{\Omega}_j} v_j L_{\tilde{\Omega}_j} \eta \d x = \frac{1}{2} \int_{\tilde{\Omega}_j} \tilde{h}_j \eta \d x ~~ \forall \eta \in C_c^{\infty}(B_{r_j^{-1}/2}).
\end{align}
This can be proved in the same way as in \cite[Proof of equation~(6.9)]{AFR23} (using \cite[equations~(6.4) and (6.5)]{AFR23}), and there are no differences due to the faster growth of $v_j$ in \eqref{eq:vj-growth} (see also \autoref{lemma:weak-distr}).

Now we note that, for any $\eta \in H_K(\R^n)$ supported in $B_R$ for some $R \ge 1$, it holds that
\begin{align*}
\int_{\tilde{\Omega}_j} \tilde{h}_j \eta \d x \le C(\eta,R) \Vert d_{\tilde{\Omega}_j}^{2s-1} \tilde{h}_j \Vert_{L^{\infty}(\tilde{\Omega}_j \cap B_{2R})} \to 0,
\end{align*}
thanks to \autoref{lemma:fv-integral-estimate} and \eqref{eq:f-bound}. 

Hence, using that $\1_{\tilde{\Omega}_j} \to \1_{\{ x_n > 0\}}$ and $K_{\tilde{\Omega}_j} \to K_{\{ x_n > 0 \}}$ a.e., we can apply Vitali's convergence theorem as in \cite[Step 3, Proposition 6.2]{AFR23} and deduce that 
\begin{align*}
\int_{\{ x_n > 0 \}} v L_{\{ x_n > 0 \}} \eta \d x = 0 ~~ \forall \eta \in C_c^{\infty}(\R^n).
\end{align*}
Furthermore, as in \cite[Step 3, Proposition 6.2]{AFR23}, we can prove that $v$ is a weak solution to
\begin{align*}
L_{\{ x_n > 0 \}} v = 0 ~~ \text{ in } \{ x_n > 0 \} 
\end{align*}
with Neumann conditions on $\{ x_n = 0 \}$.

In addition, we have that $\Vert v \Vert_{L^{\infty}(B_R)} \le C (1 + R^{\min\{2s,B_0\} - \eps})$ for any $R \ge 1$.

\textbf{Step 4:} Hence, we can apply the Liouville theorem in \autoref{thm:Neumann-Liouville-nd} and deduce that
\begin{align*}
v(x) = a + b \cdot x
\end{align*}
for some $a \in \R$ and $b \in \R^n$ with $b_n = 0$. Since $v(0) = 0$ by \eqref{eq:v-properties} we deduce that $a = 0$. Moreover, by the third property in \eqref{eq:v-properties}, we have
\begin{align*}
0 = \int_{\{ x_n > 0 \} \cap B_1} v(x) b \cdot x \d x = \int_{\{ x_n > 0 \} \cap B_1} (b \cdot x)^2 \d x,
\end{align*}
which implies that also $b = 0$, and therefore $v = 0$. This is a contradiction with the second property in \eqref{eq:v-properties}. Hence, we have shown \eqref{eq:blowup-claim}, as desired. This proves the result in case $\Omega$ is bounded.

Finally, note that if $\Omega \subset \R^n$ is unbounded, we can construct a bounded open set $\Omega' \subset \Omega \cap B_2$ with $\partial \Omega' \in C^{1,\alpha}$ and such that $\Omega' \cap B_{3/2} = \Omega \cap B_{3/2}$. Then, by \autoref{lemma:unbd-error}, we can apply \autoref{thm:bdry-expansion} with $\Omega'$ instead of $\Omega$, using that
\begin{align*}
L_{\Omega'} u = h + Q \quad \text{in } \Omega' \cap B_{1/2},
\end{align*}
where 
\begin{align*}
|Q(x)| \le C \Vert u \Vert_{L^{\infty}(\R^n)} (1 + |\log d_{\Omega}(x)|)\quad \text{in } \Omega \cap B_{1/2}.
\end{align*}
In particular, we have that $d^{2s - \min\{2s,B_0\} + \eps} Q \in L^{\infty}(\Omega \cap B_{1/2})$. This concludes the proof.
\end{proof}

\subsection{Regularity of order less than $2s$}
\label{subsec:reg-less-2s}

In this subsection we prove our main result in case $B_0 < 2s$:

\begin{theorem}
\label{thm:main-Neumann}
Let $s \in (0,1)$, $\eps \in (0,\min\{2s,B_0\}-s)$, $\alpha \in [\max\{ 0 , \min\{2s,B_0\} - \eps - 1\} , 1)$, and $\Omega \subset \R^n$ be an open set with $\partial \Omega \in C^{1,\alpha}$.
Let $u \in H_K(\Omega)$ be a weak solution to
\begin{align*}
\begin{cases}
(-\Delta)^s u &= h ~~ \text{ in } \Omega \cap B_1,\\
\mathcal{N}_{\Omega}^s u &= 0 ~~ \text{ in } B_1 \setminus \Omega,
\end{cases}
\end{align*}
where $d_{\Omega}^{2s-\min\{2s,B_0\}+\eps} h \in L^{\infty}(\Omega \cap B_1)$. 

Then, there exists $C > 0$, depending only on $n,s,\eps,\alpha$, and $\Omega$, such that 
\begin{align*}
\Vert u \Vert_{C^{\min\{2s,B_0\}-\eps}(\overline{\Omega} \cap B_{1/2})} \le C \left( \Vert u \Vert_{L^{\infty}(\R^n)} + \Vert d_{\Omega}^{2s-\min\{2s,B_0\}+\eps} h \Vert_{L^{\infty}(\Omega \cap B_1)} \right).
\end{align*}
\end{theorem}

\begin{remark}
We stress that if $\min\{2s,B_0\}-\eps < 1$, our result allows for $\alpha = 0$, i.e. it suffices to have $\partial \Omega \in C^1$.
\end{remark}

\begin{proof}[Proof of \autoref{thm:main-Neumann}]
In case $\min\{2s,B_0\}-\eps < 1$ the proof follows from \autoref{thm:bdry-expansion} (up to a rescaling, rotation, and translation) in the exact same way as \cite[Theorem 1.1]{AFR23} follows from \cite[Proposition 6.2]{AFR23}.

Hence, we can assume from now on that $\min\{2s,B_0\} - \eps \ge 1$, which implies in particular that $s > \frac{1}{2}$. Furthermore, up to a normalization, we can assume that
\begin{align}
\label{eq:normalization-thm}
\Vert u \Vert_{L^{\infty}(\R^n)} + \Vert d_{\Omega}^{2s-\min\{2s,B_0\}+\eps} h \Vert_{L^{\infty}(\Omega \cap B_1)} \le 1.
\end{align}
We note that, since $\cN^s_{\Omega} u = 0$ in $B_1 \setminus \Omega$ (see \cite[Proposition A.3]{AFR23}), we have that, for all~$x \in B_1 \setminus \Omega$,
\begin{align}
\label{eq:bdness-outside}
|u(x)| \le \left| \left( \int_{\Omega} |x-y|^{-n-2s} \d y \right)^{-1} \int_{\Omega} u(y)|x-y|^{-n-2s} \d y \right| \le \Vert u \Vert_{L^{\infty}(\Omega)} \le 1.
\end{align}

We assume that $0 \in \partial \Omega$, $\nu(0) = e_n$, and that there exists a $C^{1,\alpha}$ diffeomorphism $\Phi$ as in \eqref{eq:Pb-normal-zero}. We claim that for any $x_0 \in \Omega_{1/8} := \Omega \cap B_{1/8}$ with $|x_0| = d_{\Omega}(x_0) = :r$ there exists $b \in \{ b_n = 0 \}$ with $|b| \le C$ such that
\begin{align}
\label{eq:claim-thm}
[u - u(0) - P_b]_{C^{\min\{2s,B_0\}-\eps}(B_{r/2}(x_0))} \le C,
\end{align}
where $C > 0$ depends only on $n,s,\eps$, and the $C^{1,\alpha}$ norm of $\Omega$, and $P_b$ is as in \autoref{subsec:correction}, i.e. in particular $P_b = b \cdot \Phi^{-1}$ in $\Omega_{1}$. 

We suppose for a moment that~\eqref{eq:claim-thm} holds and we
use it to complete the proof of the desired result.
To this end, we point out that, as a consequence of~\eqref{eq:claim-thm},
\begin{align*}
[u]_{C^{\min\{2s,B_0\}-\eps}(B_{r/2}(x_0))} \le C + [P_b]_{C^{\min\{2s,B_0\}-\eps}(B_{r/2}(x_0))} \le C(1 + |b|) \le C,
\end{align*}
where we used that $|b| \le C$. From here, applying the previous arguments again, up to a rotation, translation, and rescaling, we obtain that there exists $\rho > 0$, depending only on $\Omega$, such that, for any $x_0 \in \Omega \cap B_{1/2}$ with $d_{\Omega}(x_0) \le 2\rho$,
\begin{align}
\label{eq:u-bdry-est}
[u]_{C^{\min\{2s,B_0\}-\eps}(B_{d_{\Omega}(x_0)/2}(x_0))} \le C.
\end{align}
For $x_0 \in \Omega$ with $d_{\Omega}(x_0) > 2\rho$ we use that, by interior regularity estimates for the fractional Laplacian (see e.g. \cite[Lemma 6.4]{AFR23}),
\begin{align*}
[u]_{C^{\min\{2s,B_0\}-\eps}(B_{d_{\Omega}(x_0)/2}(x_0))} \le C(\rho) \left( \Vert u \Vert_{L^{\infty}(B_{\rho}(x_0))} + \Vert (1 + |\cdot|)^{-n-2s} u \Vert_{L^1(\R^n)} + \Vert h \Vert_{L^{\infty}(B_{\rho}(x_0))} \right),
\end{align*}
and since we assumed \eqref{eq:normalization-thm}, we deduce that the right-hand side in the previous estimate is bounded.
Therefore, \eqref{eq:u-bdry-est} holds true for all $x_0 \in \Omega \cap B_{1/2}$ (with $C > 0$ depending also on $\rho$). From here, the desired result follows by \cite[Lemma A.1.4]{FeRo24}.

It remains to prove \eqref{eq:claim-thm}. To do so, let us apply the interior regularity estimate for the fractional Laplacian (see e.g. \cite[Lemma 6.4]{AFR23}) to the function
\begin{align*}
v(x): = u(x) - u(0) - P_b(x),
\end{align*}
where $b \in \{ b_n = 0 \}$ is the vector from \autoref{thm:bdry-expansion}, which solves
\begin{align*}
\begin{cases}
(-\Delta)^s v &= h - L_{\Omega} P_b ~~ \text{ in } \Omega \cap B_1,\\
\mathcal{N}_{\Omega}^s v &= 0 ~~\qquad\quad ~~ \text{ in } B_1 \setminus \Omega,
\end{cases}
\end{align*}
and conclude that
\begin{align}
\label{eq:int-reg-appl}
\begin{split}
& [v]_{C^{\min\{2s,B_0\}-\eps}(B_{r/2}(x_0))} \le C r^{\eps-\min\{2s,B_0\}}  \Vert v \Vert_{L^{\infty}(B_{2r/3}(x_0))} \\
&\qquad + C r^{2s-\min\{2s,B_0\}+\eps} \Vert |\cdot -x_0|^{-n-2s} v \Vert_{L^1(\R^n \setminus B_{2r/3}(x_0))} \\
&\qquad + C r^{2s-\min\{2s,B_0\}+\eps} \Vert h \Vert_{L^{\infty}(B_{2r/3}(x_0))} + C r^{2s-\min\{2s,B_0\}+\eps} \Vert  L_{\Omega} P_b \Vert_{L^{\infty}(B_{2r/3}(x_0))} .
\end{split}
\end{align}

Also, for any $x \in B_{2r/3}(x_0)$ we have that $d_{\Omega}(x) \ge \frac{r}{3}$ and accordingly
\begin{align*}
r^{\eps} \Vert h \Vert_{L^{\infty}(B_{2r/3}(x_0))} \le C \Vert d_{\Omega}^{\eps} h \Vert_{L^{\infty}(\Omega)} \le C.
\end{align*}
We now observe that \autoref{lemma:LPb-estimate} is applicable here,
since $B_{2r/3}(x_0) \subset \Omega_{2r} \subset \Omega_{1/4}$. 
Besides, 
given $x \in B_{2r/3}(x_0)$,
we point out
that $|x| \le 2 r$ and $\frac{r}{3} \le d_{\Omega}(x) 2r$. Also, we know that~$1 + \alpha \ge \min\{2s,B_0\} - \eps$.

Hence,
using again that $|b| \le C$ (thanks to~\autoref{thm:bdry-expansion} and \eqref{eq:normalization-thm}), we find that,
for any $\eps' \in (0,\eps)$,
\begin{align*}
r^{2s-\min\{2s,B_0\}+\eps} & \Vert L_{\Omega} P_b \Vert_{L^{\infty}(B_{2r/3}(x_0))} \\
&\le C r^{\eps + 1 - \min\{2s,B_0\}} \Vert d_{\Omega}^{2s-1} L_{\Omega} P_b  \Vert_{L^{\infty}(B_{2r/3}(x_0))} \\
&\le C r^{\eps + 1 - \min\{2s,B_0\}} \Vert d_{\Omega}^{2s-1} (1 + |\log d_{\Omega}| + d_{\Omega}^{1+\alpha - 2s} + |\cdot|^{\alpha} d_{\Omega}^{1-2s} ) \Vert_{L^{\infty}(B_{2r/3}(x_0))} \\
&\le C r^{\eps + 1 - \min\{2s,B_0\}} \Vert d_{\Omega}^{2s-1} + d_{\Omega}^{2s-1-\eps'} + d_{\Omega}^{\alpha} + |\cdot|^{\alpha} ) \Vert_{L^{\infty}(B_{2r/3}(x_0))} \\
&\le C r^{2s-\min\{2s,B_0\}+\eps} + C r^{2s-\min\{2s,B_0\}+\eps - \eps'} + C r^{-\min\{2s,B_0\} + \eps + 1 + \alpha} \le C,
\end{align*}as desired.

It remains to estimate the growth of $v$. To do so, we observe that, by appropriately localizing the proof of \cite[Proposition A.3]{AFR23} and using the equation for $v$, it holds
\begin{align*}
L_{\Omega} v = h - L_{\Omega} P_b + g \quad \text{in } \Omega \cap B_{3/4},
\end{align*}
where
\begin{align*}
g(x) = c \int_{\Omega^c \setminus B_1} \frac{u(z) - Eu(z)}{|x-z|^{n+2s}} \d z, \quad \text{ where } \quad Eu(z) = \left(\int_{\Omega} |z-y|^{-n-2s} \d y \right)^{-1} \int_{\Omega} u(y) |z-y|^{-n-2s} \d y. 
\end{align*}
Note that $\Vert E u \Vert_{L^{\infty}(\Omega^c \setminus B_1)} \le C$, and therefore, 
\begin{align*}
\Vert g \Vert_{L^{\infty}(\Omega \cap B_{3/4}} \le C \Vert u \Vert_{L^{\infty}(\R^n)}.
\end{align*}
Hence, \autoref{thm:bdry-expansion} is applicable and using also \eqref{eq:normalization-thm}, we derive
\begin{align}
\label{eq:v-growth-claim-Omega}
|v(x)| \le C |x|^{\min\{2s,B_0\}-\eps} ~~ \forall x \in \Omega_{1/4}.
\end{align}
In particular, 
\begin{align*}
r^{\eps-\min\{2s,B_0\}}  \Vert v \Vert_{L^{\infty}(B_{2r/3}(x_0))} \le r^{\eps-\min\{2s,B_0\}}  \Vert v \Vert_{L^{\infty}(\Omega_{2r})} \le C.
\end{align*}
Next, we claim that
\begin{align}
\label{eq:v-growth-claim}
|v(x)| \le C |x|^{\min\{2s,B_0\}-\eps} ~~ \forall x \in \R^n.
\end{align}
For $x \in B_{1/4} \setminus \Omega$, the claim follows by the same reasoning as in \eqref{eq:bdness-outside}, using that $\cN^s_{\Omega} P_b = 0$ in $B_1 \setminus \Omega$ by construction.
Note that by \eqref{eq:normalization-thm} and since $|P_b| \le C$ by construction, this implies \eqref{eq:v-growth-claim} for $x \in \R^n \setminus B_{1/4}$. 

Therefore, using that $|x-x_0| \ge \frac{2r}{3}$ in $B_{2r} \setminus B_{2r/3}(x_0)$ and $|x-x_0| \ge \frac{|x|}{3}$ in $\R^n \setminus B_{2r}$, we deduce
\begin{align*}
r^{2s-\min\{2s,B_0\}+\eps} & \Vert |\cdot -x_0|^{-n-2s} v \Vert_{L^1(\R^n \setminus B_{2r/3}(x_0))} \\
&\le C r^{2s-\min\{2s,B_0\}+\eps} |B_{2r} \setminus B_{2r/3}(x_0)| r^{-n-2s} r^{\min\{2s,B_0\}-\eps} \\
&\quad + C r^{2s-\min\{2s,B_0\}+\eps} \int_{\R^n \setminus B_{2r}} |x|^{-n-2s} |x|^{\min\{2s,B_0\}-\eps} \d x \le C.
\end{align*}

Plugging the previous estimates into \eqref{eq:int-reg-appl}, we deduce \eqref{eq:claim-thm}, as desired. 
\end{proof}

\subsection{An expansion at boundary points of order larger than $2s$}
\label{subsec:expansion-bigger-2s}

Throughout this subsection we assume that $B_0 > 2s$ and we prove an expansion of the solutions to the nonlocal Neumann problem of order larger than $2s$, thereby complementing the result in \autoref{thm:bdry-expansion}. Let us recall from \autoref{prop:zeros-final} that $B_0 < 2s + \frac{1}{2}$ when $s \in (0,\frac{1}{2}]$ and $B_0 < s+1$ when $s \in [\frac{1}{2},1)$. In particular, 
for every~$ s \in (0,1)$,
\begin{align}
\label{eq:B0-trivial-bound}
B_0 < s + 1 < 2 .
\end{align}

The proof in this higher regularity regime is significantly more involved since the blow-up argument leads to functions that grow too fast for the nonlocal operators to be evaluated in a classical sense. We solve this problem by a two-step procedure. First, we take incremental quotients of the blow-up sequence in the tangential directions (which exhibit slower growth at infinity) and pass them to the limit to deduce that the blow-up limit must be a 1D function, depending only on the normal direction. Using this information, in a second step we justify that the blow-up limit is a solution to the 1D nonlocal Neumann problem in the sense of \autoref{def:faster-growth}, which allows us to conclude the proof by an application of \autoref{thm:Neumann-Liouville-growth}.

Moreover, we first prove the expansion in case $\Omega \subset \R^n$ is bounded (see \autoref{thm:bdry-expansion-higher}). The result for unbounded sets follows by a combination of \autoref{thm:bdry-expansion-higher} with \autoref{lemma:unbd-error-2}.

\begin{theorem}
\label{thm:bdry-expansion-higher}
Let $s \in (0,1)$, $\beta \in (0,\min\{1,2s\})$, and $\gamma \in (s,1)$ be such that $\beta+\gamma \in (2s , B_0)$ and $\beta + \gamma \not = 1$. 
Additionally,
when $\beta + \gamma > 1$
we assume that $\gamma < 2s$. 

Let $\Omega \subset \R^n$ be a bounded open set with $\partial \Omega \in C^{1,1}$ with $0 \in \partial \Omega$ and $\nu(0) = e_n$. 

Let $u$ be a weak solution in the sense of \autoref{def:weak-sol-Neumann} to
\begin{align*}
\begin{cases}
(-\Delta)^s u &= h ~~\qquad \text{ in } \Omega \cap B_{1/2},\\
\cN_{\Omega}^s u &= a + b u ~~ \text{ in } B_{1/2} \setminus \Omega,
\end{cases}
\end{align*}
and $h \in C^{\beta + \gamma - 2s}(\Omega \cap B_{1/2})$, $a,b \in C^{2}(B_{1/2} \setminus \Omega)$. 

Then, there exist $C > 0$, and $b \in \R^n$ with $b \cdot e_n = 0$, such that it holds for any $r \in (0,\frac{1}{4}]$,
\begin{align*}
[u - u(0) - \1_{\{ \beta + \gamma > 1 \}} P_b]_{C^{\gamma}(B_{r})} \le C r^{\beta} \Big( \Vert u \Vert_{L^{\infty}(\R^n)} + \Vert u \Vert_{C^{\gamma}(B_{1/2})} &+ \Vert h \Vert_{C^{\beta + \gamma - 2s}(\Omega \cap B_{1/2})} \\
&+ \Vert a \Vert_{C^{2}(B_{1/2} \setminus \Omega)} \Big),
\end{align*}
where $B_0$ is as in \eqref{eq:B0} and $P_b$ is as in \autoref{subsec:correction} satisfying \eqref{eq:Pb-integral-vanish} and \eqref{eq:Pb-complement-rep} in $\Omega^c$. 

Moreover, $C > 0$ depends only on $n,s,\beta,\gamma,C_0$, an upper bound of $\Vert b \Vert_{C^{2}(B_{1/2} \setminus \Omega)}$, and the $C^{1,1}$ norm of $\Omega$.
\end{theorem}

Note that, differently than \autoref{thm:bdry-expansion}, here we prove an expansion inside and outside $\Omega$ at the same time. However, we also assume implicitly that $u \in C^{\gamma}(B_1)$. We will verify this assumption by an application of \autoref{lemma:outside-bound}.

Before we prove \autoref{thm:bdry-expansion-higher}, we need to state the following three technical lemmas. The first one says that classical solutions and weak solutions to $L_{\{ x_n > 0 \}} u = 0$ coincide in case they possess sufficient regularity up to the boundary. The second one shows that if $u$ is a solution to the nonlocal Neumann problem in the half-space up to a constant, then its tangential increments are strong solutions. Finally, the third one is reminiscent of \cite[Lemma 5.11]{AFR23} and reduces an $n$-dimensional Neumann problem in the half-space to a 1D problem if the solution is 1D.

\begin{lemma}
\label{lemma:strong-weak}
Let $\Omega \subset \R^n$ be a bounded open set such that $\partial \Omega \in C^{1,\eps}$ for some $\eps > 0$ and let $B \subset \R^n$ be compact. Let $u \in C^{2s+\eps}_{loc}(\Omega \cap B) \cap C^{\gamma}(\bar{\Omega} \cap B)$ for some $\eps > 0$ and $\gamma \in (s , 1]$ be such that 
\begin{align}
\label{eq:growth-ass-stong-weak}
\Vert u \Vert_{L^{\infty}(\Omega \cap B_R)} &\le C (1 + R)^{2s-\eps} ~~ \forall R \ge 1,\\
\label{eq:bdry-ass-strong-weak}
[u]_{C^{2s+\eps}(B_r(x_0))} &\le C r^{\eps - 1} \qquad\qquad \forall x_0 \in \Omega \cap B, ~~ r = \min\{d_{\Omega}(x_0)/2 , 1\}.
\end{align}

Let $h \in L^{\infty}(\Omega \cap B)$. Then, $u$ is a classical solution to
\begin{align*}
L_{\Omega} u &= h ~~ \text{ in } \Omega \cap B
\end{align*}
if and only if it is a weak solution to $L_{\Omega} u = h$ in $\Omega \cap B$ with Neumann boundary conditions on $\partial \Omega \cap B$ in the sense of \autoref{def:weak-sol}.
\end{lemma}

\begin{proof}
We claim that there exists~$\delta > 0$ such that, for any $R > 1$, 
\begin{align}
\label{eq:strong-weak-claim}
|L_{\Omega} u(x)| \le C d_{\Omega}^{-1 + \delta}(x) ~~ \forall x \in \Omega \cap B.
\end{align}
This follows by refining the computation in \cite[equation~(6.4)]{AFR23}. Indeed, for $x \in \Omega \cap B$ and $r = \min\{ d_{\Omega}(x)/2 , 1\}$ we have that
\begin{align*}
|L_{\Omega} u(x)| &\le C[u]_{C^{2s+\eps}(B_r(x))} + \int_{B_r(x)} |u(x) - u(y)| k_{\Omega}(x,y) \d y + \int_{\Omega \setminus B_r(x)} |u(x) - u(y)| K_{\Omega}(x,y) \d y  \\
&=: I_1 + I_2 + I_3.
\end{align*}
To estimate $I_2$, we observe that if~$y \in B_r(x)$ then~$k_{\Omega}(x,y) \le C d_{\Omega}^{-n-2s}(x)$ and accordingly
\begin{align*}
I_2 \le C [u]_{C^{\gamma}(B_r(x))} r^{-2s+\gamma}.
\end{align*}
For $I_3$, we remark that
\begin{align}
\label{eq:K-bound}
K_{\Omega}(x,y) \le C \left(1 + \log_- \left( \frac{\min\{d_{\Omega}(x) , d_{\Omega}(y) \}}{|x-y|} \right) \right) |x-y|^{-n-2s} ~~ \forall x,y \in \Omega.
\end{align}
Hence, recalling \eqref{eq:growth-ass-stong-weak},
\begin{align*}
I_3 &\le \int_{(\Omega \cap B_1(x)) \setminus B_r(x)} |u(x) - u(y)| K_{\Omega}(x,y) \d y \\
&\quad + C \int_{\Omega \setminus B_1(x)} |u(x)| K_{\Omega}(x,y) \d y + C \int_{\Omega \setminus B_1(x)} K_{\Omega}(x,y) \d y \\
&\le C [u]_{C^{\gamma}(\overline{\Omega} \cap B_1(x))} \log(r) r^{-2s+\gamma} + C \log(r) (1 + |u(x)|). 
\end{align*}

As a result, using \eqref{eq:bdry-ass-strong-weak}, we get
\begin{align*}
|L_{\Omega} u(x)| &\le C \big( [u]_{C^{2s+\eps}(B_r(x))} + [u]_{C^{\gamma}(B_1(x))} (1 + |\log(r)|) r^{-2s+\gamma} + (1 + |\log(r)|) (1 + |u(x)|) \big)\\
&\le C(R) \big( r^{\eps - 1} + (1 + |\log(r)|) r^{-2s+\gamma} \big),
\end{align*}
thereby proving \eqref{eq:strong-weak-claim} (since $\gamma > \max\{2s-1,0\}$), as claimed.

Let us now take $\eta \in C^{\infty}_c(B)$. Then, proceeding as in \cite[equation~(6.20)]{AFR23}, i.e. integrating over $\{ (x,y) \in \Omega \times \Omega : |x-y| \ge \eps \}$ and using again \eqref{eq:growth-ass-stong-weak} to deduce an estimate analogous to \cite[equation~(6.21)]{AFR23},
after sending $\eps \searrow 0$ we obtain that
\begin{align}
\label{eq:weak-strong}
\int_{\Omega} \eta(x) L_{\Omega} u(x) \d x = \int_{\Omega} \int_{\Omega} (\eta(x) - \eta(y)) (u(x) - u(y)) K_{\Omega}(x,y) \d y \d x.
\end{align}
Moreover, both integrals converge absolutely, owing to \eqref{eq:strong-weak-claim}. Since $\gamma > s$, we also have that~$u \in H_K(\Omega \cap B)$ for any compact set $B \subset \R^n$. Indeed, by \eqref{eq:K-bound}, we deduce that
\begin{align*}
& \int_{\Omega \cap B} \int_{\Omega \cap B} |u(x) - u(y)| K_{\Omega}(x,y) \d y \d x \\
&\quad \le C [u]_{C^{\gamma}(\Omega \cap B)}^2 \int_{\Omega \cap B} \int_{\Omega \cap B} |x-y|^{-n-2s+2\gamma} \left(1 + \log_- \left( \frac{\min\{d_{\Omega}(x) , d_{\Omega}(x) \}}{|x-y|} \right) \right) \d y \d x \le C(B) < \infty.
\end{align*}

Thus, if $u$ is a classical solution, then it is also weak solution in the sense of \autoref{def:weak-sol}, since $u \in H_K(\Omega \cap B)$ and the double integral in \eqref{eq:weak-strong} equals to $\int_{\Omega \cap B} \eta h \d x$ by assumption. Furthermore, if $u$ is a weak solution, then the left-hand side in \eqref{eq:weak-strong} equals to $\int_{\Omega \cap B} \eta h \d x$, and therefore the desired result follows by density.
\end{proof}

\begin{lemma}
\label{lemma:incr-sol}
Let $u \in C^{2s+\eps}_{loc}(\{ x_n > 0 \})$ for some $\eps > 0$ be a  solution to
\begin{align*}
(-\Delta)^s u \overset{1}{=} 0 ~~ \text{ in } \{ x_n > 0 \}
,
\end{align*}
in the sense of \autoref{def:up-to-poly} with 
\begin{align}
\label{eq:u-uh-growth}
\Vert u^{(h)} \Vert_{L^{\infty}(B_R)} \le C (1 + R)^{2s-\eps}, \qquad \Vert u \Vert_{L^{\infty}(B_R)} \le C (1 + R)^{2s+1-\eps} ~~ \forall R \ge 2,
\end{align}
where $h \in B_1$ and $u^{(h)} = u(\cdot + h) - u(\cdot)$.

Then, $u^{(h)}$ is a classical solution to
\begin{align*}
(-\Delta)^s u^{(h)} = 0 ~~ \text{ in } \{ x_n > (h_n)_- \}.
\end{align*}
\end{lemma}

\begin{proof}
Let $h \in B_1$. By assumption, $u^{(h)} \in C^{2s+\eps}_{loc}(\{ x_n > 0 \}) \cap L^1_{2s}(\R^n)$ and  therefore $(-\Delta)^s u^{(h)}(x)$ is well-defined for any $x \in \{ x_n > (h_n)_- \}$.

Moreover, for any $R > 4$ and $x \in B_{R/2} \cap \{ x_n > 0 \}$,
\begin{align*}
(-\Delta)^s (u \1_{B_R})(x) = E_R(x) + c_R ~~ \forall x \in \{ x_n > (h_n)_- \},
\end{align*}
where $E_R \to 0$ as $R \to \infty$ locally uniformly in $\{ x_n > (h_n)_- \}$ and $c_R \in \R$. 

Hence,
\begin{align}
\label{eq:incr-quot-help}
\begin{split}
(-\Delta)^s u^{(h)}(x) &= (-\Delta)^s (u^{(h)} \1_{B_R}) (x) + (-\Delta)^s (u^{(h)} \1_{B_R^c}) (x)  \\
&= (-\Delta)^s ((u\1_{B_R})^{(h)})(x) - (-\Delta)^s (u(\cdot + h) \1^{(h)}_{B_R}) (x)  + (-\Delta)^s (u^{(h)} \1_{B_R^c}) (x) \\
&=: I_1^{R}(x) + I_2^R(x) + I_3^R(x).
\end{split}
\end{align}
In particular, one finds that $I_1^R = ((-\Delta)^s (u\1_{B_R}))^{(h)} = E_R^{(h)} \to 0$ as $R \to \infty$ locally uniformly in $x \in \{ x_n > (h_n)_- \}$. 

Additionally, by the growth assumption on $u^{(h)}$ in \eqref{eq:u-uh-growth}, for any $x \in B_{R/2}$
we have that
\begin{align*}
|I_3^R(x)| \le C \int_{B_R^c} |u^{(h)}(y)| |x-y|^{-n-2s} \d y \le C \int_{B_R^c} |y|^{-n-\eps} \d y \le C R^{-\eps} \to 0 ~~ \text{ as } R \to \infty.
\end{align*}
Consequently, $I_3^R \to 0$ locally uniformly in $\{ x_n > 0 \}$. 

To deal with $I_2^R$, we observe that $\supp(\1_{B_R}^{(h)}) \subset B_{R+|h|} \setminus B_{R - |h|}$, and thus, using also the growth assumption on $u$ from \eqref{eq:u-uh-growth}, we deduce that,
when $x \in B_{R/2} \cap \{ x_n > (h_n)_- \}$,
\begin{align*}
|I_2^R(x)| &\le C \int_{B_{R + |h|} \setminus B_{R - |h|}} |u(y+h)| |x-y|^{-n-2s} \d y \\
&\le C |B_{R + |h|} \setminus B_{R - |h|}| R^{2s+1-\eps} R^{-n-2s} \le C |h| R^{n-1} R^{1-\eps-n} \le C |h| R^{-\eps} \to 0 ~~ \text{ as } R \to \infty,
\end{align*}
yielding that $I_3^R \to 0$ locally uniformly in $\{ x_n > (h_n)_- \}$.

As a result, we deduce that $(-\Delta)^s u^{(h)} = 0$ in $\{ x_n > (h_n)_- \}$, as desired.
\end{proof}

\begin{lemma}
\label{lemma:incr-sol-Neumann}
Let $\eps > 0$ and $u$ be a solution to
\begin{align*}
\cN_{\{ x_n > 0 \}}^s u &\overset{1}{=} 0 ~~ \text{ in } \{ x_n < 0 \}
\end{align*}
in the sense of \autoref{def:up-to-poly}, with 
\begin{align*}
\Vert u^{(h)} \Vert_{L^{\infty}(B_R \cap \{ x_n > 0 \})} \le C (1 + R)^{2s-\eps}, \qquad \Vert u \Vert_{L^{\infty}(B_R \cap \{ x_n > - (h_n)_- \})} \le C (1 + R)^{2s+1-\eps} ~~ \forall R \ge 2,
\end{align*}
where $h \in B_1$, and $u^{(h)} = u(\cdot + h) - u(\cdot)$. 

Then, $u^{(h)}$ is a solution to
\begin{align*}
\cN_{\{ x_n > 0 \}}^s u^{(h)} &= \1_{\{h_n \not= 0\}}f_h ~~ \text{ in } \{ x_n < -(h_n)_+ \},
\end{align*}
where
\begin{align*}
f_h(x) = - c_{n,s} \int_{0}^{h_n} \int_{\R^{n-1}} (u(x+h) - u(y)) |x + h -y|^{-n-2s} \d y'\d y_n.
\end{align*}
\end{lemma}

\begin{proof}
In case $h_n = 0$ the proof goes exactly as in \autoref{lemma:incr-sol}, using crucially that the property $h_n = 0$ keeps the integration domain $\{ x_n > 0 \}$ invariant under the shift $x \mapsto x + h$.

In case $h_n < 0$, we compute, in analogy to \eqref{eq:incr-quot-help},
\begin{align*}
\cN_{\{ x_n > 0 \}}^s u^{(h)}(x) & = \cN_{\{ x_n > 0 \}}^s ((u\1_{B_R})^{(h)})(x) - \cN_{\{ x_n > 0 \}}^s (u(\cdot + h) \1^{(h)}_{B_R}) (x)  + \cN_{\{ x_n > 0 \}}^s (u^{(h)} \1_{B_R^c}) (x) \\
&=: \bar{I}_1^{R}(x) + \bar{I}_2^R(x) + \bar{I}_3^R(x).
\end{align*}
The quantities in $\bar{I}_2^R, \bar{I}_3^R$ satisfy the exact same bounds since the $L^{\infty}$-bound for $u$ from the assumption holds in $\{ x_n > - (h_n)_- \}$, but the integration takes place only over $B_R^c \cap \{ y_n > 0 \}$. 

In contrast, for $\bar{I}_1^R$, we have that
\begin{align}
\label{eq:Neumann-fh-term}
\begin{split}
\bar{I}_1^R(x) &= \cN_{\{ x_n > 0 \}}^s ((u\1_{B_R})^{(h)})(x) \\
&=  (\cN_{\{ x_n > 0 \}}^s (u\1_{B_R}))^{(h)}(x) \\
&\quad + c_{n,s} \int_{0}^{h_n} \int_{\R^{n-1}} (u(x+h) - (u\1_{B_R})(y)) |x + h -y|^{-n-2s} \d y'\d y_n,
\end{split}
\end{align}
and the desired result follows by sending $R \to \infty$.
\end{proof}

\begin{lemma}
\label{lemma:1D-reduction}
Let $u$ be a solution to
\begin{align*}
\begin{cases}
(-\Delta)^s u &\overset{1}{=} 0 ~~ \text{ in } \{ x_n > 0 \},\\
\cN_{\{ x_n  > 0 \}} u &\overset{1}{=} 0 ~~ \text{ in } \{ x_n < 0 \}
\end{cases}
\end{align*}
in the sense of \autoref{def:up-to-poly}, satisfying for some $\eps > 0$,
\begin{align}
\label{eq:growth-ass-1D-reduction}
\Vert u \Vert_{L^{\infty}(B_R)} &\le C (1 + R)^{2s+1-\eps}, \qquad \Vert u^{(h)} \Vert_{L^{\infty}(B_R)} \le C (1 + R)^{2s-\eps} ~~ \forall R \ge 1,\\
\label{eq:bdry-ass-1D-reduction}
[u]_{C^{2s+\eps}(B_r(x_0))} &\le C(R) r^{\eps - 1} \qquad \forall x_0 \in \Omega \cap B_R, ~~ r = \min\{|(x_0)_n|/2 , 1\}.
\end{align}
Moreover, assume that there exists a function $U : \R \to \R$ such that $u(x) = U(x_n)$ for any $x \in \R^n$. 

Then, for any $\eta \in C_{c}^{\infty}(\R)$ with $\int_{-\infty}^0 \eta(t) \d t = \int_{\R} \eta(t) \d t = 0$,
\begin{align*}
\int_{0}^{\infty} (-\Delta)^s \eta (t) U(t) \d t + \int_{-\infty}^{0} \cN^s_{(0,\infty)} \eta(t) U(t) \d t = 0. 
\end{align*}
\end{lemma}

\begin{proof}
We deduce from \autoref{lemma:incr-sol} that, for $h = t e_n$, 
\begin{align*}
(-\Delta)^s u^{(h)} = 0 ~~ \text{ in } \{ x_n > t_- \}.
\end{align*}
Since $u^{(h)}(x) = U^{(t)}(x_n)$, we deduce from \cite[Lemma B.1.5]{FeRo24} that 
\begin{align*}
(-\Delta)^s_{\R} U^{(t)} = 0 ~~ \text{ in } (t_- , \infty).
\end{align*}
Let now $R > 2$ and $x \in (0,\frac{R}{2})$. Then, as in \eqref{eq:incr-quot-help},
we compute that, for any $t \in (-x,\infty)$,
\begin{align}
\label{eq:error-finding}
\begin{split}
((-\Delta)^s_{\R} (U \1_{B_R}))^{(t)}(x) &= (-\Delta)^s_{\R} ((U \1_{B_R})^{(t)})(x) \\
&= I_2^{R,t}(x) + I_3^{R,t}(x) + (-\Delta)^s_{\R} U^{(t)}(x) = I_2^{R,t}(x) + I_3^{R,t}(x),
\end{split}
\end{align}
where $I_2^{R,t}$ and $I_3^{R,t}$ are as in \eqref{eq:incr-quot-help} and satisfy, for any $|t| < 1$,
\begin{align}
\label{eq:error-to-zero}
|I_2^{R,t}(x)| \le C |t| R^{-\eps}, \qquad |I_3^{R,t}(x)| \le C  R^{-\eps}.
\end{align}
Hence, given any $M \ge 1$, whenever $x \in (0,M)$ and $R \ge 2M$ we have that
\begin{align}
\label{eq:Laplacian-up-to-poly}
(-\Delta)^s_{\R} (U \1_{B_R})(x) = E_R(x) + c_R,
\end{align}
where $c_R = -(-\Delta)^s_{\R} (U \1_{B_R})(M) \in \R$, and by \eqref{eq:error-to-zero},
\begin{align*}
E_R(x) &:= -\sum_{k = 0}^{\lfloor M - x \rfloor } \big( I_2^{R,1}(x+k) + I_3^{R,1}(x+k) \big) \\
&\quad - \big( I_2^{R,M - x - \lfloor M - x \rfloor}(x + \lfloor M - x \rfloor) + I_3^{R,M - x - \lfloor M - x \rfloor}(x + \lfloor M - x \rfloor) \big) \to 0 ~~ \text{ as } R \to \infty,
\end{align*}
uniformly for $x \in (0,M)$.
Analogously, we show that for $x \in (-M,0)$, as long as $R \ge 2M$, 
\begin{align}
\label{eq:Neumann-up-to-poly}
\cN^s_{(0,\infty)} (U\1_{B_R})(x) = F_R(x) + d_R,
\end{align}
where $d_R = -\cN^s_{(0,\infty)} (U \1_{B_R})(-M) \in \R$ and $F_R(x) \to 0$ as $R \to \infty$, uniformly for $x \in (-M,0)$.

This result follows by application of \autoref{lemma:incr-sol-Neumann}, observing that 
\begin{align*}
\cN^s_{\{ x_n > 0 \}} u^{(h)} = f_h ~~ \text{ in } \{ x_n < - t_+ \},
\end{align*}
where, since $u$ is one-dimensional,
\begin{align*}
f_h(x) &= - c_{n,s}\int_{0}^{h_n} \int_{\R^{n-1}} (u(x+h) - u(y)) |x + h -y|^{-n-2s} \d y'\d y_n \\
&= - c_{n,s}\int_0^{h_n} (U(x_n + t) - U(y_n)) \int_{\R^{n-1}} |x + t e_n - y|^{-n-2s} \d y' \d y_n \\
&= -C_1 c_{1,s} \int_0^{h_n} (U(x_n + t) - U(y_n)) |x_n + t - y_n|^{-1-2s} \d y_n =: C_1 F_t(x_n),
\end{align*}
as well as
\begin{align*}
\cN^s_{\{ x_n > 0 \}} u^{(h)}(x) = c_{n,s} \int_0^{\infty} (U^{(t)}(x_n) - U^{(t)}(y_n)) \int_{\R^{n-1}} |x-y| ^{-n-2s} \d y' \d y_n  = C_1 \cN^s_{(0,\infty)} U^{(t)}(x_n)
\end{align*}
for some $C_1 = C(n,s) > 0$. 

This yields that
\begin{align*}
\cN^s_{(0,\infty)} U^{(t)} = F_t ~~ \text{ in } (-\infty , -t_+).
\end{align*}
Then, redoing the computation in \eqref{eq:error-finding} for $\cN^s_{(0,\infty)}$ and the computation in \eqref{eq:Neumann-fh-term}, we observe that the term $F_t$ cancels, namely
\begin{align*}
(\cN^s_{(0,\infty)} (U \1_{B_R}))^{(t)}(x) &= \cN^s_{(0,\infty)} ((U \1_{B_R})^{(t)})(x) - F_t^R(x) \\
&= I_2^{R,t}(x) + I_3^{R,t}(x) + \cN^s_{(0,\infty)} U^{(t)}(x) - F_t^R(x) = I_2^{R,t}(x) + I_3^{R,t}(x) + (F_t - F_t^R)(x),
\end{align*}
where
\begin{align*}
F_t^R(x) = c_{1,s} \int_0^{h_n} (U(x_n + t) - (U\1_{B_R})(y_n)) |x_n + t - y_n|^{-1-2s} \d y_n \to F_t(x) ~~ \text{ as } R \to \infty,
\end{align*}
uniformly in $(-M,0)$. 

From here, \eqref{eq:Neumann-up-to-poly} follows in the same way as \eqref{eq:Laplacian-up-to-poly}.

Thus, \eqref{eq:Laplacian-up-to-poly} and \eqref{eq:Neumann-up-to-poly} yield that
\begin{align*}
\begin{cases}
(-\Delta)^s_{\R} U &\overset{1}{=} 0 ~~ \text{ in } (0,\infty),\\
\cN^s_{(0,\infty)} U &\overset{1}{=} 0 ~~ \text{ in } (-\infty,0).
\end{cases}
\end{align*}

To conclude the proof, we multiply \eqref{eq:Laplacian-up-to-poly} and \eqref{eq:Neumann-up-to-poly} by $\eta$ as in the statement of the theorem, with $\supp(\eta) \subset (-M,M)$, integrate them over $(0,\infty)$ and $(-\infty,0)$, respectively, and add up the resulting identities. 

Then, observing that~$ c_R \int_0^{\infty} \eta(x) \d x = d_R \int_{-\infty}^{0} \eta(x) \d x = 0$,
we obtain, for all~$R > 2M$,
\begin{align*}
\int_0^{\infty} (-\Delta)^s_{\R} \eta(x) U(x) \1_{B_R}(x) \d x &+ \int_{-\infty}^0 \cN_{(0,\infty)}^s \eta(x) U(x) \1_{B_R}(x) \d x \\
&= \int_0^{\infty} \eta(x) (-\Delta)^s_{\R} (U \1_{B_R})(x) \d x + \int_{-\infty}^0 \eta(x) \cN_{(0,\infty)}^s  (U \1_{B_R})(x) \d x \\
&= \int_0^{\infty} \eta(x) E_R(x) \d x + \int_{-\infty}^0 \eta(x) F_R(x) \d x,
\end{align*}

As a technical observation, we note that in the first equality here above
we applied the nonlocal integration by parts formula (see \cite[Lemma 3.3]{DRV}), whose proof can easily be extended to the function $\eta$ and $U \1_{B_R}$, using that $(-\Delta)^s_{\R} (U \1_{B_R}) \in L^1_{loc}([0,R/2])$ and $\cN^s_{(0,\infty)}(U \1_{B_R}) \in L^1_{loc}([-R/2,0])$ as a consequence of \eqref{eq:bdry-ass-1D-reduction}. 

Next, we use the fact that $\eta \in C^{\infty}_c(\R)$ and \autoref{lemma:Neumann-L-growth} to
check the assumptions of the dominated convergence theorem, taking the limit $R \to \infty$ on both sides, and obtain the desired result. 
\end{proof}

\begin{proof}[Proof of \autoref{thm:bdry-expansion-higher}]
We restrict ourselves to proving the result in case $\beta + \gamma > 1$, since this is the most difficult case. The case $\beta + \gamma < 1$ goes in the same way, but the functions $P_b$ do not need to be subtracted, which simplifies some of the arguments. In particular, in this case the analogs of \autoref{lemma:aux-theta} and \autoref{lemma:theta-increasing} do not require the assumption $\gamma < 2s$ and they can be found in \cite{AbRo20}.

\textbf{Step 1:} We assume that $\Vert u \Vert_{L^{\infty}(\R^n)} + \Vert u \Vert_{C^{\gamma}(B_{1/2})} + \Vert h \Vert_{C^{\beta + \gamma - 2s}(\Omega \cap B_{1/2})} + \Vert a \Vert_{C^{2}(B_{1/2} \setminus \Omega)} \le 1$. Hence, it suffices to show
\begin{align}
\label{eq:blowup-claim-growth}
[u - u(0) - P_b]_{C^{\gamma}( B_{r})} \le C r^{\beta} ~~ \forall r \in (0,1/4].
\end{align}

Let us assume by contradiction that \eqref{eq:blowup-claim-growth} does not hold true. In that case there are sequences $(\Omega_k)$ as in the statement of the lemma, as well as $(u_k) \subset V^s(\Omega_k | \R^n)$, and $(h_k) \subset C^{\beta + \gamma -2s}(\Omega \cap B_{1/2})$, $(a_k), (b_k) \in C^2(B_{1/2} \setminus \Omega)$, such that
\begin{align}
\label{eq:blowup-normalization-growth}
\begin{split}
\Vert u_k \Vert_{L^{\infty}(\R^n)} + \Vert u_k \Vert_{C^{\gamma}(B_{1/2})} + \Vert h_k \Vert_{C^{\beta + \gamma - 2s}(\Omega_k \cap B_{1/2})} + \Vert a_k \Vert_{C^{2}(\cap B_{1/2} \setminus \Omega_k)} &\le 1,\\
 \Vert b_k \Vert_{C^{2}(\cap B_{1/2} \setminus \Omega_k)} &\le M
\end{split}
\end{align}
for some $M > 0$ and 
\begin{align*}
\begin{cases}
(-\Delta)^s u_k &= h_k ~~~~~\qquad \text{ in } \Omega_k \cap B_{1/2},\\
\cN^s_{\Omega_k} u_k &= a_k + u_k b_k ~~ \text{ in } B_{1/2} \setminus \Omega_k
\end{cases}
\end{align*}
in the weak sense. 

Moreover, one can find $C_k \to +\infty$ such that 
\begin{align*}
\inf_{b \in \{  b_n = 0 \}} \sup_{r \in [0,\frac{1}{4}]} r^{-\beta} [ u_k - u_k(0) - P_b^k ]_{C^{\gamma}(B_r)} \ge C_k,
\end{align*}
where we denote by $P_b^k$ the function from \autoref{subsec:correction} with respect to $\Omega_k$ satisfying \eqref{eq:Pb-complement-rep} in $\R^n \setminus \Omega_k$ and \eqref{eq:Pb-integral-vanish}. 

Note that the space $\cP_{k,r} := \{ P_b^k : b \in \{ b_n = 0 \} \}$ is a linear subspace of $L^2(B_r)$, thanks to \autoref{remark:Pb-subspace}.
Then, for any $k \in \N$ and $r \in (0,\frac{1}{4}]$ we consider the $L^2( B_{r})$ projections of $u_k(\cdot) - u_k(0)$ over the space. We denote these projections by $P_{k,r}$ and extend them to $B_r^c$, in the same way as in \autoref{subsec:correction}, such that \autoref{lemma:LPb-estimate-2} is applicable. 

Thus, in analogy to the proof of \autoref{thm:bdry-expansion-higher} we have
\begin{align*}
\Vert u_k - u_k(0) - P_{k,r} \Vert_{L^2(B_{r})} &\le \Vert  u_k - u_k(0) - P \Vert_{L^2(B_{r})} ~~ \forall P \in \cP_{k,r},\\
\int_{B_{r}} (u_k(x) - u_k(0) - P_{k,r}(x)) P(x) \d x &= 0 ~~ \forall P \in \cP_{k,r},
\end{align*}
and it follows from \autoref{lemma:theta-increasing} that the function 
\begin{align*}
\theta(r) := \sup_{k \in \N} \sup_{\rho \in [r , \frac{1}{4}]} \rho^{-\beta} [ u_k - u_k(0) - P_{k,\rho} ]_{C^{\gamma}(B_{\rho})}
\end{align*} 
satisfies 
\begin{align*}
\theta(r) \nearrow +\infty ~~ \text{ as } ~~ r \searrow 0.
\end{align*}As a result, there exist subsequences $(k_j) \subset \N$ and $r_j \searrow 0$ with $r_j \in (0,\frac{1}{8}]$ such that
\begin{align*}
\frac{[ u_{k_j} - u_{k_j}(0) - P_{k_j,r_j} ]_{C^{\gamma}(B_{r_j})}}{r_j^{\beta} \theta(r_j)} \ge \frac{1}{2} ~~ \forall j \in \N.
\end{align*}

\textbf{Step 2:} Next, we introduce the blow-up sequence
\begin{align*}
v_j(x):= \frac{u_{k_j}(r_j x) - u_{k_j}(0) - P_{k_j,r_j}(r_j x)}{r_j^{\beta + \gamma} \theta(r_j)}, ~~ j \in \N,
\end{align*}
and observe that, for all $j \in \N$,
\begin{align}
\label{eq:vj-properties-2}
v_j(0) = 0, \qquad [ v_j ]_{C^{\gamma}(B_1)} \ge \frac{1}{2}, \qquad \int_{ B_1} v_j(x) r_j^{-1} P(r_j x) \d x = 0 ~~ \forall P \in \cP_{k_j,r_j}.
\end{align}
We denote $\tilde{\Omega}_j = r_{_j}^{-1} \Omega_{k_j}$. 
We claim that
\begin{align}
\label{eq:vj-growth-2}
[ v_j ]_{C^{\gamma}(B_R)} \le C R^{\beta}, \qquad \Vert v_j \Vert_{L^{\infty}( B_R)} \le C R^{\beta + \gamma} ~~ \forall R \in [1 , r_j^{-1}/4 ],
\end{align}
and
\begin{align}
\label{eq:b-coeff-convergence-2}
\frac{|b_{k_j,r_j}|}{\theta(r_j)} \to 0 ~~ \text{ as } j \to \infty.
\end{align}
The claims in \eqref{eq:vj-growth-2} as well as \eqref{eq:b-coeff-convergence-2} follow from \autoref{lemma:aux-theta}, applied with $u := u_{k_j}- u_{k_j}(0)$, exactly as in the proof of \autoref{thm:bdry-expansion}. Note that the assumptions of \autoref{lemma:aux-theta} are satisfied, since by the definition of $\theta(r)$ and by \eqref{eq:blowup-normalization-growth}, 
for any  $r \le \frac{1}{4}$ one has that
\begin{align*}
[ u_{k_j} - u_{k_j}(0) - P_{k_j,r_j} ]_{C^{\gamma}(B_{r})} \le \theta(r) r^{\beta}, \qquad [ u_{k_j} - u_{k_j}(0) ]_{C^{\gamma}(B_{1/4})} \le 1.
\end{align*}
The second claim in \eqref{eq:vj-growth-2} follows from the first claim in \eqref{eq:vj-growth-2}, which also yields $\Vert v_j \Vert_{L^{\infty}(B_1)} \le C$. 

Next, we observe that $v_j$ satisfies in the weak sense
\begin{align}
\label{eq:vj-Neumann-PDE}
\begin{cases}
(-\Delta)^s v_j &= \tilde{h}_j ~~~~~\qquad \text{ in } \tilde{\Omega}_j \cap B_{r_j^{-1}/2},\\
\cN^s_{\tilde{\Omega}_j} v_j &= \tilde{a}_j + v_j \tilde{b_j} ~~~~ \text{ in } B_{r_j^{-1}/2} \setminus \tilde{\Omega}_j,
\end{cases}
\end{align}
where
\begin{align*}
\tilde{h}_j(x) &= \frac{r_j^{2s - \beta - \gamma}}{\theta(r_j)} h_{j}(r_j x) - \frac{r_j^{2s - \beta - \gamma}}{\theta(r_j)} (L_{\Omega_j} P_{k_j,r_j})(r_j x), \\
\tilde{a}_j(x) &= \frac{r_j^{2s - \beta - \gamma}}{\theta(r_j)}  \big[ a_{k_j}(r_j x) + b_{k_j}(r_j x)( u_{k_j}(0) + P_{k_j,r_j}(r_j x)) \big], \\
\tilde{b}_j(x) &= r_j^{2s} b_{k_j}(r_j x).
\end{align*}

Furthermore, 
$1 + s - \beta - \gamma > 1 + s - B_0 > 0$, due to \eqref{eq:B0-trivial-bound}.
Hence,
by recalling \autoref{lemma:LPb-estimate-2} (applied with $\alpha = 1$, which yields $\frac{1+\alpha-2s}{1+\alpha} = 1-s$), and employing also~\eqref{eq:b-coeff-convergence-2}
and \eqref{eq:blowup-normalization-growth}, we have that, for any $\rho > 0$, 
\begin{align}
\label{eq:f-Holder-to-zero}
\begin{split}
[\tilde{h}_j]_{C^{\beta + \gamma - 2s}(\tilde{\Omega}_j \cap B_{R} \cap \{ d_{\tilde{\Omega}_j} \ge \rho \})} &\le \frac{[h_j]_{C^{\beta + \gamma - 2s}(\Omega_j \cap B_{r_j R})}} {\theta(r_j)} + \frac{[L_{\Omega_j} P_{k_j,r_j}]_{C^{\beta + \gamma - 2s}(\Omega_j \cap B_{r_j R} \cap \{ d_{\Omega_j} \ge r_j \rho \} )}}{\theta(r_j)}  \\
&\le \theta(r_j)^{-1} + C \frac{|b_{k_j,r_j}|}{\theta(r_j)} \left( 1 + (R r_j)^{1-s} (r_j \rho)^{2s -\beta - \gamma} \right) \\
&\le C(R) (1 + \rho^{2s-\beta - \gamma}) \frac{1 + |b_{k_j,r_j}|}{\theta(r_j)} \to 0.
\end{split}
\end{align}

Moreover, we notice immediately that for any $R > 0$, it holds
\begin{align}
\label{eq:bj}
\Vert \tilde{b_j} \Vert_{C^{2}(B_R \setminus \tilde{\Omega}_j)} \to 0.
\end{align}
Furthermore, we claim that there are there are $d_j \in \R$ such that for any $R > 0$,
\begin{align}
\label{eq:aj-dj}
\Vert \tilde{a_j} - d_j \Vert_{L^{\infty}(B_R \setminus \tilde{\Omega}_j)} \to 0.
\end{align}
Indeed, setting
\begin{align*}
d_j := \frac{r_j^{2s - \beta - \gamma}}{\theta(r_j)} \big[ a_{k_j}(0) + b_{k_j}(0)u_{k_j}(0) \big],
\end{align*}
we obtain, using that $P_{k_j,r_j}(0) = 0$ and that $\beta + \gamma -2s \in (0,1)$,
\begin{align*}
\Vert \tilde{a}_j - d_j \Vert_{L^{\infty}(B_R \setminus \tilde{\Omega}_j)} &\le \frac{r_j^{2s - \beta - \gamma}}{\theta(r_j)} \Vert a_{k_j}(r_j \cdot) - a_{k_j}(0) \Vert_{L^{\infty}(B_R \setminus \tilde{\Omega}_j)} \\
&\quad + \frac{r_j^{2s - \beta - \gamma}}{\theta(r_j)} \Vert b_{k_j}(r_j \cdot) - b_{k_j}(0) \Vert_{L^{\infty}(B_R \setminus \tilde{\Omega}_j)} |u_{k_j}(0)| \\
&\quad + \frac{r_j^{2s - \beta - \gamma}}{\theta(r_j)}  \Vert b_{k_j}(r_j \cdot) \Vert_{L^{\infty}(B_R \setminus \tilde{\Omega}_j)} \Vert P_{k_j,r_j}(r_j \cdot) - P_{k_j,r_j}(0) \Vert_{L^{\infty}(B_R \setminus \tilde{\Omega}_j)} \\
&\le C(R)  \theta(r_j)^{-1} \big( 1 + |b_{k_j,r_j}| \big) \to 0,
\end{align*}
where we also applied \eqref{eq:blowup-normalization-growth}, \eqref{eq:b-coeff-convergence-2} in the last step. This proves \eqref{eq:aj-dj}

Besides, we claim that $v_j$ is actually a strong solution to the equation in \eqref{eq:vj-Neumann-PDE}. To prove it, we observe that for any $\tau \in (\gamma , \gamma + \beta)$ and any 
 $R \in [2,r_j^{-1}/2]$ and $x_0 \in B_R \cap \{ x_n > 0 \}$,
\begin{align}
\label{eq:vj-2spluseps}
[v_j]_{C^{\tau}(B_{\rho_j}(x_0))} \le C(R) \rho_j^{\gamma - \tau} , \quad \text{ where } \quad  \rho_j = \min\{ d_{\tilde{\Omega}_j}(x_0)/4 , 1\},
\end{align}
as long as~$j$ is large enough.

To check \eqref{eq:vj-2spluseps}, it is useful to use the assumption $\Vert u_j \Vert_{L^{\infty}(\R^n)} \le 1$ in \eqref{eq:blowup-normalization-growth}
and infer that, for $|x| \ge r_j^{-1}/4$,
\begin{align*}
|v_j(x)| \le C r_j^{-\beta - \gamma} \theta(r_j)^{-1} \le C |x|^{\beta + \gamma}.
\end{align*}
We also recall~\eqref{eq:vj-growth-2} and~\eqref{eq:f-Holder-to-zero}.
In this way, for any $r \in (0,1]$, since $\gamma \le 1$,
\begin{align*}
\Vert v_j - v_j(x_0) \Vert_{L^{\infty}(B_r(x_0))} \le C r^{\gamma} [v_j]_{C^{\gamma}(B_r(x_0))} \le  C r^{\gamma} [v_j]_{C^{\gamma}(B_{2R})}\le C(R) r^{\gamma}
\end{align*}
and therefore (using~\eqref{eq:vj-growth-2} with~$|x| \le r_j^{-1}/4$),
\begin{align*}
\left\Vert \frac{v_j - v_j(x_0)}{|\cdot-x_0|^{\beta + \gamma}} \right\Vert_{L^{\infty}(B_{2\rho_j}(x_0)^c)} &\le \left\Vert \frac{v_j - v_j(x_0)}{|\cdot-x_0|^{\beta + \gamma}} \right\Vert_{L^{\infty}(B_1(x_0) \setminus B_{2\rho_j}(x_0))} + C \left\Vert \frac{1 + |\cdot|^{\beta + \gamma}}{|\cdot-x_0|^{\beta + \gamma}} \right\Vert_{L^{\infty}(B_{1}(x_0)^c)} \\
&\le C(R) \sup_{r \in [2\rho_j , 1]} r^{\gamma-(\beta + \gamma)} + C \le C(R) (1 + \rho_j^{- \beta}).
\end{align*}

Thanks to these observations, the desired result in~\eqref{eq:vj-2spluseps}
can now be deduced by an
application of the interior $C^{\beta + \gamma}$ estimate from \cite[Proposition 3.9]{AbRo20} for solutions with finite tails of order $\beta + \gamma \in (2s , 2s + 1)$ with $v_j - v_j(x_0)$, which yields 
\begin{align*}
[v_j]_{C^{\beta + \gamma}(B_{\rho_j}(x_0))} &\le C \rho_j^{-\beta - \gamma} \Vert v_j - v_j(x_0) \Vert_{L^{\infty}(B_{2\rho_j}(x_0))} + C\left\Vert \frac{v_j - v_j(x_0)}{|\cdot-x_0|^{\beta + \gamma}} \right\Vert_{L^{\infty}(B_{2\rho_j}(x_0)^c)} \\
&\quad + C [\tilde{h}_j]_{C^{\beta + \gamma - 2s}(B_{2\rho_j}(x_0))} \\
&\le C(R) (\rho_j^{-\beta} + 1 + \rho_j^{2s-\beta-\gamma}) \le C(R) \rho_j^{-\beta},
\end{align*}
as claimed in \eqref{eq:vj-2spluseps}, since $- \beta > -1$ and $2s-\beta-\gamma > -1$ (we can now find $\delta > 0$ small enough by interpolation, since $\beta + \gamma > 2s$).

Altogether, \eqref{eq:vj-2spluseps} shows that the first equation in \eqref{eq:vj-Neumann-PDE}  holds true in the classical sense after choosing $\tau > 2s$. Moreover, testing the weak formulation with test functions compactly supported in $\tilde{\Omega}_j^c$ implies that also the second equation in \eqref{eq:vj-Neumann-PDE} holds in the classical sense.

Moreover, we note that by following the proof of \autoref{lemma:outside-bound}, we can actually show the same regularity as in \eqref{eq:vj-2spluseps} for $x_0 \in B_R \setminus \tilde{\Omega}_j$, and thereby obtain the following more general version of \eqref{eq:vj-2spluseps}, also establishing regularity outside $\tilde{\Omega}_j$:
\begin{align}
\label{eq:extended-vj-est}
[v_j]_{C^{\tau}(B_{\rho_j}(x_0))} \le C(R) \rho_j^{\gamma - \tau} \quad \forall x_0 \in B_R \setminus \partial \tilde{\Omega}_j.
\end{align}
To see it, let us first define as in \autoref{lemma:outside-bound} the function $W_j(x) = \int_{\tilde{\Omega}_j} |x-y|^{-n-2s} \d y \ge C \rho_j^{-2s}$, and note that we get by the Neumann condition of $P_{k_j,r_j}$ that for $i = 1,2$ and $x_0 \in B_R \setminus \tilde{\Omega}_j$,
\begin{align*}
W_j(x_0) (D_x^{(i)} P_{k_j,r_j}(r_j \cdot))(x_0)  &=  \int_{\tilde{\Omega}_j} (P_{k_j,r_j}(r_j y) - P_{k_j,r_j}(r_j x_0)) D^{(i)}_x |x_0-y|^{-n-2s} \d y \\
&\quad - 2\1_{\{i=2\}} D_x W_j(x_0) D_x P_{k_j,r_j}(r_j \cdot)(x_0).
\end{align*}
Then, using that by \eqref{eq:Phi-Cgamma-est-2}, it holds
\begin{align*}
[P_{k_j,r_j}(r_j  \cdot)]_{C^{\gamma}(B_{2R})} \le C |b_{k_j,r_j}| r_j ,
\end{align*}
and that $|D_x W_j(x_0)| \le C \rho_j^{-2s-1}$, we can derive first for $i=1$ and then for $i=2$ that
\begin{align}
\label{eq:der-proj}
|D_x^{(i)} (P_{k_j,r_j}(r_j \cdot))(x_0)|\le C |b_{k_j,r_j}| r_j \rho_j^{\gamma-i},
\end{align}
Analogously, we get
\begin{align*}
(c W_j(x) - \tilde{b}_j(x)) v_j(x) = c \int_{\tilde{\Omega}_j} v_j(y) |x-y|^{-n-2s} \d y + \tilde{a}_j(x)
\end{align*}
and  we get for $i \in \{1,2\}$ and $x_0 \in B_R \setminus \tilde{\Omega}_j$,
\begin{align*}
(c W_j(x_0) - \tilde{b}_j(x_0))  |D^{(i)} v_j(x_0)| &\le \Big|\int_{\tilde{\Omega}_j} (v_j(y) - v_j(x_0)) D^{(i)}_x |x_0-y|^{-n-2s} \d y\Big| \\
&\quad + |D_x^{(i)} \tilde{a}_j(x_0)| + |D_x^{(i)} \tilde{b}_j(x_0)| |v_j(x_0)| \\
&\quad + C \1_{\{i=2\}} \big( | D_x \tilde{b}_j(x_0) D_x v_j(x_0) | + |D_x W_j(x_0) D_x v_j(x_0)| \big).
\end{align*}
Now, note that by \eqref{eq:der-proj} and by the definition of $\tilde{a}_j$ and $\tilde{b}_j$, it holds
\begin{align*}
|D^{(i)}_x \tilde{a}_j(x_0)| \le C (R)\frac{r_j^{i+2s-\beta-\gamma}}{\theta(r_j)} + C(R) \frac{r_j^{1+2s-\beta-\gamma} |b_{k_j,r_j}| }{\theta(r_j)} \rho_j^{\gamma-i} \le C(R) \rho_j^{\gamma-i}, \qquad |D_x^{(i)} \tilde{b}_j(x_0)| \le C r_j^{2s+i}.
\end{align*}
Moreover, we have by \eqref{eq:vj-growth-2} and interior estimates.
\begin{align*}
\Big|\int_{\tilde{\Omega}_j} (v_j(y) - v_j(x_0)) D^{(i)}_x |x_0-y|^{-n-2s} \d y\Big| \le C(R) \rho_j^{\gamma-2s-i}.
\end{align*}
Since $\tilde{b}_j \to 0$, we can estimate $|(c W_j(x_0) - \tilde{b}_j(x_0))| \ge c \rho_j^{-2s}$ for large enough $j$, which yields 
\begin{align*}
|D_x v_j(x_0)| \le C(R) \rho_j^{2s} \Big( \rho_j^{\gamma-2s-1} + \rho_j^{\gamma-1} + r_j^{2s+1} \Big) \le C(R) \rho_j^{\gamma-1}, \qquad |D_x^{(2)} v_j(x_0)| \le  C(R) \rho_j^{\gamma-2}.
\end{align*}

From here, the desired result \eqref{eq:extended-vj-est} follows easily by scaling.

\textbf{Step 3:} In this section, we send~$j \to \infty$ and deduce an equation for the limit.
Clearly, by Arzel\`a-Ascoli's theorem, \eqref{eq:vj-growth-2} implies that there exists $v \in C^{\gamma}_{loc}(\R^n) \cap C^{2s + \delta}_{loc}(\{ x_n > 0 \})$ such that $v_j \to v$ locally uniformly and in $C^{\gamma'}(B_{1/2})$ for any $\gamma' < \gamma$, and from \eqref{eq:vj-properties-2}, \eqref{eq:vj-growth-2}, and \eqref{eq:vj-2spluseps}, we see that for some $\delta > 0$
\begin{align}
\label{eq:v-properties-2}
v(0) = 0, \qquad 
\qquad \int_{B_1} v(x) (b \cdot x) \d x &= 0 ~~ \forall b \in \{ b_n = 0 \},\\
\label{eq:v-growth-2}
\forall R \ge 1 : \qquad \Vert v \Vert_{L^{\infty}(  B_R ) } \le C R^{\beta + \gamma}, \qquad [ v ]_{C^{\gamma}(B_R)} &\le C R^{\beta},\\
\label{eq:v-reg-2}
\forall \rho \le 1 : \qquad [v]_{C^{2s+\delta}(B_R \cap \{ x_n \ge \rho \})} &\le C(R) \rho^{\delta - 1},
\end{align}

We stress that, to prove the third property in \eqref{eq:v-properties-2}, we used that all function $P_b$ with respect to $\{ x_n > 0 \}$ are of the form $b \cdot x$ for some $b \in \{ b_n =  0\}$.

Moreover, by \eqref{eq:vj-growth-2}, \eqref{eq:f-Holder-to-zero}, and \eqref{eq:vj-2spluseps},
we can pass the equation in \eqref{eq:vj-Neumann-PDE} to the limit as $j \to \infty$ and conclude that
\begin{align}
\label{eq:v-Neumann-PDE}
\begin{cases}
(-\Delta)^s v &\overset{1}{=} 0 ~~ \text{ in } \{ x_n > 0 \},\\
\cN_{\{ x_n > 0 \}}^s v &\overset{1}{=} 0 ~~ \text{ in } \{ x_n < 0 \}.
\end{cases}
\end{align}
We remark that the first condition passes to the limit by application of \cite[Lemma 3.5]{AbRo20}. The Neumann condition can be passed to the limit by mimicking the proof of \cite[Lemma 3.5]{AbRo20} 
(note that it is even easier since $\cN^s_{\{ x_n > 0 \}}$ is not singular) in a straightforward fashion, using \eqref{eq:vj-growth-2} and writing
\begin{align*}
\cN^s_{\tilde{\Omega}_j} v_j &= \tilde{a}_j + v_j \tilde{b_j} = d_j + (\tilde{a}_j - d_j) + v_j \tilde{b_j} ~~~~ \text{ in } B_{r_j^{-1}/2} \setminus \tilde{\Omega}_j,
\end{align*}
where $(\tilde{a}_j - d_j) \to 0$ by \eqref{eq:aj-dj} and $v_j \tilde{b_j} \to 0$ locally uniformly by \eqref{eq:bj}.

\textbf{Step 4:} The goal of this step is to prove that
\begin{align}
\label{eq:v-1D-growth}
v(x) = W(x_n) + b \cdot x
\end{align}
for some function $W : \R \to \R$ and $b \in \R^n$ with $b_n = 0$. To prove it we take increments of $v^{(h)}(x) = v(x+h) - v(x)$ in tangential directions $h \in B_{1/8}$ with $h_n = 0$. Clearly, by \eqref{eq:v-reg-2}, \eqref{eq:v-growth-2}, and H\"older interpolation we have for any $R \ge 1$ and $\rho \le 1$, since $\gamma \le 1$,
\begin{align}
\label{eq:incr-quot-est-1}
[v^{(h)}]_{C^{\gamma}(  B_R )} &\le C R^{\beta}, \qquad  [v^{(h)}]_{C^{2s+\delta}(B_R \cap \{ x_n \ge \rho \})} \le C(R) \rho^{\delta - 1}, \\
\label{eq:incr-quot-est-2}
 \Vert v^{(h)} \Vert_{L^{\infty} ( B_R)} &\le C |h|^{\gamma} [ v ]_{C^{\gamma}( B_R)} \le C R^{\beta}.
\end{align}

Now we recall \eqref{eq:v-growth-2}  (using that $\beta + \gamma < B_0 < 2s + 1$ due to \eqref{eq:B0-trivial-bound}) and \eqref{eq:incr-quot-est-2} (using that $\beta < 2s$).
This allows us to apply \autoref{lemma:incr-sol-Neumann} and obtain that $v^{(h)}$ is a solution to
\begin{align}
\cN_{\{ x_n > 0 \}}^s v^{(h)} &= 0 ~~ \text{ in } \{ x_n < 0 \}.
\end{align}

Additionally, by \eqref{eq:incr-quot-est-1}, \eqref{eq:incr-quot-est-2}, and \eqref{eq:v-growth-2}, we can apply \autoref{lemma:incr-sol} and also deduce that 
\begin{align}
(-\Delta)^s v^{(h)} &= 0 ~~ \text{ in } \{ x_n > 0 \}
\end{align}
in the classical sense. 

In particular, we have that $v^{(h)}$ is a strong solution to
\begin{align}
\label{eq:vh-equation-tangential-regional}
L_{\{ x_n > 0 \}} v^{(h)} = 0 ~~ \text{ in } \{ x_n > 0 \}.
\end{align}
Using \eqref{eq:incr-quot-est-1} and \eqref{eq:incr-quot-est-2}, we can apply \autoref{lemma:strong-weak} and get that $v^{(h)}$ is also a weak solution to \eqref{eq:vh-equation-tangential-regional} with Neumann boundary condition on $\{ x_n = 0 \}$ in the sense of \autoref{def:weak-sol}. 

Also, since $v^{(h)}$ grows slower than $t^{2s-\eps}$, we can apply the Liouville theorem in \autoref{thm:Neumann-Liouville} to see that there exist $a^{(h)} \in \R$ and $b^{(h)} \in \R^n$ with $b^{(h)}_n = 0$ such that
\begin{align*}
v^{(h)}(x) = a^{(h)} + b^{(h)} \cdot x ~~ \text{ in } \{ x_n > 0 \}.
\end{align*}
Since $v^{(h)}$ also grows sub-linearly, as a consequence of $\beta < 1$, we deduce that $b^{(h)} \equiv 0$. Since $h \in B_1$ with $h_n = 0$ was arbitrary, this implies \eqref{eq:v-1D-growth}, as claimed.

\textbf{Step 5:} In this step, we will show that $v=0$ by applying \autoref{thm:Neumann-Liouville-growth} to $W(x) := v(x) - b \cdot x$. It follows from \eqref{eq:v-1D-growth} that $W(x) = W(x_n)$, i.e. $W$ is one-dimensional. 

Moreover, since $x \mapsto b \cdot x$ solves \eqref{eq:v-Neumann-PDE}, we find that $W$ solves (in the sense of \autoref{def:up-to-poly})
\begin{align*}
\begin{cases}
(-\Delta)^s W &\overset{1}{=} 0 ~~ \text{ in } \{ x_n > 0 \},\\
\cN_{\{ x_n  > 0 \}} W&\overset{1}{=} 0 ~~ \text{ in } \{ x_n < 0 \}.
\end{cases}
\end{align*}

Then, owing to \eqref{eq:v-growth-2}, \eqref{eq:v-reg-2}, and \eqref{eq:incr-quot-est-2}, we can apply \autoref{lemma:1D-reduction} and deduce that,
for any $\eta \in C_{c}^{\infty}(\R)$ with $\int_{-\infty}^0 \eta(x) \d x = \int_{\R} \eta(x) \d x = 0$,
\begin{align*}
\int_{0}^{\infty} (-\Delta)^s_{\R} \eta (x) W(x) \d x + \int_{-\infty}^{0} \cN^s_{(0,\infty)} \eta(x) W(x) \d x = 0.
\end{align*}
Hence, we conclude that $W - W(0)$ is a distributional solution with faster growth in the sense of \autoref{def:faster-growth}. Thus, due to \eqref{eq:v-growth-2}, we can apply \autoref{thm:Neumann-Liouville-growth} and deduce that $W - W(0) = 0$. Therefore, we have shown that $v(x) = a + b \cdot x$ for some $a \in \R$. Since $v(0) = 0$ by \eqref{eq:v-properties-2}, it must be $a = 0$. 

As a result, it follows from the third property in \eqref{eq:v-properties-2} that
\begin{align*}
0 = \int_{B_1} v(x) (b \cdot x) \d x = \int_{B_1} (b \cdot x)^2 \d x ,
\end{align*}
which implies that $b = 0$. Consequently, $v \equiv 0$ in $B_1$.

\textbf{Step 6:} In this step, we conclude the proof by arguing that $v \equiv 0$ leads to a contradiction. Recall that by Step 5,  \eqref{eq:vj-properties-2}, and \eqref{eq:vj-growth-2}, we know that
\begin{align*}
v_j \to 0 ~~\text{loc. uniformly}, \qquad [v_j]_{C^{\gamma}(B_1)} \ge \frac{1}{2}.
\end{align*}
In particular, we can find $x_j,y_j \in B_1$ such that
\begin{align}
\label{eq:hj-prop}
|v_j(x_j) - v_j(y_j)| \ge \frac{1}{4} |x_j - y_j|^{\gamma}, \qquad \kappa_j:= |x_j - y_j| \to 0,
\end{align}
where the latter property follows from the fact that $v = 0$.
Next, we claim that after passing to a subsequence, it holds for some constant $C > 0$
\begin{align}
\label{eq:dist-hj}
d_{\tilde{\Omega}_j}(x_j) \le C \kappa_j.
\end{align}
In fact, if \eqref{eq:dist-hj} fails, then we can extract a subsequence such that $d_{\tilde{\Omega}_j}(x_j) \kappa_j^{-1} \to \infty$, which means that $\kappa_j/\ell_j \to 0$, where $\ell_j = \min \{d_{\tilde{\Omega}_j}(x_j)/8,1/8 \}$. In other words, it must be $y_j \in B_{\ell_j}(x_j)$ for large enough $j$. But then, by \eqref{eq:extended-vj-est}, we conclude
\begin{align*}
\frac{|v_j(x_j) - v_j(y_j)|}{\kappa_j^{\gamma}} \le [v_j]_{C^{\tau}(B_{\ell_j}(x_j))} \kappa_j^{\tau-\gamma} \le C \left(\frac{ \kappa_j }{\ell_j} \right)^{\tau - \gamma} \to 0,
\end{align*}
since $\tau > \gamma$, which contradicts the first claim in \eqref{eq:hj-prop}. Hence, it must be \eqref{eq:dist-hj}.

Hence, we can choose the nearest boundary point $z_j \in \partial \tilde{\Omega}_j$ to $x_j$ and, up to a rotation, we can assume that the normal of $\tilde{\Omega}_j$ at $z_j$ is $e_n$. Hence, defining 
\begin{align*}
E_j := \kappa_j^{-1} (\tilde{\Omega}_j - z_j) , \qquad w_j(x) = \frac{v_j(z_j + \kappa_j x) - v_j(z_j)}{\kappa_j^{\gamma}},
\end{align*}
i.e. considering a second blow-up of $v_j$ at $z_j$ with respect to $\kappa_j$, we conclude that $E_j \to \{ x_n > 0\}$, as well as for any $R \ge 1$ and $j$ large enough,
\begin{align}
\label{eq:wj-prop}
w_j(0) = 0, \qquad [w_j]_{C^{\gamma}(B_R)} \le C, \qquad \forall \rho \le 1 : \qquad [w_j]_{C^{2s+\delta}(B_R \cap \{ x_n \ge \rho \})} &\le C(R) \rho^{\delta - 1},
\end{align}
for some $C > 0$, independent of $R$. The latter property is due to the fact that $z_j + \kappa_j x \in B_2$ for $j$ large enough, since $\kappa_j \to 0$. Hence, we can extract another subsequence such that
\begin{align}
\label{eq:w-prop}
\begin{split}
w_j \to w ~~ \text{ loc. unif.}, \qquad [w]_{C^{\gamma}(\R^n)} \le C, \qquad w(0) = 0, \qquad |w(x)| &\le C |x|^{\gamma}, \\
 \forall \rho \le 1 : \qquad [w]_{C^{2s+\delta}(B_R \cap \{ x_n \ge \rho \})} &\le C(R) \rho^{\delta - 1}.
 \end{split}
\end{align}
Next, let us define
\begin{align*}
X_j = \kappa_j^{-1}(x_j - z_j), \qquad Y_j = \kappa_j^{-1}(y_j - z_j),
\end{align*}
and observe that by \eqref{eq:dist-hj} we have
\begin{align*}
|X_j| \le d_{\tilde{\Omega}_j}(x_j) \kappa_j^{-1} \le C \qquad |X_j - Y_j| = \kappa_j^{-1} |x_j - y_j| = 1, \qquad |Y_j| \le |X_j - Y_j| + |X_j| \le 1 + C,
\end{align*}
as well as by \eqref{eq:hj-prop}
\begin{align*}
|w_j(X_j) - w_j(Y_j)| = \kappa_j^{-\gamma} |v_j(x_j) - v_j(y_j)| \ge \frac{1}{4}.
\end{align*}
Hence, we can extract subsequences such that $X_j \to X$, $Y_j \to Y$ such that 
\begin{align}
\label{eq:w-prop-2}
|X-Y| = 1, \qquad |w(X) - w(Y)| \ge \frac{1}{4}.
\end{align}
Moreover, note that  since $w_j$ is just a translation an rescaling of $v_j$, we have that
\begin{align*}
\begin{cases}
(-\Delta)^s w_j &= \tilde{g}_j \qquad\qquad ~~~ \text{ in } E_j \cap B_{r_j^{-1}/2},\\
\cN^s_{E_j} w_j &= A_j + B_j w_j ~~~~ \text{ in } B_{r_j^{-1}/2} \setminus E_j,
\end{cases}
\end{align*}
where
\begin{align*}
\tilde{g}_j(x) &= \kappa_j^{2s-\gamma} \tilde{h}_j(z_j + \kappa_j x) , \\
A_j(x) &= \kappa_j^{2s-\gamma}  \big[ \tilde{a}_j(z_j + \kappa_j x) + \tilde{b}_j(z_j + \kappa_j x) v_j(z_j) \big], \\
B_j(x) &= \kappa_j^{2s} \tilde{b}_j(z_j + \kappa_j x)
\end{align*}
satisfies by \eqref{eq:f-Holder-to-zero}
\begin{align*}
[\tilde{g}_j]_{C^{\beta + \gamma - 2s}(E_j \cap B_R \cap \{ d_{E_j} \ge \rho \})} &\le \theta(r_j)^{-1} \kappa_j^{\beta} + C(R) \kappa_j^{\beta} \frac{|b_{k_j,r_j}|}{\theta(r_j)} \Big(1 + (\kappa_j \rho)^{2s-\beta-\gamma} \Big) \\
&\le C \theta(r_j)^{-1} \kappa_j^{\beta} + C(R) (\kappa_j^{2s - \gamma} \rho^{2s-\beta-\gamma} + \kappa_j^{\beta} ) \frac{|b_{k_j,r_j}|}{\theta(r_j)} \to 0,
\end{align*}
since $\gamma < 2s$, as well as by \eqref{eq:bj} and \eqref{eq:aj-dj},
\begin{align*}
B_j \to 0, \qquad A_j - \tilde{d}_j \to 0 ~~ \text{loc. uniformly, } \quad \text{where }~~ \tilde{d}_j = \kappa_j^{2s-\gamma}[\tilde{a}_j(z_j) + \tilde{b}_j(z_j) v_j(z_j)],
\end{align*}
using that $\gamma < 2s$ and
\begin{align*}
|A_j(x) - \tilde{d}_j| \le C \kappa_j^{2s-\gamma+1} \big( \Vert \tilde{a}_j \Vert_{C^{0,1}(B_R)} + \Vert \tilde{b}_j \Vert_{C^{0,1}(B_R)} |v_j(z_j)| \big)|x|.
\end{align*}
Hence, by \eqref{eq:vj-growth-2}, \eqref{eq:f-Holder-to-zero}, and \eqref{eq:vj-2spluseps}, and using the same argument as in Step 3, we can pass the equation for $w_j$ to the limit and conclude that
\begin{align*}
\begin{cases}
(-\Delta)^s w &\overset{1}{=} 0 ~~ \text{ in } \{ x_n > 0 \},\\
\cN_{\{ x_n > 0 \}}^s w &\overset{1}{=} 0 ~~ \text{ in } \{ x_n < 0 \}.
\end{cases}
\end{align*}
Moreover, by growth of $w$ in the third property in \eqref{eq:w-prop}, $w$ is a classical solution and there must be $c_\pm \in \R$ such that
\begin{align}
\label{eq:limiting-eq-w}
\begin{cases}
(-\Delta)^s w &= c_+ ~~ \text{ in } \{ x_n > 0 \},\\
\cN_{\{ x_n > 0 \}}^s w &= c_- ~~ \text{ in } \{ x_n < 0 \}.
\end{cases}
\end{align}
It is easy to see that the $C^{\gamma}$ regularity of $w$ from \eqref{eq:w-prop} yields for any $t > 0$
\begin{align*}
|c_-| = |\cN_{\{ x_n > 0 \}}^s w(-t e_n)| \le C \int_{\{ y_n > 0 \}} \frac{|w(-t e_n) - w(y)|}{|x + t e_n|^{n+2s}} \d y \le C \int_{\{ y_n > 0 \}} |y + t e_n|^{-n-2s+\gamma} \d y \le C t^{\gamma - 2s}.
\end{align*}
Upon sending $t \to \infty$, this implies $c_- = 0$ since $\gamma < 2s$. Moreover, given a function $\phi \in C^{\infty}_c(\{ x_n > 0\})$ we have for $\phi_R(x) =  R^{-n} \phi(x/R)$,
\begin{align*}
|c_+| = \Big| \int_{\R^n} w(x) (-\Delta)^s  \phi_R(x) \d x \Big| &\le R^{-2s} \int_{\R^n} |w(R y)| |(-\Delta)^s \phi(y)| \d y \\
&\le C R^{\gamma - 2s} \int_{\R^n} |y|^{\gamma} |(-\Delta)^s \phi(y)| \d y \le C \Vert \phi \Vert_{C^{2s+\delta}(\R^n)} R^{\gamma -2s},
\end{align*}
where we used the growth of $w$ from \eqref{eq:w-prop}. Taking the limit $R \to \infty$, we conclude that $c_+ = 0$ since $\gamma < 2s$. Moreover, we can apply \cite[Lemma A.1]{AFR23} and \autoref{lemma:strong-weak} to see that $w$ is a weak solution to \eqref{eq:limiting-eq-w} with $c_\pm = 0$, and hence, the Liouville theorem from \autoref{thm:Neumann-Liouville-nd} allows to conclude that $w \equiv a + b \cdot x$ for some $a \in \R$ and $b \in \R^n$ with $b_n = 0$. Since $w(0) = 0$, we conclude that $a = 0$. Moreover, note that $b = 0$ since $\gamma < 1$. Hence, $w \equiv 0$.

This contradicts the second property in \eqref{eq:w-prop-2}. In this way, we obtain \eqref{eq:blowup-claim-growth}, and the proof is complete.
\end{proof}

Finally, we state the expansion in unbounded solution domains:

\begin{corollary}
\label{cor:bdry-expansion-higher}
The statement of \autoref{thm:bdry-expansion-higher} remains true if $\Omega \subset \R^n$ is an unbounded open set with $\partial \Omega \in C^{1,1}$. The function $\Phi$ in the expansion is given by choosing $\Omega' \subset B_2$ with $C^{1,1}$ boundary such that $\Omega \cap B_{3/2} = \Omega' \cap B_{3/2}$ and choosing $\Phi$ with respect to $\Omega'$ as in \autoref{subsec:correction}.
\end{corollary}

\begin{proof}
First, we choose $\Omega' \subset \R^n$ with $\Omega' \subset B_2$ and with $C^{1,1}$ boundary such that $\Omega \cap B_{3/2} = \Omega' \cap B_{3/2}$. Then, by \autoref{lemma:unbd-error-2}, we find that $u$ is a weak solution to
\begin{align*}
\begin{cases}
(-\Delta)^s u &= h ~~\qquad \text{ in } \Omega' \cap B_{1/2},\\
\cN_{\Omega'}^s u &= a + b u ~~ \text{ in } B_{1/2} \setminus \Omega',
\end{cases}
\end{align*}
where
\begin{align*}
\Vert a \Vert_{C^2(B_{1/2} \setminus \Omega')}  \le C \Vert u \Vert_{L^\infty(\Omega)}, \quad \Vert b \Vert_{C^2(B_{1/2} \setminus \Omega')} \le C.
\end{align*}
Hence, the desired result follows by application of \autoref{thm:bdry-expansion-higher} with $\Omega'$.
\end{proof}

\subsection{Regularity of order larger than $2s$}
\label{subsec:reg-bigger-2s}

In this subsection, we prove the optimal regularity of solutions to the nonlocal Neumann problem in case $B_0 > 2s$.

\begin{theorem}
\label{thm:main-Neumann-higher}
Let $s \in (0,1)$, $\eps > 0$ be such that $B_0 - \eps \not= 1$ and $B_0 - \eps > 2s$. Let $\Omega \subset \R^n$ be an open set with $\partial \Omega \in C^{1,1}$.
Let $u \in H_K(\Omega)$ be a weak solution to
\begin{align*}
\begin{cases}
(-\Delta)^s u &= h ~~ \text{ in } \Omega \cap B_1,\\
\mathcal{N}_{\Omega}^s u &= 0 ~~ \text{ in } B_1 \setminus \Omega,
\end{cases}
\end{align*}
where $h \in C^{B_0 - \eps - 2s}(\Omega)$. 

Then, there exists $C > 0$, depending only on $n,s,\eps$, and $\Omega$, such that 
\begin{align*}
\Vert u \Vert_{C^{B_0-\eps}(\overline{\Omega} \cap B_{1/2})} \le C \left( \Vert u \Vert_{L^{\infty}(\R^n)} + \Vert h \Vert_{C^{B_0 - \eps - 2s}(\Omega \cap B_1)} \right).
\end{align*}
\end{theorem}

\begin{proof}[Proof of \autoref{thm:main-Neumann-higher}]
Up to a normalization, we can suppose that
\begin{align}
\label{eq:normalization-thm-2}
\Vert u \Vert_{L^{\infty}(\R^n)} + \Vert h \Vert_{C^{B_0 - \eps - 2s}(\Omega \cap B_1)} \le 1.
\end{align}

We also assume that $0 \in \partial \Omega$, $\nu(0) = e_n$, and that there exists a $C^{1,1}$ diffeomorphism $\Phi$ as in \autoref{subsec:correction} satisfying \eqref{eq:Pb-integral-vanish} and \eqref{eq:Pb-normal-zero} in $\Omega^c$. 

We claim that, for any $x_0 \in \Omega_{1/8} := \Omega \cap B_{1/8}$ with $|x_0| = d_{\Omega}(x_0) = :r$, there exists $b \in \{ b_n = 0 \}$ such that
\begin{align}
\label{eq:claim-thm-2}
[u]_{C^{\gamma}(B_{1/2})} \le C_1 \qquad \Rightarrow \qquad [u
 - u(0) - \1_{\{\beta+\gamma > 1 \}}P_b]_{C^{\beta + \gamma}(B_{r/2}(x_0))} \le C( 1+ C_1),
\end{align}
for any $\beta \in (0,\min\{2s,1\})$ and $\gamma \in (s, 1)$ such that $\beta + \gamma \le B_0 - \eps$ and $\beta + \gamma \not= 1$, with $\gamma \le 2s$ in case $\beta + \gamma >1$. 

Here, $C > 0$ depends only on $n,s,\eps,\beta,\gamma$, and the $C^{1,1}$ norm of $\Omega$, and $P_b$ is as in \autoref{subsec:correction}.  

We also claim that, if $x_0 \in B_{1/8} \setminus \Omega$ with $|x_0| = d_{\Omega}(x_0) =:r$, and also $\beta + \gamma \le 1$, then
\begin{align}
\label{eq:claim-thm-2-outside}
[u]_{C^{\gamma}(B_{1/2})} \le C_1 \qquad \Rightarrow \qquad [u
 - u(0)]_{C^{\beta + \gamma}(B_{r/2}(x_0))} \le C( 1+ C_1).
\end{align}

Let us suppose for a moment that the above claims hold true.
In this case, denoting $\overline{\Omega^{\kappa}} := \{x \in \R^n : \dist(x,\Omega) \le \kappa \}$, we deduce that 
\begin{align}
\label{eq:claim-thm-beta-gamma}
[u]_{C^{\gamma}(\overline{\Omega^{1/2}} \cap B_{3/4})} \le C_1 \qquad \Rightarrow \qquad \begin{cases} [u]_{C^{\beta + \gamma}(\overline{\Omega} \cap B_{1/2})} &\le C( 1+ C_1), \\
[u]_{C^{\beta + \gamma}(\overline{\Omega^{1/8}} \cap B_{1/2})} &\le C( 1+ C_1), ~~ \text{ if } \beta + \gamma \le 1
\end{cases}
\end{align}
by the exact same arguments as in the proof of \autoref{thm:main-Neumann}.
 
Let us first explain how \eqref{eq:claim-thm-beta-gamma} implies the desired result. Clearly, in case $s \in [\frac{1}{2},1)$, since $B_0 < s + 1$ and since $B_0 - \eps > 2s$, we can choose $\beta = s - \bar{\eps} \in (0,\min\{1,2s\})$ for some $\bar{\eps} \in (\eps,s)$, and  $\gamma = B_0 - \eps - \beta \in (s,1)$ such that $\beta + \gamma = B_0 - \eps$. 

Therefore, since $[u]_{C^{\gamma}(\overline{\Omega} \cap B_{7/8})} \le C$ by \autoref{thm:main-Neumann} and $\gamma < \min\{1,2s\} = 1$, we can use the same argument as in the proof of \cite[Lemma 6.3]{AFR23} to obtain 
\begin{align*}
|u(x) - u(z)| \le C |x-z|^{\gamma} \quad \forall x \in B_{3/4} \setminus \Omega,
\end{align*}
with $z \in \partial \Omega$ being such that $|x-z| = d_{\Omega}(x)$.
Then, one finds that \autoref{lemma:outside-bound} is applicable, and accordingly $[u]_{C^{\gamma}(\overline{\Omega^{1/2}} \cap B_{3/4})} \le C_1$. 

Thus, the implication
\eqref{eq:claim-thm-beta-gamma} yields the desired result in that case. Note that this choice of $\bar{\eps}$ is always possible since $B_0 - \eps > 2s$ yields that $\eps < B_0 - 2s < 1 - s \le s$.

In case $s \in (\frac{1}{4},\frac{1}{2})$, we observe that $B_0 < 2s + \frac{1}{2} < 4s$. Hence, we can take $\gamma = 2s - \bar{\eps}$ for some $\bar{\eps} \in (0,\eps)$ and $\beta = B_0 - \eps - \gamma < 4s - \eps - 2s + \bar{\eps} < 2s - \eps + \bar{\eps} < 2s \le \min\{2s , 1\}$. Then, we clearly have $\gamma \in (s,\min\{1,2s\})$. Moreover, by \autoref{thm:main-Neumann}, we have $[u]_{C^{\gamma}(\overline{\Omega} \cap B_{3/4})} \le C_1$. Since $\gamma < 2s$, \cite[Lemma 6.3]{AFR23} verifies the assumptions of \autoref{lemma:outside-bound}, which gives that~$[u]_{C^{\gamma}(\overline{\Omega_{1/2}})} \le C_1$. Thus, \eqref{eq:claim-thm-beta-gamma} implies the desired result, since $\beta + \gamma = B_0 - \eps$.

Finally, in case $s \in (0,\frac{1}{4}]$, we need to iterate \eqref{eq:claim-thm-beta-gamma} several times. Let us observe that by \autoref{prop:zeros-final}, it holds $B_0 - \eps < 2s + \frac{1}{2} < 1$, which is why we will never apply \eqref{eq:claim-thm-beta-gamma} with $\beta+\gamma > 1$, $\gamma \ge 1$, or $\beta > 1$, and we are therefore allowed to choose $\gamma \ge 2s$.

Hence, we define $k_0 := \min \{ k \in \N : (2s - \bar{\eps})k \le B_0 - \eps \}$ for some $\bar{\eps} \in (0,s)$ to be chosen later, and set for some small $\delta > 0$,
\begin{align*}
\gamma_i := \gamma_{i-1} + \beta_{i-1}, \qquad \beta_i := 2s - \bar{\eps}, \qquad \gamma_0 := 2s - \bar{\eps}, \qquad i \in \{1,\dots,k_0\}.
\end{align*}
Note that since $\gamma_0 < 2s$, by the same arguments as in the previous case we have $[u]_{C^{\gamma_0}(\overline{\Omega_{1/2}})} \le C$. Hence, we can iterate $k_0-1$ times
the implication \eqref{eq:claim-thm-beta-gamma} and, after a scaling argument, infer that 
\begin{align*}
[u]_{C^{(2s-\bar{\eps})k_0}(\overline{\Omega^{1/2}} \cap B_{1/2})} \le C.
\end{align*}
Let us now choose $\bar{\eps}$ such that $(2s - \bar{\eps})k_0 = B_0 - \eps$, concluding the proof also in this case. 

It remains to prove \eqref{eq:claim-thm-2} and \eqref{eq:claim-thm-2-outside}, which we will explain now in detail. 
First, by assumption, we have that 
\begin{align*}
[u]_{C^{\gamma}(B_{1/2})} \le C.
\end{align*}
We also remark that the function
$v(x) = u(x) - u(0) - \1_{\{\beta + \gamma > 1\}} P_b(x)$
satisfies $v(0) = 0$.
Hence, recalling that $\gamma < 1$,
we are in a position to apply \autoref{thm:bdry-expansion-higher} and deduce that, for any $\rho \in (0,\frac{1}{4})$, 
\begin{align}
\label{eq:v-expansion-main-proof}
\Vert v \Vert_{L^{\infty}(B_{\rho})} \le \rho^{\gamma} [v]_{C^{\gamma}(B_{\rho})} \le C \rho^{\gamma} \rho^{\beta} \le C \rho^{\beta + \gamma}.
\end{align}

We stress that \eqref{eq:v-expansion-main-proof} allows us to apply \autoref{lemma:outside-bound} in case $\beta + \gamma \le 1$ and deduce the second claim \eqref{eq:claim-thm-2-outside}.

To prove the claim in~\eqref{eq:claim-thm-2}, we observe that the estimate \eqref{eq:v-expansion-main-proof} remains true for $\rho \ge \frac{1}{4}$ due to \eqref{eq:normalization-thm-2}, and therefore, since $|x-x_0| \ge \frac{2r}{3}$ in $B_{2r} \setminus B_{2r/3}(x_0)$ and $|x-x_0| \ge \frac{|x|}{3}$ in $\R^n \setminus B_{2r}$, we have that
\begin{align}
\label{eq:growth-est-thm-2}
\Vert |\cdot -x_0|^{-B_0 + \eps} v \Vert_{L^{\infty}(\R^n \setminus B_{2r/3}(x_0))} \le C.
\end{align}

Next, let us observe that
\begin{align*}
\begin{cases}
(-\Delta)^s v &= h - \1_{\{\beta + \gamma > 1\} }L_{\Omega} P_b =: \tilde{f} ~~ \text{ in } \Omega \cap B_1,\\
\cN^s_{\Omega} v &= 0 \qquad\qquad\qquad\qquad\qquad ~  \text{ in } B_1 \setminus \Omega.
\end{cases}
\end{align*}
Moreover, by \autoref{lemma:LPb-estimate-2}, 
\begin{align}
\label{eq:f-est-thm-2}
[\tilde{h}]_{C^{B_0 - \eps - 2s}(B_{r/2}(x_0))} \le 1 + \1_{\{\beta + \gamma > 1\}}  \left( C r^{1-s-(B_0 - \eps - 2s)} + C |x_0|^{1-s} r^{-(B_0 - \eps - 2s)} \right) \le C.
\end{align}
Here, we used that $B_0 - \eps - 2s < 1 - s$ due to \autoref{prop:zeros-final} and that $|x_0| = r$ by construction, as well as $|b| \le C$, which follows from \autoref{thm:bdry-expansion-higher},  \eqref{eq:normalization-thm-2}, and the construction of the function $P_b$.

Thus, using the interior regularity estimate with growth from \cite[Proposition 3.9]{AbRo20}, 
it follows from \eqref{eq:growth-est-thm-2} and \eqref{eq:f-est-thm-2}
(in analogy to the proof of \eqref{eq:int-reg-appl}) that
\begin{align*}
[v]_{C^{B_0-\eps}(B_{r/2}(x_0))} &\le C r^{\eps-B_0}  \Vert v \Vert_{L^{\infty}(B_{2r/3}(x_0))} + C \Vert |\cdot -x_0|^{-B_0 + \eps} v \Vert_{L^{\infty}(\R^n \setminus B_{2r/3}(x_0))} \\
&\qquad + C r^{2s-B_0+\eps} [ \tilde{h} ]_{C^{B_0 - \eps - 2s}(B_{2r/3}(x_0))} \\
&\le C.
\end{align*}
This implies \eqref{eq:claim-thm-2} and allows us to complete the proof.
\end{proof}

\subsection{Proof of main regularity results}
\label{subsec:main-proofs}

In this subsection, we prove our main results. We distinguish between bounded and unbounded open sets $\Omega \subset \R^n$.

\begin{proof}[Proof of \autoref{thm0} and \autoref{thm:main-intro} if $\Omega$ is unbounded]
To prove the regularity estimate for $u$ in \autoref{thm:main-intro}, we combine \autoref{thm:main-Neumann} and \autoref{thm:main-Neumann-higher}. To deduce the regularity result in \autoref{thm0}, we use \autoref{prop:zeros-final}. Finally, the property $\partial_{\nu} u = 0$, whenever $B_0 - \eps > 1$, follows immediately from the expansions \autoref{thm:bdry-expansion} and \autoref{cor:bdry-expansion-higher}, using that $b \cdot \Phi^{-1}(0) = 0 = \partial_{\nu} (b \cdot \Phi^{-1})(0)$, by construction (see \autoref{subsec:correction}).
\end{proof}

\begin{proof}[Proof of \autoref{thm0} and \autoref{thm:main-intro} if $\Omega$ is bounded]
To prove the regularity estimate for $u$ in \autoref{thm:main-intro}, we combine \autoref{thm:main-Neumann} and \autoref{thm:main-Neumann-higher}. The property $\partial_{\nu} u = 0$, whenever $B_0 - \eps > 1$, follows as in the case of unbounded solution domains. Moreover with the help of \eqref{eq:bdness-outside}, \eqref{eq:v-Linfty-estimate}, and \eqref{eq:v-Hs-estimate}, we can estimate
\begin{align*}
\Vert u \Vert_{L^{\infty}(\R^n)} \le C \Vert u \Vert_{L^{\infty}(\Omega)} \le C \Vert u \Vert_{L^2(\Omega)} + C \Vert h \Vert_{\cX}. 
\end{align*}
Moreover, since $v := u- (u)_{\Omega}$ is again a solution, where $(u)_{\Omega} = \dashint_\Omega u$, we can  deduce
\begin{align*}
[v]_{H^s(\Omega)}^2 = B_{\Omega}(v,v) = \int_{\Omega} h v \le C \Vert h \Vert_{\cX} [v]_{H^s(\Omega)},
\end{align*}
where we used \autoref{lemma:fv-integral-estimate} or H\"older's inequality in the last inequality, depending on whether $B_0 - \eps > 2s$ or $B_0 - \eps < 2s$. Thus, using also the Poincar\'e inequality,
\begin{align*}
\Vert v \Vert_{L^2(\Omega)} \le C [v]_{H^s(\Omega)}  \le C \Vert h \Vert_{\cX}.
\end{align*}
Altogether, this shows the last claim of \autoref{thm:main-intro}.
To deduce the regularity result in \autoref{thm0}, we use \autoref{prop:zeros-final}. 
\end{proof}


\begin{thebibliography}{alpha}




\bibitem{AbRo20} N. Abatangelo, X. Ros-Oton, \emph{Obstacle problems for integro-differential operators: higher regularity of free boundaries}, Adv. Math. \textbf{360} (2020), 106931, 61.

\bibitem{AG} H. Abels, G. Grubb, \emph{Fractional-order operators on nonsmooth domains}, J. Lond. Math. Soc. \textbf{107} (2023), 1297-1350.

\bibitem{ADEJ20} N. Alibaud, F. Del Teso, J. Endal, E. Jakobsen, \emph{The Liouville theorem and  linear operators satisfying the maximum principle}, J. Math. Pures Appl. \textbf{142}  (2020), 229-242.


\bibitem{AFR23} A. Audrito, J.-C. Felipe-Navarro, X. Ros-Oton, \emph{The Neumann problem for the fractional Laplacian: regularity up to the boundary}, Ann. Sc. Norm. Super. Pisa Cl. Sci. (5) \textbf{24} (2023), 1155--1222.

\bibitem{BCI2} G. Barles, E. Chasseigne, C. Imbert, \emph{H\"older continuity of solutions of second-order elliptic integro-differential equations}, J. Eur. Math. Soc. \textbf{13} (2011), 1-26.


\bibitem{BKK08} K. Bogdan, T. Kulczycki, and M. Kwasnicki, \textit{Estimates and structure of $\alpha$-harmonic functions}, Probab. Theory Relat. Fields \textbf{140} (2008), 345-381.

\bibitem{BKK15} K. Bogdan, T. Kumagai, M. Kwasnicki, \textit{Boundary Harnack inequality for Markov processes with jumps}, Trans. Amer. Math. Soc. \textbf{367} (2015), 477-517.


\bibitem{CS0} L. Caffarelli, L. Silvestre, An extension problem related to the fractional Laplacian, Comm. Partial Differential Equations {\bf 32} (2007), no.~7-9, 1245--1260.

\bibitem{CS1} L. Caffarelli, L. Silvestre, \textit{Regularity theory for fully nonlinear integro-differential equations}, Comm. Pure Appl. Math. {\bf 62} (2009), no.~5, 597--638.

\bibitem{CS2} L. Caffarelli, L. Silvestre, \textit{Regularity results for nonlocal equations by approximation}, Arch. Ration. Mech. Anal. \textbf{200} (2011), 59-88.

\bibitem{CS3} L. Caffarelli, L. Silvestre, \textit{The Evans-Krylov theorem for nonlocal fully nonlinear equations}, Ann. of Math. (2) {\bf 174} (2011), no.~2, 1163--1187.

\bibitem{CKS} Z. Chen, P. Kim, R. Song, \emph{Heat kernel estimates for the Dirichlet fractional Laplacian}, J. Eur. Math. Soc. \textbf{12} (2010), 1307-1329.

\bibitem{ChSo03} Z.-Q. Chen, R. Song, \emph{Hardy inequality for censored stable processes}, Tohoku Math. J.  \textbf{55} (2003), 439--450.

\bibitem{DRSV} S. Dipierro, X. Ros-Oton, J. Serra, E. Valdinoci, \textit{Non-symmetric stable operators: regularity theory and integration by parts}, Adv. Math. \textbf{401} (2022), 108321, 100pag.

\bibitem{DRV} S. Dipierro, X. Ros-Oton, E. Valdinoci, \textit{Nonlocal problems with Neumann boundary conditions}, Rev. Mat. Iberoam. \textbf{33} (2017), 377-416.

\bibitem{DSV2} S. Dipierro, O. Savin, E. Valdinoci, \emph{Definition of fractional Laplacian for functions with polynomial growth}, Rev. Mat. Iberoam. \textbf{35} (2019), 1079-1122.

\bibitem{DGLZ} Q. Du, M. Gunzburger, R. B. Lehoucq, K. Zhou, \emph{Analysis and approximation of nonlocal diffusion problems with volume constraints}, SIAM Review \textbf{54} (2012), 667-696.

\bibitem{Dyd04} B. Dyda, \emph{A fractional order Hardy inequality}, Illinois J. Math. \textbf{48} (2004), 575--588.


\bibitem{FaRo22} M. Fall, X. Ros-Oton, \emph{Global Schauder theory for minimizers of the $H^s(\Omega)$ energy}, J. Funct. Anal. \textbf{283} (2022), Paper No. 109523, 50.

\bibitem{FW16} M. M. Fall, T. Weth, \emph{Liouville theorems for a general class of nonlocal operators}, Potential Anal. \textbf{45} (2016), 187-200.

\bibitem{FeRo24} X. Fern\'andez-Real, X. Ros-Oton, \emph{Integro-Differential Elliptic Equations}, Progress in Mathematics, Vol. \textbf{350} (2024), xvi+395.

\bibitem{GrRy07} I. S. Gradshteyn, I. M. Ryzhik, \emph{Table of Integrals, Series, and Products}, (2007), xlviii+1171.

\bibitem{Grubb2} G. Grubb, \textit{Local and nonlocal boundary conditions for $\mu$-transmission and fractional elliptic pseudodifferential operators}, Anal. PDE \textbf{7} (2014), 1649-1682.

\bibitem{Grubb} G. Grubb, \emph{Fractional Laplacians on domains, a development of H\"ormander's theory of $\mu$-transmission pseudodifferential operators}, Adv. Math. \textbf{268} (2015), 478-528.

\bibitem{Gru22} G. Grubb, \textit{Fourier methods for fractional-order operators}, Proceedings of the RIMS Symposium ``Harmonic Analysis and Nonlinear Partial Differential equations'', 2022.

\bibitem{GrHe24} F. Grube, T. Hensiek, \emph{Robust nonlocal trace spaces and Neumann problems}, Nonlinear Anal. \textbf{241} (2024), Paper No. 113481, 35.

\bibitem{Hil43} A. P. Hillman, H. E. Salzer, \emph{Roots of $\sin z=z$}, Philos. Mag. (7) \textbf{34} (1943), 575.

\bibitem{KD} M. Kassmann, B. Dyda, \emph{Function spaces and extension results for nonlocal Dirichlet problems}, J. Funct. Anal. \textbf{277} (2019), 108134.

\bibitem{KaWe22} M. Kassmann, M. Weidner, \emph{Nonlocal operators related to nonsymmetric forms I: H\"older estimates}, Math. Ann. \textbf{393} (2025), 2307-2389.

\bibitem{KiWe24} M. Kim, M. Weidner, \emph{Optimal boundary regularity and Green function estimates for nonlocal equations in divergence form}, J. Eur. Math. Soc. (2025), No.1760.

\bibitem{McL63} N. McLachlan, \emph{Complex variable theory and transform calculus with technical applications}, (1963), xi+388.

\bibitem{MiLa86} O. Misra, J. Lavoine, \emph{Transform Analysis of Generalized Functions}, North-Holland Mathematics Studies, Vol. \textbf{119} (1986), xiv+332.

\bibitem{PaKa01} R. Paris, D. Kaminski, \emph{Asymptotics and Mellin-Barnes Integrals}, Encyclopedia of Mathematics and its Applications, Vol. \textbf{85} (2001), xvi+422.

\bibitem{ReSi75} M. Reed, B. Simon, \emph{Methods of Modern Mathematical Physics. II. Fourier Analysis, Self-Adjointness}, (1975), xv+361.

\bibitem{RS-Cs} X. Ros-Oton, J. Serra, \textit{The Dirichlet problem for the fractional Laplacian: regularity up to the boundary}, J. Math. Pures Appl. \textbf{101} (2014), 275-302.

\bibitem{RS-Duke} X. Ros-Oton, J. Serra, \textit{Boundary regularity for fully nonlinear integro-differential equations}, Duke Math. J. \textbf{165} (2016), 2079-2154.

\bibitem{RoWe24} X. Ros-Oton, M. Weidner, \emph{Optimal regularity for nonlocal elliptic equations and free boundary problems}, Amer. J. Math. (2024), to appear.

\bibitem{RoWe25} X. Ros-Oton, M. Weidner, \emph{Optimal regularity for kinetic Fokker-Planck equations in domains}, arXiv:2505.11943 (2025).

\bibitem{Rud87} W. Rudin, \emph{Real and Complex Analysis}, (1987), xiv+416.

\bibitem{StSh03} E. Stein, R. Shakarchi, \emph{Complex Analysis}, Princeton Lectures in Analysis, Vol. \textbf{2} (2003), xviii+379.

\bibitem{WhWa62} E. T. Whittaker, G. N. Watson, \emph{A Course of Modern Analysis. An introduction to the general theory of infinite processes and of analytic functions: with an account of the principal transcendental functions}, (1962), vii+608.


\end{thebibliography}
\end{document}